\documentclass{article}
\usepackage[english]{babel}
\usepackage{amsmath,amssymb,stmaryrd,enumerate,bbm,latexsym,theorem}
\usepackage[top=2.5cm,bottom=2.5cm,right=2cm,left=2cm]{geometry}
\usepackage{authblk}

\numberwithin{equation}{section}

\catcode`\<=\active \def<{
\fontencoding{T1}\selectfont\symbol{60}\fontencoding{\encodingdefault}}
\catcode`\>=\active \def>{
\fontencoding{T1}\selectfont\symbol{62}\fontencoding{\encodingdefault}}
\newcommand{\assign}{:=}
\newcommand{\nin}{\not\in}
\newcommand{\nobracket}{}
\newcommand{\nosymbol}{}
\newcommand{\tmem}[1]{{\em #1\/}}
\newcommand{\tmmathbf}[1]{\ensuremath{\boldsymbol{#1}}}
\newcommand{\tmop}[1]{\ensuremath{\operatorname{#1}}}
\newcommand{\tmrsup}[1]{\textsuperscript{#1}}
\newcommand{\tmtextbf}[1]{{\bfseries{#1}}}
\newcommand{\tmtextit}[1]{{\itshape{#1}}}
\newcommand{\tmverbatim}[1]{{\ttfamily{#1}}}
\newenvironment{enumerateroman}{\begin{enumerate}[i.] }{\end{enumerate}}
\newenvironment{itemizeminus}{\begin{itemize} }{\end{itemize}}
\newenvironment{proof}{\noindent\textbf{Proof\ }}{\hspace*{\fill}$\Box$\medskip}
\newenvironment{tmindent}{\begin{tmparmod}{1.5em}{0pt}{0pt} }{\end{tmparmod}}
\newenvironment{tmparmod}[3]{\begin{list}{}{\setlength{\topsep}{0pt}\setlength{\leftmargin}{#1}\setlength{\rightmargin}{#2}\setlength{\parindent}{#3}\setlength{\listparindent}{\parindent}\setlength{\itemindent}{\parindent}\setlength{\parsep}{\parskip}} \item[]}{\end{list}}
\newtheorem{theorem}{Theorem}[section]
\newtheorem{corollary}[theorem]{Corollary}
\newtheorem{lemma}[theorem]{Lemma}
\newtheorem{proposition}[theorem]{Proposition}

\usepackage[hidelinks]{hyperref}

\date{}
\author{David Chiron}
\author{Eliot Pacherie}
\affil{Universit{\'e} C\^{o}te d'Azur, CNRS, LJAD, France}

\begin{document}

\title{Coercivity for travelling waves in the Gross-Pitaevskii equation in
$\mathbbm{R}^2$ for small speed}

\maketitle

\begin{abstract}
  In the previous paper {\cite{CP1}}, we constructed a smooth branch of
  travelling waves for the 2 dimensional Gross-Pitaevskii equation. Here, we
  continue the study of this branch. We show some coercivity results, and we
  deduce from them the kernel of the linearized operator, a spectral stability
  result, as well as a uniqueness result in the energy space. In particular,
  our result proves the non degeneracy of these travelling waves, which is a
  key step in the classification of these waves and for the construction of
  multi-travelling waves.
\end{abstract}

\section{Introduction and statement of the results}

We consider the Gross-Pitaevskii equation
\[ 0 = (\tmop{GP}) (\mathfrak{u}) \assign i \partial_t \mathfrak{u}+ \Delta
   \mathfrak{u}- (| \mathfrak{u} |^2 - 1) \mathfrak{u} \]
in dimension 2 for $\mathfrak{u}: \mathbbm{R}_t \times \mathbbm{R}_x^2
\rightarrow \mathbbm{C}$. The Gross-Pitaevskii equation is a physical model
for Bose-Einstein condensate {\cite{Ginz-Pit_58a}}, {\cite{Neu_90}}, and is
associated with the Ginzburg-Landau energy
\[ E (v) \assign \frac{1}{2} \int_{\mathbbm{R}^2} | \nabla v |^2 + \frac{1}{4}
   \int_{\mathbbm{R}^2} (1 - | v |^2)^2 . \]
The condition at infinity for $(\tmop{GP})$ will be
\[ | \mathfrak{u} | \rightarrow 1 \quad \tmop{as} \quad | x | \rightarrow +
   \infty . \]
The equation $(\tmop{GP})$ has some well known stationary solutions of
infinite energy called vortices, which are solutions of $(\tmop{GP})$ of
degrees $n \in \mathbbm{Z}^{\ast}$ (see {\cite{CX}}):
\[ V_n (x) = \rho_n (r) e^{i n \theta}, \]
where $x = r e^{i \theta}$, solving
\[ \left\{ \begin{array}{l}
     \Delta V_n - (| V_n |^2 - 1) V_n = 0\\
     | V_n | \rightarrow 1 \tmop{as} | x | \rightarrow \infty .
   \end{array} \right. \]
Amongst other properties, $V_1 \tmop{and} V_{- 1}$ have exactly one zero
($\rho_n (r) = 0$ only if $r = 0$), and we call it the center of the vortex.
Since the equation is invariant by translation, we can define a vortex by its
degree and its center (the only point where its value is zero).

\

We are interested here in travelling wave solutions of $(\tmop{GP})$:
\[ \mathfrak{u} (t, x) = v (x_1, x_2 + c t), \]
where $x = (x_1, x_2)$ and $c > 0$ is the speed of the travelling wave, which
moves along the direction $- \overrightarrow{e_2}$. The equation on $v$ is
then
\[ 0 = (\tmop{TW}_c) (v) \assign - i c \partial_{x_2} v - \Delta v - (1 - | v
   |^2) v. \]
In this paper, we use all along the following notations. We denote, for
functions $f, g \in L^2_{\tmop{loc}} (\mathbbm{R}^2, \mathbbm{C})$ such that
$\mathfrak{R}\mathfrak{e} (f \bar{g}) \in L^1 (\mathbbm{R}^2, \mathbbm{C})$,
the quantity
\[ \langle f, g \rangle \assign \int_{\mathbbm{R}^2} \mathfrak{R}\mathfrak{e}
   (f \bar{g}), \]
even if $f, g \nin L^2 (\mathbbm{R}^2, \mathbbm{C})$. We also use the notation
$B (x, r)$ to define the closed ball in $\mathbbm{R}^2$ of center $x \in
\mathbbm{R}^2$ and radius $r > 0$ for the Euclidean norm. We define between
two vectors $X = (X_1, X_2), Y = (Y_1, Y_2) \in \mathbbm{C}^2$ the quantity
\[ X.Y \assign X_1 Y_1 + X_2 Y_2 . \]

\subsection{Branch of travelling waves at small speed}

In the previous paper {\cite{CP1}}, we constructed solutions of
$(\tmop{TW}_c)$ for small values of $c$ as a perturbation of two
well-separated vortices (the distance between their centers is large when $c$
is small). We have shown the following result.

\begin{theorem}[{\cite{CP1}}, Theorem 1.1]
  \label{th1}There exists $c_0 > 0$ a small constant such that for any $0 < c
  \leqslant c_0$, there exists a solution of $(\tmop{TW}_c)$ of the form
  \[ Q_c = V_1 (. - d_c \overrightarrow{e_1}) V_{- 1} (. + d_c
     \overrightarrow{e_1}) + \Gamma_c, \]
  where $d_c = \frac{1 + o_{c \rightarrow 0} (1)}{c}$ is a continuous function
  of c. This solution has finite energy $(E (Q_c) < + \infty)$ and $Q_c
  \rightarrow 1$ at infinity.
  
  Furthermore, for all $+ \infty \geqslant p > 2$, there exists $c_0 (p) > 0$
  such that if $c \leqslant c_0 (p)$, for the norm
  \[ \| h \|_p \assign \| h \|_{L^p (\mathbbm{R}^2)} + \| \nabla h \|_{L^{p -
     1} (\mathbbm{R}^2)} \]
  and the space $X_p \assign \{ f \in L^p (\mathbbm{R}^2), \nabla f \in L^{p -
  1} (\mathbbm{R}^2) \}$, one has
  \[ \| \Gamma_c \|_p = o_{c \rightarrow 0} (1) . \]
  In addition,
  \[ c \mapsto Q_c - 1 \in C^1 (] 0, c_0 (p) [, X_p), \]
  with the estimate
  \[ \left\| \partial_c Q_c + \left( \frac{1 + o_{c \rightarrow 0} (1)}{c^2}
     \right) \partial_d (V_1 (. - d \overrightarrow{e_1}) V_{- 1} (. + d
     \overrightarrow{e_1}))_{| d = d_c \nobracket} \right\|_p = o_{c
     \rightarrow 0} \left( \frac{1}{c^2} \right) . \]
\end{theorem}

The main idea of the proof of Theorem \ref{th1} is to use perturbative methods
around a quasi-solution $V_1 (. - d_{} \overrightarrow{e_1}) V_{- 1} (. + d
\overrightarrow{e_1})$, get $\Gamma_c$ by a fixed point theorem and the value
of $d_c$ by the cancellation of a Lagrange multiplier. With an implicit
function theorem, we can show that this construction gives us a $C^1$ branch
with respect to the speed $c$. In {\cite{CP1}}, we showed additional and more
precise estimates on $Q_c$ and $\partial_c Q_c$ in some weighed $L^{\infty}$
norms that will be useful in the proof of the next results (they will be
recalled later on). Still in {\cite{CP1}}, we wrote the perturbation
$\Gamma_{c, d_c}$ to make the dependence on $c$ and $d_c$ clearer, but it is
no longer needed here, and we will only write $\Gamma_c$.

\

With this solution $Q_c$, we can construct travelling waves of any small
speed, i.e. solutions of
\[ (\tmop{TW}_{\vec{c}}) (v) \assign i \vec{c} . \nabla v - \Delta v - (1 - |
   v |^2) v \]
for any $\vec{c} \in \mathbbm{R}^2$ of small modulus. For $\vec{c} = | \vec{c}
| e^{i (\theta_{\vec{c}} - \pi / 2)} \in \mathbbm{R}^2$, $| \vec{c} |
\leqslant c_0$, we have that
\begin{equation}
  Q_{\vec{c}} \assign Q_{| \vec{c} |} \circ R_{- \theta_{\vec{c}}}
  \label{CP2rotQc}
\end{equation}
is a solution of $(\tmop{TW}_{\vec{c}})$, with $R_{\alpha}$ being the rotation
of angle $\alpha$ and $Q_{| \vec{c} |}$ defined in Theorem \ref{th1}.
Furthermore, the equation is invariant by translation and by changing the
phase. Thus, we have a family of solutions of (GP) depending on five real
parameters, $\vec{c} \in \mathbbm{R}^2$, $| \vec{c} | \leqslant c_0$, $X \in
\mathbbm{R}^2$ and $\gamma \in \mathbbm{R}$:
\[ Q_{\vec{c}} (. - X - \vec{c} t) e^{i \gamma} . \]
We remark that, for a vortex of degree $\pm 1$, the family of solutions has
three parameters (the two translations and the phase): $V_{\pm 1} (. - X) e^{i
\gamma}$ is solution of $(\tmop{GP})$ for $X \in \mathbbm{R}^2, \gamma \in
\mathbbm{R}$. In particular, between a travelling wave and the two vortices
that compose it, we lose a parameter (since the phase is global). This is one
of the difficulty that will appear when we study the stability of this branch.

\

First, we give additional results on this branch of travelling waves: we will
study the position of its zeros, its energy and momentum, as well as some
particular values appearing in the linearization. The (additive) linearized
operator around $Q_c$ is
\[ L_{Q_c} (\varphi) \assign - \Delta \varphi - i c \partial_{x_2} \varphi -
   (1 - | Q_c |^2) \varphi + 2\mathfrak{R}\mathfrak{e} (\overline{Q_c}
   \varphi) Q_c . \]
We want to define and use four particular directions for the linearized
operator around $Q_c$, which are
\[ \partial_{x_1} Q_c, \partial_{x_2} Q_c, \]
related to the translations (i.e. related to the parameter $X \in
\mathbbm{R}^2$ in the family of travelling waves), and
\[ \partial_c Q_c, \partial_{c^{\bot}} Q_c, \]
related to the variation of speed (i.e. related to the parameter $\vec{c} \in
\mathbbm{R}^2$), if we change respectively its modulus or its direction. The
functions $\partial_{x_1} Q_c, \partial_{x_2} Q_c$ and $\partial_c Q_c$ are
defined in Theorem \ref{th1}, and we will show that
\[ \partial_{c^{\bot}} Q_c (x) \assign \partial_{\alpha} (Q_c \circ R_{-
   \alpha})_{| \alpha = 0 \nobracket} = - x^{\bot} . \nabla Q_c (x), \]
with $x^{\bot} = (- x_2, x_1)$ (see Lemma \ref{CP2a}). We infer the following
properties.

\begin{proposition}
  \label{CP2prop5}There exists $c_0 > 0$ such that, for $0 < c \leqslant c_0$,
  the momentum $\vec{P} (Q_c) = (P_1 (Q_c), P_2 (Q_c))$ of $Q_c$ from Theorem
  \ref{th1}, defined by
  \[ P_1 (Q_c) \assign \frac{1}{2} \langle i \partial_{x_1} Q_c, Q_c - 1
     \rangle, \]
  \[ P_2 (Q_c) \assign \frac{1}{2} \langle i \partial_{x_2} Q_c, Q_c - 1
     \rangle, \]
  verifies $c \mapsto \vec{P} (Q_c) \in C^1 (] 0, c_0 [, \mathbbm{R}^2)$,
  \[ P_1 (Q_c) = \partial_c P_1 (Q_c) = 0, \]
  \[ P_2 (Q_c) = \frac{2 \pi + o_{c \rightarrow 0} (1)}{c} \]
  and
  \[ \partial_c P_2 (Q_c) = \frac{- 2 \pi + o_{c \rightarrow 0} (1)}{c^2} . \]
  Furthermore, the energy satisfies $c \mapsto E (Q_c) \in C^1 (] 0, c_0 [,
  \mathbbm{R})$, and
  \[ E (Q_c) = (2 \pi + o_{c \rightarrow 0} (1)) \ln \left( \frac{1}{c}
     \right) . \]
  Additionally, $\mathfrak{R}\mathfrak{e} (L_{Q_c} (A) \bar{A}) \in L^1
  (\mathbbm{R}^2, \mathbbm{R})$ for $A \in \{ \partial_{x_1} Q_c,
  \partial_{x_2} Q_c, \partial_c Q_c, \partial_{c^{\bot}} Q_c \}$, and
  \[ \langle L_{Q_c} (\partial_{x_1} Q_c), \partial_{x_1} Q_c \rangle =
     \langle L_{Q_c} (\partial_{x_2} Q_c), \partial_{x_2} Q_c \rangle = 0, \]
  \[ \langle L_{Q_c} (\partial_c Q_c), \partial_c Q_c \rangle = \partial_c P_2
     (Q_c) = \frac{- 2 \pi + o_{c \rightarrow 0} (1)}{c^2}, \]
  \[ \langle L_{Q_c} (\partial_{c^{\bot}} Q_c), \partial_{c^{\bot}} Q_c
     \rangle = c P_2 (Q_c) = 2 \pi + o_{c \rightarrow 0} (1) \]
  and
  \[ \partial_c E (Q_c) = c \partial_c P_2 (Q_c) = \frac{- 2 \pi + o_{c
     \rightarrow 0} (1)}{c} . \]
  Finally, the function $Q_c$ has exactly two zeros. Their positions are $\pm
  \widetilde{d_c} \overrightarrow{e_1}$, with
  \[ | d_c - \tilde{d}_c | = o_{c \rightarrow 0} (1), \]
  where $d_c$ is defined in Theorem \ref{th1}.
\end{proposition}

The momentum has a generalized definition for finite energy functions (see
{\cite{MR3043579}} in $3 d$ and {\cite{MR3686002}}). For travelling waves
going to $1$ at infinity, it is equal to the quantity defined in Proposition
\ref{CP2prop5}. The proof of Proposition \ref{CP2prop5} is done in section
\ref{CP2sec2N}.

The equality $\langle L_{Q_c} (\partial_c Q_c), \partial_c Q_c \rangle =
\partial_c P_2 (Q_c)$ is a general property for Hamiltonian system, see
{\cite{MR901236}}. The equality $\partial_c E (Q_c) = c \partial_c P_2 (Q_c)$
has been conjectured and formally shown in {\cite{0305-4470-15-8-036}},
provided we have a smooth branch $c \mapsto Q_c$, which is precisely shown in
Theorem \ref{th1}. We remark that the energy $E (Q_c)$ is of same order as the
energy of the travelling waves constructed in {\cite{MR1669387}}, which also
exhibit two vortices at distance of order $\frac{1}{c}$. We believe that both
construction give the same branch, and that this branch minimises globally the
energy at fixed momentum. However, we were not able to show even a local
minimisation result of the energy for $Q_c$ defined in Theorem \ref{th1}.

In the limit $c \rightarrow 0$, the four directions ($\partial_{x_1} Q_c,
\partial_{x_2} Q_c, c^2 \partial_c Q_c, c \partial_{c^{\bot}} Q_c$) are going
to zeros of the quadratic form (while being of size of order one), and we see
here the splitting for small values of $c$. In particular, two directions give
zero ($\partial_{x_1} Q_c$ and $\partial_{x_2} Q_c$), one becomes positive
($\partial_{c^{\bot}} Q_c$) and one negative ($\partial_c Q_c$).

\subsection{Coercivity results}

One of the main ideas is to reduce the problem of the coercivity of a
travelling wave to the coercivity of vortices. We will first state such a
result for vortices (Proposition \ref{CP2P811}) before the results on the
travelling waves (see in particular Theorem \ref{CP2th2}).

\subsubsection{Coercivity in the case of one vortex}

A coercivity result for one vortex of degree $\pm 1$ is already known, see
{\cite{DP}}, and in particular equation (2.42) there. We consider both
vortices of degrees $+ 1$ and $- 1$ here at the same time, since $V_1 =
\overline{V_{- 1}}$. Here, we present a slight variation of the results in
{\cite{DP}} that will be useful for the coercivity of the travelling waves. We
recall from {\cite{DP}} the quadratic form around $V_1$:
\[ B_{V_1} (\varphi) = \int_{\mathbbm{R}^2} | \nabla \varphi |^2 - (1 - | V_1
   |^2) | \varphi |^2 + 2\mathfrak{R}\mathfrak{e}^2 (\overline{V_1} \varphi),
\]
for functions in the energy space
\[ H_{V_1} = \left\{ \varphi \in H^1_{\tmop{loc}} (\mathbbm{R}^2,
   \mathbbm{C}), \| \varphi \|_{H_{V_1}}^2 \assign \int_{\mathbbm{R}^2} |
   \nabla \varphi |^2 + (1 - | V_1 |^2) | \varphi |^2
   +\mathfrak{R}\mathfrak{e}^2 (\overline{V_1} \varphi) < + \infty \right\} .
\]
As the family of vortices has three parameters, we expect a coercivity result
under three orthogonality conditions. The three associated directions are
$\partial_{x_1} V_1, \partial_{x_2} V_1$ (for the translations) and $i V_1$
(for the phase).

\begin{proposition}
  \label{CP2P811}There exist $K > 0$, $R > 5$, such that, if the following
  three orthogonality conditions are satisfied for $\varphi = V_1 \psi \in
  C^{\infty}_c (\mathbbm{R}^2 \backslash \{ 0 \}, \mathbbm{C})$,
  \[ \int_{B (0, R)} \mathfrak{R}\mathfrak{e} (\partial_{x_1} V_1
     \overline{V_1 \psi}) = \int_{B (0, R)} \mathfrak{R}\mathfrak{e}
     (\partial_{x_2} V_1 \overline{V_1 \psi}) = \int_{B (0, R) \backslash B
     (0, R / 2)} \mathfrak{I}\mathfrak{m} (\psi) = 0, \]
  then,
  \[ B_{V_1} (\varphi) \geqslant K \left( \int_{B (0, 10)} | \nabla \varphi
     |^2 + | \varphi |^2 + \int_{\mathbbm{R}^2 \backslash B (0, 5)} | \nabla
     \psi |^2 | V_1 |^2 +\mathfrak{R}\mathfrak{e}^2 (\psi) | V_1 |^4 + \frac{|
     \psi |^2}{r^2 \ln^2 (r)} \right) . \]
\end{proposition}

The same result holds if we replace $V_1$ by $V_{- 1}$. We remark that the
coercivity norm is not $\| . \|_{H_{V_1}}$, but is weaker (the decay in
position is stronger), and this is due to the fact that $i V_1 \nin H_{V_1}$.
That is why this result is stated for compactly supported function. The fact
that the support of $\varphi$ avoids $0$ is technical and can be removed by
density (see Lemma \ref{CP2Ndensity}).

\

Proposition \ref{CP2P811} is shown in subsection \ref{CP2s24}. The proofs
there are mostly slight variations or improvements of proofs given in
{\cite{DP}}.

\subsubsection{Coercivity and kernel in the energy space}

The main part of this section consists of coercivity results for the family of
travelling waves constructed in Theorem \ref{th1}. We will show it on $Q_c$
defined in Theorem \ref{th1}, and with (\ref{CP2rotQc}), it extends to all
speed values $\vec{c}$ of small norm. We recall the linearized operator around
$Q_c$:
\[ L_{Q_c} (\varphi) = - \Delta \varphi - i c \partial_{x_2} \varphi - (1 - |
   Q_c |^2) \varphi + 2\mathfrak{R}\mathfrak{e} (\overline{Q_c} \varphi) Q_c .
\]
The natural associated energy space is
\[ H_{Q_c} \assign \left\{ \varphi \in H^1_{\tmop{loc}} (\mathbbm{R}^2), \|
   \varphi \|_{H_{Q_c}} < + \infty \right\}, \]
where
\[ \| \varphi \|_{H_{Q_c}}^2 \assign \int_{\mathbbm{R}^2} | \nabla \varphi |^2
   + | 1 - | Q_c^{\nosymbol} |^2 | | \varphi |^2 +\mathfrak{R}\mathfrak{e}^2
   (\overline{Q_c} \varphi) . \]
First, there are difficulties in the definition of the quadratic form for
$\varphi \in H_{Q_c}$, because of the transport term. A natural definition for
the associated quadratic form for $\varphi \in H_{Q_c}$ could be
\begin{equation}
  \int_{\mathbbm{R}^2} | \nabla \varphi |^2 - (1 - | Q_c |^2) | \varphi |^2 +
  2\mathfrak{R}\mathfrak{e}^2 (\overline{Q_c} \varphi)
  -\mathfrak{R}\mathfrak{e} (i c \partial_{x_2} \varphi \bar{\varphi})
  \label{CP2120403},
\end{equation}
unfortunately the last term is not well defined for $\varphi \in H_{Q_c}$,
because we lack a control on $\mathfrak{I}\mathfrak{m} (\overline{Q_c}
\varphi)$\tmverbatim{{\tmem{}}} in $L^2 (\mathbbm{R}^2)$ in $\| .
\|_{H_{Q_c}}$, see {\cite{MR3043579}}. We can resolve this issue by
decomposing this term and doing an integration by parts, but the proof of the
integration by parts can not be done if we only suppose $\varphi \in H_{Q_c}$
(see section \ref{CP2NSEC4} for more details). We therefore define the
quadratic form with the integration by parts already done. Take a smooth
cutoff function $\eta$ such that $\eta (x) = 0$ on $B (\pm \widetilde{d_c}
\overrightarrow{e_1}, 1)$, $\eta (x) = 1$ on $\mathbbm{R}^2 \backslash B (\pm
\widetilde{d_c} \overrightarrow{e_1}, 2)$, where $\pm \widetilde{d_c}
\overrightarrow{e_1}$ are the zeros of $Q_c$. We define, for $\varphi = Q_c
\psi \in H_{Q_c}$,
\begin{eqnarray}
  B_{Q_c} (\varphi) & \assign & \int_{\mathbbm{R}^2} | \nabla \varphi |^2 - (1
  - | Q_c |^2) | \varphi |^2 + 2\mathfrak{R}\mathfrak{e}^2 (\overline{Q_c}
  \varphi) \nonumber\\
  & - & c \int_{\mathbbm{R}^2} (1 - \eta) \mathfrak{R}\mathfrak{e} (i
  \partial_{x_2} \varphi \bar{\varphi}) - c \int_{\mathbbm{R}^2 \nosymbol}
  \eta \mathfrak{R}\mathfrak{e} (i \partial_{x_2} Q_c \overline{Q_c}) | \psi
  |^2 \nonumber\\
  & + & 2 c \int_{\mathbbm{R}^2} \eta \mathfrak{R}\mathfrak{e} \psi
  \mathfrak{I}\mathfrak{m} (\partial_{x_2} \psi) | Q_c |^2 + c
  \int_{\mathbbm{R}^2} \partial_{x_2} \eta \mathfrak{R}\mathfrak{e} \psi
  \mathfrak{I}\mathfrak{m} \psi | Q_c |^2 \nonumber\\
  & + & c \int_{\mathbbm{R}^2} \eta \mathfrak{R}\mathfrak{e} \psi
  \mathfrak{I}\mathfrak{m} \psi \partial_{x_2} (| Q_c |^2) . 
  \label{CP2truebqc}
\end{eqnarray}
See subsection \ref{CP22241103} for the details of the computation. For
functions $\varphi \in H^1 (\mathbbm{R}^2)$ for instance, both quadratic forms
(\ref{CP2120403}) and (\ref{CP2truebqc}) are well defined and are equal (see
Lemma \ref{CP2L3150403}). We will show that $B_{Q_c}$ is well defined for
$\varphi \in H_{Q_c}$ (see Lemma \ref{CP2finitebilinear}), and that for $A \in
\{ \partial_{x_1} Q_c, \partial_{x_2} Q_c, \partial_c Q_c, \partial_{c^{\bot}}
Q_c \}$, $B_{Q_c} (A) = \langle L_{Q_c} (A), A \rangle$.

\

From Proposition \ref{CP2prop5}, we know that $Q_c$ has only two zeros. We
will write the quadratic form $B_{Q_c}$ around the zeros of $Q_c$ (for a
function $\varphi = Q_c \psi \in H_{Q_c}$) as the quadratic form for one
vortex (computed in Proposition \ref{CP2P811}), up to some small error. As we
want to avoid to add an orthogonality on the phase, we change the coercivity
norm to a weaker semi-norm, that avoids $i Q_c$, the direction connected to
the shift of phase.

We will therefore infer a coercivity result under four orthogonality
conditions near the zeros of $Q_c$ (two for each zero). Then, we shall show
that far from the zeros of $Q_c$, the coercivity holds, without any additional
orthogonality conditions.

\begin{proposition}
  \label{CP205218}There exists $c_0, R > 0$ such that, for $0 < c \leqslant
  c_0$, if one defines $\tilde{V}_{\pm 1}$ to be the vortices centered around
  $\pm \widetilde{d_c} \overrightarrow{e_1}$ ($\widetilde{d_c}$ is defined in
  Proposition \ref{CP2prop5}), there exist $K > 0$ such that for $\varphi =
  Q_c \psi \in H_{Q_c}$, $0 < c < c_0$, if the four orthogonality conditions
  \[ \int_{B (\tilde{d}_c \overrightarrow{e_1}, R)} \mathfrak{R}\mathfrak{e}
     \left( \partial_{x_1} \tilde{V}_1 \overline{_{} \widetilde{V_1} \psi}
     \right) = \int_{B (\tilde{d}_c \overrightarrow{e_1}, R)}
     \mathfrak{R}\mathfrak{e} \left( \partial_{x_2} \widetilde{V_1}
     \overline{\widetilde{V_1} \psi} \right) = 0, \]
  \[ \int_{B (- \tilde{d}_c \overrightarrow{e_1}, R)} \mathfrak{R}\mathfrak{e}
     (\partial_{x_1} \tilde{V}_{- 1} \overline{\tilde{V}_{- 1} \psi}) =
     \int_{B (- \tilde{d}_c \overrightarrow{e_1}, R)} \mathfrak{R}\mathfrak{e}
     (\partial_{x_2} \tilde{V}_{- 1} \overline{\tilde{V}_{- 1} \psi}) = 0 \]
  are satisfied, then, for
  \[ \| \varphi \|_{\mathcal{C}}^2 \assign \int_{\mathbbm{R}^2} | \nabla \psi
     |^2 | Q_c |^4 +\mathfrak{R}\mathfrak{e}^2 (\psi) | Q_c |^4, \]
  the following coercivity result holds:
  \[ B_{Q_c} (\varphi) \geqslant K \| \varphi \|_{\mathcal{C}}^2 . \]
\end{proposition}

We will check that $\| \varphi \|_{\mathcal{C}}$ is well defined for $\varphi
\in H_{Q_c}$ (see section \ref{CP2NSEC4}). Proposition \ref{CP205218} is
proven in subsection \ref{CP2concl}.

We point out that $\varphi = Q_c \psi \mapsto \| \varphi \|_{\mathcal{C}}$ is
not a norm but a seminorm since $\int_{\mathbbm{R}^2} | \nabla \psi |^2 | Q_c
|^4 +\mathfrak{R}\mathfrak{e}^2 (\psi) | Q_c |^4 = 0$ implies only that
$\varphi = \lambda i Q_c$ for some $\lambda \in \mathbbm{R}$, and $i Q_c$ is
the direction connected to the shift of phase.

\

Now, we want to change the orthogonality conditions in Proposition
\ref{CP205218} to quantities linked to the parameters $\vec{c}$ and $X$ of the
travelling waves, that is $\partial_{x_1} Q_c, \partial_{x_2} Q_c, \partial_c
Q_c$ and $\partial_{c^{\bot}} Q_c$. We can show that for $\varphi = Q_c \psi
\in H_{Q_c}$, for instance
\[ \left| \int_{B (\tilde{d}_c \overrightarrow{e_1}, R)}
   \mathfrak{R}\mathfrak{e} \left( \partial_{x_1} \tilde{V}_1
   \overline{\widetilde{V_1} \psi} \right) \right| \leqslant K \| \varphi
   \|_{\mathcal{C}}, \]
but such an estimate might not hold for $\mathfrak{R}\mathfrak{e} \int_{B
(\tilde{d}_c \overrightarrow{e_1}, R) \cup B (- \tilde{d}_c
\overrightarrow{e_1}, R)} \partial_{x_1} Q_c \overline{Q_c \psi}$ (because of
the lack of control on $\mathfrak{I}\mathfrak{m} (\psi)$ in $L^2
(\mathbbm{R}^2)$ in the coercivity norm $\| . \|_{\mathcal{C}}$). It is
therefore difficult to have a local orthogonality condition directly on
$\partial_{x_1} Q_c$ for instance.

To solve this issue, we shall use the harmonic decomposition around $\pm
\tilde{d}_c \overrightarrow{e_1}$. For the constructed travelling wave $Q_c$,
two distances play a particular role, they are $d_c$ (defined in Theorem
\ref{th1}) and $\widetilde{d_c}$ (defined in Proposition \ref{CP2prop5} and is
connected to the position of the zeros of $Q_c$). In particular, we define the
following polar coordinates for $x \in \mathbbm{R}^2$:
\[ r e^{i \theta} \assign x \in \mathbbm{R}^2, \]
\[ r_{\pm 1} e^{i \theta_{\pm 1}} \assign x - (\pm d_c) \overrightarrow{e_1}
   \in \mathbbm{R}^2, \]
\[ \tilde{r}_{\pm 1} e^{i \tilde{\theta}_{\pm 1}} \assign x - (\pm
   \widetilde{d_c}) \overrightarrow{e_1} \in \mathbbm{R}^2 . \]
We will also use $\tilde{r} \assign \min (r_1, r_{- 1})$ and $\check{r}
\assign \min (\widetilde{r_1}, \tilde{r}_{- 1})$. For a function $\psi$ such
that $Q_c \psi \in H^1_{\tmop{loc}} (\mathbbm{R}^2)$ and $j \in \mathbbm{Z}$,
we define its $j -$harmonic around $\pm \widetilde{d_c} \overrightarrow{e_1}$
by the radial function around $\pm \widetilde{d_c} \overrightarrow{e_1}$:
\[ \psi^{j, \pm 1} (\tilde{r}_{\pm 1}) \assign \frac{1}{2 \pi} \int_0^{2 \pi}
   \psi (\tilde{r}_{\pm 1} e^{i \tilde{\theta}_{\pm 1}}) e^{- i j
   \tilde{\theta}_{\pm 1}} d \tilde{\theta}_{\pm 1} . \]
Summing over the Fourier modes leads to
\[ \psi (x) = \sum_{j \in \mathbbm{Z}} \psi^{j, \pm 1} (\tilde{r}_{\pm 1})
   e^{i j \tilde{\theta}_{\pm 1}} . \]
and we define, to simplify the notations later on, the function $\psi^{\neq
0}$, by
\[ \psi^{\neq 0} (x) \assign \psi (x)^{^{}} - \psi^{0, 1} (\tilde{r}_1) \]
in the right half-plane, and
\[ \psi^{\neq 0} (x) \assign \psi (x)^{^{}} - \psi^{0, - 1} (\tilde{r}_{- 1})
\]
in the left half-plane. This notation will only be used far from the line $\{
x_1 = 0 \}$. We now state the main coercivity result.

\begin{theorem}
  \label{CP2th2}There exist $c_0, K, \beta_0 > 0$ such that, for $R > 0$
  defined in Proposition \ref{CP205218}, for any $0 < \beta < \beta_0$, there
  exists $c_0 (\beta), K (\beta) > 0$ such that, for $c < c_0 (\beta)$, if
  $\varphi = Q_c \psi \in H_{Q_c}$ satisfies the following three orthogonality
  conditions:
  \[ \mathfrak{R}\mathfrak{e} \int_{B (\tilde{d}_c \overrightarrow{e_1}, R)
     \cup B (- \tilde{d}_c \overrightarrow{e_1}, R)} \partial_{x_1} Q_c
     \overline{Q_c \psi^{\neq 0}} =\mathfrak{R}\mathfrak{e} \int_{B
     (\tilde{d}_c \overrightarrow{e_1}, R) \cup B (- \tilde{d}_c
     \overrightarrow{e_1}, R)} \partial_{x_2} Q_c \overline{Q_c \psi^{\neq 0}}
     = 0 \]
  and
  \[ \mathfrak{R}\mathfrak{e} \int_{B (\tilde{d}_c \overrightarrow{e_1}, R)
     \cup B (- \tilde{d}_c \overrightarrow{e_1}, R)} \partial_c Q_c
     \overline{Q_c \psi^{\neq 0}} = 0, \]
  then,
  \[ B_{Q_c} (\varphi) \geqslant K (\beta) c^{2 + \beta} \| \varphi
     \|_{\mathcal{C}}^2, \]
  with
  \[ \| \varphi \|_{\mathcal{C}}^2 = \int_{\mathbbm{R}^2} | \nabla \psi |^2 |
     Q_c |^4 +\mathfrak{R}\mathfrak{e}^2 (\psi) | Q_c |^4 . \]
  If $\varphi = Q_c \psi$ also satisfies the fourth orthogonality condition
  (with $0 < c < c_0$)
  \[ \mathfrak{R}\mathfrak{e} \int_{B (\tilde{d}_c \overrightarrow{e_1}, R)
     \cup B (- \tilde{d}_c \overrightarrow{e_1}, R)} \partial_{c^{\bot}} Q_c
     \overline{Q_c \psi^{\neq 0}} = 0, \]
  then
  \[ B_{Q_c} (\varphi) \geqslant K \| \varphi \|_{\mathcal{C}}^2 . \]
\end{theorem}

Theorem \ref{CP2th2} shows that under four orthogonality conditions, we have a
coercivity result in a weaker norm $\| . \|_{\mathcal{C}}$, instead of $\| .
\|_{H_{Q_c}}$ with a constant independent of $c$, and with only three
orthogonality conditions, we have the coercivity but the constant is a
$O^{\beta}_{c \rightarrow 0} (c^{2 + \beta})$. This is because, of the four
particular directions of the linearized operator, $\partial_{x_1} Q_c,
\partial_{x_2} Q_c$ are in its kernel, $\partial_c Q_c$ is a small negative
direction, and $\partial_{c^{\bot}} Q_c$ is a small positive direction (see
Proposition \ref{CP2prop5}). About the orthogonality conditions, we remark
that, for $\varphi = Q_c \psi \in H_{Q_c}$,
\[ \mathfrak{R}\mathfrak{e} \int_{B (\tilde{d}_c \overrightarrow{e_1}, R) \cup
   B (- \tilde{d}_c \overrightarrow{e_1}, R)} \partial_{x_1} Q_c \overline{Q_c
   \psi^{\neq 0}} \]
is close to
\[ \mathfrak{R}\mathfrak{e} \int_{B (\tilde{d}_c \overrightarrow{e_1}, R) \cup
   B (- \tilde{d}_c \overrightarrow{e_1}, R)} \partial_{x_1} Q_c \overline{Q_c
   \psi} \]
(we have $\mathfrak{R}\mathfrak{e} \int_{B (\tilde{d}_c \overrightarrow{e_1},
R)} \partial_{x_1} Q_c \overline{Q_c \psi^{0, 1}} = o_{c \rightarrow 0} (1) \|
\varphi \|_{H_{Q_c}}$ for instance), but the first quantity can be controlled
by $\| \varphi \|_{\mathcal{C}}$, and the second can not be.

Theorem \ref{CP2th2} is a consequence of Proposition \ref{CP205218}, and is
shown in section \ref{CP2s32}. From this result, we can also deduce the kernel
of the linearized operator in $H_{Q_c}$.

\begin{corollary}
  \label{CP2Cor41}There exists $c_0 > 0$ such that, for $0 < c < c_0$, $Q_c$
  defined in Theorem \ref{th1}, for $\varphi \in H_{Q_c}$, the following
  properties are equivalent:
  \begin{enumerateroman}
    \item $L_{Q_c} (\varphi) = 0$ in $H^{- 1} (\mathbbm{R}^2)$, that is,
    $\forall \varphi^{\ast} \in H^1 (\mathbbm{R}^2),$
    \[ \int_{\mathbbm{R}^2} \mathfrak{R}\mathfrak{e} (\nabla \varphi . \nabla
       \overline{\varphi^{\ast}}) - (1 - | Q_c |^2) \mathfrak{R}\mathfrak{e}
       (\varphi \overline{\varphi^{\ast}}) + 2\mathfrak{R}\mathfrak{e}
       (\overline{Q_c} \varphi) \mathfrak{R}\mathfrak{e} (\overline{Q_c}
       \varphi^{\ast}) -\mathfrak{R}\mathfrak{e} (i c \partial_{x_2} \varphi
       \overline{\varphi^{\ast}}) = 0. \]
    \item $\varphi \in \tmop{Span}_{\mathbbm{R}} (\partial_{x_1} Q_c,
    \partial_{x_2} Q_c)$.
  \end{enumerateroman}
\end{corollary}

This corollary is proven in subsection \ref{CP2proofcor}. This nondegeneracy
result is, to our knowledge, the first one on this type of model. It is a
building block in the analysis of the dynamical stability of the travelling
wave and the construction of multi-travelling wave. Here, the travelling wave
is not radial, nor has a simple profile, which means that we can not use
classical technics for radial ground states for instance (see
{\cite{MR820338}}).

\subsubsection{Spectral stability in $H^1 (\mathbbm{R}^2)$}

In this subsection, we give some result on the spectrum of $L_{Q_c} : H^2
(\mathbbm{R}^2) \rightarrow L^2 (\mathbbm{R}^2)$. In particular, we are
interested in negative eigenvalues of the linearized operator. We can show
that $H^1 (\mathbbm{R}^2) \subset H_{Q_c}$ and prove the following corollary
of Theorem \ref{CP2th2}.

\begin{corollary}
  \label{CP2cor177}There exists $c_0 > 0$ such that, for $0 < c \leqslant
  c_0$, $Q_c$ defined in Theorem \ref{th1}, if $\varphi \in H^1
  (\mathbbm{R}^2)$ satisfies
  \[ \langle \varphi, i \partial_{x_2} Q_c \rangle = 0, \]
  then
  \[ B_{Q_c} (\varphi) \geqslant 0. \]
\end{corollary}

We can show that $L_{Q_c} (\partial_c Q_c) = i \partial_{x_2} Q_c \in L^2
(\mathbbm{R}^2)$, and thus $\varphi i \overline{\partial_{x_2} Q_c} \in L^1
(\mathbbm{R}^2)$ for $\varphi \in H^1 (\mathbbm{R}^2)$. This result shows that
we expect only one negative direction for the linearized operator, and it
should also hold in $H_{Q_c}$. For $\varphi \in H^1 (\mathbbm{R}^2)$, we have
that $B_{Q_c} (\varphi)$ is equal to the expression (\ref{CP2120403}).

Now, we define $\mathfrak{G}$ to be the collection of subspaces $S \subset H^1
(\mathbbm{R}^2)$ such that $B_{Q_c} (\varphi) < 0$ for all $\varphi \neq 0,
\varphi \in S$, and we define
\[ n^- (L_{Q_c}) \assign \max \{ \dim S, S \in \mathfrak{G} \} . \]
\begin{proposition}
  \label{CP2p188}There exists $c_0 > 0$ such that, for $0 < c < c_0$, for
  $Q_c$ defined in Theorem \ref{th1},
  \[ n^- (L_{Q_c}) = 1. \]
  Furthermore, $L_{Q_c} : H^2 (\mathbbm{R}^2) \rightarrow L^2 (\mathbbm{R}^2)$
  has exactly one negative eigenvalue with eigenvector in $L^2
  (\mathbbm{R}^2)$.
\end{proposition}

With this result, Theorem \ref{th1} and Proposition \ref{CP2prop5}, we have
met all the conditions to show the spectral stability of the travelling wave:

\begin{theorem}[Theorem 11.8 $(i)$ of {\cite{LZ}}]
  \label{heyhey}For $0 < c_1 < c_2$ and $c \mapsto U_c$ a $C^1$ branch of
  solutions of $(\tmop{TW}_c) (U_c) = 0$ on $] c_1, c_2 [$ with finite energy,
  for $c_{\ast} \in] c_1, c_2 [$, under the following conditions:
  \begin{enumerateroman}
    \item for all $c \in] c_1, c_2 [$, $\mathfrak{R}\mathfrak{e} (U_c - 1) \in
    H^1 (\mathbbm{R}^2)$, $\mathfrak{I}\mathfrak{m} (\nabla U_c) \in L^2
    (\mathbbm{R}^2)$, $| U_c | \rightarrow 1$ at infinity and $\| U_c \|_{C^1
    (\mathbbm{R}^2)} < + \infty$
    
    \item $n^- \left( L_{Q_{c_{\ast}}} \right) \leqslant 1$
    
    \item $\partial_c P_2 (U_c)_{| c = c_{\ast} \nobracket} < 0$,
  \end{enumerateroman}
  then $U_{c_{\ast}}$ is spectrally stable. That is, it is not an
  exponentially unstable solution of the linearized equation in $\dot{H}^1
  (\mathbbm{R}^2, \mathbbm{C})$.
\end{theorem}

\begin{corollary}
  \label{CP2th199}There exists $c_0 > 0$ such that, for any $0 < c < c_0$, the
  function $Q_c$ defined in Theorem \ref{th1} is spectrally stable in the
  sense of Theorem \ref{heyhey}.
\end{corollary}

The notion of spectral stability of {\cite{LZ}} is the following: for any $u_0
\in H^1 (\mathbbm{R}^2, \mathbbm{C})$, the solution to the problem
\[ \left\{ \begin{array}{l}
     i \partial_t u = L_{Q_c} (u)\\
     u (t = 0) = u_0
   \end{array} \right. \]
satisfies that, for all $\lambda > 0$,
\[ \left( \int_{\mathbbm{R}^2} | \nabla u |^2 (t) d x \right) e^{- \lambda t}
   \rightarrow 0 \]
when $t \rightarrow \infty$. The result of {\cite{LZ}} is a little stronger:
the norm that does not grow exponentially in time is better than the one on
$\dot{H}^1 (\mathbbm{R}^2, \mathbbm{C})$, but weaker than the one on $H^1
(\mathbbm{R}^2, \mathbbm{C})$, and is not explicit.

\subsection{Generalisation to a larger energy space and use of the phase}

There are two main difficulties with the phase. The first one, as previously
stated, is that we lose a parameter when passing from two vortices to a
travelling wave. The second one is that for the direction linked to the phase
shift, namely $i Q_c$, we have $i Q_c \nin H_{Q_c}$ (and even for one vortex,
$i V_1 \nin H_{V_1}$). This will be an obstacle when we modulate on the phase
for the local uniqueness result. Therefore, we define here a space larger than
$H_{Q_c}$.

\subsubsection{Definition and properties of the space
$H^{\exp}_{Q_c}$}\label{CP21310811}

We define the space $H^{\exp}_{Q_c}$, the expanded energy space, by
\[ H^{\exp}_{Q_c} \assign \left\{ \varphi \in H^1_{\tmop{loc}}
   (\mathbbm{R}^2), \| \varphi \|_{H^{\exp}_{Q_c}} < + \infty \right\}, \]
with the norm, for $\varphi = Q_c \psi \in H^1_{\tmop{loc}} (\mathbbm{R}^2)$,
\[ \| \varphi \|_{H^{\exp}_{Q_c}}^2 \assign \| \varphi \|_{H^1 (\{ \tilde{r}
   \leqslant 10 \})}^2 + \int_{\{ \tilde{r} \geqslant 5 \}} | \nabla \psi |^2
   +\mathfrak{R}\mathfrak{e}^2 (\psi) + \frac{| \psi |^2}{\tilde{r}^2 \ln^2
   (\tilde{r})}, \]
where $\tilde{r} = \min (\widetilde{r_1}, \tilde{r}_{- 1})$, the minimum of
the distance to the zeros of $Q_c$. It is easy to check that that there exists
$K > 0$ independent of $c$ such that, for $\varphi = Q_c \psi \in
H_{Q_c}^{\exp}$,
\[ \frac{1}{K} \| \varphi \|^2_{H^1 (\{ 5 \leqslant \tilde{r} \leqslant 10
   \})} \leqslant \int_{\{ 5 \leqslant \tilde{r} \leqslant 10 \}} | \nabla
   \psi |^2 +\mathfrak{R}\mathfrak{e}^2 (\psi) + \frac{| \psi |^2}{\tilde{r}^2
   \ln (\tilde{r})^2} \leqslant K \| \varphi \|^2_{H^1 (\{ 5 \leqslant
   \tilde{r} \leqslant 10 \})} . \]
We will show that $H_{Q_c} \subset H^{\exp}_{Q_c}$ and $i Q_c \in
H^{\exp}_{Q_c}$, whereas $i Q_c \nin H_{Q_c}$. This space will appear in the
proof of the local uniqueness (Theorem \ref{CP2th16} below). The main
difficulty is that $B_{Q_c} (\varphi)$ is not well defined for $\varphi \in
H^{\exp}_{Q_c}$ because for instance of the term $(1 - | Q_c |^2) | \varphi
|^2$ integrated at infinity. If we write the linearized operator
multiplicatively, for $\varphi = Q_c \psi$ (using $(\tmop{TW}_c) (Q_c) = 0$),
\[ Q_c L_{Q_c}' (\psi) \assign L_{Q_c} (\varphi) = Q_c \left( - i c
   \partial_{x_2} \psi - \Delta \psi - 2 \frac{\nabla Q_c}{Q_c} . \nabla \psi
   + 2\mathfrak{R}\mathfrak{e} (\psi) | Q_c |^2 \right), \]
then there will be no problem at infinity for $\varphi \in H^{\exp}_{Q_c}$ for
the associated quadratic form (in $\psi$), but there are instead some
integrability issues near the zeros of $Q_c$. We take as before a smooth
cutoff function $\eta$ such that $\eta (x) = 0$ on $B (\pm \widetilde{d_c}
\overrightarrow{e_1}, 1)$, $\eta (x) = 1$ on $\mathbbm{R}^2 \backslash B (\pm
\widetilde{d_c} \overrightarrow{e_1}, 2)$, where $\pm \widetilde{d_c}
\overrightarrow{e_1}$ are the zeros of $Q_c$. The natural linear operator for
which we want to consider the quadratic form is then
\[ L^{\exp}_{Q_c} (\varphi) \assign (1 - \eta) L_{Q_c} (\varphi) + \eta Q_c
   L_{Q_c}' (\psi), \]
and we therefore define, for $\varphi = Q_c \psi \in H^{\exp}_{Q_c}$,
\begin{eqnarray}
  B^{\exp}_{Q_c} (\varphi) & \assign & \int_{\mathbbm{R}^2} (1 - \eta) (|
  \nabla \varphi |^2 -\mathfrak{R}\mathfrak{e} (i c \partial_{x_2} \varphi
  \bar{\varphi}) - (1 - | Q_c |^2) | \varphi |^2 + 2\mathfrak{R}\mathfrak{e}^2
  (\overline{Q_c} \varphi)) \nonumber\\
  & - & \int_{\mathbbm{R}^2} \nabla \eta . (\mathfrak{R}\mathfrak{e} (\nabla
  Q_c \overline{Q_c}) | \psi |^2 - 2\mathfrak{I}\mathfrak{m} (\nabla Q_c
  \overline{Q_c}) \mathfrak{R}\mathfrak{e} (\psi) \mathfrak{I}\mathfrak{m}
  (\psi)) \nonumber\\
  & + & \int_{\mathbbm{R}^2} c \partial_{x_2} \eta \mathfrak{R}\mathfrak{e}
  (\psi) \mathfrak{I}\mathfrak{m} (\psi) | Q_c |^2 \nonumber\\
  & + & \int_{\mathbbm{R}^2} \eta (| \nabla \psi |^2 | Q_c |^2 +
  2\mathfrak{R}\mathfrak{e}^2 (\psi) | Q_c |^4) \nonumber\\
  & + & \int_{\mathbbm{R}^2} \eta (4\mathfrak{I}\mathfrak{m} (\nabla Q_c
  \overline{Q_c}) \mathfrak{I}\mathfrak{m} (\nabla \psi)
  \mathfrak{R}\mathfrak{e} (\psi) + 2 c | Q_c |^2 \mathfrak{I}\mathfrak{m}
  (\partial_{x_2} \psi) \mathfrak{R}\mathfrak{e} (\psi)) .  \label{CP2Btilda}
\end{eqnarray}
This quantity is independent of the choice of $\eta$.

We will show that $B^{\exp}_{Q_c} (\varphi)$ is well defined for $\varphi \in
H^{\exp}_{Q_c}$ and that, if $\varphi \in H_{Q_c} \subset H^{\exp}_{Q_c}$,
then $B^{\exp}_{Q_c} (\varphi) = B_{Q_c} (\varphi)$. Writing the quadratic
form $B^{\exp}_{Q_c}$ is a way to enlarge the space of possible perturbations
to add in particular the remaining zero of the linearized operator. We infer
the following result.

\begin{proposition}
  \label{CP2prop16}There exist $c_0, K, R, \beta_0 > 0$ such that, for any $0
  < \beta < \beta_0$, there exists $c_0 (\beta), K (\beta) > 0$ such that, for
  $0 < c < c_0 (\beta)$, if $\varphi = Q_c \psi \in H^{\exp}_{Q_c}$ satisfies
  the following three orthogonality conditions:
  \[ \mathfrak{R}\mathfrak{e} \int_{B (\tilde{d}_c \overrightarrow{e_1}, R)
     \cup B (- \tilde{d}_c \overrightarrow{e_1}, R)} \partial_{x_1} Q_c
     \overline{Q_c \psi^{\neq 0}} =\mathfrak{R}\mathfrak{e} \int_{B
     (\tilde{d}_c \overrightarrow{e_1}, R) \cup B (- \tilde{d}_c
     \overrightarrow{e_1}, R)} \partial_{x_2} Q_c \overline{Q_c \psi^{\neq 0}}
     = 0 \]
  and
  \[ \mathfrak{R}\mathfrak{e} \int_{B (\tilde{d}_c \overrightarrow{e_1}, R)
     \cup B (- \tilde{d}_c \overrightarrow{e_1}, R)} \partial_c Q_c
     \overline{Q_c \psi^{\neq 0}} = 0, \]
  then,
  \[ B^{\exp}_{Q_c} (\varphi) \geqslant K (\beta) c^{2 + \beta} \| \varphi
     \|_{\mathcal{C}}^2, \]
  with
  \[ \| \varphi \|_{\mathcal{C}}^2 = \int_{\mathbbm{R}^2} | \nabla \psi |^2 |
     Q_c |^4 +\mathfrak{R}\mathfrak{e}^2 (\psi) | Q_c |^4 . \]
  If $\varphi = Q_c \psi$ also satisfies the fourth orthogonality condition
  (with $0 < c < c_0$)
  \[ \mathfrak{R}\mathfrak{e} \int_{B (\tilde{d}_c \overrightarrow{e_1}, R)
     \cup B (- \tilde{d}_c \overrightarrow{e_1}, R)} \partial_{c^{\bot}} Q_c
     \overline{Q_c \psi^{\neq 0}} = 0, \]
  then
  \[ B^{\exp}_{Q_c} (\varphi) \geqslant K \| \varphi \|_{\mathcal{C}}^2 . \]
  Furthermore, for $\varphi \in H^{\exp}_{Q_c}$, the following properties are
  equivalent:
  \begin{enumerateroman}
    \item $L_{Q_c} (\varphi) = 0$ in $H^{- 1} (\mathbbm{R}^2)$, that is,
    $\forall \varphi^{\ast} \in H^1 (\mathbbm{R}^2),$
    \[ \int_{\mathbbm{R}^2} \mathfrak{R}\mathfrak{e} (\nabla \varphi . \nabla
       \overline{\varphi^{\ast}}) - (1 - | Q_c |^2) \mathfrak{R}\mathfrak{e}
       (\varphi \overline{\varphi^{\ast}}) + 2\mathfrak{R}\mathfrak{e}
       (\overline{Q_c} \varphi) \mathfrak{R}\mathfrak{e} (\overline{Q_c}
       \varphi^{\ast}) -\mathfrak{R}\mathfrak{e} (i c \partial_{x_2} \varphi
       \overline{\varphi^{\ast}}) = 0. \]
    \item $\varphi \in \tmop{Span}_{\mathbbm{R}} (i Q_c, \partial_{x_1} Q_c,
    \partial_{x_2} Q_c)$
  \end{enumerateroman}
\end{proposition}

Proposition \ref{CP2prop16} is proven in subsection \ref{CP2HQc911}. The
additional direction in the kernel comes from the invariance of phase
$(L_{Q_c} (i Q_c) = 0)$. The main difficulties, compared to Theorem
\ref{CP2th2}, is to show that the considered quantities are well defined with
only $\varphi \in H^{\exp}_{Q_c}$, and that we can conclude by density in this
bigger space.

\subsubsection{Coercivity results with an orthogonality on the phase}

The main problem with adding a local orthogonality condition on $i Q_c$ is to
choose where to put it. Indeed, we want this condition near both zeros of
$Q_c$, or else the coercivity constant will depend on the distance between the
vortices, which itself depends on $c$.

The first option is to let the coercivity constant depend on $c$. In that
case, we can also remove the orthogonality condition on $\partial_{c^{\bot}}
Q_c$, the small positive direction. We infer the following result.

\begin{proposition}
  \label{CP2prop17}There exist universal constants $K_1, c_0 > 0$ such that,
  with $R > 0$ defined in Proposition \ref{CP205218}, for $0 < c < c_0$, for
  the function $Q_c$ defined in Theorem \ref{th1}, there exists $K_2 (c) > 0$
  depending on $c$ such that, if $\varphi = Q_c \psi \in H^{\exp}_{Q_c}$
  satisfies the following four orthogonality conditions:
  \[ \mathfrak{R}\mathfrak{e} \int_{B (\tilde{d}_c \overrightarrow{e_1}, R)
     \cup B (- \tilde{d}_c \overrightarrow{e_1}, R)} \partial_{x_1} Q_c
     \overline{Q_c \psi^{\neq 0}} =\mathfrak{R}\mathfrak{e} \int_{B
     (\tilde{d}_c \overrightarrow{e_1}, R) \cup B (- \tilde{d}_c
     \overrightarrow{e_1}, R)} \partial_{x_2} Q_c \overline{Q_c \psi^{\neq 0}}
     = 0, \]
  \[ \mathfrak{R}\mathfrak{e} \int_{B (\tilde{d}_c \overrightarrow{e_1}, R)
     \cup B (- \tilde{d}_c \overrightarrow{e_1}, R)} \partial_c Q_c
     \overline{Q_c \psi^{\neq 0}} =\mathfrak{R}\mathfrak{e} \int_{B (0, R)} i
     \psi = 0, \]
  then
  \[ K_1 \| \varphi \|^2_{H^{\exp}_{Q_c}} \geqslant B^{\exp}_{Q_c} (\varphi)
     \geqslant K_2 (c) \| \varphi \|_{H^{\exp}_{Q_c}}^2 . \]
\end{proposition}

Here, the orthogonality condition on $i Q_c$ is around $0$, between the two
vortices, but it can be chosen near one of the vortices for instance, and the
result still holds.

\

The second possibility is to work with symmetric perturbations, since the
orthogonality condition can then be at both the zeros of $Q_c$. We then study
the space
\[ H^{\exp, s}_{Q_c} \assign \{ \varphi \in H^{\exp}_{Q_c}, \forall x = (x_1,
   x_2) \in \mathbbm{R}^2, \varphi (x_1, x_2) = \varphi (- x_1, x_2) \} . \]
We show that, under three orthogonality conditions, the quadratic form is
equivalent to the norm on $H^{\exp}_{Q_c}$.

\begin{theorem}
  \label{CP2p41}There exist $R, K, c_0 > 0$ such that, for $0 < c \leqslant
  c_0$, $Q_c$ defined in Theorem \ref{th1}, if a function $\varphi \in
  H^{\exp, s}_{Q_c}$ satisfies the three orthogonality conditions:
  \[ \mathfrak{R}\mathfrak{e} \int_{B (\tilde{d}_c \overrightarrow{e_1}, R)
     \cup B (- \tilde{d}_c \overrightarrow{e_1}, R)} \partial_c Q_c
     \bar{\varphi} =\mathfrak{R}\mathfrak{e} \int_{B (\tilde{d}_c
     \overrightarrow{e_1}, R) \cup B (- \tilde{d}_c \overrightarrow{e_1}, R)}
     \partial_{x_2} Q_c \bar{\varphi} = 0, \]
  \[ \mathfrak{R}\mathfrak{e} \int_{B (\tilde{d}_c \overrightarrow{e_1}, R)
     \cup B (- \tilde{d}_c \overrightarrow{e_1}, R)} i Q_c \bar{\varphi} = 0,
  \]
  then
  \[ \frac{1}{K} \| \varphi \|^2_{H^{\exp}_{Q_c}} \geqslant B^{\exp}_{Q_c}
     (\varphi) \geqslant K \| \varphi \|^2_{H^{\exp}_{Q_c}} . \]
\end{theorem}

We remark that here, the orthogonality condition to $\partial_{x_1} Q_c$ and
$\partial_{c^{\bot}} Q_c$ are freely given by the symmetry. We also do not
need to remove the $0$-harmonic near the zeros of $Q_c$.

Propositions \ref{CP2prop17} and Theorem \ref{CP2p41} hold if we replace
$B^{\exp}_{Q_c}$ by $B_{Q_c}$ for $\varphi = Q_c \psi \in H_{Q_c}$ with the
symmetry, but the coercivity norm will still be $\| . \|_{H^{\exp}_{Q_c}}$.

\subsection{Local uniqueness result}

With Propositions \ref{CP2prop16} and \ref{CP2prop17}, we can modulate on the
five parameters $(\vec{c}, X, \gamma)$ of the travelling wave, and these
coercivity results will be enough to show the following theorem.

\begin{theorem}
  \label{CP2th16}There exist constants $K, c_0, \varepsilon_0, \mu_0 > 0$ such
  that, for $0 < c < c_0$, $Q_c$ defined in Theorem \ref{th1}, there exists
  $R_c > 0$ depending on $c$ such that, for any $\lambda > R_c$, if a function
  $Z \in C^2 (\mathbbm{R}^2, \mathbbm{C})$ satisfies, for some small constant
  $\varepsilon (c, \lambda) > 0$, depending on $c$ and $\lambda$,
  \begin{itemizeminus}
    \item $(\tmop{TW}_c) (Z) = 0$
    
    \item $E (Z) < + \infty$
    
    \item $\| Z - Q_c \|_{C^1 (\mathbbm{R}^2 \backslash B (0, \lambda))}
    \leqslant \mu_0$
    
    \item $\| Z - Q_c \|_{H^{\exp}_{Q_c}} \leqslant \varepsilon (c, \lambda)$,
  \end{itemizeminus}
  then, there exists $X \in \mathbbm{R}^2$ such that $| X | \leqslant K \| Z -
  Q_c \|_{H^{\exp}_{Q_c}}$, and
  \[ Z = Q_c (. - X) . \]
\end{theorem}

The conditions $E (Z) < + \infty$ and $\| Z - Q_c \|_{H^{\exp}_{Q_c}}
\leqslant \varepsilon (c, \lambda)$ imply that the travelling wave $Z
\rightarrow 1$ at infinity, and therefore $Z = Q_c e^{i \gamma}$ with $\gamma
\in \mathbbm{R}, \gamma \neq 0$ is excluded. The fact that $\varepsilon (c,
\lambda)$ depends on $c$ comes in part from the constant of coercivity in
Proposition \ref{CP2prop17}, which depends itself on $c$. The condition that
$\| Z - Q_c \|_{C^1 (\mathbbm{R}^2 \backslash B (0, \lambda))} \leqslant
\mu_0$ outside of $B (0, \lambda)$ is mainly technical. We believe that this
condition is automatically satisfied with the other ones (with $\lambda$
depending only on $c$), but we were not able to show it.

To the best of our knowledge, this is the first result of local uniqueness for
travelling waves in $(\tmop{GP})$. It does not suppose any symmetries on $Z$,
and therefore shows that we can not bifurcate from this branch, even to
nonsymmetric travelling waves.

We believe that, at least in the symmetric case, Theorem \ref{CP2th16} should
hold for $\| Z - Q_c \|_{H^{\exp}_{Q_c}} \leqslant \varepsilon$ with
$\varepsilon > 0$ independent of $c$ and $\lambda$. We also remark that the
condition $\| Z - Q_c \|_{H^{\exp}_{Q_c}} \leqslant \varepsilon (c, \lambda)$
is weaker than $\| Z - Q_c \|_{H_{Q_c}} \leqslant \varepsilon (c, \lambda)$,
and thus we can state a result in $H_{Q_c}$.

\subsection{Plan of the proofs}

Section \ref{CP2sec2N} is devoted to the proof of Proposition \ref{CP2prop5}.
We start by giving some estimates on the branch of travelling waves in
subsection \ref{CP2Ndecay}, we then show the equivalents when $c \rightarrow
0$ for the energy and momentum, as well as the relations between them and some
specific values of the quadratic form in subsection \ref{CP2s31}. Finally, in
subsection \ref{CP205s1}, we study the travelling wave near its zeros.

In section \ref{CP2NSEC4}, we infer some properties of the space $H_{Q_c}$.
First, we explain why we can not have a coercivity result in the energy norm
in subsection \ref{CP2CandHc}, and we show the well posedness of several
quantities in subsections \ref{CP2Nrouge} and \ref{CP22241103}. A density
argument is given in subsection \ref{CP2density}, that will be needed for the
proof of Proposition \ref{CP205218}.

Section \ref{CP2NSEC3} is devoted to the proofs of Propositions \ref{CP2P811}
and \ref{CP205218}. We start by writing the quadratic form for test functions
in a particular form (subsection \ref{CP21202CompQc}), and we then show
Proposition \ref{CP2P811} and \ref{CP205218} respectively in subsections
\ref{CP2s24} and \ref{CP2concl}. To show Proposition \ref{CP205218}, we use
Proposition \ref{CP2P811} and the fact that we know well the travelling wave
near its zeros from subsection \ref{CP205s1}.

The next part, section \ref{CP2s32}, is devoted to the proof of Theorem
\ref{CP2th2} and its corollaries. We show the coercivity under four
orthogonality conditions by showing that we can modify the initial function by
a small amount to have the four orthogonality conditions of Proposition
\ref{CP205218}, and that the error commited is small in the coercivity norm.
We then focus on the corollaries of Theorem \ref{CP2th2} in subsection
\ref{CP2proofcor}. We show there composition of the kernel of $L_{Q_c}$
(Corollary \ref{CP2Cor41}), and the results in $H^1 (\mathbbm{R}^2)$:
Corollary \ref{CP2cor177}, Proposition \ref{CP2p188} and Corollary
\ref{CP2th199}.

The penultimate section (\ref{CP2ss41}) is devoted to the proofs of
Proposition \ref{CP2prop16}, \ref{CP2prop17} and Theorem \ref{CP2p41}. In
subsection \ref{CP2HQc911}, we study the space $H_{Q_c}^{\exp}$, in particular
we give a density argument, that allows us to finish the proof of Proposition
\ref{CP2prop16}. Then, in subsection \ref{CP2ss62}, we compute how the
additional orthogonality condition improves the coercivity norm, both in the
symmetric and non symmetric case, and we can then show Proposition
\ref{CP2prop17} and Theorem \ref{CP2p41}.

Section \ref{CP2ss42} is devoted to the proof of Theorem \ref{CP2th16}. We use
here classical methods for the proof of local uniqueness, by modulating on the
five parameters of the family, and using a coercivity result. One of the main
point is to write the problem additively near the zeros of $Q_c$ and
multiplicatively far from them. The reason for that is that we do not know the
link between the speed and the position of the zeros of a travelling wave in
general, and we therefore cannot write a perturbation multiplicatively in the
whole space. Because of that, we require here an orthogonality on the phase,
and we cannot avoid it, as we did for instance the proof of Proposition
\ref{CP205218} by choosing correctly the position of the vortices.

\

We will use many cutoffs in the proofs. As a rule of thumb, a function
written as $\eta, \chi$ or $\tilde{\chi}$ will be smooth and have value $1$ at
infinity and $0$ in some compact domain. The function $\eta$ itself is
reserved for $B_{Q_c}$ and $B^{\exp}_{Q_c}$ (see equations (\ref{CP2truebqc})
and (\ref{CP2Btilda})).

{\noindent}\tmtextbf{Acknowledgments . }The authors would like to thank Pierre
Rapha{\"e}l for helpful discussions. E.P. is supported by the ERC-2014-CoG
646650 SingWave.{\hspace*{\fill}}{\medskip}

\section{Properties of the branch of travelling waves}\label{CP2sec2N}

This section is devoted to the proof of Proposition \ref{CP2prop5}. In
subsection \ref{CP2Ndecay}, we recall some estimates on $Q_c$ defined in
Theorem \ref{th1} from previous works ({\cite{CX}}, {\cite{CP1}},
{\cite{MR2086751}} and {\cite{HH}}). In subsection \ref{CP2s31}, we compute
some equalities and equivalents when $c \rightarrow 0$ on the energy, momentum
and the four particular directions ($\partial_{x_1} Q_c, \partial_{x_2} Q_c,
\partial_c Q_c$ and $\partial_{c^{\bot}} Q_c$). Finally, the properties of the
zeros of $Q_c$ are studied in subsection \ref{CP205s1}.

\subsection{Decay estimates}\label{CP2Ndecay}

\subsubsection{Estimates on vortices}\label{CP2Nvor}

We recall that vortices are stationary solutions of $(\tmop{GP})$ of degrees
$n \in \mathbbm{Z}^{\ast}$ (see {\cite{CX}}):
\[ V_n (x) = \rho_n (r) e^{i n \theta}, \]
where $x = r e^{i \theta}$, solving
\[ \left\{ \begin{array}{l}
     \Delta V_n - (| V_n |^2 - 1) V_n = 0\\
     | V_n | \rightarrow 1 \tmop{as} | x | \rightarrow \infty .
   \end{array} \right. \]
We regroup here estimates on quantities involving vortices. \ We start with
estimates on $V_{\pm 1}$.

\begin{lemma}[{\cite{CX}} and {\cite{HH}}]
  \label{lemme3new}A vortex centered around $0$, $V_1 (x) = \rho_1 (r) e^{i
  \theta}$, verifies $V_1 (0) = 0$, and there exist constants $K, \kappa > 0$
  such that
  \[ \forall r > 0, 0 < \rho_1 (r) < 1, \rho_1 (r) \sim_{r \rightarrow 0}
     \kappa r, \rho'_1 (r) \sim_{r \rightarrow 0} \kappa \]
  \[ \rho_1' (r) > 0 ; \rho_1' (r) = O_{r \rightarrow \infty} \left(
     \frac{1}{r^3} \right), | \rho'' (r) | + | \rho''' (r) | \leqslant K, \]
  \[ 1 - | V_1 (x) | = \frac{1}{2 r^2} + O_{r \rightarrow \infty} \left(
     \frac{1}{r^3} \right), \]
  \[ | \nabla V_1 | \leqslant \frac{K}{1 + r}, | \nabla^2 V_1 | \leqslant
     \frac{K}{(1 + r)^2} \]
  and
  \[ \nabla V_1 (x) = i V_1 (x) \frac{x^{\bot}}{r^2} + O_{r \rightarrow
     \infty} \left( \frac{1}{r^3} \right), \]
  where $x^{\perp} = (- x_2, x_1)$, $x = r e^{i \theta} \in \mathbbm{R}^2$.
  Furthermore, similar properties holds for $V_{- 1}$, since
  \[ V_{- 1} (x) = \overline{V_1 (x)_{}} . \]
\end{lemma}

We also define, as in {\cite{CP1}},
\[ V (.) \assign V_1 (. - d_c \overrightarrow{e_1}) V_{- 1} (. + d_c
   \overrightarrow{e_1}) \]
and
\[ \partial_d V (.) \assign \partial_d (V_1 (. - d \overrightarrow{e_1}) V_{-
   1} (. + d \overrightarrow{e_1}))_{| d = d_c \nobracket} . \]
We will also estimate
\[ \partial_d^2 V \assign \partial_d^2 (V_1 (. - d \overrightarrow{e_1}) V_{-
   1} (. + d \overrightarrow{e_1}))_{| d = d_c \nobracket} . \]
The function $V (x) = V_1 (x - d_c \overrightarrow{e_1}) V_{- 1} (x + d_c
\overrightarrow{e_1})$ is close to $V_1 (x - d_c \overrightarrow{e_1})$ in $B
(d_c \overrightarrow{e_1}, 2 d_c^{1 / 2})$, since, from Lemma \ref{lemme3new}
and {\cite{CX}}, we have, uniformly in $B (d_c \overrightarrow{e_1}, 2 d_c^{1
/ 2})$,
\begin{equation}
  V_{- 1} (. + d_c \overrightarrow{e_1}) = 1 + O_{c \rightarrow 0} (c^{1 / 2})
  \label{CP2222405}
\end{equation}
and
\begin{equation}
  | \nabla V_{- 1} (. + d_c \overrightarrow{e_1}) | \leqslant \frac{o_{c
  \rightarrow 0} (c^{1 / 2})}{(1 + \widetilde{r_1})} \label{CP2V-1est} .
\end{equation}
We recall that $B (d_c \overrightarrow{e_1}, 2 d_c^{1 / 2})$ is near the
vortex of degree $+ 1$ of $Q_c$ and that $\tilde{r} = \min (r_1, r_{- 1})$,
with $r_{\pm 1} = | x \mp d_c \overrightarrow{e_1} |$.

\subsubsection{Estimates on $Q_c$ from {\cite{CP1}}}\label{CP2NQcest}

We recall, for the function $Q_c$ defined in Theorem \ref{th1}, that
\begin{equation}
  \forall (x_1, x_2) \in \mathbbm{R}^2, Q_c (x_1, x_2) = \overline{Q_c (x_1, -
  x_2)} = Q_c (- x_1, x_2) . \label{CP2sym}
\end{equation}
In particular, $\partial_c Q_c$ enjoys the same symmetries, since
(\ref{CP2sym}) holds for any $c > 0$ small enough. We recall that $Q_c \in
C^{\infty} (\mathbbm{R}^2, \mathbbm{C})$ by standard elliptic regularity
arguments.

Finally, we recall some estimates on $Q_c$ and its derivatives, coming from
Lemma 3.8 of {\cite{CP1}}. We denote $\tilde{r} = \min (r_1, r_{- 1})$, the
minimum of the distances to $d_c \vec{e}_1$ and $- d_c \vec{e}_1$, and we
recall that $V (x) = V_1 (x - d_c \overrightarrow{e_1}) V_{- 1} (x + d_c
\overrightarrow{e_1})$.

We write $Q_c = V + \Gamma_c$ or $Q_c = (1 - \eta) V \Psi_c + \eta V
e^{\Psi_c}$, where $\Gamma_c = (1 - \eta) V \Psi_c + \eta V (e^{\Psi_c} - 1)$
(see equation (3.8) of {\cite{CP1}}). There exists $K > 0$ and, for any $0 <
\sigma < 1$, there exists $K (\sigma) > 0$ such that
\begin{equation}
  | \Gamma_c | \leqslant \frac{K (\sigma) c^{1 - \sigma}}{(1 +
  \tilde{r})^{\sigma}} \label{CP2labasique}
\end{equation}
\begin{equation}
  | \nabla \Gamma_c | \leqslant \frac{K (\sigma) c^{1 - \sigma}}{(1 +
  \tilde{r})^{1 + \sigma}} . \label{CP2225}
\end{equation}
\begin{equation}
  | 1 - | Q_c | | \leqslant \frac{K (\sigma)}{(1 + \tilde{r})^{1 + \sigma}},
  \label{CP2217}
\end{equation}
\begin{equation}
  | Q_c - V | \leqslant \frac{K (\sigma) c^{1 - \sigma}}{(1 +
  \tilde{r})^{\sigma}}, \label{CP2218}
\end{equation}
\begin{equation}
  | | Q_c |^2 - | V |^2 | \leqslant \frac{K (\sigma) c^{1 - \sigma}}{(1 +
  \tilde{r})^{1 + \sigma}}, \label{CP2219}
\end{equation}
\begin{equation}
  | \mathfrak{R}\mathfrak{e} (\nabla Q_c \overline{Q_c}) | \leqslant \frac{K
  (\sigma)}{(1 + \tilde{r})^{2 + \sigma}}, \label{CP2220}
\end{equation}
\begin{equation}
  | \mathfrak{I}\mathfrak{m} (\nabla Q_c \overline{Q_c}) | \leqslant
  \frac{K}{1 + \tilde{r}} \label{CP2221},
\end{equation}
and for $0 < \sigma < \sigma' < 1$, there exists $K (\sigma, \sigma') > 0$
such that
\begin{equation}
  | D^2 \mathfrak{I}\mathfrak{m} (\Psi_c) | + | \nabla
  \mathfrak{R}\mathfrak{e} (\Psi_c) | + | \nabla^2 \mathfrak{R}\mathfrak{e}
  (\Psi_c) | \leqslant \frac{K (\sigma, \sigma') c^{1 - \sigma'}}{(1 +
  \tilde{r})^{2 + \sigma}} \label{CP2222} .
\end{equation}
From Lemmas \ref{lemme3new}, with Theorem \ref{th1}, we deduce in particular
that for $c$ small enough, there exist universal constants $K_1, K_2 > 0$ such
that on $\mathbbm{R}^2 \backslash B (\pm d_c \overrightarrow{e_1}, 1)$ we have
\begin{equation}
  K_1 \leqslant | Q_c | \leqslant K_2 . \label{CP2Qcpaszero}
\end{equation}
To these estimates, we add two additional lemmas. We write
\[ \begin{array}{lll}
     \| \psi \|_{\sigma, d_c} & \assign & \| V \psi \|_{C^1 (\{ \tilde{r}
     \leqslant 3 \})} + \| \tilde{r}^{1 + \sigma} \mathfrak{R}\mathfrak{e}
     (\psi) \|_{L^{\infty} (\{ \tilde{r} \geqslant 2 \})} + \| \tilde{r}^{2 +
     \sigma} \nabla \mathfrak{R}\mathfrak{e} (\psi) \|_{L^{\infty} (\{
     \tilde{r} \geqslant 2 \})}\\
     & + & \| \tilde{r}^{\sigma} \mathfrak{I}\mathfrak{m} (\psi)
     \|_{L^{\infty} (\{ \tilde{r} \geqslant 2 \})} + \| \tilde{r}^{1 + \sigma}
     \nabla \mathfrak{I}\mathfrak{m} (\psi) \|_{L^{\infty} (\{ \tilde{r}
     \geqslant 2 \})},
   \end{array} \]
where $\tilde{r} = \min (r_1, r_{- 1})$, with
\begin{equation}
  r_{\pm 1} = | x \mp d_c \overrightarrow{e_1} | \label{CP2223not},
\end{equation}
and with $d_c$ defined in Theorem \ref{th1}. The first lemma is about $Q_c$
and the second one about $\partial_c Q_c$.

\begin{lemma}
  \label{CP283L33}For any $0 < \sigma < 1$, there exist $c_0 (\sigma), K
  (\sigma) > 0$ such that, for $0 < c < c_0 (\sigma)$ and $Q_c$ defined in
  Theorem \ref{th1}, if
  \[ \Gamma_c = Q_c - V, \]
  then
  \[ \left\| \frac{\Gamma_c}{V} \right\|_{\sigma, d_c} \leqslant K (\sigma)
     c^{1 - \sigma} . \]
\end{lemma}

\begin{proof}
  This estimate is a consequence of
  \[ \Gamma_c = (1 - \eta) V \Psi_c + \eta V (e^{\Psi_c} - 1) \]
  and equation (3.14) of {\cite{CP1}}.
\end{proof}

\begin{lemma}[Lemma 4.6 of {\cite{CP1}}]
  \label{CP2dcQcsigma}There exists $1 > \beta_0 > 0$ such that, for all $0 <
  \sigma < \beta_0 < \sigma' < 1$,There exists $c_0 (\sigma, \sigma') > 0$
  such that for any $0 < c < c_0 (\sigma, \sigma')$, $Q_c$ defined in Theorem
  \ref{th1}, $c \mapsto Q_c$ is a $C^1$ function from $] 0, c_0 (\sigma,
  \sigma') [$ to $C^1 (\mathbbm{R}^2, \mathbbm{C})$, and
  \[ \left\| \frac{\partial_c Q_c}{V} + \left( \frac{1 + o^{\sigma,
     \sigma'}_{c \rightarrow 0} (c^{1 - \sigma'})}{c^2} \right)
     \frac{\partial_d V_{| d = d_c \nobracket}}{V} \right\|_{\sigma, d_c} =
     o^{\sigma, \sigma'}_{c \rightarrow 0} \left( \frac{c^{1 - \sigma'}}{c^2}
     \right) . \]
\end{lemma}

These results are technical, but quite precise. They give both a decay in
position and the size in $c$ of the error term. The statement of Lemma 4.6 of
{\cite{CP1}} has $o_{c \rightarrow 0} (1)$ and $o_{c \rightarrow 0} \left(
\frac{1}{c^2} \right)$ instead of respectively $o_{c \rightarrow 0} (c^{1 -
\sigma'})$ and $o_{c \rightarrow 0} \left( \frac{c^{1 - \sigma'}}{c^2}
\right)$, but its proof gives this better estimate (given that $\sigma'$ is
close enough to $1$). We recall that $o^{\sigma, \sigma'}_{c \rightarrow 0}
(1)$ is a quantity going to $0$ when $c \rightarrow 0$ at fixed $\sigma,
\sigma'$. We recall that $\partial_c \nabla Q_c = \nabla \partial_c Q_c$. We
conclude this subsection with a link between the $\| . \|_{\sigma}$ norms and
$\| . \|_{H_{Q_c}}$. We recall
\[ \| \varphi \|_{H_{Q_c}}^2 = \int_{\mathbbm{R}^2} | \nabla \varphi |^2 + | 1
   - | Q_c^{\nosymbol} |^2 | | \varphi |^2 +\mathfrak{R}\mathfrak{e}^2
   (\overline{Q_c} \varphi) . \]
\begin{lemma}
  \label{CP2123L34}There exists a universal constant $K > 0$ (independent of
  $c$) such that, for $Q_c$ defined in Theorem \ref{th1},
  \[ \| h \|_{H_{Q_c}} \leqslant K \left\| \frac{h}{V} \right\|_{3 / 4, d_c} .
  \]
\end{lemma}

The value $\sigma = 3 / 4$ is arbitrary here, this estimate holds for other
values of $\sigma$.

\begin{proof}
  We compute, using Lemma \ref{lemme3new}, that
  \[ \int_{\mathbbm{R}^2} | \nabla h |^2 \leqslant K \left\| \frac{h}{V}
     \right\|^2_{3 / 4, d_c} + \int_{\{ \tilde{r} \geqslant 1 \}} \left|
     \nabla \left( \frac{h}{V} V \right) \right|^2 \leqslant K \left\|
     \frac{h}{V} \right\|^2_{3 / 4, d_c} + 2 \int_{\{ \tilde{r} \geqslant 1
     \}} \left| \nabla \left( \frac{h}{V} \right) \right|^2 + | \nabla V |^2
     \frac{| h |^2}{| V |^2} . \]
  From Lemma \ref{lemme3new} and the definition of $\| . \|_{3 / 4, d_c}$, we
  check that
  \[ 2 \int_{\{ \tilde{r} \geqslant 1 \}} \left| \nabla \left( \frac{h}{V}
     \right) \right|^2 + | \nabla V |^2 \frac{| h |^2}{| V |^2} \leqslant K
     \left\| \frac{h}{V} \right\|^2_{3 / 4, d_c} \int_{\{ \tilde{r} \geqslant
     1 \}} \frac{1}{(1 + \tilde{r})^{3 + 1 / 2}} \leqslant K \left\|
     \frac{h}{V} \right\|^2_{3 / 4, d_c} . \]
  Furthermore, from equation (\ref{CP2217}) with $\sigma = 1 / 2$, we have the
  estimate
  \[ \int_{\mathbbm{R}^2} | 1 - | Q_c |^2 | | h |^2 \leqslant K \left\|
     \frac{h}{V} \right\|^2_{3 / 4, d_c} \int_{\mathbbm{R}^2} \frac{1}{(1 +
     \tilde{r})^{9 / 4}} \leqslant K \left\| \frac{h}{V} \right\|^2_{3 / 4,
     d_c} . \]
  Finally, we compute
  \[ \int_{\mathbbm{R}^2} \mathfrak{R}\mathfrak{e}^2 (\overline{Q_c} h)
     \leqslant K \left\| \frac{h}{V} \right\|^2_{3 / 4, d_c} + \int_{\{
     \tilde{r} \geqslant 1 \}} \mathfrak{R}\mathfrak{e}^2 (\overline{Q_c} h),
  \]
  and
  \[ \int_{\{ \tilde{r} \geqslant 1 \}} \mathfrak{R}\mathfrak{e}^2
     (\overline{Q_c} h) = \int_{\{ \tilde{r} \geqslant 1 \}}
     \mathfrak{R}\mathfrak{e}^2 \left( V \overline{Q_c} \frac{h}{V} \right)
     \leqslant 2 \int_{\{ \tilde{r} \geqslant 1 \}} \mathfrak{R}\mathfrak{e}^2
     \left( \frac{h}{V} \right) \mathfrak{R}\mathfrak{e}^2 (V \overline{Q_c})
     +\mathfrak{I}\mathfrak{m}^2 \left( \frac{h}{V} \right)
     \mathfrak{I}\mathfrak{m}^2 (V \overline{Q_c}) . \]
  With the definition of $\| . \|_{3 / 4, d_c}$, Lemmas \ref{lemme3new} and
  \ref{CP283L33}, we check that
  \[ \int_{\{ \tilde{r} \geqslant 1 \}} \mathfrak{R}\mathfrak{e}^2 \left(
     \frac{h}{V} \right) \mathfrak{R}\mathfrak{e}^2 (V \overline{Q_c})
     \leqslant K \int_{\{ \tilde{r} \geqslant 1 \}} \mathfrak{R}\mathfrak{e}^2
     \left( \frac{h}{V} \right) \leqslant K \left\| \frac{h}{V} \right\|^2_{3
     / 4, d_c} \int_{\{ \tilde{r} \geqslant 1 \}} \frac{1}{(1 + \tilde{r})^{3
     + 1 / 2}} \leqslant K \left\| \frac{h}{V} \right\|^2_{3 / 4, d_c} . \]
  From Lemma \ref{CP283L33} with $\sigma = 1 / 2$, we check that, since
  $\mathfrak{I}\mathfrak{m}^2 (V \overline{Q_c}) =\mathfrak{I}\mathfrak{m}^2
  (V \overline{V + \Gamma_c}) =\mathfrak{I}\mathfrak{m}^2 (V \bar{\Gamma}_c)$,
  we have
  \[ \int_{\{ \tilde{r} \geqslant 1 \}} \mathfrak{I}\mathfrak{m}^2 \left(
     \frac{h}{V} \right) \mathfrak{I}\mathfrak{m}^2 (V \overline{Q_c})
     \leqslant K \left\| \frac{h}{V} \right\|^2_{3 / 4, d_c} \int_{\{
     \tilde{r} \geqslant 1 \}} \frac{1}{(1 + \tilde{r})^{2 + 1 / 2}} \leqslant
     K \left\| \frac{h}{V} \right\|^2_{3 / 4, d_c} . \]
  Combining, these estimates, we end the proof of this lemma.
\end{proof}

\subsubsection{Faraway estimates on $Q_c$}

Since $E (Q_c) < + \infty$ thanks to Theorem \ref{th1}, from Theorem 7 of
{\cite{MR2086751}}, we have the following result.

\begin{theorem}[{\cite{MR2086751}}, Theorem 7]
  \label{CP2Qcbehav}There exists a constant $C_{\nosymbol} (c) > 0$ (depending
  on $c$) such that, for $Q_c$ defined in Theorem \ref{th1},
  \[ | 1 - | Q_c |^2 | \leqslant \frac{C (c)}{(1 + r)^2}, \]
  \[ | 1 - Q_c | \leqslant \frac{C_{} (c)_{\nosymbol}}{1 + r}, \]
  \[ | \nabla Q_c | \leqslant \frac{C_{} (c)}{(1 + r)^2} \]
  and
  \[ | \nabla | Q_c | | \leqslant \frac{C (c)}{(1 + r)^3} . \]
  Furthermore, such estimates hold for any travelling waves with finite energy
  (but then the constant $C (c)$ also depends on the travelling wave, and not
  only on its speed).
\end{theorem}

This result is crucial to show that some terms are well defined, since it
gives better decay estimates in position than the estimates in subsection
\ref{CP2NQcest} (but with no smallness in $c$). Remark that $1 - | Q_c |^2$ is
not necessarily positive. In fact it is not at infinity (see
{\cite{MR2191764}}). In particular, the estimate
\[ | 1 - | Q_c |^2 | \geqslant \frac{C (c)}{1 + r^2} \]
does not hold because of the possibility of $| Q_c | = 1$. This happens, but
only for few directions and it can be catched up. We show the following
sufficient result, which is needed to show that some quantities we will use
are well defined. Furthermore, in these estimates, the constant depends on
$c$, and thus can not be used in error estimates (since the smallness of the
errors there will depend on $c$).

\begin{lemma}
  \label{CP2sensmanqu}There exists $c_0 > 0$ such that, for $0 < c < c_0$,
  there exists $C (c) > 0$ such that for $\varphi \in H_{Q_c}$ and the
  function $Q_c$ defined in Theorem \ref{th1},
  \[ \int_{\mathbbm{R}^2} \frac{| \varphi |^2}{(1 + | x |)^2} d x \leqslant C
     (c) \left( \int_{\mathbbm{R}^2} | \nabla \varphi |^2 + | 1 - | Q_c |^2 |
     | \varphi |^2 \right) . \]
\end{lemma}

\begin{proof}
  From Propositions 5 and 7 of {\cite{MR2191764}} (where $\eta = 1 - | Q_c
  |^2$), we have in our case, for $x = r \sigma \in \mathbbm{R}^2$ with $r \in
  \mathbbm{R}^+, | \sigma | = 1$, $\sigma = (\sigma_1, \sigma_2) \in
  \mathbbm{R}^2$, that
  \[ r^2 (1 - | Q_c |^2) (r \sigma) \rightarrow c \alpha (c) \left( \frac{1}{1
     - \frac{c^2}{2} + \frac{c^2 \sigma_2^2}{2}} - \frac{2 \sigma_2^2}{\left(
     1 - \frac{c^2}{2} + \frac{c^2 \sigma^2_2}{2} \right)^2} \right) \]
  uniformly in $\sigma \in S^1$ when $r \rightarrow + \infty$, where $\alpha
  (c) > 0$ depends on $c$ and $Q_c$. Remark that our travelling wave is
  axisymmetric around axis $x_2$ (and not $x_1$ for which the results of
  {\cite{MR2191764}} are given), hence the swap between $\sigma_1$ and
  $\sigma_2$ between the two papers. We have
  \[ \frac{1}{1 - \frac{c^2}{2} + \frac{c^2 \sigma_2^2}{2}} - \frac{2
     \sigma_2^2}{\left( 1 - \frac{c^2}{2} + \frac{c^2 \sigma^2_2}{2}
     \right)^2} = \frac{1 - \frac{c^2}{2} - \left( 2 - \frac{c^2}{2} \right)
     \sigma_2^2}{\left( 1 - \frac{c^2}{2} + \frac{c^2 \sigma_2^2}{2}
     \right)^2}, \]
  this shows in particular that $| Q_c | = 1$ when $r \gg \frac{1}{c}$ is
  possible only in cones around $\sin (\theta) = \sigma_2 = \pm \sqrt{\frac{1
  - c^2 / 2}{2 - c^2 / 2}}$. Therefore, for $c$ small enough, for some $\gamma
  > 0$ small and $R > 0$ large (that may depend on $c$), we have
  \[ \int_{\mathbbm{R}^2} | 1 - | Q_c |^2 | | \varphi |^2 \geqslant K (c,
     \beta, R) \int_{\mathbbm{R}^2 \backslash (B (0, R) \cup D (\gamma))}
     \frac{| \varphi |^2}{(1 + r)^2}, \]
  where $D (\gamma) = \left\{ r e^{i \theta} \in \mathbbm{R}^2, \left| \sin
  (\theta) \pm \sqrt{\frac{1 - c^2 / 2}{2 - c^2 / 2}} \right| \leqslant \gamma
  \right\}$. We want to show that for $\varphi \in H_{Q_c}$,
  \[ \int_{D (\gamma) \cup (\mathbbm{R}^2 \backslash B (0, R))} \frac{|
     \varphi |^2}{(1 + r)^2} \leqslant C (c, \gamma, R) \left(
     \int_{\mathbbm{R}^2} | \nabla \varphi |^2 + \int_{\mathbbm{R}^2
     \backslash (B (0, R) \cup D (\gamma))} \frac{| \varphi |^2}{(1 + r)^2}
     \right) . \]
  For $\theta_0$ any of the four angles such that $\sin (\theta) \pm
  \sqrt{\frac{1 - c^2 / 2}{2 - c^2 / 2}} = 0$, we fix $r > 0$ and look at
  $\varphi (\theta)$ as a function of the angle only. We compute, for $\theta
  \in [\theta_0 - 2 \beta, \theta_0 + 2 \beta]$ ($\beta > 0$ being a small
  constant depending on $\gamma$ such that $\{ x = r e^{i \theta} \in
  \mathbbm{R}^2, \theta \in [\theta_0 + 3 \beta, \theta_0 + \beta] \} \cap D
  (\gamma) = \emptyset$, and such that $D (\gamma)$ is included in the union
  of the $[\theta_0 - \beta, \theta_0 + \beta]$ for the four possible values
  of $\theta_0$),
  \[ \varphi (\theta) = \varphi (2 \beta + \theta) - \int_{\theta}^{2 \beta +
     \theta} \partial_{\theta} \varphi (\Theta) d \Theta, \]
  hence,
  \[ | \varphi (\theta) | \leqslant | \varphi (2 \beta + \theta) | +
     \int_{\theta_0 - \beta}^{\theta_0 + 3 \beta} | \partial_{\theta} \varphi
     (\Theta) | d \Theta . \]
  This implies that
  \[ | \varphi (\theta) |^2 \leqslant 2 | \varphi (2 \beta + \theta) |^2 + K
     \int_0^{2 \pi} | \partial_{\theta} \varphi (\Theta) |^2 d \Theta \]
  by Cauchy-Schwarz, and, integrating between $\theta_0 - \beta$ and $\theta_0
  + \beta$ yields
  \[ \int_{\theta_0 - \beta}^{\theta_0 + \beta} | \varphi (\theta) |^2 d
     \theta \leqslant 2 \int_{\theta_0 + \beta}^{\theta_0 + 3 \beta} | \varphi
     (\theta) |^2 d \theta + K \int_0^{2 \pi} | \partial_{\theta} \varphi
     (\theta) |^2 d \theta . \]
  Now multiplying by $\frac{r}{(1 + r)^2}$ and integrating in $r$ on $[R, +
  \infty [$, we infer
  \begin{eqnarray*}
    \int_{\theta - \theta_0 \in [- \beta, \beta]} \int_{r \in [R, + \infty [}
    \frac{| \varphi |^2}{(1 + r)^2} r d r d \theta & \leqslant & 2
    \int_{\theta - \theta_0 \in [\beta, 3 \beta]} \int_{r \in [R, + \infty [}
    \frac{| \varphi |^2}{(1 + r)^2} r d r d \theta\\
    & + & K (c, \beta, R) \int_{\mathbbm{R}^2} | \nabla \varphi |^2 d x\\
    & \leqslant & 2 \int_{\mathbbm{R}^2 \backslash (B (0, R) \cup D
    (\gamma))} \frac{| \varphi |^2 d x}{(1 + | x |)^2} + K (c, \beta, R)
    \int_{\mathbbm{R}^2} | \nabla \varphi |^2 d x,
  \end{eqnarray*}
  using
  \[ \frac{| \partial_{\theta} \varphi |^2}{(1 + r)^2} \leqslant \frac{|
     \partial_{\theta} \varphi |^2}{r^2} \leqslant | \nabla \varphi |^2 . \]
  Therefore,
  \[ \int_{D (\gamma) \cup (\mathbbm{R}^2 \backslash B (0, R))} \frac{|
     \varphi |^2}{(1 + r)^2} \leqslant K \int_{\mathbbm{R}^2 \backslash (B (0,
     R) \cup D (\gamma))} \frac{| \varphi |^2}{(1 + r)^2} d x + K (c, \beta,
     \gamma, R) \int_{\mathbbm{R}^2} | \nabla \varphi |^2 d x, \]
  and thus
  \[ \int_{\nobracket \mathbbm{R}^2 \backslash B (0, R \nobracket)} \frac{|
     \varphi |^2}{(1 + r)^2} \leqslant K (c, \beta, \gamma, R)
     \int_{\mathbbm{R}^2} | \nabla \varphi |^2 + | 1 - | Q_c |^2 | | \varphi
     |^2 . \]
  We are left with the proof of
  \begin{equation}
    \int_{\nobracket B (0, R \nobracket)} \frac{| \varphi |^2}{(1 + r)^2}
    \leqslant K (c, \beta, R) \left( \int_{\mathbbm{R}^2} | \nabla \varphi |^2
    + \int_{\nobracket \mathbbm{R}^2 \backslash B (0, R \nobracket)} \frac{|
    \varphi |^2}{(1 + r)^2} \right) . \label{CP22190110}
  \end{equation}
  We argue by contradiction. We suppose that there exists a sequence
  $\varphi_n \in H_{Q_c}$ such that $\int_{\nobracket B (0, R \nobracket)}
  \frac{| \varphi_n |^2}{(1 + r)^2} = 1$ and $\int_{\mathbbm{R}^2} | \nabla
  \varphi_n |^2 + \int_{\nobracket \mathbbm{R}^2 \backslash B (0, R
  \nobracket)} \frac{| \varphi_n |^2}{(1 + r)^2} \rightarrow 0$. Since
  $\varphi_n$ is bounded in $H^1 (B (0, R + 1))$, by Rellich's Theorem, up to
  a subsequence, we have the convergences $\varphi_n \rightarrow \varphi$
  strongly in $L^2$ and weakly in $H^1$ to some function $\varphi$ in $B (0, R
  + 1)$. In particular $\int_{B (0, R + 1)} | \nabla \varphi |^2 = 0$, hence
  $\varphi$ is constant on $B (0, R + 1)$, and with $\int_{\nobracket B (0, R
  + 1) \backslash B (0, R \nobracket)} \frac{| \varphi |^2}{(1 + r)^2} = 0$ we
  have $\varphi = 0$, which is in contradiction with \ $1 = \int_{\nobracket B
  (0, R \nobracket)} \frac{| \varphi_n |^2}{(1 + r)^2} \rightarrow
  \int_{\nobracket B (0, R \nobracket)} \frac{| \varphi |^2}{(1 + r)^2}$ by
  $L^2 (B (0, R + 1))$ strong convergence. This concludes the proof of this
  lemma.
\end{proof}

\subsection{Construction and properties of the four particular
directions}\label{CP2s31}

\subsubsection{Definitions}

The four directions we want to study here are $\partial_{x_1} Q_c,
\partial_{x_2} Q_c, \partial_c Q_c$ and $\partial_{c^{\bot}} Q_c$. The first
two are derivatives of $Q_c$ with respect to the position, the third one is
the derivative of $Q_c$ with respect of the speed, and we have its first order
term in Theorem \ref{th1}. The fourth direction is defined in Lemma \ref{CP2a}
below. The directions $\partial_{x_1} Q_c$ and $\partial_{x_2} Q_c$ correspond
to the translations of the travelling wave, $\partial_c Q_c$ and
$\partial_{c^{\bot}} Q_c$ to changes respectively in the modulus and direction
of its speed. These directions will also appear in the orthogonality
conditions for some of the coercivity results.

\begin{lemma}
  \label{CP2a}Take $\vec{c} \in \mathbbm{R}^2$ such that $| \vec{c} | < c_0$
  for $c_0$ defined in Theorem \ref{th1}. Define $\alpha$ such that $\vec{c} =
  | \vec{c} | R_{\alpha} (- \vec{e}_2)$, where $R_{\theta} : \mathbbm{R}^2
  \rightarrow \mathbbm{R}^2$ is the rotation of angle $\theta$. Then,
  $Q_{\vec{c}} \assign Q_{| \vec{c} |} \circ R_{- \alpha}$ solves
  \[ \left\{ \begin{array}{l}
       (\tmop{TW}_{\vec{c}}) (v) = i \vec{c} . \nabla v - \Delta v - (1 - | v
       |^2) v = 0\\
       | v | \rightarrow 1 \tmop{as} | x | \rightarrow + \infty,
     \end{array} \right. \]
  where $Q_{| \vec{c} |}$ is the solution of $(\tmop{TW}_{| \vec{c} |})$ in
  Theorem $\ref{th1}$. In particular, $Q_{\vec{c}}$ is a $C^1$ function of
  $\alpha$ and
  \[ \partial_{\alpha} Q_{\vec{c}} (x) = - R_{- \alpha} (x^{\bot}) . \nabla
     Q_{| \vec{c} |} (R_{- \alpha} (x)) . \]
  Furthermore, at $\alpha = 0$, the quantity
  \[ \partial_{c^{\bot}} Q_c \assign (\partial_{\alpha} Q_{\vec{c}})_{| \alpha
     = 0 \nobracket} \]
  satisfies
  \[ \partial_{c^{\bot}} Q_c (x) = - x^{\bot} . \nabla Q_c (x), \]
  is in $C^{\infty} (\mathbbm{R}^2, \mathbbm{C})$ and
  \[ L_{Q_c} (\partial_{c^{\bot}} Q_c) = - i c \partial_{x_1} Q_c . \]
\end{lemma}

\begin{proof}
  Since the Laplacian operator is invariant by rotation, it is easy to check
  that $Q_{| \vec{c} |} \circ R_{- \alpha}$ solves $(\tmop{TW}_{\vec{c}})
  (Q_{| \vec{c} |} \circ R_{- \alpha}) = 0$. The function $\theta \mapsto
  R_{\theta}$ is $C^1$, hence $(\alpha, x) \mapsto Q_{\vec{c}} (x)$ is a $C^1$
  function, and we compute
  \[ (\partial_{\alpha} Q_{\vec{c}}) (x) = \partial_{\alpha} (Q_{| \vec{c} |}
     \circ R_{- \alpha}) (x) = \partial_{\alpha} (R_{- \alpha} (x)) . \nabla
     Q_{| \vec{c} |} (R_{- \alpha} (x)) . \]
  We remark that
  \[ \partial_{\alpha} (R_{- \alpha} (x)) = - R_{- \alpha} (x^{\bot}), \]
  where $x^{\bot} = (- x_2, x_1)$, hence
  \[ \partial_{\alpha} Q_{\vec{c}} (x) = - R_{- \alpha} (x^{\bot}) . \nabla
     Q_{| \vec{c} |} (R_{- \alpha} (x)) . \]
  In particular, for $\alpha = 0$,
  \[ \partial_{\alpha} Q_{\vec{c}} (x)_{| \alpha = 0 \nobracket} = - x^{\bot}
     . \nabla Q_c (x) . \]
  We recall that $Q_{\vec{c}}$ solves
  \[ i \vec{c} . \nabla Q_{\vec{c}} - \Delta Q_{\vec{c}} - (1 - | Q_{\vec{c}}
     |^2) Q_{\vec{c}} = 0, \]
  and when we differentiate this equation with respect to $\alpha$ (with $|
  \vec{c} | = c$), we have
  \[ - i \partial_{\alpha} \vec{c} . (\nabla Q_{\vec{c}}) + L_{Q_{\vec{c}}}
     (\partial_{\alpha} Q_{\vec{c}}) = 0. \]
  At $\alpha = 0$, $Q_{\vec{c}} = Q_c$, $\partial_{\alpha} \vec{c} = - c
  \vec{e}_1$ and $\partial_{\alpha} Q_{\vec{c} | \alpha = 0 \nobracket} =
  \partial_{c^{\bot}} Q_c$, therefore
  \[ L_{Q_c} (\partial_{c^{\bot}} Q_c) = - i c \partial_{x_1} Q_c . \]
\end{proof}

\subsubsection{Estimates on the four directions}

We shall now show that the functions $\partial_{x_1} Q_c, \partial_{x_2} Q_c,
\partial_c Q_c$ and $\partial_{c^{\bot}} Q_c$ are in the energy space and we
will also compute their values through the linearized operator around $Q_c$,
namely
\[ L_{Q_c} (\varphi) = - \Delta \varphi - i c \partial_{x_2} \varphi - (1 - |
   Q_c |^2) \varphi + 2\mathfrak{R}\mathfrak{e} (\overline{Q_c} \varphi) Q_c .
\]
\begin{lemma}
  \label{CP20703L222}There exists $c_0 > 0$ such that, for $0 < c < c_0$,
  $Q_c$ defined in Theorem \ref{th1}, we have
  \[ \partial_{x_1} Q_c, \partial_{x_2} Q_c, \partial_c Q_c,
     \partial_{c^{\bot}} Q_c \in H_{Q_c}, \]
  and
  \[ L_{Q_c} (\partial_{x_1} Q_c) = L_{Q_c} (\partial_{x_2} Q_c) = 0, \]
  \[ L_{Q_c} (\partial_c Q_c) = i \partial_{x_2} Q_c, \]
  \[ L_{Q_c} (\partial_{c^{\bot}} Q_c) = - i c \partial_{x_1} Q_c . \]
\end{lemma}

We could check that we also have $\partial_{x_1} Q_c, \partial_{x_2} Q_c \in
H^1 (\mathbbm{R}^2)$ (see {\cite{MR2191764}}), but we expect that $\partial_c
Q_c, \partial_{c^{\bot}} Q_c \nin L^2 (\mathbbm{R}^2)$. For
$\partial_{c^{\bot}} Q_c$, this can be shown with Lemma \ref{CP2a} and
{\cite{MR2191764}}.

\begin{proof}
  We have defined
  \[ \| \varphi \|_{H_{Q_c}}^2 = \int_{\mathbbm{R}^2} | \nabla \varphi |^2 + |
     1 - | Q_c |^2 | | \varphi |^2 +\mathfrak{R}\mathfrak{e}^2 (\overline{Q_c}
     \varphi) . \]
  For any of the four functions, since they are in $C^{\infty} (\mathbbm{R}^2,
  \mathbbm{C})$, the only possible problem for the integrability is at
  infinity.

  \begin{tmindent}
    Step 1.  We have $\partial_{x_1} Q_c, \partial_{x_2} Q_c \in H_{Q_c}$.
  \end{tmindent}

  From Lemma \ref{lemme3new} and equation (\ref{CP2222}) (for $1 > \sigma' >
  \sigma = 3 / 4$), we have
  \[ \int_{\mathbbm{R}^2} | \nabla \partial_{x_1} Q_c |^2 +
     \int_{\mathbbm{R}^2} | \nabla \partial_{x_2} Q_c |^2 \leqslant
     \int_{\mathbbm{R}^2} \frac{K (c, \sigma')}{(1 + r)^{7 / 2}} < + \infty .
  \]
  From Theorem \ref{CP2Qcbehav}, we have
  \[ \int_{\mathbbm{R}^2} | 1 - | Q_c |^2 | | \nabla Q_c |^2
     +\mathfrak{R}\mathfrak{e}^2 (\overline{Q_c} \nabla Q_c) \leqslant
     \int_{\mathbbm{R}^2} \frac{K (c)}{(1 + r)^4} < + \infty, \]
  hence $\partial_{x_1} Q_c, \partial_{x_2} Q_c \in H_{Q_c}$.

  \begin{tmindent}
    Step 2.  We have $\partial_c Q_c \in H_{Q_c}$.
  \end{tmindent}

  From Lemmas \ref{CP2dcQcsigma} and \ref{CP2123L34}, we have that for $\sigma
  > 0$ small enough
  \[ \partial_c Q_c + \frac{1 + o^{\sigma}_{c \rightarrow 0}
     (c^{\sigma})}{c^2} \partial_d V_{| d = d_c \nobracket} \in H_{Q_c}, \]
  therefore we just have to check that $\partial_d V_{| d = d_c \nobracket}
  \in H_{Q_c}$, which is a direct consequence of Lemma 2.6 of {\cite{CP1}}.

  \begin{tmindent}
    Step 3.  We have $\partial_{c^{\bot}} Q_c \in H_{Q_c}$.
  \end{tmindent}

  From Lemma \ref{CP2a}, we have $\partial_{c^{\bot}} Q_c = - x^{\bot} .
  \nabla Q_c$. With Theorem \ref{CP2Qcbehav}, Lemma \ref{lemme3new} and
  equation (\ref{CP2222}), we check that
  \[ \int_{\mathbbm{R}^2} | \nabla \partial_{c^{\bot}} Q_c |^2 + | (1 - | Q_c
     |^2) | | \partial_{c^{\bot}} Q_c |^2 < + \infty . \]
  Now, from Lemma \ref{lemme3new} and equation (\ref{CP2217}) (with $\sigma =
  1 / 2$), we have
  \[ \int_{\mathbbm{R}^2} \mathfrak{R}\mathfrak{e}^2 (\overline{Q_c}
     \partial_{c^{\bot}} Q_c) \leqslant K \int_{\mathbbm{R}^2} (1 + r^2)
     \mathfrak{R}\mathfrak{e}^2 (\nabla Q_c \overline{Q_c}) \leqslant K (c)
     \int_{\mathbbm{R}^2} \frac{1}{(1 + r)^3} < + \infty, \]
  thus $\partial_{c^{\bot}} Q_c \in H_{Q_c}$.

  \begin{tmindent}
    Step 4.  Computation of the linearized operator on $\partial_{x_1} Q_c,
    \partial_{x_2} Q_c, \partial_c Q_c, \partial_{c^{\bot}} Q_c$.
  \end{tmindent}

  For the values in the linearized operator, since
  \[ - i c \partial_{x_2} Q_c - \Delta Q_c - (1 - | Q_c |^2) Q_c =
     (\tmop{TW}_c) (Q_c) = 0, \]
  by differentiating it with respect to $x_1$ and $x_2$, we have
  \[ L_{Q_c} (\partial_{x_1} Q_c) = L_{Q_c} (\partial_{x_2} Q_c) = 0. \]
  By differentiating it with respect to $c$, we have (we recall that
  $\partial_c Q_c \in C^{\infty} (\mathbbm{R}^2, \mathbbm{C})$)
  \[ - i \partial_{x_2} Q_c + L_{Q_c} (\partial_c Q_c) = 0. \]
  Finally, the quantity $L_{Q_c} (\partial_{c^{\bot}} Q_c)$ is given by Lemma
  \ref{CP2a}.
\end{proof}

The next two lemmas are additional estimates on the four directions that will
be useful later on. They estimate in particular the dependence on $c$ of $\| .
\|_{\mathcal{C}}$ on these four directions.

\begin{lemma}
  \label{CP283L223}There exists $K > 0$ a universal constant, independent of
  $c$, such that, for $Q_c$ defined in Theorem \ref{th1},
  \[ \| \partial_{x_1} Q_c \|_{\mathcal{C}} + \| \partial_{x_2} Q_c
     \|_{\mathcal{C}} + \| c^2 \partial_c Q_c \|_{\mathcal{C}} \leqslant K. \]
  Furthermore, for any $1 > \beta > 0$,
  \[ \| c \partial_{c^{\bot}} Q_c \|_{\mathcal{C}} = o_{c \rightarrow
     0}^{\beta} (c^{- \beta}) . \]
\end{lemma}

\begin{proof}
  We defined, for $\varphi = Q_c \psi \in H_{Q_c}$,
  \[ \| \varphi \|_{\mathcal{C}}^2 = \int_{\mathbbm{R}^2} | \nabla \psi |^2 |
     Q_c |^4 +\mathfrak{R}\mathfrak{e}^2 (\psi) | Q_c |^4 . \]
  We recall that, since $\varphi = Q_c \psi$,
  \begin{equation}
    \int_{\mathbbm{R}^2} | \nabla \psi |^2 | Q_c |^4 = \int_{\mathbbm{R}^2} |
    \nabla \varphi - \nabla Q_c \psi |^2 | Q_c |^2 \leqslant K
    \int_{\mathbbm{R}^2} | \nabla \varphi |^2 | Q_c |^2 + | \nabla Q_c |^2 |
    \varphi |^2 \label{CP2llabel4}
  \end{equation}
  
  \begin{tmindent}
    Step 1.  We have $\| \partial_{x_1} Q_c \|_{\mathcal{C}} + \|
    \partial_{x_2} Q_c \|_{\mathcal{C}} \leqslant K$.
  \end{tmindent}

  From Lemmas \ref{lemme3new} and \ref{CP283L33} and equations (\ref{CP2220})
  to (\ref{CP2222}), we have that, for $\tilde{r} = \min (r_1, r_{- 1})$,
  \[ | \nabla Q_c | \leqslant \frac{K}{(1 + \tilde{r})} \quad \tmop{and} \quad
     | \nabla^2 Q_c | \leqslant \frac{K}{(1 + \tilde{r})^2} . \]
  Therefore,
  \[ \int_{\mathbbm{R}^2} | \nabla (\partial_{x_1} Q_c) |^2 | Q_c |^2 + |
     \nabla (\partial_{x_2} Q_c) |^2 | Q_c |^2 \leqslant K, \]
  and we also have
  \[ \int_{\mathbbm{R}^2} | \nabla Q_c |^2 | \nabla Q_c |^2 \leqslant K, \]
  thus, with equation (\ref{CP2llabel4}),
  \[ \int_{\mathbbm{R}^2} \left| \nabla \left( \frac{\partial_{x_1} Q_c}{Q_c}
     \right) \right|^2 | Q_c |^4 + \int_{\mathbbm{R}^2} \left| \nabla \left(
     \frac{\partial_{x_2} Q_c}{Q_c} \right) \right|^2 | Q_c |^4 \leqslant K.
  \]
  By equation (\ref{CP2220}) (for $\sigma = 1 / 4$), we have
  \[ \int_{\mathbbm{R}^2} \mathfrak{R}\mathfrak{e}^2 \left( \frac{\nabla
     Q_c}{Q_c} \right) | Q_c |^4 \leqslant K \int_{\mathbbm{R}^2}
     \mathfrak{R}\mathfrak{e}^2 (\nabla Q_c \overline{Q_c}) \leqslant K
     \int_{\mathbbm{R}^2} \frac{1}{(1 + \tilde{r})^{5 / 2}} \leqslant K. \]
  We conclude that $\| \partial_{x_1} Q_c \|_{\mathcal{C}} + \| \partial_{x_2}
  Q_c \|_{\mathcal{C}} \leqslant K$.

  \begin{tmindent}
    Step 2.  We have $\| c^2 \partial_c Q_c \|_{\mathcal{C}} \leqslant K$.
  \end{tmindent}

  From Lemma \ref{CP2dcQcsigma}, we have, writing $c^2 \partial_c Q_c = (1 +
  o_{c \rightarrow 0} (1)) \partial_d V_{| d = d_c \nobracket} + h$, that
  $\left\| \frac{h}{V} \right\|_{\sigma, d_c} = o_{c \rightarrow 0} (1)$. In
  particular if we show that $\| \partial_d V_{| d = d_c \nobracket}
  \|_{\mathcal{C}} \leqslant K$ and $\| h \|_{\mathcal{C}} \leqslant K$, then
  $\| c^2 \partial_c Q_c \|_{\mathcal{C}} \leqslant K$. From Lemma 2.6 of
  {\cite{CP1}}, we check directly that
  \[ \int_{\mathbbm{R}^2} | \nabla \partial_d V_{| d = d_c \nobracket} |^2 +
     \frac{| \partial_d V_{| d = d_c \nobracket} |^2}{(1 + \tilde{r})^{3 / 2}}
     +\mathfrak{R}\mathfrak{e}^2 (V \partial_d V_{| d = d_c \nobracket})
     \leqslant K. \]
  In particular, with (\ref{CP2llabel4}), it implies that
  \[ \int_{\mathbbm{R}^2} \left| \nabla \left( \frac{\partial_d V_{| d = d_c
     \nobracket}}{Q_c} \right) \right|^2 | Q_c |^4 \leqslant K \]
  and we estimate
  \[ \int_{\mathbbm{R}^2} \mathfrak{R}\mathfrak{e}^2 \left( \frac{\partial_d
     V_{| d = d_c \nobracket}}{Q_c} \right) | Q_c |^4 \leqslant K
     \int_{\mathbbm{R}^2} \mathfrak{R}\mathfrak{e}^2 (\bar{V} \partial_d V_{|
     d = d_c \nobracket}) + | V - Q_c |^2 | \partial_d V_{| d = d_c
     \nobracket} |^2 \leqslant K \]
  with the same arguments and equation (\ref{CP2218}). Similarly,
  \[ \int_{\mathbbm{R}^2} \left| \nabla \frac{\partial_d V_{| d = d_c
     \nobracket}}{Q_c} \right|^2 | Q_c |^4 \leqslant 2 \int_{\mathbbm{R}^2} |
     \nabla \partial_d V_{| d = d_c \nobracket} |^2 | Q_c |^2 + | \nabla Q_c
     \partial_d V_{| d = d_c \nobracket} |^2 \leqslant K, \]
  therefore $\| \partial_d V_{| d = d_c \nobracket} \|_{\mathcal{C}} \leqslant
  K$. We now have to estimate $\| h \|_{\mathcal{C}}$. The computations are
  similar, since we check easily that
  \[ \int_{\mathbbm{R}^2} | \nabla h |^2 + | \nabla Q_c |^2 | h |^2 \leqslant
     K \left\| \frac{h}{V} \right\|^2_{3 / 4, d_c} \]
  and
  \[ \int_{\mathbbm{R}^2} \mathfrak{R}\mathfrak{e}^2 (\bar{Q}_c h) \leqslant K
     \int_{\mathbbm{R}^2} \mathfrak{R}\mathfrak{e}^2 (\bar{V} h) + | V - Q_c
     |^2 | h |^2 \leqslant K \left\| \frac{h}{V} \right\|^2_{3 / 4, d_c} . \]
  
  \begin{tmindent}
    Step 3.  We have $\| c \partial_{c^{\bot}} Q_c \|_{\mathcal{C}} = o_{c
    \rightarrow 0}^{\beta} (c^{- \beta})$.
  \end{tmindent}

  By definition, $c \partial_{c^{\bot}} Q_c = - c x^{\bot} . \nabla Q_c (x)$,
  and we check by triangular inequality that $c | x^{\bot} | \leqslant K (1 +
  \tilde{r})$ since $\tilde{r} = \min (| x - \widetilde{d_c}
  \overrightarrow{e_1} |, | x + \widetilde{d_c} \overrightarrow{e_1} |)$ and
  $c \tilde{d}_c \rightarrow 1$. Therefore,
  \[ \int_{\mathbbm{R}^2} | \nabla (c \partial_{c^{\bot}} Q_c) |^2 \leqslant
     c^2 \int_{\mathbbm{R}^2} | \nabla Q_c |^2 + \int_{\mathbbm{R}^2} (c |
     x^{\bot} |)^2 | \nabla^2 Q_c |^2 \leqslant K \left( 1 +
     \int_{\mathbbm{R}^2} | \nabla^2 Q_c |^2 (1 + \tilde{r})^2 \right) . \]
  We have $| \nabla^2 Q_c | \leqslant | \nabla^2 V | + | \nabla^2 \Gamma_c |$,
  and with equation (\ref{CP2222}), we check that $\int_{\mathbbm{R}^2} |
  \nabla^2 \Gamma_c |^2 (1 + \tilde{r})^2 \leqslant K$. With computations
  similar to the ones of Lemmas 2.3 of {\cite{CP1}} and \ref{lemme3new}, we
  can show that
  \[ | \nabla^2 V | \leqslant \frac{K}{(1 + \tilde{r})^2} \quad \tmop{and}
     \quad | \nabla^2 V | \leqslant \frac{K}{c (1 + \tilde{r})^3}, \]
  therefore, for any $1 > \beta > 0$,
  \[ | \nabla^2 V | \leqslant \frac{K c^{- \beta}}{(1 + \tilde{r})^{2 +
     \beta}}, \]
  and thus, by (\ref{CP2llabel4}),
  \[ \int_{\mathbbm{R}^2} \left| \nabla \left( \frac{c \partial_{c^{\bot}}
     Q_c}{Q_c} \right) \right|^2 | Q_c |^4 \leqslant K \int_{\mathbbm{R}^2} |
     \nabla c \partial_{c^{\bot}} Q_c |^2 | Q_c |^2 + | \nabla Q_c |^2 | c
     \partial_{c^{\bot}} Q_c |^2 \leqslant K (\beta) c^{- 2 \beta} . \]
  Furthermore, by equations (\ref{CP2220}) (for $\sigma = 1 / 2$) and
  (\ref{CP2Qcpaszero}), we have
  \[ \int_{\mathbbm{R}^2} \mathfrak{R}\mathfrak{e}^2 \left( \frac{c x^{\bot} .
     \nabla Q_c (x)}{Q_c} \right) | Q_c |^4 \leqslant K \int_{\mathbbm{R}^2}
     (1 + \tilde{r})^2 \mathfrak{R}\mathfrak{e}^2 (\nabla Q_c \overline{Q_c})
     \leqslant K \int_{\mathbbm{R}^2} \frac{1}{(1 + \tilde{r})^3} \leqslant K.
  \]
  We conclude that $\| c \partial_{c^{\bot}} Q_c \|_{\mathcal{C}} = o_{c
  \rightarrow 0}^{\beta} (c^{- \beta})$.
\end{proof}

\subsubsection{Link with the energy and momentum and computations of
equivalents}\label{CP2s33}

In this subsection, we compute the value of the four previous particular
direction $\partial_{x_1} Q_c, \partial_{x_2} Q_c, \partial_c Q_c,
\partial_{c^{\bot}} Q_c$ on the quadratic form. In particular, we shall show
that one of them is negative.

\begin{lemma}
  \label{CP2nend}There exists $c_0 > 0$ such that for $0 < c < c_0$, and for
  $Q_c$ defined in Theorem \ref{th1}, for $A \in \{ \partial_{x_1} Q_c,
  \partial_{x_2} Q_c, \partial_c Q_c, \partial_{c^{\bot}} Q_c \}$,
  $\mathfrak{R}\mathfrak{e} (L_{Q_c} (A) \bar{A}) \in L^1 (\mathbbm{R}^2)$ and
  \[ \langle L_{Q_c} (\partial_{x_1} Q_c), \partial_{x_1} Q_c \rangle =
     \langle L_{Q_c} (\partial_{x_2} Q_c), \partial_{x_2} Q_c \rangle = 0, \]
  \[ \langle L_{Q_c} (\partial_c Q_c), \partial_c Q_c \rangle = \frac{- 2 \pi
     + o_{c \rightarrow 0} (1)}{c^2}, \]
  \[ \langle L_{Q_c} (\partial_{c^{\bot}} Q_c), \partial_{c^{\bot}} Q_c
     \rangle = 2 \pi + o_{c \rightarrow 0} (1) . \]
\end{lemma}

\begin{proof}
  For $A \in \{ \partial_{x_1} Q_c, \partial_{x_2} Q_c, \partial_c Q_c,
  \partial_{c^{\bot}} Q_c \}$, we recall from Lemma \ref{CP20703L222} that $A
  \in H_{Q_c}$. To show that $\mathfrak{R}\mathfrak{e} (L_{Q_c} (A) \bar{A})
  \in L^1 (\mathbbm{R}^2)$, we need to show that
  \[ -\mathfrak{R}\mathfrak{e} (\Delta A \bar{A}) -\mathfrak{R}\mathfrak{e} (i
     c \partial_{x_2} A \bar{A}) - (1 - | Q_c |^2) | A |^2 +
     2\mathfrak{R}\mathfrak{e}^2 (\overline{Q_c} A) \in L^1 (\mathbbm{R}^2) .
  \]
  For that, we check that, for some $\sigma > 1 / 2$,
  \begin{eqnarray}
    &  & \| (1 + r)^{\sigma} A \|_{L^{\infty} (\mathbbm{R}^2)} + \| (1 +
    r)^{1 + \sigma} (| \nabla A | + | \mathfrak{R}\mathfrak{e} (A) |)
    \|_{L^{\infty} (\mathbbm{R}^2)} \nonumber\\
    & + & \| (1 + r)^{2 + \sigma} \mathfrak{I}\mathfrak{m} (\Delta A)
    \|_{L^{\infty} (\mathbbm{R}^2)} + \| (1 + r)^{1 + \sigma}
    \mathfrak{R}\mathfrak{e} (\Delta A) \|_{L^{\infty} (\mathbbm{R}^2)}
    \nonumber\\
    & < & + \infty .  \label{CP2N216}
  \end{eqnarray}
  For $\partial_{x_1} Q_c$ and $\partial_{x_2} Q_c$, this follows from Theorem
  \ref{CP2Qcbehav}, and, since $L_{Q_c} (\partial_{x_{1, 2}} Q_c) = 0$, from
  \[ \Delta (\partial_{x_{1, 2}} Q_c) = - i c \partial^2_{x_2 x_{1, 2}} Q_c -
     (1 - | Q_c |^2) \partial_{x_{1, 2}} Q_c + 2\mathfrak{R}\mathfrak{e}
     (\overline{Q_c} \partial_{x_{1, 2}} Q_c) Q_c, \]
  which allows to estimate $\Delta (\partial_{x_{1, 2}} Q_c)$ with Theorem
  \ref{CP2Qcbehav}, Lemma \ref{lemme3new} and equation (\ref{CP2222}) for any
  $\sigma > 1 / 2$.
  
  Now, for $\partial_c Q_c$, the estimates not on its Laplacian are a
  consequence of Lemma \ref{CP2dcQcsigma}, Theorem \ref{CP2Qcbehav} and Lemma
  2.6 of {\cite{CP1}}. Then, from Lemma \ref{CP20703L222}, we have $L_{Q_c}
  (\partial_c Q_c) = i \partial_{x_2} Q_c$, thus
  \[ \Delta (\partial_c Q_c) = - i \partial_{x_2} Q_c - i c \partial_{x_2}
     \partial_c Q_c - (1 - | Q_c |^2) \partial_c Q_c +
     2\mathfrak{R}\mathfrak{e} (\overline{Q_c} \partial_c Q_c) Q_c . \]
  By Theorem \ref{CP2Qcbehav} and Lemma \ref{CP2dcQcsigma}, we have, for any
  $\sigma > 1 / 2$,
  \[ | (1 - | Q_c |^2) \partial_c Q_c | + | 2\mathfrak{R}\mathfrak{e}
     (\overline{Q_c} \partial_c Q_c) Q_c | \leqslant \frac{K (c, \sigma)}{(1 +
     r)^{2 + \sigma}}, \]
  \[ | \partial_{x_2} Q_c | + | \partial_{x_2} \partial_c Q_c | \leqslant
     \frac{K (c, \sigma)}{(1 + r)^{1 + \sigma}} \]
  and
  \[ | \mathfrak{R}\mathfrak{e} (\partial_{x_2} Q_c) | + |
     \mathfrak{R}\mathfrak{e} (\partial_{x_2} \partial_c Q_c) | \leqslant
     \frac{K (c, \sigma)}{(1 + r)^{2 + \sigma}}, \]
  which is enough to show the estimates for $\partial_c Q_c$.
  
  Finally, from Lemma \ref{CP2a} we recall that
  \[ \partial_{c^{\bot}} Q_c = - x^{\bot} . \nabla Q_c (x) \]
  and
  \[ L_{Q_c} (\partial_{c^{\bot}} Q_c) = - i c \partial_{x_1} Q_c . \]
  Similarly, the estimates not on its Laplacian follow from Theorem
  \ref{CP2Qcbehav}, Lemmas \ref{lemme3new} and \ref{CP283L33} and equation
  (\ref{CP2222}). We also have
  \[ \Delta (\partial_{c^{\bot}} Q_c) = i c \partial_{x_1} Q_c - i c
     \partial_{x_2} \partial_{c^{\bot}} Q_c - (1 - | Q_c |^2)
     \partial_{c^{\bot}} Q_c + 2\mathfrak{R}\mathfrak{e} (\overline{Q_c}
     \partial_{c^{\bot}} Q_c) Q_c, \]
  and with the same previous estimates, we conclude that $\partial_{c^{\bot}}
  Q_c$ satisfies the required estimates. With the definition $\| .
  \|_{H_{Q_c}}$, we check that the last two terms are in $L^1
  (\mathbbm{R}^2)$, and for the first two, the integrands are in $L^1
  (\mathbbm{R}^2, \mathbbm{R})$ by estimates in subsections \ref{CP2Nvor} and
  (\ref{CP2N216}).

  \begin{tmindent}
    Step 1.  We have $\langle L_{Q_c} (\partial_{x_1} Q_c), \partial_{x_1} Q_c
    \rangle = \langle L_{Q_c} (\partial_{x_2} Q_c), \partial_{x_2} Q_c \rangle
    = 0$.
  \end{tmindent}

  From Lemma \ref{CP20703L222}, we have $L_{Q_c} (\partial_{x_1} Q_c) =
  L_{Q_c} (\partial_{x_2} Q_c) = 0$, hence
  \[ \langle L_{Q_c} (\partial_{x_1} Q_c), \partial_{x_1} Q_c \rangle =
     \langle L_{Q_c} (\partial_{x_2} Q_c), \partial_{x_2} Q_c \rangle = 0. \]
  
  \begin{tmindent}
    Step 2.  We have $\langle L_{Q_c} (\partial_c Q_c), \partial_c Q_c \rangle
    = \frac{- 2 \pi + o_{c \rightarrow 0} (1)}{c^2}$.
  \end{tmindent}

  From Lemma \ref{CP20703L222}, we have
  \[ L_{Q_c} (\partial_c Q_c) = i \partial_{x_2} Q_c, \]
  therefore
  \begin{equation}
    \langle L_{Q_c} (\partial_c Q_c), \partial_c Q_c \rangle = \langle i
    \partial_{x_2} Q_c, \partial_c Q_c \rangle . \label{CP21239312}
  \end{equation}
  From Lemma \ref{CP2dcQcsigma}, we can write $\partial_c Q_c = - \left(
  \frac{1 + o_{c \rightarrow 0} (1)}{c^2} \right) \partial_d V_{| d = d_c
  \nobracket} + h$ with $\left\| \frac{h}{V} \right\|_{\sigma, d_c} = o_{c
  \rightarrow 0} \left( \frac{1}{c^2} \right) .$ Similarly, from Lemma
  \ref{CP283L33}, we write $Q_c = V + \Gamma_c$ with $\left\|
  \frac{\Gamma_c}{V} \right\|_{\sigma, d_c} = o_{c \rightarrow 0} (1)$, and we
  compute
  \begin{eqnarray}
    \langle L_{Q_c} (\partial_c Q_c), \partial_c Q_c \rangle & = &
    \left\langle i \partial_{x_2} V, - \left( \frac{1 + o_{c \rightarrow 0}
    (1)}{c^2} \right) \partial_d V_{| d = d_c \nobracket} \right\rangle +
    \langle i \partial_{x_2} V, h \rangle \nonumber\\
    & + & \left\langle i \partial_{x_2} \Gamma_c, - \left( \frac{1 + o_{c
    \rightarrow 0} (1)}{c^2} \right) \partial_d V_{| d = d_c \nobracket}
    \right\rangle + \langle i \partial_{x_2} \Gamma_c, h \rangle . 
    \label{CP283232}
  \end{eqnarray}
  By symmetry in $x_1$ of $V$, we compute
  \[ \langle i \partial_{x_2} V, \partial_d V_{| d = d_c \nobracket} \rangle =
     - 2 \langle i \partial_{x_2} V_1 V_{- 1}, \partial_{x_1} V_1 V_{- 1}
     \rangle + 2 \langle i \partial_{x_2} V_1 V_{- 1}, \partial_{x_1} V_{- 1}
     V_1 \rangle . \]
  In equation (2.43) of {\cite{CP1}}, we computed
  \[ \langle i \partial_{x_2} V_1 V_{- 1}, \partial_{x_1} V_1 V_{- 1} \rangle
     = - \pi + o_{c \rightarrow 0} (1) . \]
  Furthermore,
  \[ | \langle i \partial_{x_2} V_1 V_{- 1}, \partial_{x_1} V_{- 1} V_1
     \rangle | = \left| \int_{\mathbbm{R}^2} \mathfrak{R}\mathfrak{e} \left( i
     \partial_{x_2} V_1 \overline{V_1}  \overline{\partial_{x_1} V_{- 1}} V_{-
     1} \right) \right| \leqslant \]
  \[ \left| \int_{\mathbbm{R}^2} \mathfrak{R}\mathfrak{e} (\partial_{x_2} V_1
     \overline{V_1} ) \mathfrak{I}\mathfrak{m} \left( \overline{\partial_{x_1}
     V_{- 1}} V_{- 1} \right) \right| + \left| \int_{\mathbbm{R}^2}
     \mathfrak{I}\mathfrak{m} (\partial_{x_2} V_1 \overline{V_1} )
     \mathfrak{R}\mathfrak{e} \left( \overline{\partial_{x_1} V_{- 1}} V_{- 1}
     \right) \right| . \]
  From Lemma \ref{lemme3new}, we have the estimates
  \[ | \mathfrak{R}\mathfrak{e} (\partial_{x_2} V_{- 1} \overline{V_{- 1}} ) |
     \leqslant \frac{K}{(1 + r_{- 1})^3} \quad \tmop{and} \quad \left|
     \mathfrak{R}\mathfrak{e} \left( \overline{\partial_{x_1} V_1} V_1 \right)
     \right| \leqslant \frac{K}{(1 + r_1)^3}, \]
  as well as
  \[ | \mathfrak{I}\mathfrak{m} (\partial_{x_2} V_{- 1} \overline{V_{- 1}} ) |
     \leqslant \frac{K}{1 + r_{- 1}^{}} \quad \tmop{and} \quad \left|
     \mathfrak{I}\mathfrak{m} \left( \overline{\partial_{x_1} V_1} V_1 \right)
     \right| \leqslant \frac{K}{1 + r_1^{}} . \]
  We deduce, in the right half-plane, where $r_{- 1} \geqslant d_c$, that $|
  \mathfrak{I}\mathfrak{m} (\nabla V_{- 1} \overline{V_{- 1}} ) | = o_{c
  \rightarrow 0} (1)$ and thus
  \[ \left| \int_{\{ x_1 \geqslant 0 \}} \mathfrak{R}\mathfrak{e}
     (\partial_{x_2} V_1 \overline{V_1} ) \mathfrak{I}\mathfrak{m} \left(
     \overline{\partial_{x_1} V_{- 1}} V_{- 1} \right) \right| \leqslant o_{c
     \rightarrow 0} (1) \int_{\{ x_1 \geqslant 0 \}} \frac{1}{(1 + r_1)^3} =
     o_{c \rightarrow 0} (1) . \]
  In the left half-plane, we have $\frac{1}{1 + r_1^{}} \leqslant \frac{K}{1 +
  r_{- 1}^{}}$ and $\frac{1}{1 + r_1^{}} = o_{c \rightarrow 0} (1)$, therefore
  \[ \left| \int_{\{ x_1 \leqslant 0 \}} \mathfrak{R}\mathfrak{e}
     (\partial_{x_2} V_1 \overline{V_1} ) \mathfrak{I}\mathfrak{m} \left(
     \overline{\partial_{x_1} V_{- 1}} V_{- 1} \right) \right| \leqslant o_{c
     \rightarrow 0} (1) \int_{\{ x_1 \leqslant 0 \}} \frac{1}{(1 + r_{- 1})^3}
     = o_{c \rightarrow 0} (1) . \]
  We therefore have
  \[ \left| \int_{\mathbbm{R}^2} \mathfrak{R}\mathfrak{e} (\partial_{x_2} V_1
     \overline{V_1} ) \mathfrak{I}\mathfrak{m} \left( \overline{\partial_{x_1}
     V_{- 1}} V_{- 1} \right) \right| = o_{c \rightarrow 0} (1), \]
  and by similar estimates,
  \[ \left| \int_{\mathbbm{R}^2} \mathfrak{I}\mathfrak{m} (\partial_{x_2} V_1
     \overline{V_1} ) \mathfrak{R}\mathfrak{e} \left( \overline{\partial_{x_1}
     V_{- 1}} V_{- 1} \right) \right| = o_{c \rightarrow 0} (1) . \]
  We can thus conclude that $\langle i \partial_{x_2} V_1 V_{- 1},
  \partial_{x_1} V_{- 1} V_1 \rangle = o_{c \rightarrow 0} (1) .$ Therefore,
  \begin{equation}
    \left( \frac{1 + o_{c \rightarrow 0} (1)}{c^2} \right) \langle i
    \partial_{x_2} V, - \partial_d V_{| d = d_c \nobracket} \rangle = \frac{-
    2 \pi}{c^2} + o \left( \frac{1}{c^2} \right) . \label{CP283233}
  \end{equation}
  Now, we estimate
  \begin{eqnarray*}
    | \langle i \partial_{x_2} V, h \rangle | & = & \left|
    \int_{\mathbbm{R}^2} \mathfrak{R}\mathfrak{e} (i \partial_{x_2} V \bar{h})
    \right|\\
    & \leqslant & o_{c \rightarrow 0} (1) + \left| \int_{\{ \tilde{r}
    \geqslant 1 \}} \mathfrak{R}\mathfrak{e} (i \partial_{x_2} V \bar{h})
    \right|\\
    & \leqslant & o_{c \rightarrow 0} (1) + \left| \int_{\{ \tilde{r}
    \geqslant 1 \}} \mathfrak{R}\mathfrak{e} \left( i \partial_{x_2} V \bar{V}
    \overline{\left( \frac{h}{V} \right)} \right) \right|
  \end{eqnarray*}
  because $\| h \|_{L^{\infty}} = o_{c \rightarrow 0} (1)$ and $|
  \partial_{x_2} V |$ is bounded near $\tilde{d}_c$ by a universal constant.
  Furthermore,
  \[ \left| \int_{\{ \tilde{r} \geqslant 1 \}} \mathfrak{R}\mathfrak{e} \left(
     i \partial_{x_2} V \bar{V} \overline{\left( \frac{h}{V} \right)} \right)
     \right| \leqslant \left| \int_{\{ \tilde{r} \geqslant 1 \}}
     \mathfrak{R}\mathfrak{e} (\partial_{x_2} V \bar{V})
     \mathfrak{I}\mathfrak{m} \left( \frac{h}{V} \right) \right| + \left|
     \int_{\{ \tilde{r} \geqslant 1 \}} \mathfrak{I}\mathfrak{m}
     (\partial_{x_2} V \bar{V}) \mathfrak{R}\mathfrak{e} \left( \frac{h}{V}
     \right) \right| . \]
  From Lemmas \ref{lemme3new} and \ref{CP2dcQcsigma} (taking $\sigma = 1 /
  2$), we have
  \[ \left| \int_{\{ \tilde{r} \geqslant 1 \}} \mathfrak{R}\mathfrak{e}
     (\partial_{x_2} V \bar{V}) \mathfrak{I}\mathfrak{m} \left( \frac{h}{V}
     \right) \right| \leqslant K \left\| \frac{h}{V} \right\|_{1 / 2, d_c}
     \int_{\{ \tilde{r} \geqslant 1 \}} \frac{1}{(1 + \tilde{r})^{3 + 1 / 2}}
     = o_{c \rightarrow 0} \left( \frac{1}{c^2} \right) \]
  and
  \[ \left| \int_{\{ \tilde{r} \geqslant 1 \}} \mathfrak{I}\mathfrak{m}
     (\partial_{x_2} V \bar{V}) \mathfrak{R}\mathfrak{e} \left( \frac{h}{V}
     \right) \right| \leqslant K \left\| \frac{h}{V} \right\|_{1 / 2, d_c}
     \int_{\{ \tilde{r} \geqslant 1 \}} \frac{1}{(1 + \tilde{r})^{2 + 1 / 2}}
     = o_{c \rightarrow 0} \left( \frac{1}{c^2} \right), \]
  therefore
  \begin{equation}
    | \langle i \partial_{x_2} V, h \rangle | = o_{c \rightarrow 0} \left(
    \frac{1}{c^2} \right) .
  \end{equation}
  Now, by Lemmas \ref{lemme3new} and \ref{CP283L33} (taking $\sigma = 1 / 2$),
  we have
  \begin{equation}
    \left( \frac{1 + o_{c \rightarrow 0} (1)}{c^2} \right) | \langle i
    \partial_{x_2} \Gamma_c, \partial_d V_{| d = d_c \nobracket} \rangle |
    \leqslant \frac{K}{c^2} \left\| \frac{\Gamma_c}{V} \right\|_{1 / 2, d_c}
    \int_{\mathbbm{R}^2} \frac{1}{(1 + \tilde{r})^{2 + 1 / 2}} = o_{c
    \rightarrow 0} \left( \frac{1}{c^2} \right) .
  \end{equation}
  Finally, by Lemmas \ref{CP283L33} and \ref{CP2dcQcsigma}, we check easily
  that
  \begin{equation}
    | \langle i \partial_{x_2} \Gamma_c, h \rangle | \leqslant K \left\|
    \frac{\Gamma_c}{V} \right\|_{3 / 4, d_c} \left\| \frac{h}{V} \right\|_{1 /
    2, d_c} \int_{\mathbbm{R}^2} \frac{1}{(1 + \tilde{r})^{2 + 1 / 4}} = o_{c
    \rightarrow 0} \left( \frac{1}{c^2} \right) . \label{CP283236}
  \end{equation}
  Combining (\ref{CP283233}) to (\ref{CP283236}) in (\ref{CP283232}), we
  conclude that
  \[ \langle L_{Q_c} (\partial_c Q_c), \partial_c Q_c \rangle = \frac{- 2 \pi
     + o_{c \rightarrow 0} (1)}{c^2} . \]
  
  \begin{tmindent}
    Step 3.  We have $\langle L_{Q_c} (\partial_{c^{\bot}} Q_c),
    \partial_{c^{\bot}} Q_c \rangle = 2 \pi + o_{c \rightarrow 0} (1)$.
  \end{tmindent}

  From Lemma \ref{CP20703L222}, we have $L_{Q_c} (\partial_{c^{\bot}} Q_c) = -
  i c \partial_{x_1} Q_c$ and from Lemma \ref{CP2a}, we have
  $\partial_{c^{\bot}} Q_c = - x^{\bot} . \nabla Q_c$. Therefore,
  \[ \langle L_{Q_c} (\partial_{c^{\bot}} Q_c), \partial_{c^{\bot}} Q_c
     \rangle = c \langle i \partial_{x_1} Q_c, x^{\bot} . \nabla Q_c \rangle .
  \]
  We have
  \[ \langle i \partial_{x_1} Q_c, - x_2 \partial_{x_1} Q_c \rangle = -
     \int_{\mathbbm{R}^2} \mathfrak{R}\mathfrak{e} (i x_2 | \partial_{x_1} Q_c
     |^2) = 0, \]
  hence
  \begin{equation}
    \langle L_{Q_c} (\partial_{c^{\bot}} Q_c), \partial_{c^{\bot}} Q_c \rangle
    = c \langle i \partial_{x_1} Q_c, x_1 \partial_{x_2} Q_c \rangle .
    \label{CP2123317}
  \end{equation}
  From Lemma \ref{CP283L33}, we write $Q_c = V + \Gamma_c$ with $\left\|
  \frac{\Gamma_c}{V} \right\|_{\sigma, d_c} \leqslant K (\sigma) c^{1 -
  \sigma}$ for any $0 < \sigma < 1$, and we compute
  \[ \langle i \partial_{x_1} Q_c, x_1 \partial_{x_2} Q_c \rangle = \langle i
     \partial_{x_1} V, x_1 \partial_{x_2} V \rangle + \langle i \partial_{x_1}
     V, x_1 \partial_{x_2} \Gamma_c \rangle + \langle i \partial_{x_1}
     \Gamma_c, x_1 \partial_{x_2} V \rangle + \langle i \partial_{x_1}
     \Gamma_c, x_1 \partial_{x_2} \Gamma_c \rangle . \]
  We write $x_1 = d_c + y_1$, therefore
  \[ \langle i \partial_{x_1} V, x_1 \partial_{x_2} V \rangle = d_c \langle i
     \partial_{x_1} V, \partial_{x_2} V \rangle + \langle i \partial_{x_1} V,
     y_1 \partial_{x_2} V \rangle . \]
  We have
  \begin{eqnarray*}
    \langle i \partial_{x_1} V, \partial_{x_2} V \rangle & = & \langle i
    \partial_{x_1} V_1 V_{- 1}, \partial_{x_2} V_1 V_{- 1} \rangle + \langle i
    \partial_{x_1} V_{- 1} V_1, \partial_{x_2} V_{- 1} V_1 \rangle\\
    & + & \langle i \partial_{x_1} V_1 V_{- 1}, \partial_{x_2} V_{- 1} V_1
    \rangle + \langle i \partial_{x_1} V_{- 1} V_1, \partial_{x_2} V_1 V_{- 1}
    \rangle,
  \end{eqnarray*}
  and, from the previous step and by symmetry, we have
  \[ \langle i \partial_{x_1} V_1 V_{- 1}, \partial_{x_2} V_1 V_{- 1} \rangle
     = \langle i \partial_{x_1} V_{- 1} V_1, \partial_{x_2} V_{- 1} V_1
     \rangle = \pi + o_{c \rightarrow 0} (1) \]
  and
  \[ | \langle i \partial_{x_1} V_1 V_{- 1}, \partial_{x_2} V_{- 1} V_1
     \rangle | + | \langle i \partial_{x_1} V_{- 1} V_1, \partial_{x_2} V_1
     V_{- 1} \rangle | = o_{c \rightarrow 0} (1), \]
  thus
  \[ \langle i \partial_{x_1} V, \partial_{x_2} V \rangle = 2 \pi + o_{c
     \rightarrow 0} (1) . \]
  With $V_{\pm 1}$ centered around $\pm d_c \overrightarrow{e_1}$, we write $V
  = V_1 V_{- 1}$ and we compute
  \begin{eqnarray*}
    \langle i \partial_{x_1} V, y_1 \partial_{x_2} V \rangle & = &
    \int_{\mathbbm{R}^2} \mathfrak{R}\mathfrak{e} \left( i y_1 \partial_{x_1}
    V_1 \overline{\partial_{x_2} V_1} | V_{- 1} |^2 + i y_1 \partial_{x_1}
    V_{- 1} \overline{\partial_{x_2} V_{- 1}} | V_1 |^2 \right)\\
    & + & \int_{\mathbbm{R}^2} \mathfrak{R}\mathfrak{e} \left( i y_1
    \partial_{x_1} V_1 \overline{V_1} \overline{\overline{V_{- 1}}
    \partial_{x_2} V_{- 1}} + i y_1 \partial_{x_1} V_{- 1} \overline{V_{- 1}}
    \overline{\overline{V_1} \partial_{x_2} V_1} \right) .
  \end{eqnarray*}
  By decomposition in polar coordinates, with the notation of
  (\ref{CP2223not}) and Lemma \ref{lemme3new}, we compute
  \[ \int_{\mathbbm{R}^2} \mathfrak{R}\mathfrak{e} \left( i y_1 \partial_{x_1}
     V_1 \overline{\partial_{x_2} V_1} | V_{- 1} |^2 \right) = \int_0^{+
     \infty} \int_0^{2 \pi} | V_{- 1} |^2 \rho_1 (r_1) \rho'_1 (r_1) \cos
     (\theta_1)_{} r_1 d r_1 d \theta_1 . \]
  By integration in polar coordinates, we check that
  \[ \int_0^{+ \infty} \int_0^{2 \pi} \rho_1 (r_1) \rho'_1 (r_1) \cos
     (\theta_1)_{} = 0, \]
  hence
  \[ \int_{\mathbbm{R}^2} \mathfrak{R}\mathfrak{e} \left( i y_1 \partial_{x_1}
     V_1 \overline{\partial_{x_2} V_1} | V_{- 1} |^2 \right) =
     \int_{\mathbbm{R}^2} (1 - | V_{- 1} |^2) \mathfrak{R}\mathfrak{e} \left(
     i y_1 \partial_{x_1} V_1 \overline{\partial_{x_2} V_1} \right) . \]
  In particular, since, from Lemma \ref{lemme3new}, we have
  \[ (1 - | V_{- 1} |^2) \leqslant \frac{K}{(1 + r_{- 1})^2} \]
  and
  \[ | \rho'_1 (r_1) | \leqslant \frac{K}{(1 + r_1)^3}, \]
  we can deduce that
  \[ \int_{\mathbbm{R}^2} \mathfrak{R}\mathfrak{e} \left( i y_1
     \partial_{x_1} V_1 \overline{\partial_{x_2} V_1} | V_{- 1} |^2 \right) =
     o_{c \rightarrow 0} (1) \]
  and, similarly,
  \[ \int_{\mathbbm{R}^2} \mathfrak{R}\mathfrak{e} \left( i y_1 \partial_{x_1}
     V_{- 1} \overline{\partial_{x_2} V_{- 1}} | V_1 |^2 \right) = o_{c
     \rightarrow 0} (1) . \]
  Therefore, we conclude that
  \[ \langle i \partial_{x_1} V, x_1 \partial_{x_2} V \rangle = (2 \pi + o_{c
     \rightarrow 0} (1)) \tilde{d}_c = \frac{2 \pi + o_{c \rightarrow 0}
     (1)}{c} . \]
  Now, we want to show that
  \[ | \langle i \partial_{x_1} V, x_1 \partial_{x_2} \Gamma_c \rangle | + |
     \langle i \partial_{x_1} \Gamma_c, x_1 \partial_{x_2} V \rangle | + |
     \langle i \partial_{x_1} \Gamma_c, x_1 \partial_{x_2} \Gamma_c \rangle |
     = o_{c \rightarrow 0} \left( \frac{1}{c} \right), \]
  which is enough to end the proof of this step.
  
  By triangular inequality, we have $| x_1 | \leqslant \frac{K (1 +
  \tilde{r})}{c}$, and with Lemmas \ref{lemme3new} and \ref{CP283L33} (for
  $\sigma = 1 / 2$), we estimate
  \begin{eqnarray*}
    | \langle i \partial_{x_1} V, x_1 \partial_{x_2} \Gamma_c \rangle | & = &
    \left| \int_{\mathbbm{R}^2} x_1 \mathfrak{R}\mathfrak{e} (\partial_{x_1} V
    \bar{V}) \mathfrak{I}\mathfrak{m} \left( \overline{\partial_{x_2} \Gamma_c
    \bar{V}} \right) \right| + \left| \int_{\mathbbm{R}^2} x_1
    \mathfrak{I}\mathfrak{m} (\partial_{x_1} V \bar{V})
    \mathfrak{R}\mathfrak{e} \left( \overline{\partial_{x_2} \Gamma_c \bar{V}}
    \right) \right|\\
    & \leqslant & \frac{K}{c} \left( \int_{\mathbbm{R}^2} \frac{(1 +
    \tilde{r})}{(1 + \tilde{r})^3} \times \frac{c^{1 / 2}}{(1 + \tilde{r})^{3
    / 2}} + \frac{(1 + \tilde{r})}{(1 + \tilde{r})} \times \frac{c^{1 / 2}}{(1
    + \tilde{r})^{5 / 2}} \right)\\
    & = & o_{c \rightarrow 0} \left( \frac{1}{c} \right) .
  \end{eqnarray*}
  Similarly, we check with the same computations that $| \langle i
  \partial_{x_1} \Gamma_c, x_1 \partial_{x_2} V \rangle | = o_{c \rightarrow
  0} \left( \frac{1}{c} \right)$.
  
  Finally, using Lemma \ref{CP283L33} (for $\sigma = 1 / 4$), we estimate
  \[ | \langle i \partial_{x_1} \Gamma_c, x_1 \partial_{x_2} \Gamma_c \rangle
     | \leqslant K c^{3 / 2} \| x_1 \|_{L^{\infty} (\{ \tilde{r} \leqslant 1
     \})} + K \left| \int_{\{ \tilde{r} \geqslant 1 \}}
     \mathfrak{R}\mathfrak{e} \left( i x_1 \frac{\partial_{x_1} \Gamma_c}{V}
     \overline{\frac{\partial_{x_2} \Gamma_c}{V}} \right) \right| . \]
  We have $\| x_1 \|_{L^{\infty} (\{ \tilde{r} \leqslant 1 \})} \leqslant
  \frac{K}{c}$. Moreover, we infer
  \begin{eqnarray*}
    \left| \int_{\{ \tilde{r} \geqslant 1 \}} \mathfrak{R}\mathfrak{e} \left(
    i x_1 \frac{\partial_{x_1} \Gamma_c}{V} \overline{\frac{\partial_{x_2}
    \Gamma_c}{V}} \right) \right| & \leqslant & \int_{\{ \tilde{r} \geqslant 1
    \}} | x_1 | \left| \mathfrak{R}\mathfrak{e} \left( \frac{\partial_{x_1}
    \Gamma_c}{V} \right) \mathfrak{I}\mathfrak{m} \left( \frac{\partial_{x_2}
    \Gamma_c}{V} \right) \right|\\
    & + & \int_{\{ \tilde{r} \geqslant 1 \}} | x_1 | \left|
    \mathfrak{I}\mathfrak{m} \left( \frac{\partial_{x_1} \Gamma_c}{V} \right)
    \mathfrak{R}\mathfrak{e} \left( \frac{\partial_{x_2} \Gamma_c}{V} \right)
    \right|,
  \end{eqnarray*}
  and, with Lemma \ref{CP283L33} (for $\sigma = 1 / 4$), we have
  \[ \left| \int_{\{ \tilde{r} \geqslant 1 \}} \mathfrak{R}\mathfrak{e}
     \left( i x_1 \frac{\partial_{x_1} \Gamma_c}{V}
     \overline{\frac{\partial_{x_2} \Gamma_c}{V}} \right) \right| \leqslant K
     \int_{\{ \tilde{r} \geqslant 1 \}} | x_1 | \frac{c^{3 / 2}}{(1 +
     \tilde{r})^{3 + 1 / 2}} = o_{c \rightarrow 0} (1), \]
  since $\frac{| x_1 | c}{(1 + \tilde{r})} \leqslant K$ by triangular
  inequality. We conclude that
  \[ \langle i \partial_{x_1} \Gamma_c, x_1 \partial_{x_2} \Gamma_c \rangle =
     o_{c \rightarrow 0} (1), \]
  which, together with the previous estimates, shows that
  \[ \langle L_{Q_c} (\partial_{c^{\bot}} Q_c), \partial_{c^{\bot}} Q_c
     \rangle = 2 \pi + o_{c \rightarrow 0} (1) . \]
\end{proof}

These quantities are connected to the energy and momentum. This is shown in
this next lemma.

\begin{lemma}
  \label{CP2cor212N}There exists $c_0 > 0$ such that for $0 < c < c_0$, $Q_c$
  defined in Theorem \ref{th1}, we have
  \[ P_1 (Q_c) = \partial_c P_1 (Q_c) = 0, \]
  \[ P_2 (Q_c) = \frac{1}{c} B_{Q_c} (\partial_{c^{\bot}} Q_c) = \frac{2 \pi +
     o_{c \rightarrow 0} (1)}{c} \]
  and
  \[ \partial_c P_2 (Q_c) = B_{Q_c} (\partial_c Q_c) = \frac{- 2 \pi + o_{c
     \rightarrow 0} (1)}{c^2} . \]
  Furthermore,
  \[ \partial_c E (Q_c) = c \partial_c P_2 (Q_c), \]
  and
  \[ E (Q_c) = (2 \pi + o_{c \rightarrow 0} (1)) \ln \left( \frac{1}{c}
     \right) . \]
\end{lemma}

\begin{proof}
  We have
  \[ P_1 (Q_c) = \frac{1}{2} \langle i \partial_{x_1} Q_c, Q_c - 1 \rangle, \]
  by the symmetries (\ref{CP2sym}), $\partial_{x_1} Q_c$ is odd in $x_1$ and
  $Q_c - 1$ is even. Therefore,
  \[ P_1 (Q_c) = \partial_c P_1 (Q_c) = 0. \]
  We have
  \[ P_2 (Q_c) = \frac{1}{2} \langle i \partial_{x_2} Q_c, Q_c - 1 \rangle,
  \]
  and from Lemma \ref{CP2nend} and (\ref{CP2123317}), we have
  \[ 2 \pi + o_{c \rightarrow 0} (1) = B_{Q_c} (\partial_{c^{\bot}} Q_c) = c
     \langle i \partial_{x_1} Q_c, x_1 \partial_{x_2} Q_c \rangle . \]
  By integration by parts (which can be done thanks to Theorem
  \ref{CP2Qcbehav}, Lemma \ref{lemme3new} and equation (\ref{CP2222})), we
  compute
  \[ \langle i \partial_{x_1} Q_c, x_1 \partial_{x_2} Q_c \rangle = - \langle
     i (Q_c - 1), \partial_{x_2} Q_c \rangle - \langle i (Q_c - 1), x_1
     \partial_{x_1 x_2} Q_c \rangle, \]
  and
  \[ \langle i (Q_c - 1), x_1 \partial_{x_1 x_2} Q_c \rangle = - \langle i
     \partial_{x_2} Q_c, x_1 \partial_{x_1} Q_c \rangle = \langle i
     \partial_{x_1} Q_c, x_1 \partial_{x_2} Q_c \rangle . \]
  Therefore,
  \[ P_2 (Q_c) = \frac{1}{2} \langle i \partial_{x_1} Q_c, x_1 \partial_{x_2}
     Q_c \rangle = \frac{1}{c} B_{Q_c} (\partial_{c^{\bot}} Q_c) = \frac{2 \pi
     + o_{c \rightarrow 0} (1)}{c} . \]
  We have $P_2 (Q_c) = \frac{1}{2} \int_{\mathbbm{R}^2}
  \mathfrak{R}\mathfrak{e} (i \partial_{x_2} Q_c (\overline{Q_c} - 1))$, and
  we check that, with Lemmas \ref{CP283L33} and \ref{CP2dcQcsigma} that
  \[ | \partial_c \partial_{x_2} Q_c (\overline{Q_c} - 1) | + | \partial_{x_2}
     Q_c \partial_c \overline{Q_c} | \leqslant \frac{K}{(1 + \tilde{r})^{5 /
     2}}, \]
  and is therefore dominated by an integrable function independent of $c \in]
  c_1, c_2 [$ given that $c_1, c_2 > 0$ are small enough. We deduce that $c
  \mapsto P_2 (Q_c) \in C^1 (] 0, c_0 [, \mathbbm{R})$ for some small $c_0 >
  0$ and that, by integration by parts,
  \[ 2 \partial_c P_2 (Q_c) = \langle i \partial_{x_2} \partial_c Q_c, Q_c - 1
     \rangle + \langle i \partial_{x_2} Q_c, \partial_c Q_c \rangle = 2
     \langle i \partial_{x_2} Q_c, \partial_c Q_c \rangle, \]
  and, from Lemma \ref{CP2nend} and equation (\ref{CP21239312}), we have
  \[ \langle i \partial_{x_2} Q_c, \partial_c Q_c \rangle = B_{Q_c}
     (\partial_c Q_c) = \frac{- 2 \pi + o_{c \rightarrow 0} (1)}{c^2}, \]
  therefore
  \[ \partial_c P_2 (Q_c) = \frac{- 2 \pi + o_{c \rightarrow 0} (1)}{c^2} . \]

  We recall that
  \[ E (Q_c) = \frac{1}{2} \int_{\mathbbm{R}^2} | \nabla Q_c |^2 + \frac{1}{4}
     \int_{\mathbbm{R}^2} (1 - | Q_c |^2)^2 . \]
  We check with Lemmas \ref{CP283L33}, \ref{CP2dcQcsigma} that
  \[ | \partial_c \nabla Q_c . \overline{\nabla Q_c} | + | \partial_c (| Q_c
     |^2) (1 - | Q_c |^2) | \leqslant \frac{K}{(1 + \tilde{r})^{5 / 2}} \]
  and is therefore dominated by an integrable function independent of $c \in]
  c_1, c_2 [$ given that $c_1, c_2 > 0$ are small enough. We deduce that $c
  \mapsto E (Q_c) \in C^1 (] 0, c_0 [, \mathbbm{R})$ for some small $c_0 > 0$
  and that,
  \[ \partial_c \left( \frac{1}{2} \int_{\mathbbm{R}^2} | \nabla Q_c |^2
     \right) = \frac{1}{2} \int_{\mathbbm{R}^2} \mathfrak{R}\mathfrak{e}
     (\nabla Q_c \overline{\nabla \partial_c Q_c}) +\mathfrak{R}\mathfrak{e}
     (\nabla \partial_c Q_c \overline{\nabla Q_c}) . \]
  We check, with Theorem \ref{CP2Qcbehav} and $(\tmop{TW}_c) (Q_c) = 0$, that
  we can do the integration by parts, which yields
  \[ \partial_c \left( \frac{1}{2} \int_{\mathbbm{R}^2} | \nabla Q_c |^2
     \right) = \langle - \Delta Q_c, \partial_c Q_c \rangle . \]
  We check similarly that
  \[ \partial_c \left( \frac{1}{4} \int_{\mathbbm{R}^2} (1 - | Q_c |^2)^2
     \right) = - \int_{\mathbbm{R}^2} (1 - | Q_c |^2) \mathfrak{R}\mathfrak{e}
     (\partial_c Q_c \overline{Q_c}), \]
  hence
  \[ \partial_c \left( \frac{1}{4} \int_{\mathbbm{R}^2} (1 - | Q_c |^2)^2
     \right) = \langle - (1 - | Q_c |^2) Q_c, \partial_c Q_c \rangle . \]
  Now, since $- i c \partial_{x_2} Q_c - \Delta Q_c - (1 - | Q_c |^2) Q_c =
  0$, we have
  \[ \partial_c E (Q_c) = \langle - \Delta Q_c - (1 - | Q_c |^2) Q_c,
     \partial_c Q_c \rangle = c \langle - i \partial_{x_2} Q_c, \partial_c Q_c
     \rangle . \]
  Now, since $P_2 (Q_c) = \frac{1}{2} \langle i \partial_{x_2} Q_c, Q_c - 1
  \rangle$, we have
  \[ \partial_c P_2 (Q_c) = \frac{1}{2} (\langle i \partial_{x_2} \partial_c
     Q_c, Q_c - 1 \rangle + \langle i \partial_{x_2} Q_c, \partial_c Q_c
     \rangle) . \]
  By integrations by parts, we compute
  \[ \partial_c P_2 (Q_c) = \langle - i \partial_{x_2} Q_c, \partial_c Q_c
     \rangle . \]
  We deduce that $\partial_c E (Q_c) = c \partial_c P_2 (Q_c)$, and in
  particular, we deduce that
  \[ \partial_c E (Q_c) = \frac{- 2 \pi + o_{c \rightarrow 0} (1)}{c} . \]
  By integration (from some fixed $c_0 > c > 0$), we check that $E (Q_c) = (2
  \pi + o_{c \rightarrow 0} (1)) \ln \left( \frac{1}{c} \right)$.
\end{proof}

We conclude this subsection with an estimate on $Q_c$ connected to the energy
that will be useful later on.

\begin{lemma}
  \label{CP293L37}There exists $K > 0$, a universal constant independent of
  $c$, such that, if $c$ is small enough, for $Q_c$ defined in Theorem
  \ref{th1},
  \[ \int_{\mathbbm{R}^2} \frac{| \mathfrak{I}\mathfrak{m} (\nabla Q_c
     \overline{Q_c}) |^2}{| Q_c |^2} \leqslant K \ln \left( \frac{1}{c}
     \right) . \]
\end{lemma}

\begin{proof}
  We recall that $r_{\pm 1} = | x \mp d_c \overrightarrow{e_1} |$. Since
  $\nabla Q_c$ is bounded near the zeros of $Q_c$ (by Lemmas \ref{lemme3new}
  and \ref{CP283L33}), and $| Q_c | \geqslant K$ on $\mathbbm{R}^2 \backslash
  B (\pm \widetilde{d_c} \overrightarrow{e_1}, 1)$ by (\ref{CP2Qcpaszero}), we
  have
  \[ \int_{\mathbbm{R}^2} \frac{| \mathfrak{I}\mathfrak{m} (\nabla Q_c
     \overline{Q_c}) |^2}{| Q_c |^2} \leqslant K \left( 1 + \int_{\{ \tilde{r}
     \geqslant 1 \}} | \mathfrak{I}\mathfrak{m} (\nabla Q_c \overline{Q_c})
     |^2 \right) . \]
  Now, by (\ref{CP2Qcpaszero}), Lemma \ref{CP2cor212N} and the definition of
  the energy,
  \[ \int_{\{ \tilde{r} \geqslant 1 \}} | \mathfrak{I}\mathfrak{m} (\nabla Q_c
     \overline{Q_c}) |^2 \leqslant \int_{\{ \tilde{r} \geqslant 1 \}} | \nabla
     Q_c |^2 | Q_c |^2 \leqslant K \int_{\mathbbm{R}^2} | \nabla Q_c |^2
     \leqslant K E (Q_c) \leqslant K \ln \left( \frac{1}{c} \right) . \]
\end{proof}

We could check that this estimate is optimal with respect to its growth in $c$
when $c \rightarrow 0$.

\subsection{Zeros of $Q_c$}\label{CP205s1}

In this subsection, we show that $Q_c$ has only two zeros and we compute
estimates on $Q_c$ around them. In a bounded domain, a general result about
the zeros of solutions to the Ginzburg-Landau problem is already known, see
{\cite{MR1831928}}.

\begin{lemma}
  \label{CP2zerosofQc}For $c > 0$ small enough, the function $Q_c$ defined in
  Theorem \ref{th1} has exactly two zeros. Their positions are $\pm
  \widetilde{d_c} \overrightarrow{e_1}$, and, for any $0 < \sigma < 1$,
  \[ | d_c - \widetilde{d_c} | = o^{\sigma}_{c \rightarrow 0} (c^{1 -
     \sigma}), \]
  where $d_c$ is defined in Theorem \ref{th1}.
\end{lemma}

The notation $o^{\sigma}_{c \rightarrow 0} (1)$ denotes a quantity going to
$0$ when $c \rightarrow 0$ at fixed $\sigma$. Combining Lemmas \ref{CP2nend},
\ref{CP2cor212N} and \ref{CP2zerosofQc}, we end the proof of Proposition
\ref{CP2prop5}.

\begin{proof}
  From (\ref{CP2sym}), we know that $Q_c$ enjoys the symmetry $Q_c (x_1, x_2)
  = Q_c (- x_1, x_2)$ for $(x_1, x_2) \in \mathbbm{R}^2$, hence we look at
  zeros only in the right half-plane. From Theorem \ref{th1}, we have $Q_c =
  V_1 (. - d_c \overrightarrow{e_1}) V_{- 1} (. + d_c \overrightarrow{e_1}) +
  \Gamma_c$ with $\| \Gamma_c \|_{L^{\infty} (\mathbbm{R}^2)} + \| \nabla
  \Gamma_c \|_{L^{\infty} (\mathbbm{R}^2)} = o_{c \rightarrow 0} (1)$. In the
  right half-plane and outside of $B (d_c \overrightarrow{e_1}, \Lambda)$ for
  any $\Lambda > 0$, by Lemma \ref{lemme3new}, we estimate
  \[ | Q_c | \geqslant | V_1 (. - d_c \overrightarrow{e_1}) V_{- 1} (. + d_c
     \overrightarrow{e_1}) | - o_{c \rightarrow 0} (1) \geqslant K (\Lambda) >
     0 \]
  if $c$ is small enough (depending on $\Lambda$). Now, we consider the smooth
  function $F : \mathbbm{R} \times \mathbbm{R}^2 \rightarrow \mathbbm{C}$
  defined by
  \[ F (\mu, z) \assign (V_1 (. - d_c \overrightarrow{e_1}) V_{- 1} (. + d_c
     \overrightarrow{e_1}) + \mu \Gamma_c (.)) (z + d_c \vec{e}_1) . \]
  We have $F (0, 0) = V_1 (0) V_{- 1} (2 d_c \overrightarrow{e_1}) = 0$ by
  Lemma \ref{lemme3new} and $F (1, z) = Q_c (z + d_c \vec{e}_1)$. For $| \mu |
  \leqslant 1$ and $| z | \leqslant 1$, since $\| \nabla \Gamma_c
  \|_{L^{\infty} (\mathbbm{R}^2)} = o^{\sigma}_{c \rightarrow 0} (c^{1 -
  \sigma})$ by equation (\ref{CP2225}), with Lemma \ref{lemme3new} and
  equation (\ref{CP2222405}), we check that
  \begin{equation}
    | d_z F_{(\mu, z)} (\xi) - \nabla V_1 (z) . \xi | = o_{c \rightarrow 0}
    (1) | \xi | \label{CP2llabel1}
  \end{equation}
  uniformly in $\mu \in [0, 1]$.
  
  Now, from Lemma \ref{lemme3new}, we estimate (for $x = r e^{i \theta} \neq 0
  \in \mathbbm{R}^2$)
  \begin{eqnarray*}
    \partial_{x_1} V_1 (x) & = & \left( \cos (\theta) \rho' (r) - \frac{i}{r}
    \sin (\theta) \rho (r) \right) e^{i \theta}\\
    & = & \kappa (\cos (\theta) - i \sin (\theta)) e^{i \theta} + o_{r
    \rightarrow 0} (1)\\
    & = & \kappa + o_{r \rightarrow 0} (1),
  \end{eqnarray*}
  and thus, by continuity, $\partial_{x_1} V_1 (0) = \kappa > 0$. Similarly,
  we check that $\partial_{x_2} V_1 (0) = - i \kappa$, and therefore,
  \[ \nabla V_1 (z) = \kappa \left(\begin{array}{c}
       1\\
       - i
     \end{array}\right) + o_{| z | \rightarrow 0} (1) . \]
  Identifying $\mathbbm{C}$ with $\mathbbm{R}^2$ canonically, we deduce that
  the Jacobian determinant of $F$ in $z$, $J (F)$, satisfies
  \[ J (F_{}) (\mu, z) = J (V_1) (z) + o_{c \rightarrow 0} (1) = - \kappa^2 +
     o_{c \rightarrow 0} (1) + o_{| z | \rightarrow 0} (1) \neq 0, \]
  given that $c$ and $| z |$ are small enough. By the implicit function
  theorem, there exists $\mu_0 > 0$ such that, for $| \mu | \leqslant \mu_0$,
  there exists a unique value $z (\mu)$ in a vicinity of $0$ such that $F
  (\mu, z (\mu)) = 0$, and since $\partial_{\mu} F (\mu, z) = \Gamma_c (d_c
  \vec{e}_1 + z) = o^{\sigma}_{c \rightarrow 0} (c^{1 - \sigma})$ uniformly in
  $z$ (by (\ref{CP2labasique})), it satisfies additionally $z (\mu) =
  o^{\sigma}_{c \rightarrow 0} (c^{1 - \sigma})$.
  
  Now, let us show that we can take $\mu_0 = 1$. Indeed, if we define $\mu_0
  = \sup \{ \nu > 0, \mu \rightarrow z (\mu) \in C^1 ([0, \nu], \mathbbm{R}^2)
  \} > 0$ and we have $\mu_0 < 1$, since $\mu \rightarrow z (\mu) \in C^1 ([0,
  \mu_0], \mathbbm{R}^2)$ with $| d_{\mu} z | (\mu) = o^{\sigma}_{c
  \rightarrow 0} (c^{1 - \sigma})$ uniformly in $[0, \mu_0]$, it can be
  continuously extended to $\mu_0$ with $F (\mu_0, z (\mu_0)) = 0$ and $z
  (\mu_0) = o^{\sigma}_{c \rightarrow 0} (c^{1 - \sigma})$. Then, by the
  implicit function theorem at $(\mu_0, z (\mu_0))$ (since $\mu_0 < 1$ with
  equation (\ref{CP2llabel1})), it can be extended above $\mu_0$, which is in
  contradiction with the definition of $\mu_0$.
  
  Since $F (1, .) = Q_c (. + d_c \vec{e}_1)$, we have shown that there exists
  $z \in \mathbbm{R}^2$ with $| z | = o^{\sigma}_{c \rightarrow 0} (c^{1 -
  \sigma})$ such that $Q_c (z + d_c \vec{e}_1) = 0$. Now, for $c$ small enough
  and $| \xi | \leqslant 1$, we have
  \[ \nabla (Q_c (\xi + z + d_c \vec{e}_1)) = \nabla V_1 (z) + o_{c
     \rightarrow 0} (1) + o_{| \xi | \rightarrow 0} (1) = \kappa
     \left(\begin{array}{c}
       1\\
       - i
     \end{array}\right) + o_{c \rightarrow 0} (1) + o_{| \xi | \rightarrow 0}
     (1) . \]
  We deduce, with $Q_c (\zeta + z + d_c \vec{e}_1) = \int_0^{| \zeta |} \nabla
  Q_c \left( \xi \frac{\zeta}{| \zeta_{} |} + z + d_c \vec{e}_1 \right) .
  \frac{\zeta}{| \zeta_{} |} d \xi$, that
  \[ \left| Q_c (\zeta + z + d_c \vec{e}_1) - \zeta . \left(\begin{array}{c}
       1\\
       - i
     \end{array}\right) \kappa \right| = o_{| \zeta | \rightarrow 0} (| \zeta
     |) + o_{c \rightarrow 0} (1) | \zeta | . \]
  Therefore, $Q_c$ has no other zeros in $B (z + d_c \vec{e}_1, \Lambda)$ for
  some $\Lambda > 0$ independent of $c$. Therefore, since for $c$ small
  enough, $| Q_c | > K (\Lambda) > 0$ outside of $B (z + d_c \vec{e}_1,
  \Lambda)$ in the right half-plane, $Q_c$ has only one zero in the right
  half-plane.
  
  By the symmetry $Q_c (x_1, x_2) = \overline{Q_c (x_1, - x_2)}$ (see
  (\ref{CP2sym})), $z$ must be colinear to $\overrightarrow{e_1}$, therefore
  we define $\widetilde{d_c} \in \mathbbm{R}$ by $\widetilde{d_c}
  \overrightarrow{e_1} \assign z + d_c \vec{e}_1$, and we conclude that, since
  $| z | = o^{\sigma}_{c \rightarrow 0} (c^{1 - \sigma})$,
  \[ | d_c - \widetilde{d_c} | = o^{\sigma}_{c \rightarrow 0} (c^{1 - \sigma})
     . \]
\end{proof}

We define the vortices around the zeros of $Q_c$ by
\[ \tilde{V}_{\pm 1} (x) \assign V_{\pm 1} (x \mp \widetilde{d_c}
   \overrightarrow{e_1}), \]
and we will use the already defined polar coordinates around $\pm
\widetilde{d_c} \overrightarrow{e_1}$ of $x \in \mathbbm{R}^2$, namely
\[ \tilde{r}_{\pm 1} = | x \mp \widetilde{d_c} \overrightarrow{e_1} |,
   \tilde{\theta}_{\pm 1} = \arg (x \mp \widetilde{d_c} \overrightarrow{e_1})
   . \]
One of the idea of the proof is to understand how $Q_c$ is close,
multiplicatively, to vortices $\tilde{V}_{\pm 1}$ centered at its zeros, since
by construction it is close to a vortex centered around $\pm d_c
\overrightarrow{e_1}$, which is itself close to $\pm \widetilde{d_c}
\overrightarrow{e_1}$. In particular, Lemma \ref{CP2closecall} below will show
that the ratio $\left| \frac{Q_c}{\tilde{V}_1} \right|$ is bounded and close
to 1 near $\widetilde{d_c} \overrightarrow{e_1}$.

\

In Lemma \ref{CP20524} to follow, we compute the additive perturbation
between derivatives of $Q_c$ and a vortex $\tilde{V}_{\pm 1}$ centered around
one of its zeros. In Lemma \ref{CP2closecall}, we compute the multiplicative
perturbation. All along, we work in $B (\widetilde{d_c} \overrightarrow{e_1},
\tilde{d}^{1 / 2}_c)$, the size of the ball $\tilde{d}^{1 / 2}_c$ being
arbitrary (any quantity that both goes to infinity when $c \rightarrow 0$ and
is a $o_{c \rightarrow 0} (\widetilde{d_c})$ should work). We recall that
$\tilde{r}_{\pm 1} = | x \mp \widetilde{d_c} \overrightarrow{e_1} |$.

\begin{lemma}
  \label{CP20524}Uniformly in $B (\widetilde{d_c} \overrightarrow{e_1},
  \tilde{d}^{1 / 2}_c)$, for $Q_c$ defined in Theorem \ref{th1}, one has
  \[ | Q_c - \widetilde{V_1} | = o_{c \rightarrow 0} (1), \]
  \[ | \nabla Q_c - \nabla \widetilde{V_1} | \leqslant \frac{o_{c \rightarrow
     0} (1)}{1 + \tilde{r}_1} \]
  and
  \[ | \nabla^2 Q_c - \nabla^2 \widetilde{V_1} | \leqslant \frac{o_{c
     \rightarrow 0} (1)}{1 + \tilde{r}_1} . \]
\end{lemma}

\begin{proof}
  From equations (\ref{CP2218}) and (\ref{CP2222405}), as well as Lemmas 2.6
  of {\cite{CP1}}, \ref{CP2zerosofQc} and the mean value theorem, in $B
  (\widetilde{d_c} \overrightarrow{e_1}, \tilde{d}^{1 / 2}_c)$,
  \begin{eqnarray}
    | Q_c - \widetilde{V_1} | & \leqslant & | Q_c - V | + | V -
    \widetilde{V_1} | \nonumber\\
    & \leqslant & o_{c \rightarrow 0} (1) + | V_1 (. - \widetilde{d_c}
    \overrightarrow{e_1}) - \widetilde{V_1} | \nonumber\\
    & \leqslant & o_{c \rightarrow 0} (1) + | d_c - \widetilde{d_c} |  \|
    \partial_{x_1} V \|_{L^{\infty} (\mathbbm{R}^2)} \nonumber\\
    & \leqslant & o_{c \rightarrow 0} (1),  \label{CP2llabel2}
  \end{eqnarray}
  which is the first statement.
  
  \
  
  For the second statement, we write $Q_c = V_1 (. - d_c
  \overrightarrow{e_1}) V_{- 1} (. - d_c \overrightarrow{e_1}) + \Gamma_c$,
  and from equation (\ref{CP2225}) (with some margin), we have
  \[ | \nabla \Gamma_c | \leqslant \frac{o_{c \rightarrow 0} (1)}{1 +
     \tilde{r}_1} . \]
  Furthermore, since $\tilde{V}_1 = V_1 (. - \widetilde{d_c}
  \overrightarrow{e_1})$,
  \[ \nabla (V_1 (. - d_c \overrightarrow{e_1}) V_{- 1} (. + d_c
     \overrightarrow{e_1})) - \nabla \widetilde{V_1} = \]
  \[ \nabla V_1 (. - d_c \overrightarrow{e_1}) V_{- 1} (. + d_c
     \overrightarrow{e_1}) - \nabla \widetilde{V_1} + V_1 (. - d_c
     \overrightarrow{e_1}) \nabla V_{- 1} (. + d_c \overrightarrow{e_1}), \]
  and from (\ref{CP2V-1est}), in $B (\widetilde{d_c} \overrightarrow{e_1},
  \tilde{d}^{1 / 2}_c)$, we have
  \[ | \nabla V_{- 1} (. + d_c \overrightarrow{e_1}) | \leqslant \frac{o_{c
     \rightarrow 0} (1)}{1 + \tilde{r}_1} . \]
  We compute
  \[ \nabla V_1 (. - d_c \overrightarrow{e_1}) V_{- 1} (. + d_c
     \overrightarrow{e_1}) - \nabla \widetilde{V_1} = \nabla V_1 (. - d_c
     \overrightarrow{e_1}) (V_{- 1} (. + d_c \overrightarrow{e_1}) - 1) -
     \nabla \widetilde{V_1} + \nabla V_1 (. - d_c \overrightarrow{e_1}) \]
  and, from (\ref{CP2222405}), in $B (\widetilde{d_c} \overrightarrow{e_1},
  \tilde{d}^{1 / 2}_c)$, we have $| V_{- 1} (. + d_c \overrightarrow{e_1}) - 1
  | = o_{c \rightarrow 0} (1)$. Finally, from Lemmas \ref{lemme3new} and
  \ref{CP2zerosofQc}, we estimate (with the mean value theorem)
  \[ | \nabla V_1 (. - d_c \overrightarrow{e_1}) - \nabla \widetilde{V_1} |
     \leqslant | d_c - \widetilde{d_c} | \sup_{d \in [d_c, \tilde{d}_c] \cup
     [\tilde{d}_c, d_c]} | \nabla^2 V_1 (x - d) | \leqslant K \frac{| d_c -
     \widetilde{d_c} |}{(1 + \tilde{r}_1)^2} = \frac{o_{c \rightarrow 0}
     (1)}{(1 + \tilde{r}_1)^2}, \]
  hence
  \begin{equation}
    | \nabla Q_c - \nabla \widetilde{V_1} | \leqslant \frac{o_{c \rightarrow
    0} (1)}{1 + \tilde{r}_1} . \label{CP2llabel3}
  \end{equation}
  Now, writing $w = Q_c - \widetilde{V_1}$, in $B (\widetilde{d_c}
  \overrightarrow{e_1}, 2 \tilde{d}^{1 / 2}_c)$, we estimate (since
  $\tmop{TW}_c (Q_c) = 0$ and $\Delta \widetilde{V_1} - (| \widetilde{V_1} |^2
  - 1) \widetilde{V_1} = 0$)
  \[ | \Delta w | = | - i c \partial_{x_2} Q_c - (1 - | Q_c |^2) Q_c + (1 - |
     \widetilde{V_1} |^2) \widetilde{V_1} | \leqslant \frac{o_{c \rightarrow
     0} (1)}{1 + \widetilde{r_1}} \]
  by equations (\ref{CP2217}) to (\ref{CP2221}) and (\ref{CP2222405}).
  Furthermore, by equations (\ref{CP2217}) to (\ref{CP2V-1est}), we have
  \[ | \nabla (\Delta w) | \leqslant \frac{o_{c \rightarrow 0} (1)}{(1 +
     \tilde{r}_1)} . \]
  We check, as the proof of (\ref{CP2llabel2}), that, in $B (\widetilde{d_c}
  \overrightarrow{e_1}, 2 \tilde{d}^{1 / 2}_c)$,
  \[ | w | = o_{c \rightarrow 0} (1), \]
  and, similarly, with equations (\ref{CP2V-1est}) and (\ref{CP2llabel3}),
  that
  \[ | \nabla w | = o_{c \rightarrow 0} (1) \]
  in $B (\widetilde{d_c} \overrightarrow{e_1}, 2 \tilde{d}^{1 / 2}_c)$. By
  Theorem 6.2 of {\cite{MR1814364}} (taking a domain $\Omega = B \left( x -
  \widetilde{d_c} \overrightarrow{e_1}, \frac{| x - \tilde{d}_c
  \overrightarrow{e_1} |}{2} \right)$, and $\alpha = 1 / 2$, but it also holds
  for any $0 < \alpha < 1$), we have, for $x \in B (\widetilde{d_c}
  \overrightarrow{e_1}, 2 \tilde{d}^{1 / 2}_c)$,
  \[ (1 + \tilde{r}_1)^2 | \nabla^2 w (x - \widetilde{d_c}
     \overrightarrow{e_1}) | \leqslant K (\| w \|_{C^1 (\Omega)} + (1 +
     \widetilde{r_1})^2 \| \Delta w \|_{C^1 (\Omega)}), \]
  and from the previous estimates, we have $\| w \|_{C^1 (\Omega)} = o_{c
  \rightarrow 0} (1)$ and $\| \Delta w \|_{C^1 (\Omega)} \leqslant \frac{o_{c
  \rightarrow 0} (1)}{(1 + \tilde{r}_1)}$, therefore
  \[ | \nabla^2 (Q_c - \widetilde{V_1}) | = | \nabla^2 w | \leqslant
     \frac{o_{c \rightarrow 0} (1)}{(1 + \widetilde{r_1})} . \]
\end{proof}

\begin{lemma}
  \label{CP2closecall}In $B (\widetilde{d_c} \overrightarrow{e_1},
  \tilde{d}^{1 / 2}_c)$, for $Q_c$ defined in Theorem \ref{th1}, we have
  \[ \left| \frac{Q_c}{\widetilde{V_1}} - 1 \right| = o_{c \rightarrow 0}
     (c^{1 / 10}) . \]
  In particular,
  \[ \left| \frac{Q_c}{\widetilde{V_1}} \right| = 1 + o_{c \rightarrow 0}
     (c^{1 / 10}) . \]
\end{lemma}

The power $1 / 10$ is arbitrary, but enough here for the upcoming estimations.

\begin{proof}
  We recall that both $Q_c$ and $\tilde{V}_1$ are $C^{\infty}$ since they are
  solutions of elliptic equations. We have that $Q_c (\widetilde{d_c}
  \overrightarrow{e_1}) = 0$ by Lemma \ref{CP2zerosofQc}, thus, for $x \in
  \mathbbm{R}^2$, by Taylor expansion, for $| x | \leqslant 1$,
  \[ Q_c (x + \widetilde{d_c} \overrightarrow{e_1}) = x. \nabla Q_c
     (\widetilde{d_c} \overrightarrow{e_1}) + O_{| x | \rightarrow 0} (| x
     |^2) . \]
  From Theorem \ref{th1}, we have $Q_c = V_1 (. - d_c \overrightarrow{e_1})
  V_{- 1} (. + d_c \overrightarrow{e_1}) + \Gamma_c$, therefore, with $V_{\pm
  1}$ being centered around $\pm d_c \overrightarrow{e_1}$ for the rest of the
  proof,
  \[ \nabla Q_c (\widetilde{d_c} \overrightarrow{e_1}) = \nabla V_1
     (\widetilde{d_c} \overrightarrow{e_1}) V_{- 1} (\widetilde{d_c}
     \overrightarrow{e_1}) + V_1 (\widetilde{d_c} \overrightarrow{e_1}) \nabla
     V_{- 1} (\widetilde{d_c} \overrightarrow{e_1}) + \nabla \Gamma_c
     (\widetilde{d_c} \overrightarrow{e_1}) . \]
  We have $V_1 (\widetilde{d_c} \overrightarrow{e_1}) \nabla V_{- 1}
  (\widetilde{d_c} \overrightarrow{e_1}) + \nabla \Gamma_c (\widetilde{d_c}
  \overrightarrow{e_1}) = o_{c \rightarrow 0} (c^{1 / 2})$ by Theorem
  \ref{th1}, Lemma \ref{lemme3new} and (\ref{CP2V-1est}). Furthermore, by
  (\ref{CP2222405}), Lemmas \ref{lemme3new} and \ref{CP2zerosofQc},
  \begin{eqnarray*}
    \nabla V_1 (\widetilde{d_c} \overrightarrow{e_1}) V_{- 1} (\widetilde{d_c}
    \overrightarrow{e_1}) & = & \nabla V_1 (\widetilde{d_c}
    \overrightarrow{e_1}) + o_{c \rightarrow 0} (c^{1 / 4})
  \end{eqnarray*}
  We deduce that
  \begin{equation}
    Q_c (x + \widetilde{d_c} \overrightarrow{e_1}) = x. (\nabla V_1 (d_c
    \overrightarrow{e_1}) + o_{c \rightarrow 0} (c^{1 / 4})) + O_{x
    \rightarrow 0} (| x |^2) . \label{CP2eq222502}
  \end{equation}
  We also have $\widetilde{V_1} (x + \widetilde{d_c} \overrightarrow{e_1}) =
  x. \nabla \widetilde{V_1} (\widetilde{d_c} \overrightarrow{e_1}) + O_{x
  \rightarrow 0} (| x |^2)$ (since $\widetilde{V_1} (\widetilde{d_c}
  \overrightarrow{e_1}) = 0$) and $\nabla V_1 (d_c \overrightarrow{e_1}) =
  \nabla \widetilde{V_1} (\widetilde{d_c} \overrightarrow{e_1})$, hence
  \[ Q_c (x + \widetilde{d_c} \overrightarrow{e_1}) = \widetilde{V_1} (x +
     \widetilde{d_c} \overrightarrow{e_1}) + x.o_{c \rightarrow 0} (c^{1 / 4})
     + O_{| x | \rightarrow 0} (| x |^2) . \]
  Now, by Lemma \ref{lemme3new}, there exists $K > 0$ such that, in $B
  (\widetilde{d_c} \overrightarrow{e_1}, c^{1 / 4})$ for $c$ small enough, $|
  \widetilde{V_1} (x + \widetilde{d_c} \overrightarrow{e_1}) | \geqslant K | x
  |$. We deduce that
  \begin{eqnarray*}
    \left| \frac{Q_c}{\widetilde{V_1}} - 1 \right| & \leqslant & \frac{| x |
    o_{c \rightarrow 0} (c^{1 / 4})}{| \widetilde{V_1} (x + \widetilde{d_c}
    \overrightarrow{e_1}) |} + \frac{O_{| x | \rightarrow 0} (| x |^2)}{|
    \widetilde{V_1} (x + \widetilde{d_c} \overrightarrow{e_1}) |}\\
    & \leqslant & o_{c \rightarrow 0} (c^{1 / 4}) + O_{| x | \rightarrow 0}
    (| x |)\\
    & \leqslant & o_{c \rightarrow 0} (c^{1 / 5}) .
  \end{eqnarray*}
  Outside of $B (\widetilde{d_c} \overrightarrow{e_1}, c^{1 / 4})$ and in $B
  (\widetilde{d_c} \overrightarrow{e_1}, \widetilde{d_c}^{1 / 2})$, we have $|
  \widetilde{V_1} | \geqslant K c^{1 / 4}$ by Lemma \ref{lemme3new}, and
  \[ Q_c = V_1 + O_{c \rightarrow 0} (c^{1 / 2}) \]
  by Theorem \ref{th1}, equations (\ref{CP2218}) and (\ref{CP2222405}). We
  deduce
  \[ \left| \frac{Q_c}{\widetilde{V_1}} - 1 \right| (x) = \left| \frac{V_1 +
     O_{c \rightarrow 0} (c^{1 / 2})}{\widetilde{V_1}} - 1 \right| (x) =
     \left| \frac{V_1 (x)}{\widetilde{V_1} (x)} - 1 \right| + o_{c \rightarrow
     0} (c^{1 / 10}) . \]
  Furthermore, by Lemma \ref{CP2zerosofQc} (for $\sigma = 1 / 2$), we have
  \[ \left| \frac{V_1 (x)}{\widetilde{V_1} (x)} - 1 \right| = \left|
     \frac{\widetilde{V_1} (x) + O_{| d_c - \tilde{d}_c | \rightarrow 0} (|
     d_c - \widetilde{d_c} |)}{\widetilde{V_1} (x)} - 1 \right| = \frac{O_{|
     d_c - \tilde{d}_c | \rightarrow 0} (| d_c - \widetilde{d_c} |)}{c^{1 /
     4}} = o_{c \rightarrow 0} (c^{1 / 10}) . \]
  We conclude that $\left| \frac{Q_c}{\tilde{V}_1} - 1 \right| = o_{c
  \rightarrow 0} (c^{1 / 10})$ in $B (\widetilde{d_c} \overrightarrow{e_1},
  \widetilde{d_c}^{1 / 2})$.
\end{proof}

By the symmetries of $Q_c$ (see (\ref{CP2sym})), the result of Lemma
\ref{CP2closecall} holds if we change $\overrightarrow{e_1}$ by $-
\overrightarrow{e_1}$ and $\widetilde{V_1}$ by $\tilde{V}_{- 1}$.

\

We conclude this section with the proof that in $B (\pm \tilde{d}_c
\overrightarrow{e_1}, \tilde{d}_c^{1 / 2})$, we have, for $\psi \in
C^{\infty}_c (\mathbbm{R}^2 \backslash \{ \pm \tilde{d}_c \overrightarrow{e_1}
\}, \mathbbm{C})$,
\begin{equation}
  \int_0^{2 \pi} | \psi^{\neq 0} |^2 d \tilde{\theta}_{\pm 1} \leqslant
  \tilde{r}_{\pm 1}^2 \int_0^{2 \pi} | \nabla \psi |^2 d \tilde{\theta}_{\pm
  1} . \label{CP2jesaispasfairedesmaths}
\end{equation}
We recall that the function $\psi^{\neq 0}$ is defined by
\[ \psi^{\neq 0} (x) = \psi (x)^{^{}} - \psi^{0, 1} (\tilde{r}_1) \]
in the right half-plane, and
\[ \psi^{\neq 0} (x) = \psi (x)^{^{}} - \psi^{0, - 1} (\tilde{r}_{- 1}) \]
in the left half-plane.

To show (\ref{CP2jesaispasfairedesmaths}), it is enough to show that, for
$\psi \in C^{\infty}_c (\mathbbm{R}^2 \backslash \{ 0 \}, \mathbbm{C})$, we
have, with $x = r e^{i \theta}$,
\[ \int_0^{2 \pi} \left| \psi - \int_0^{2 \pi} \psi d \gamma \right|^2 d
   \theta \leqslant r^2 \int_0^{2 \pi} | \nabla \psi |^2 d \theta . \]
This is a Poincar{\'e} inequality. By decomposition in harmonics and
Parseval's equality, we have
\begin{eqnarray*}
  \int_0^{2 \pi} \left| \psi - \int_0^{2 \pi} \psi (\gamma) d \gamma \right|^2
  d \theta & = & \int_0^{2 \pi} \left| \sum_{n \in \mathbbm{Z}^{\ast}} \psi_n
  (r) e^{i n \theta} \right|^2 d \theta\\
  & = & \int_0^{2 \pi} \sum_{n \in \mathbbm{Z}^{\ast}} | \psi_n (r) |^2 d
  \theta,
\end{eqnarray*}
and
\begin{eqnarray*}
  \int_0^{2 \pi} | \nabla \psi |^2 d \theta & \geqslant & \int_0^{2 \pi}
  \frac{1}{r^2} | \partial_{\theta} \psi |^2 d \theta\\
  & \geqslant & \int_0^{2 \pi} \left| \sum_{n \in \mathbbm{Z}^{\ast}} i
  \frac{n \psi_n (r)}{r} e^{i n \theta} \right|^2 d \theta\\
  & \geqslant & \frac{1}{r^2} \int_0^{2 \pi} \sum_{n \in \mathbbm{Z}^{\ast}}
  n^2 | \psi_n (r) |^2 d \theta\\
  & \geqslant & \frac{1}{r^2} \int_0^{2 \pi} \sum_{n \in \mathbbm{Z}^{\ast}}
  | \psi_n (r) |^2 d \theta .
\end{eqnarray*}

This concludes the proof of (\ref{CP2jesaispasfairedesmaths}). With $| Q_c (x
\pm \tilde{d}_c \overrightarrow{e_1}) | = O_{\tilde{r}_{\pm 1} \rightarrow 0}
(\tilde{r}_{\pm 1})$ and (\ref{CP2jesaispasfairedesmaths}), we have, for
$\tilde{r}_{\pm 1} \leqslant R$,
\begin{eqnarray}
  \int_0^{2 \pi} | Q_c |^2 | \psi^{\neq 0} |^2 d \tilde{\theta}_{\pm 1} &
  \leqslant & K \int_0^{2 \pi} \tilde{r}_{\pm 1}^2 | \psi^{\neq 0} |^2 d
  \tilde{\theta}_{\pm 1} \nonumber\\
  & \leqslant & K \int_0^{2 \pi} \tilde{r}_{\pm 1}^4 | \nabla \psi |^2 d
  \tilde{\theta}_{\pm 1} \nonumber\\
  & \leqslant & K (R) \int_0^{2 \pi} | Q_c |^4 | \nabla \psi |^2 d
  \tilde{\theta}_{\pm 1} . \label{CP2jesaistoujourspasfairedesmaths} 
\end{eqnarray}
This result will be usefull to estimate the quantities in the orthogonality
conditions.

\section{Estimations in $H_{Q_c}$}\label{CP2NSEC4}

We give several estimates for functions in $H_{Q_c}$. They will in particular
allow us to use a density argument to show Proposition \ref{CP205218} once it
is shown for test function in section \ref{CP2NSEC3}. We will also explain why
a coercivity result with the energy norm $\| . \|_{H_{Q_c}}$ is impossible
with any number of local orthogonality conditions, and show that the quadratic
form and the coercivity norm are well defined for functions in $H_{Q_c}$.

\subsection{Comparaison of the energy and coercivity norms}\label{CP2CandHc}

In the introduction, we have defined the quadratic form by
\begin{eqnarray*}
  B_{Q_c} (\varphi) & = & \int_{\mathbbm{R}^2} | \nabla \varphi |^2 - (1 - |
  Q_c |^2) | \varphi |^2 + 2\mathfrak{R}\mathfrak{e}^2 (\overline{Q_c}
  \varphi)\\
  & - & c \int_{\mathbbm{R}^2} (1 - \eta) \mathfrak{R}\mathfrak{e} (i
  \partial_{x_2} \varphi \bar{\varphi}) - c \int_{\mathbbm{R}^2 \nosymbol}
  \eta \mathfrak{R}\mathfrak{e} (i \partial_{x_2} Q_c \overline{Q_c}) | \psi
  |^2\\
  & + & 2 c \int_{\mathbbm{R}^2} \eta \mathfrak{R}\mathfrak{e} \psi
  \mathfrak{I}\mathfrak{m} \partial_{x_2} \psi | Q_c |^2 + c
  \int_{\mathbbm{R}^2} \partial_{x_2} \eta \mathfrak{R}\mathfrak{e} \psi
  \mathfrak{I}\mathfrak{m} \psi | Q_c |^2\\
  & + & c \int_{\mathbbm{R}^2} \eta \mathfrak{R}\mathfrak{e} \psi
  \mathfrak{I}\mathfrak{m} \psi \partial_{x_2} (| Q_c |^2)
\end{eqnarray*}
(see (\ref{CP2truebqc})). We will show in Lemma \ref{CP2finitebilinear} below
that this quantity is well defined for $\varphi \in H_{Q_c}$. As we have seen,
the natural energy space $H_{Q_c}$ is given by the norm
\[ \| \varphi \|^2_{H_{Q_c}} = \int_{\mathbbm{R}^2} | \nabla \varphi |^2 + | 1
   - | Q_c |^2 | | \varphi |^2 +\mathfrak{R}\mathfrak{e}^2 (\overline{Q_c}
   \varphi) . \]
We could expect to remplace Theorem \ref{CP2th2} by a result of the form: up
to some local orthogonality conditions, for $\varphi \in H_{Q_c}$ we have
\[ B_{Q_c} (\varphi) \geqslant K (c) \| \varphi \|^2_{H_{Q_c}} . \]
However such a result can not hold. This is because of a formal zero of
$L_{Q_c}$ which is not in the space $H_{Q_c}$: $i Q_c$ (which comes from the
phase invariance of the equation). We have $L_{Q_c} (i Q_c) = 0$ and $i Q_c
\nin H_{Q_c}$ because
\[ (1 - | Q_c |^2) | i Q_c |^2 \]
is not integrable at infinity (see {\cite{MR2191764}}, where it is shown that
this quantity decays like $1 / r^2$). We can then create functions in
$H_{Q_c}$ getting close to $i Q_c$, for instance
\[ f_R = \eta_R i Q_c, \]
where $\eta_R$ is a $C^{\infty}$ real function with value $1$ if $R_0 < | x |
< R$ and value $0$ if $| x | < R_0 - 1 \tmop{or} | x | > 2 R$. In that case,
when $R \rightarrow + \infty$, $\| f_R \|_{H_{Q_c}} \rightarrow + \infty$ and
$B_{Q_c} (f_R) \rightarrow C$ a constant independent of $R$, making the
inequality $B_{Q_c} (\varphi) \geqslant K \| \varphi \|^2_{H_{Q_c}}$
impossible (and the local orthogonality conditions are verified for $R_0$
large enough since $f_R = 0$ on $B (0, R_0 - 1)$). That is why we get the
result in a weaker norm in Proposition \ref{CP2prop17}: we will only get for
$\varphi \in H_{Q_c}$, up to some local orthogonality conditions,
\[ B_{Q_c} (\varphi) \geqslant K (c) \| \varphi \|_{H^{\exp}_{Q_c}}^2, \]
where $\| . \|_{H^{\exp}_{Q_c}}$ is defined in subsection \ref{CP21310811}. In
particular, $\| . \|_{H^{\exp}_{Q_c}}$ is not equivalent to $\| .
\|_{H_{Q_c}}$.

\

\subsection{The coercivity norm and other quantities are well defined in
$H_{Q_c}$}\label{CP2Nrouge}

We have defined the energy space $H_{Q_c}$ by the norm
\[ \| \varphi \|^2_{H_{Q_c}} = \int_{\mathbbm{R}^2} | \nabla \varphi |^2 + | 1
   - | Q_c |^2 | | \varphi |^2 +\mathfrak{R}\mathfrak{e}^2 (\overline{Q_c}
   \varphi) . \]
By Lemma \ref{CP2sensmanqu}, we have that, for $\varphi \in H_{Q_c}$,
\begin{equation}
  \int_{\mathbbm{R}^2} \frac{| \varphi |^2}{(1 + | x |)^2} d x \leqslant C (c)
  \| \varphi \|^2_{H_{Q_c}} \label{CP22102est} .
\end{equation}
The goal of this subsection is to show that for $\varphi \in H_{Q_c}$, $\|
\varphi \|_{\mathcal{C}}$ and $B_{Q_c} (\varphi)$, as well as the quantities
in the orthogonality conditions of Proposition \ref{CP205218} and Theorem
\ref{CP2th2}, are well defined. This is done in Lemmas \ref{CP2farnormexist}
to \ref{CP2finitebilinear}.

\begin{lemma}
  \label{CP2farnormexist}There exists $c_{0 \nosymbol} > 0$ such that for $0 <
  c \leqslant c_0$, there exists $C (c) > 0$ such that, for $Q_c$ defined in
  Theorem \ref{th1} and for any $\varphi = Q_c \psi \in H_{Q_c}$,
  \[ \| \varphi \|_{\mathcal{C}}^2 = \int_{\mathbbm{R}^2} | \nabla \psi |^2 |
     Q_c |^4 +\mathfrak{R}\mathfrak{e}^2 (\psi) | Q_c |^4 \leqslant C (c) \|
     \varphi \|_{H_{Q_c}}^2 . \]
\end{lemma}

\begin{proof}
  We estimate for $\varphi = Q_c \psi \in H_{Q_c}$, using equations
  (\ref{CP2Qcpaszero}), (\ref{CP22102est}) and $| \nabla Q_c | \leqslant
  \frac{C (c)}{(1 + r)^2}$ from Theorem \ref{CP2Qcbehav}, that
  \begin{eqnarray*}
    \int_{\mathbbm{R}^2} | \nabla \psi |^2 | Q_c |^4 & = & 
    \int_{\mathbbm{R}^2} | \nabla \varphi - \nabla Q_c \psi |^2 | Q_c |^2\\
    & \leqslant & K \int_{\mathbbm{R}^2} | \nabla \varphi |^2 | Q_c |^2 + |
    \nabla Q_c |^2 | Q_c \psi |^2\\
    & \leqslant & K (c) \int_{\mathbbm{R}^2} | \nabla \varphi |^2 + \frac{|
    \varphi |^2}{(1 + r)^4}\\
    & \leqslant & K (c) \| \varphi \|_{H_{Q_c}}^2 .
  \end{eqnarray*}
  Similarly, for $\varphi = Q_c \psi$,
  \[ \int_{\mathbbm{R}^2} \mathfrak{R}\mathfrak{e}^2 (\psi) | Q_c |^4 =
     \int_{\mathbbm{R}^2} \mathfrak{R}\mathfrak{e}^2 (\overline{Q_c} \varphi)
     \leqslant \| \varphi \|^2_{H_{Q_c}} . \]
  We conclude that
  \begin{equation}
    \int_{\mathbbm{R}^2} | \nabla \psi |^2 | Q_c |^4
    +\mathfrak{R}\mathfrak{e}^2 (\psi) | Q_c |^4 \leqslant C (c) \| \varphi
    \|^2_{H_{Q_c}} . \label{CP212372202}
  \end{equation}
\end{proof}

We conclude this subsection with the proof that the quantities in the
orthogonality conditions are well defined for $\varphi \in H_{Q_c}$.

\begin{lemma}
  \label{CP2L2111706}There exists $K > 0$ and, for $c$ small enough, there
  exists $K (c) > 0$ such that, for $Q_c$ defined in Theorem \ref{th1} and
  $\varphi = Q_c \psi \in H_{Q_c}$, $0 < R < \tilde{d}_c^{1 / 2}$, we have
  \[ \int_{B (\pm \tilde{d}_c \overrightarrow{e_1}, R)} |
     \mathfrak{R}\mathfrak{e} (\partial_{x_1} \tilde{V}_{\pm 1} \overline{_{}
     \tilde{V}_{\pm 1} \psi}) | + \int_{B (\pm \tilde{d}_c
     \overrightarrow{e_1}, R)} | \mathfrak{R}\mathfrak{e} (\partial_{x_2}
     \tilde{V}_{\pm 1} \overline{_{} \tilde{V}_{\pm 1} \psi}) | \leqslant K
     (c) \| \varphi \|_{H_{Q_c}}, \]
  \[ \int_{B (\tilde{d}_c \overrightarrow{e_1}, R) \cup B (- \tilde{d}_c
     \overrightarrow{e_1}, R)} | \mathfrak{R}\mathfrak{e} (\partial_{x_{1, 2}}
     Q_c \overline{Q_c \psi^{\neq 0}}) | \leqslant K (c) \| \varphi
     \|_{H_{Q_c}}, \]
  \[ \int_{B (\tilde{d}_c \overrightarrow{e_1}, R) \cup B (- \tilde{d}_c
     \overrightarrow{e_1}, R)} | \mathfrak{R}\mathfrak{e} (\partial_c Q_c
     \overline{Q_c \psi^{\neq 0}}) | \leqslant K (c) \| \varphi \|_{H_{Q_c}}
  \]
  and
  \[ \int_{B (\tilde{d}_c \overrightarrow{e_1}, R) \cup B (- \tilde{d}_c
     \overrightarrow{e_1}, R)} | \mathfrak{R}\mathfrak{e} (- x^{\bot} . \nabla
     Q_c \overline{Q_c \psi^{\neq 0}}) | \leqslant K (c) \| \varphi
     \|_{H_{Q_c}} . \]
\end{lemma}

We recall that $\psi^{\neq 0} (x) = \psi (x)^{^{}} - \psi^{0, 1}
(\tilde{r}_1)$ in the right half-plane and $\psi^{\neq 0} (x) = \psi (x)^{^{}}
- \psi^{0, - 1} (\tilde{r}_{- 1})$ in the left half-plane, with
$\tilde{r}_{\pm 1} = | x \mp \widetilde{d_c} \overrightarrow{e_1} |$ and
$\psi^{0, \pm 1} (\tilde{r}_{\pm 1})$ the $0$-harmonic of $\psi$ around $\pm
\widetilde{d_c} \overrightarrow{e_1}$.

\begin{proof}
  From Lemma \ref{CP2closecall}, we have, for $\varphi = Q_c \psi \in
  H_{Q_c}$,
  \[ | \tilde{V}_{\pm 1} \psi | = | \varphi | \times \left|
     \frac{\tilde{V}_{\pm 1}}{Q_c} \right| \leqslant 2 | \varphi | \]
  given that $c$ is small enough. We deduce by Cauchy-Schwarz, Lemmas
  \ref{lemme3new} and \ref{CP2sensmanqu} that
  \begin{eqnarray*}
    \int_{B (\pm \tilde{d}_c \overrightarrow{e_1}, R)} |
    \mathfrak{R}\mathfrak{e} (\partial_{x_1} \tilde{V}_{\pm 1} \overline{_{}
    \tilde{V}_{\pm 1} \psi}) | & \leqslant & 2 \int_{B (\pm \tilde{d}_c
    \overrightarrow{e_1}, R)} | \partial_{x_1} \tilde{V}_{\pm 1} | \times |
    \varphi | \leqslant \tmmathbf{} K (c) \| \varphi \|_{H^1 (B (\pm
    \tilde{d}_c \overrightarrow{e_1}, R))}\\
    & \leqslant & K (c) \| \varphi \|_{H_{Q_c}},
  \end{eqnarray*}
  and similarly $\int_{B (\pm \tilde{d}_c \overrightarrow{e_1}, R)} |
  \mathfrak{R}\mathfrak{e} (\partial_{x_2} \tilde{V}_{\pm 1} \overline{_{}
  \tilde{V}_{\pm 1} \psi}) | \leqslant K (c) \| \varphi \|_{H_{Q_c}}$.
  
  \
  
  \
  
  By Cauchy-Schwarz, equation (\ref{CP212372202}) and Theorem \ref{th1} (for
  $p = + \infty$), we conclude that
  \begin{eqnarray*}
    \int_{B (\tilde{d}_c \overrightarrow{e_1}, R) \cup B (- \tilde{d}_c
    \overrightarrow{e_1}, R)} | \mathfrak{R}\mathfrak{e} (\partial_c Q_c
    \overline{Q_c \psi^{\neq 0}}) | & \leqslant & K (c) \sqrt{\int_{B
    (\tilde{d}_c \overrightarrow{e_1}, R) \cup B (- \tilde{d}_c
    \overrightarrow{e_1}, R)} | \nabla \psi |^2 | Q_c |^4}\\
    & \leqslant & K (c) \| \varphi \|_{H_{Q_c}} .
  \end{eqnarray*}
  We can estimate the other terms similarly.
\end{proof}

\subsection{On the definition of $B_{Q_c}$}\label{CP22241103}

We start by explaining how to get $B_{Q_c} (\varphi)$ from the ``natural''
quadratic form
\[ \int_{\mathbbm{R}^2} | \nabla \varphi |^2 - (1 - | Q_c |^2) | \varphi |^2 +
   2\mathfrak{R}\mathfrak{e}^2 (\overline{Q_c} \varphi)
   -\mathfrak{R}\mathfrak{e} (i c \partial_{x_2} \varphi \bar{\varphi}) . \]
For the first three terms of this quantity, it is obvious that they are well
defined for $\varphi \in H_{Q_c}$, but the term $-\mathfrak{R}\mathfrak{e} (i
c \partial_{x_2} \varphi \bar{\varphi})$ is not clearly integrable.

Take a smooth cutoff function $\eta$ such that $\eta (x) = 0$ on $B (\pm
\widetilde{d_c} \overrightarrow{e_1}, 1)$, $\eta (x) = 1$ on $\mathbbm{R}^2
\backslash B (\pm \widetilde{d_c} \overrightarrow{e_1}, 2)$. Then, taking for
now $\varphi = Q_c \psi \in C^{\infty}_c (\mathbbm{R}^2)$,
\[ \mathfrak{R}\mathfrak{e} (i \partial_{x_2} \varphi \bar{\varphi}) = \eta
   \mathfrak{R}\mathfrak{e} (i \partial_{x_2} \varphi \bar{\varphi}) + (1 -
   \eta) \mathfrak{R}\mathfrak{e} (i \partial_{x_2} \varphi \bar{\varphi}), \]
and writing $\varphi = Q_c \psi$,
\begin{eqnarray*}
  \eta \mathfrak{R}\mathfrak{e} (i \partial_{x_2} \varphi \bar{\varphi}) & = &
  \eta \mathfrak{R}\mathfrak{e} (i \partial_{x_2} Q_c \overline{Q_c}) | \psi
  |^2 + \eta \mathfrak{R}\mathfrak{e} (i \partial_{x_2} \psi \bar{\psi}) | Q_c
  |^2\\
  & = & _{\nosymbol} \eta \mathfrak{R}\mathfrak{e} (i \partial_{x_2} Q_c
  \overline{Q_c}) | \psi |^2 - \eta \mathfrak{R}\mathfrak{e} \psi
  \mathfrak{I}\mathfrak{m} \partial_{x_2} \psi | Q_c |^2\\
  & + & \eta \mathfrak{R}\mathfrak{e} \partial_{x_2} \psi
  \mathfrak{I}\mathfrak{m} \psi | Q_c |^2 .
\end{eqnarray*}
Furthermore,
\begin{eqnarray*}
  \eta \mathfrak{R}\mathfrak{e} \partial_{x_2} \psi \mathfrak{I}\mathfrak{m}
  \psi | Q_c |^2 & = & \partial_{x_2} (\eta \mathfrak{R}\mathfrak{e} \psi
  \mathfrak{I}\mathfrak{m} \psi | Q_c |^2)\\
  & - & \partial_{x_2} \eta \mathfrak{R}\mathfrak{e} \psi
  \mathfrak{I}\mathfrak{m} \psi | Q_c |^2 - \eta \mathfrak{R}\mathfrak{e} \psi
  \mathfrak{I}\mathfrak{m} \partial_{x_2} \psi | Q_c |^2\\
  & - & \eta \mathfrak{R}\mathfrak{e} \psi \mathfrak{I}\mathfrak{m} \psi
  \partial_{x_2} (| Q_c |^2),
\end{eqnarray*}
thus we can write
\begin{eqnarray*}
  \int_{\mathbbm{R}^2} \mathfrak{R}\mathfrak{e} (i \partial_{x_2} \varphi
  \bar{\varphi}) & = & \int_{\mathbbm{R}^2} \partial_{x_2} (\eta
  \mathfrak{R}\mathfrak{e} \psi \mathfrak{I}\mathfrak{m} \psi | Q_c |^2)\\
  & + & \int_{\mathbbm{R}^2} (1 - \eta) \mathfrak{R}\mathfrak{e} (i
  \partial_{x_2} \varphi \bar{\varphi}) + \int_{\mathbbm{R}^2 \nosymbol} \eta
  \mathfrak{R}\mathfrak{e} (i \partial_{x_2} Q_c \overline{Q_c}) | \psi |^2\\
  & - & 2 \int_{\mathbbm{R}^2} \eta \mathfrak{R}\mathfrak{e} \psi
  \mathfrak{I}\mathfrak{m} \partial_{x_2} \psi | Q_c |^2 -
  \int_{\mathbbm{R}^2} \partial_{x_2} \eta \mathfrak{R}\mathfrak{e} \psi
  \mathfrak{I}\mathfrak{m} \psi | Q_c |^2\\
  & - & \int_{\mathbbm{R}^2} \eta \mathfrak{R}\mathfrak{e} \psi
  \mathfrak{I}\mathfrak{m} \psi \partial_{x_2} (| Q_c |^2) .
\end{eqnarray*}
The only difficulty here is that the first integral is not well defined for
$\varphi \in H_{Q_c}$, but it is the integral of a derivative. Therefore, this
is why we defined instead the quadratic form
\begin{eqnarray*}
  B_{Q_c} (\varphi) & = & \int_{\mathbbm{R}^2} | \nabla \varphi |^2 - (1 - |
  Q_c |^2) | \varphi |^2 + 2\mathfrak{R}\mathfrak{e}^2 (\overline{Q_c}
  \varphi)\\
  & - & c \int_{\mathbbm{R}^2} (1 - \eta) \mathfrak{R}\mathfrak{e} (i
  \partial_{x_2} \varphi \bar{\varphi}) - c \int_{\mathbbm{R}^2 \nosymbol}
  \eta \mathfrak{R}\mathfrak{e} (i \partial_{x_2} Q_c \overline{Q_c}) | \psi
  |^2\\
  & + & 2 c \int_{\mathbbm{R}^2} \eta \mathfrak{R}\mathfrak{e} \psi
  \mathfrak{I}\mathfrak{m} \partial_{x_2} \psi | Q_c |^2 + c
  \int_{\mathbbm{R}^2} \partial_{x_2} \eta \mathfrak{R}\mathfrak{e} \psi
  \mathfrak{I}\mathfrak{m} \psi | Q_c |^2\\
  & + & c \int_{\mathbbm{R}^2} \eta \mathfrak{R}\mathfrak{e} \psi
  \mathfrak{I}\mathfrak{m} \psi \partial_{x_2} (| Q_c |^2) .
\end{eqnarray*}
It is easy to check that this quantity is independent of the choice of $\eta$.
We will show in Lemma \ref{CP2finitebilinear} that this quantity is well
defined for $\varphi \in H_{Q_c}$. By adding some conditions on $\varphi$, for
instance if $\varphi \in H^1 (\mathbbm{R}^2)$, we can show that
$\int_{\mathbbm{R}^2} \partial_{x_2} (\eta \mathfrak{R}\mathfrak{e} \psi
\mathfrak{I}\mathfrak{m} \psi | Q_c |^2)$ is well defined and is $0$. In these
cases, we therefore have
\[ B_{Q_c} (\varphi) = \int_{\mathbbm{R}^2} | \nabla \varphi |^2 - (1 - | Q_c
   |^2) | \varphi |^2 + 2\mathfrak{R}\mathfrak{e}^2 (\overline{Q_c} \varphi)
   -\mathfrak{R}\mathfrak{e} (i c \partial_{x_2} \varphi \bar{\varphi}) . \]
This is a classical situation for Schr{\"o}dinger equations with nonzero limit
at infinity (see {\cite{MR3686002}} or {\cite{MR3043579}}): the quadratic form
is defined up to a term which is a derivative of some function in some $L^p$
space.

\begin{lemma}
  \label{CP2finitebilinear}There exists $c_{0 \nosymbol} > 0$ such that, for
  $0 < c \leqslant c_0$, $Q_c$ defined in Theorem \ref{th1}, \ there exists a
  constant $C_{\nosymbol} (c) > 0$ such that, for $\varphi = Q_c \psi \in
  H_{Q_c}$ and $\eta$ a smooth cutoff function such that $\eta (x) = 0$ on $B
  (\pm \widetilde{d_c} \overrightarrow{e_1}, 1)$, $\eta (x) = 1$ on
  $\mathbbm{R}^2 \backslash B (\pm \widetilde{d_c} \overrightarrow{e_1}, 2)$,
  we have
  \begin{eqnarray*}
    &  & \int_{\mathbbm{R}^2} | (1 - \eta) \mathfrak{R}\mathfrak{e} (i
    \partial_{x_2} \varphi \bar{\varphi}) | + \int_{\mathbbm{R}^2 \nosymbol} |
    \eta \mathfrak{R}\mathfrak{e} (i \partial_{x_2} Q_c \overline{Q_c}) | \psi
    |^2 |\\
    & + & \int_{\mathbbm{R}^2} | \eta \mathfrak{R}\mathfrak{e} \psi
    \mathfrak{I}\mathfrak{m} (\partial_{x_2} \psi) | Q_c |^2 | +
    \int_{\mathbbm{R}^2} | \partial_{x_2} \eta \mathfrak{R}\mathfrak{e} \psi
    \mathfrak{I}\mathfrak{m} \psi | Q_c |^2 |\\
    & + & \int_{\mathbbm{R}^2} | \eta \mathfrak{R}\mathfrak{e} \psi
    \mathfrak{I}\mathfrak{m} \psi \partial_{x_2} (| Q_c |^2) |\\
    & \leqslant & C (c) \| \varphi \|_{H_{Q_c}}^2 .
  \end{eqnarray*}
\end{lemma}

\begin{proof}
  Since $| 1 - | Q_c |^2 | \geqslant K > 0$ on $B (\pm \widetilde{d_c}
  \overrightarrow{e_1}, 2)$ for $c$ small enough by Lemma \ref{lemme3new} and
  Theorem \ref{th1}, we estimate
  \[ \int_{\nosymbol \mathbbm{R}^2} | (1 - \eta) \mathfrak{R}\mathfrak{e} (i c
     \partial_{x_2} \varphi \bar{\varphi}) | \leqslant C (c) \int_{\nosymbol B
     (\tilde{d}_c \overrightarrow{e_1}, 2) \cup B (- \tilde{d}_c
     \overrightarrow{e_1}, 2)} | 1 - | Q_c |^2 | | \varphi | | \partial_{x_2}
     \varphi | \leqslant C (c) \| \varphi \|^2_{H_{Q_c}} . \]
  Furthermore, by (\ref{CP2Qcpaszero}) and Lemma \ref{CP2sensmanqu},
  \[ \int_{\mathbbm{R}^2} | \eta \mathfrak{R}\mathfrak{e} (i c \partial_{x_2}
     Q_c \overline{Q_c}) | \psi |^2 | \leqslant C (c) \int_{\nosymbol
     \mathbbm{R}^2} \eta | \nabla Q_c | | \psi |^2 \leqslant C (c)
     \int_{\mathbbm{R}^2 \nosymbol} \eta | \nabla Q_c | | \varphi |^2
     \leqslant C (c) \| \varphi \|^2_{H_{Q_c}} \]
  since $| \nabla Q_c | \leqslant \frac{C (c)}{(1 + r)^2}$ from Theorem
  \ref{CP2Qcbehav}. By Cauchy-Schwarz, equations (\ref{CP2Qcpaszero}) and
  Lemma \ref{CP2farnormexist},
  \begin{equation}
    \int_{\mathbbm{R}^2} | \eta \mathfrak{R}\mathfrak{e} \psi
    \mathfrak{I}\mathfrak{m} \partial_{x_2} \psi | Q_c |^2 | \leqslant K
    \sqrt{\int_{\mathbbm{R}^2} \eta \mathfrak{R}\mathfrak{e}^2 (\psi) 
    \int_{\mathbbm{R}^2} \eta | \nabla \psi |^2} \leqslant C (c) \| \varphi
    \|^2_{H_{Q_c}} \label{CP2raz1} .
  \end{equation}
  Now, still by equations (\ref{CP2Qcpaszero}) and Lemma
  \ref{CP2farnormexist}, since $\partial_{x_2} \eta$ is supported in $B (\pm
  \widetilde{d_c} \overrightarrow{e_1}, 2) \backslash B (\pm \widetilde{d_c}
  \overrightarrow{e_1}, 1)$,
  \[ \int_{\mathbbm{R}^2} | \partial_{x_2} \eta \mathfrak{R}\mathfrak{e} \psi
     \mathfrak{I}\mathfrak{m} \psi | Q_c |^2 | \leqslant K \| \varphi
     \|^2_{H_{Q_c}} . \]
  Finally, since $| \nabla Q_c | \leqslant \frac{C (c)}{(1 + r)^2}$ by Theorem
  \ref{CP2Qcbehav}, by Cauchy-Schwarz and Lemma \ref{CP2sensmanqu},
  \[ \int_{\mathbbm{R}^2} | \eta \mathfrak{R}\mathfrak{e} \psi
     \mathfrak{I}\mathfrak{m} \psi \partial_{x_2} (| Q_c |^2) | \leqslant C
     (c) \sqrt{\int_{\mathbbm{R}^2} \eta \mathfrak{R}\mathfrak{e}^2 (\psi) 
     \int_{\mathbbm{R}^2} \eta \frac{\mathfrak{I}\mathfrak{m}^2 \psi}{(1 +
     r)^4}} \leqslant C (c) \| \varphi \|^2_{H_{Q_c}} . \]
\end{proof}

\subsection{Density of test functions in $H_{Q_c}$}\label{CP2density}

We shall prove the coercivity with test functions, that are $0$ in a vicinity
of the zeros of $Q_c$. This will allow us to divide by $Q_c$ in several
computations. We give here a density result to show that it is not a problem
to remove a vicinity of the zeros of $Q_c$ for test functions.

\begin{lemma}
  \label{CP2Ndensity}$C^{\infty}_c (\mathbbm{R}^2 \backslash \{
  \widetilde{d_c} \overrightarrow{e_1}, - \widetilde{d_c} \overrightarrow{e_1}
  \}, \mathbbm{C})$ is dense in $H_{Q_c}$ for the norm $\| . \|_{H_{Q_c}}$.
\end{lemma}

This result uses similar arguments as {\cite{DP}} for the density in
$H_{V_1}$. For the sake of completeness, we give a proof of it.

\begin{proof}
  We recall that
  \[ \| \varphi \|^2_{H_{Q_c}} = \int_{\mathbbm{R}^2} | \nabla \varphi |^2 + |
     1 - | Q_c^{\nosymbol} |^2 | | \varphi |^2 +\mathfrak{R}\mathfrak{e}^2
     (\overline{Q_c} \varphi), \]
  and since, for all $\lambda > 0$,
  \[ K_1 (\lambda) \int_{B (0, \lambda)} | \nabla \varphi |^2 + | \varphi |^2
     \leqslant \int_{B (0, \lambda)} | \nabla \varphi |^2 + | 1 - |
     Q_c^{\nosymbol} |^2 | | \varphi |^2 +\mathfrak{R}\mathfrak{e}^2
     (\overline{Q_c} \varphi) \leqslant K_2 (\lambda) \int_{B (0, \lambda)} |
     \nabla \varphi |^2 + | \varphi |^2, \]
  by standard density argument, we have that $C^{\infty}_c (\mathbbm{R}^2,
  \mathbbm{C})$ is dense in $H_{Q_c}$ for the norm $\| . \|_{H_{Q_c}}$.
  
  We are therefore left with the proof that $C^{\infty}_c (\mathbbm{R}^2
  \backslash \{ \widetilde{d_c} \overrightarrow{e_1}, - \widetilde{d_c}
  \overrightarrow{e_1} \}, \mathbbm{C})$ is dense in $C^{\infty}_c
  (\mathbbm{R}^2, \mathbbm{C})$ for the norm $\| . \|_{H_{Q_c}}$. For that, it
  is enough to check that $C^{\infty}_c (B (0, 2) \backslash \{ 0 \},
  \mathbbm{C})$ is dense in $C^{\infty}_c (B (0, 2), \mathbbm{C})$ for the
  norm $\| . \|_{H^1 (B (0, 2))}$. This result is a consequence of the fact
  that the capacity of a point in a ball in dimension 2 is 0. For the sake of
  completeness, we give here a proof of this result.
  
  We define $\eta_{\varepsilon} \in C^0 (B (0, 2), \mathbbm{R})$ the radial
  function with $\eta_{\varepsilon} (x) = 0$ if $| x | \leqslant \varepsilon$,
  $\eta_{\varepsilon} (x) = - \frac{\ln (| x |)}{\ln (\varepsilon)} + 1$ if $|
  x | \in [\varepsilon, 1]$ and $\eta_{\varepsilon} (x) = 1$ if $2 \geqslant |
  x | \geqslant 1$. Then, we define $\eta_{\varepsilon, \lambda} \in
  C^{\infty} (B (0, 2), \mathbbm{R})$ a radial regularisation of
  $\eta_{\varepsilon}$ with $\eta_{\varepsilon, \lambda} (x) = 0$ if $| x |
  \leqslant \varepsilon / 2$ such that $\eta_{\varepsilon, \lambda}
  \rightarrow \eta_{\varepsilon}$ in $H^1 (B (0, 2))$ when $\lambda
  \rightarrow 0$. Finally, we define $\eta_{\varepsilon, \lambda, \delta} =
  \eta_{\varepsilon, \lambda} \left( \frac{x}{\delta} \right)$ for a small
  $\delta > 0$.
  
  Now, given $\varphi \in C^{\infty}_c (B (0, 2), \mathbbm{C})$,
  $\eta_{\varepsilon, \lambda, \delta} \varphi \in C^{\infty}_c (B (0, 2)
  \backslash \{ 0 \}, \mathbbm{C})$ for all $\varepsilon > 0, \lambda > 0,
  \delta > 0$, and by dominated convergence, we check that
  \[ \int_{B (0, 2)} | \eta_{\varepsilon, \lambda, \delta} \varphi |^2
     \rightarrow \int_{B (0, 2)} | \varphi |^2 \]
  when $\delta \rightarrow 0$. Furthermore, we compute by integration by parts
  \begin{eqnarray*}
    \int_{B (0, 2)} | \nabla (\eta_{\varepsilon, \lambda, \delta} \varphi) |^2
    & = & \int_{B (0, 2)} \eta_{\varepsilon, \lambda, \delta}^2 | \nabla
    \varphi |^2 + 2 \int_{B (0, 2)} \nabla \eta_{\varepsilon, \lambda, \delta}
    \eta_{\varepsilon, \lambda, \delta} \mathfrak{R}\mathfrak{e} (\nabla
    \varphi \bar{\varphi})\\
    & + & \int_{B (0, 2)} | \nabla \eta_{\varepsilon, \lambda, \delta} |^2 |
    \varphi |^2\\
    & = & \int_{B (0, 2)} \eta_{\varepsilon, \lambda, \delta}^2 | \nabla
    \varphi |^2 - \int_{B (0, 2)} | \varphi |^2 \Delta \eta_{\varepsilon,
    \lambda, \delta} \eta_{\varepsilon, \lambda, \delta} .
  \end{eqnarray*}
  Now, extending $\varphi$ to $\mathbbm{R}^2$ by $\varphi = 0$ outside of $B
  (0, 2)$, we have by change of variables
  \[ \int_{B (0, 2)} | \varphi |^2 \Delta \eta_{\varepsilon, \lambda, \delta}
     \eta_{\varepsilon, \lambda, \delta} = \int_{\mathbbm{R}^2} | \varphi |^2
     \Delta \eta_{\varepsilon, \lambda, \delta} \eta_{\varepsilon, \lambda,
     \delta} = \int_{\mathbbm{R}^2} | \varphi |^2 (x \delta) \Delta
     \eta_{\varepsilon, \lambda} \eta_{\varepsilon, \lambda} . \]
  When $\delta \rightarrow 0$, we have by dominated convergence that $\int_{B
  (0, 2)} \eta_{\varepsilon, \lambda, \delta}^2 | \nabla \varphi |^2
  \rightarrow \int_{B (0, 2)} | \nabla \varphi |^2$ and
  \[ \int_{\mathbbm{R}^2} | \varphi |^2 (x \delta) \Delta \eta_{\varepsilon,
     \lambda} \eta_{\varepsilon, \lambda} \rightarrow | \varphi |^2 (0)
     \int_{\mathbbm{R}^2} \Delta \eta_{\varepsilon, \lambda}
     \eta_{\varepsilon, \lambda} = - | \varphi |^2 (0) \int_{\mathbbm{R}^2} |
     \nabla \eta_{\varepsilon, \lambda} |^2 . \]
  Now, taking $\lambda \rightarrow 0$, we deduce that
  \[ \lim_{\lambda \rightarrow 0} \lim_{\delta \rightarrow 0} \int_{B (0, 2)}
     | \nabla (\eta_{\varepsilon, \lambda, \delta} \varphi) |^2 = \int_{B (0,
     2)} | \nabla \varphi |^2 - | \varphi |^2 (0) \int_{\mathbbm{R}^2} |
     \nabla \eta_{\varepsilon} |^2 . \]
  From the definition of $\eta_{\varepsilon}$, we compute
  \begin{eqnarray*}
    \int_{\mathbbm{R}^2} | \nabla \eta_{\varepsilon} |^2 & = &
    \int_{\varepsilon}^1 \frac{1}{\ln (\varepsilon)^2 r^2} r d r\\
    & = & \frac{1}{\ln (\varepsilon)^2} \int_{\varepsilon}^1 \frac{1}{r} d
    r\\
    & = & \frac{- 1}{\ln (\varepsilon)} \rightarrow 0
  \end{eqnarray*}
  when $\varepsilon \rightarrow 0$. We deduce that
  \[ \lim_{\varepsilon \rightarrow 0} \lim_{\lambda \rightarrow 0}
     \lim_{\delta \rightarrow 0} \int_{B (0, 2)} | \nabla (\eta_{\varepsilon,
     \lambda, \delta} \varphi) |^2 = \int_{B (0, 2)} | \nabla \varphi |^2 . \]
  This concludes the proof of this lemma.
\end{proof}

\section{Coercivity results in $H_{Q_c}$}\label{CP2NSEC3}

This section is devoted to the proofs of Propositions \ref{CP2P811} and
\ref{CP205218}. Here, we will do most of the computations with test functions,
that is functions in $C^{\infty}_c (\mathbbm{R}^2 \backslash \{ \tilde{d}_c
\vec{e}_1, - \tilde{d}_c \vec{e}_1 \}, \mathbbm{C})$. This will allow to do
many computations, including dividing by $Q_c$ in some quantities.

\subsection{Expression of the quadratic forms\label{CP21202CompQc}}

We recall that $\eta$ if a smooth cutoff function such that $\eta (x) = 0$ on
$B (\pm \widetilde{d_c} \overrightarrow{e_1}, 1)$, $\eta (x) = 1$ on
$\mathbbm{R}^2 \backslash (B (\widetilde{d_c} \overrightarrow{e_1}, 2) \cup B
(- \widetilde{d_c} \overrightarrow{e_1}, 2))$, where $\pm \widetilde{d_c}
\overrightarrow{e_1}$ are the zeros of $Q_c$. Furthermore, from {\cite{DP}},
we recall the quadratic form around a vortex $V_1$:
\[ B_{V_1} (\varphi) = \int_{\mathbbm{R}^2} | \nabla \varphi |^2 - (1 - | V_1
   |^2) | \varphi |^2 + 2\mathfrak{R}\mathfrak{e}^2 (\overline{V_1} \varphi) .
\]
We want to write the quadratic form around $V_1$ and $Q_c$ in a special form.
For the one around $Q_c$, it will be of the form $B_{Q_c}^{\exp}$, defined in
(\ref{CP2Btilda}).

\begin{lemma}
  \label{CP2L412602}For $\varphi = Q_c \psi \in C^{\infty}_c (\mathbbm{R}^2
  \backslash \{ \widetilde{d_c} \overrightarrow{e_1}, - \widetilde{d_c}
  \overrightarrow{e_1} \}, \mathbbm{C})$, we have
  \[ \langle L_{Q_c} (\varphi), \varphi \rangle = B_{Q_c}^{\exp} (\varphi), \]
  where $B_{Q_c}^{\exp} (\varphi)$ is defined in (\ref{CP2Btilda}).
  Furthermore, for $\varphi = V_1 \psi \in C^{\infty}_c (\mathbbm{R}^2
  \backslash \{ 0 \}, \mathbbm{C})$, where $V_1$ is centered at $0$, and
  $\tilde{\eta}$ a smooth radial cutoff function with value $0$ in $B (0, 1)$,
  and value $1$ outside of $B (0, 2)$,
  \begin{eqnarray*}
    B_{V_1} (\varphi) & = & \int_{\mathbbm{R}^2} (1 - \tilde{\eta}) (| \nabla
    \varphi |^2 - (1 - | V_1 |^2) | \varphi |^2 + 2\mathfrak{R}\mathfrak{e}^2
    (\overline{V_1} \varphi))\\
    & - & \int_{\mathbbm{R}^2} \nabla \tilde{\eta} .
    (\mathfrak{R}\mathfrak{e} (\nabla V_1 \overline{V_1}) | \psi |^2 -
    2\mathfrak{I}\mathfrak{m} (\nabla V_1 \overline{V_1})
    \mathfrak{R}\mathfrak{e} (\psi) \mathfrak{I}\mathfrak{m} (\psi))\\
    & + & \int_{\mathbbm{R}^2} \tilde{\eta} (| \nabla \psi |^2 | V_1 |^2 +
    2\mathfrak{R}\mathfrak{e}^2 (\psi) | V_1 |^4 + 4\mathfrak{I}\mathfrak{m}
    (\nabla V_1 \overline{V_1}) \mathfrak{I}\mathfrak{m} (\nabla \psi)
    \mathfrak{R}\mathfrak{e} (\psi)) .
  \end{eqnarray*}
\end{lemma}

\begin{proof}
  We recall that $L_{Q_c} (\varphi) = - i c \partial_{x_2} \varphi - \Delta
  \varphi - (1 - | Q_c |^2) \varphi + 2\mathfrak{R}\mathfrak{e}
  (\overline{Q_c} \varphi) Q_c$. Writing $\varphi = Q_c \psi \in C^{\infty}_c
  (\mathbbm{R}^2 \backslash \{ \widetilde{d_c} \overrightarrow{e_1}, -
  \widetilde{d_c} \overrightarrow{e_1} \}, \mathbbm{C})$, we decompose
  \[ L_{Q_c} (\varphi) = - i c \partial_{x_2} \psi Q_c - \Delta \psi Q_c - 2
     \nabla Q_c . \nabla \psi + 2\mathfrak{R}\mathfrak{e} (\psi) | Q_c |^2 Q_c
     + \tmop{TW}_c (Q_c) \psi . \]
  Since $\tmop{TW}_c (Q_c) = 0$,
  \begin{eqnarray*}
    &  & \langle L_{Q_c} (\varphi), \varphi \rangle\\
    & = & \langle (1 - \eta) L_{Q_c} (\varphi), \varphi \rangle + \langle
    \eta L_{Q_c} (\varphi), Q_c \psi \rangle\\
    & = & \int_{\mathbbm{R}^2} (1 - \eta) \mathfrak{R}\mathfrak{e} ((- i c
    \partial_{x_2} \varphi - \Delta \varphi - (1 - | Q_c |^2) \varphi +
    2\mathfrak{R}\mathfrak{e} (\overline{Q_c} \varphi) Q_c) \bar{\varphi})\\
    & + & \int_{\mathbbm{R}^2} \eta \mathfrak{R}\mathfrak{e} ((- i c
    \partial_{x_2} \psi Q_c - \Delta \psi Q_c - 2 \nabla Q_c . \nabla \psi +
    2\mathfrak{R}\mathfrak{e} (\psi) | Q_c |^2 Q_c) \overline{Q_c \psi}) .
  \end{eqnarray*}
  By integration by parts,
  \begin{eqnarray*}
    &  & \int_{\mathbbm{R}^2} (1 - \eta) \mathfrak{R}\mathfrak{e} ((- i c
    \partial_{x_2} \varphi - \Delta \varphi - (1 - | Q_c |^2) \varphi +
    2\mathfrak{R}\mathfrak{e} (\overline{Q_c} \varphi) Q_c) \bar{\varphi})\\
    & = & \int_{\mathbbm{R}^2} (1 - \eta) (| \nabla \varphi |^2
    -\mathfrak{R}\mathfrak{e} (i c \partial_{x_2} \varphi \bar{\varphi}) - (1
    - | Q_c |^2) | \varphi |^2 + 2\mathfrak{R}\mathfrak{e}^2 (\overline{Q_c}
    \varphi))\\
    & - & \int_{\mathbbm{R}^2} \nabla \eta .\mathfrak{R}\mathfrak{e} (\nabla
    \varphi \bar{\varphi}) .
  \end{eqnarray*}
  Similarly, we compute
  \begin{eqnarray*}
    &  & \int_{\mathbbm{R}^2} \eta \mathfrak{R}\mathfrak{e} ((- i c
    \partial_{x_2} \psi Q_c - \Delta \psi Q_c - 2 \nabla Q_c . \nabla \psi +
    2\mathfrak{R}\mathfrak{e} (\psi) | Q_c |^2 Q_c) \overline{Q_c \psi})\\
    & = & \int_{\mathbbm{R}^2} \eta (\mathfrak{R}\mathfrak{e} (- i c
    \partial_{x_2} \psi \bar{\psi} | Q_c |^2) -\mathfrak{R}\mathfrak{e}
    (\Delta \psi \bar{\psi}) | Q_c |^2 + 2\mathfrak{R}\mathfrak{e}^2 (\psi) |
    Q_c |^4 - 2\mathfrak{R}\mathfrak{e} (\nabla Q_c . \nabla \psi
    \overline{Q_c \psi}))\\
    & = & \int_{\mathbbm{R}^2} \eta (c | Q_c |^2 (\mathfrak{I}\mathfrak{m}
    (\partial_{x_2} \psi) \mathfrak{R}\mathfrak{e} (\psi)
    -\mathfrak{R}\mathfrak{e} (\partial_{x_2} \psi) \mathfrak{I}\mathfrak{m}
    (\psi)) + 2\mathfrak{R}\mathfrak{e}^2 (\psi) | Q_c |^4 -
    2\mathfrak{R}\mathfrak{e} (\nabla Q_c . \nabla \psi \overline{Q_c
    \psi}))\\
    & + & \int_{\mathbbm{R}^2} \eta | \nabla \psi |^2 | Q_c |^2 + 2
    \int_{\mathbbm{R}^2} \eta \mathfrak{R}\mathfrak{e} (\nabla Q_c
    \overline{Q_c}) .\mathfrak{R}\mathfrak{e} (\nabla \psi \bar{\psi}) +
    \int_{\mathbbm{R}^2} \nabla \eta .\mathfrak{R}\mathfrak{e} (\nabla \psi
    \bar{\psi}) | Q_c |^2 .
  \end{eqnarray*}
  We continue, we have
  \begin{eqnarray*}
    &  & - \int_{\mathbbm{R}^2} \eta | Q_c |^2 \mathfrak{R}\mathfrak{e}
    (\partial_{x_2} \psi) \mathfrak{I}\mathfrak{m} (\psi)\\
    & = & \int_{\mathbbm{R}^2} \eta | Q_c |^2 \mathfrak{R}\mathfrak{e} (\psi)
    \mathfrak{I}\mathfrak{m} (\partial_{x_2} \psi) + \int_{\mathbbm{R}^2}
    \partial_{x_2} \eta | Q_c |^2 \mathfrak{R}\mathfrak{e} (\psi)
    \mathfrak{I}\mathfrak{m} (\psi) + 2 \int_{\mathbbm{R}^2} \eta
    \mathfrak{R}\mathfrak{e} (\partial_{x_2} Q_c \overline{Q_c})
    \mathfrak{R}\mathfrak{e} (\psi) \mathfrak{I}\mathfrak{m} (\psi),
  \end{eqnarray*}
  as well as
  \[ \int_{\mathbbm{R}^2} \eta \mathfrak{R}\mathfrak{e} (\nabla Q_c . \nabla
     \psi \overline{Q_c \psi}) = \int_{\mathbbm{R}^2} \eta
     \mathfrak{R}\mathfrak{e} (\nabla Q_c \overline{Q_c})
     .\mathfrak{R}\mathfrak{e} (\nabla \psi \bar{\psi}) + \int_{\mathbbm{R}^2}
     \eta \mathfrak{I}\mathfrak{m} (\nabla Q_c \overline{Q_c})
     \mathfrak{I}\mathfrak{m} (\nabla \psi \bar{\psi}), \]
  therefore
  \begin{eqnarray*}
    &  & \int_{\mathbbm{R}^2} \eta \mathfrak{R}\mathfrak{e} ((- i c
    \partial_{x_2} \psi Q_c - \Delta \psi Q_c - 2 \nabla Q_c . \nabla \psi +
    2\mathfrak{R}\mathfrak{e} (\psi) | Q_c |^2 Q_c) \overline{Q_c \psi})\\
    & = & \int_{\mathbbm{R}^2} \eta (| \nabla \psi |^2 | Q_c |^2 +
    2\mathfrak{R}\mathfrak{e}^2 (\psi) | Q_c |^4 + 2 c\mathfrak{I}\mathfrak{m}
    (\partial_{x_2} \psi) \mathfrak{R}\mathfrak{e} (\psi))\\
    & + & \int_{\mathbbm{R}^2} \eta (2 c\mathfrak{R}\mathfrak{e}
    (\partial_{x_2} Q_c \overline{Q_c}) \mathfrak{R}\mathfrak{e} (\psi)
    \mathfrak{I}\mathfrak{m} (\psi) - 2\mathfrak{I}\mathfrak{m} (\nabla Q_c
    \overline{Q_c}) \mathfrak{I}\mathfrak{m} (\nabla \psi \bar{\psi}))\\
    & + & c \int_{\mathbbm{R}^2} \partial_{x_2} \eta \mathfrak{R}\mathfrak{e}
    (\psi) \mathfrak{I}\mathfrak{m} (\psi) | Q_c |^2 + \int_{\mathbbm{R}^2}
    \nabla \eta .\mathfrak{R}\mathfrak{e} (\nabla \psi \bar{\psi}) | Q_c |^2 .
  \end{eqnarray*}
  Since $i c \partial_{x_2} Q_c = \Delta Q_c + (1 - | Q_c |^2) Q_c$, we have
  $c\mathfrak{R}\mathfrak{e} (\partial_{x_2} Q_c \overline{Q_c})
  =\mathfrak{R}\mathfrak{e} (i \Delta Q_c \overline{Q_c})$. By integration by
  parts,
  \begin{eqnarray*}
    &  & 2 \int_{\mathbbm{R}^2} \eta \mathfrak{R}\mathfrak{e} (i \Delta Q_c
    \overline{Q_c}) \mathfrak{R}\mathfrak{e} (\psi) \mathfrak{I}\mathfrak{m}
    (\psi)\\
    & = & 2 \int_{\mathbbm{R}^2} \nabla \eta .\mathfrak{I}\mathfrak{m}
    (\nabla Q_c \overline{Q_c}) \mathfrak{R}\mathfrak{e} (\psi)
    \mathfrak{I}\mathfrak{m} (\psi)\\
    & - & 2 \int_{\mathbbm{R}^2} \eta \mathfrak{I}\mathfrak{m} (\nabla Q_c
    \overline{Q_c}) .\mathfrak{R}\mathfrak{e} (\nabla \psi)
    \mathfrak{I}\mathfrak{m} (\psi) - 2 \int_{\mathbbm{R}^2} \eta
    \mathfrak{I}\mathfrak{m} (\nabla Q_c \overline{Q_c})
    .\mathfrak{R}\mathfrak{e} (\psi) \mathfrak{I}\mathfrak{m} (\nabla \psi),
  \end{eqnarray*}
  and
  \begin{eqnarray*}
    &  & - 2 \int_{\mathbbm{R}^2} \eta \mathfrak{I}\mathfrak{m} (\nabla Q_c
    \overline{Q_c}) \mathfrak{I}\mathfrak{m} (\nabla \psi \bar{\psi})\\
    & = & - 2 \int_{\mathbbm{R}^2} \eta \mathfrak{I}\mathfrak{m} (\nabla Q_c
    \overline{Q_c}) (\mathfrak{I}\mathfrak{m} (\nabla \psi)
    \mathfrak{R}\mathfrak{e} (\psi) -\mathfrak{I}\mathfrak{m} (\psi)
    \mathfrak{R}\mathfrak{e} (\nabla \psi)) .
  \end{eqnarray*}
  Combining these estimates, with
  \[ \int_{\mathbbm{R}^2} \nabla \eta .\mathfrak{R}\mathfrak{e} (\nabla
     \varphi \bar{\varphi}) = \int_{\mathbbm{R}^2} \nabla \eta .
     (\mathfrak{R}\mathfrak{e} (\nabla Q_c \overline{Q_c}) | \psi |^2
     +\mathfrak{R}\mathfrak{e} (\nabla \psi \bar{\psi}) | Q_c |^2), \]
  we conclude the proof of
  \[ \langle L_{Q_c} (\varphi), \varphi \rangle = B_{Q_c}^{\exp} (\varphi) .
  \]
  Now, for the proof for $B_{V_1} (\varphi)$, the computations are identical,
  simply replacing $c$ by $0$, $\eta$ by $\tilde{\eta}$, and $Q_c$ by $V_1$.
\end{proof}

\subsection{A coercivity result for the quadratic form around one
vortex}\label{CP2s24}

This subsection is devoted to the proof of Proposition \ref{CP2P811}, and a
localized version of it (see Lemma \ref{CP2L422602}).

\subsubsection{Coercivity for test functions}

\begin{proof}[of Proposition \ref{CP2P811}]
  We recall the result from {\cite{DP}}, see Lemma 3.1 and equation (2.42)
  there. If $\varphi = V_1 \psi \in C^{\infty}_c (\mathbbm{R}^2 \backslash \{
  0 \}, \mathbbm{C})$ with the two orthogonality conditions
  \[ \int_{B (0, R)} \mathfrak{R}\mathfrak{e} (\partial_{x_1} V_1
     \bar{\varphi}) = \int_{B (0, R)} \mathfrak{R}\mathfrak{e} (\partial_{x_1}
     V_1 \bar{\varphi}) = 0, \]
  then, writing $\psi^0 (x) = \frac{1}{2 \pi} \int_0^{2 \pi} \psi (| x | \cos
  (\theta), | x | \sin (\theta) d \theta)$, the 0-harmonic around $0$ of
  $\psi$, and $\psi^{\neq 0} = \psi - \psi^0$, then
  \[ B_{V_1} (\varphi) \geqslant K \int_{\mathbbm{R}^2} | \nabla (V_1
     \psi^{\neq 0}) |^2 + | \nabla \psi^0 |^2 | V_1 |^2 + \frac{| V_1
     \psi^{\neq 0} |^2}{(1 + r)^2} +\mathfrak{R}\mathfrak{e}^2 (\psi) | V_1
     |^4 . \]
  We recall from Lemma \ref{lemme3new} that there exists $K_1 > 0$ such that
  $K_1 \leqslant \frac{| V_1 |}{r} \leqslant \frac{1}{K_1}$, and that $| V_1
  |$ is a radial function around $0$. Therefore, by Hardy inequality in
  dimension 4,
  \[ \int_{B (0, 1)} | \psi^0 |^2 \leqslant K \left( \int_{B (0, 2)} | \nabla
     \psi^0 |^2 | V_1 |^2 + \int_{B (0, 2) \backslash B (0, 1)} | \psi^0 |^2
     \right) . \]
  By Poincar{\'e} in{\'e}quality, using $\int_{B (0, R) \backslash B (0, R /
  2)} \mathfrak{I}\mathfrak{m} (\psi) = 0$ and $| V_1 |^2 \geqslant K$ outside
  of $B (0, 1)$, we have
  \[ \int_{B (0, 10) \backslash B (0, 1)} | \psi^0 |^2 \leqslant K \left(
     \int_{B (0, R)} | \nabla \psi^0 |^2 | V_1 |^2 +\mathfrak{R}\mathfrak{e}^2
     (\psi^0) | V_1 |^4 \right) . \]
  Here, the constant $K > 0$ depends on $R > 0$, but we consider $R$ as a
  universal constant. We deduce that
  \begin{eqnarray*}
    \int_{B (0, 10)} | \varphi |^2 & \leqslant & \int_{B (0, 10)} | V_1 \psi
    |^2\\
    & \leqslant & K \left( \int_{B (0, 10)} | V_1 \psi^0 |^2 + \int_{B (0,
    10)} | V_1 \psi^{\neq 0} |^2 \right)\\
    & \leqslant & K \left( \int_{\mathbbm{R}^2} | \nabla (V_1 \psi^{\neq 0})
    |^2 + | \nabla \psi^0 |^2 | V_1 |^2 + \frac{| V_1 \psi^{\neq 0} |^2}{(1 +
    r)^2} +\mathfrak{R}\mathfrak{e}^2 (\psi) | V_1 |^4 \right) .
  \end{eqnarray*}
  Similarly,
  \begin{eqnarray*}
    \int_{B (0, 10)} | \nabla \varphi |^2 & \leqslant & \int_{B (0, 10)} |
    \nabla (V_1 (\psi^0 + \psi^{\neq 0})) |^2\\
    & \leqslant & K \left( \int_{B (0, 10)} | \nabla (V_1 \psi^0) |^2 +
    \int_{B (0, 10)} | \nabla (V_1 \psi^{\neq 0}) |^2 \right)\\
    & \leqslant & K \left( \int_{B (0, 10)} | \nabla \psi^0 |^2 | V_1 |^2 + |
    \psi^0 |^2 | \nabla V_1 |^2 + \int_{B (0, 10)} | \nabla (V_1 \psi^{\neq
    0}) |^2 \right)\\
    & \leqslant & K \left( \int_{\mathbbm{R}^2} | \nabla (V_1 \psi^{\neq 0})
    |^2 + | \nabla \psi^0 |^2 | V_1 |^2 + \frac{| V_1 \psi^{\neq 0} |^2}{(1 +
    r)^2} +\mathfrak{R}\mathfrak{e}^2 (\psi) | V_1 |^4 \right) .
  \end{eqnarray*}
  Finally, outside of $B (0, 5)$, we have, by Lemma \ref{lemme3new}, that
  \[ \int_{\mathbbm{R}^2 \backslash B (0, 5)} | \nabla \psi |^2 \leqslant K
     \int_{\mathbbm{R}^2 \backslash B (0, 5)} | \nabla \psi |^2 | V_1 |^2 . \]
  Let us show that
  \[ \int_{\mathbbm{R}^2 \backslash B (0, 5)} \frac{| \psi |^2}{r^2 \ln^2 (r)}
     \leqslant K \left( \int_{\mathbbm{R}^2 \backslash B (0, 5)} | \nabla \psi
     |^2 + \int_{B (0, 10) \backslash B (0, 5)} | \psi |^2 \right) . \]
  This is a Hardy type inequality, and it would conclude the proof of this
  proposition. Remark that for the harmonics other than zeros, this is a
  direct consequence of
  \[ \int_{\mathbbm{R}^2 \backslash B (0, 5)} \frac{| \psi^{\neq 0} |^2}{r^2}
     \leqslant \int_{\mathbbm{R}^2 \backslash B (0, 5)} | \nabla \psi |^2 . \]
  We therefore suppose that $\psi$ is a radial compactly supported function.
  We define $\chi$ a smooth radial cutoff function with $\chi (r) = 0$ if $r
  \leqslant 4$ and $\chi (r) = 1$ if $r \geqslant 5$. Then, by Cauchy-Schwarz,
  \begin{eqnarray*}
    \left| \int_{\mathbbm{R}^2 \backslash B (0, 5)} \frac{\chi (r) | \psi
    |^2}{r^2 \ln^2 (r)} \right| & = & \left| - \int_5^{+ \infty} \chi (r) |
    \psi |^2 (r) \partial_r \left( \frac{1}{\ln (r)} \right) d r \right|\\
    & = & \left| \int_5^{+ \infty} \partial_r (\chi | \psi |^2) (r) \frac{d
    r}{\ln (r)} \right|\\
    & \leqslant & K \left( \int_{B (0, 10) \backslash B (0, 5)} | \psi |^2 +
    \int_5^{+ \infty} \chi (r) | \psi | (r) \partial_r | \psi | (r) \frac{d
    r}{\ln (r)} \right)\\
    & \leqslant & K \left( \int_{B (0, 10) \backslash B (0, 5)} | \psi |^2 +
    \sqrt{\int_{\mathbbm{R}^2 \backslash B (0, 5)} \frac{\chi (r) | \psi
    |^2}{r^2 \ln^2 (r)} \int_{\mathbbm{R}^2 \backslash B (0, 5)} | \nabla \psi
    |^2} \right) .
  \end{eqnarray*}
  The proof is complete.
\end{proof}

\subsubsection{Localisation of the coercivity for one vortex}

Now, we want to localize the coercivity result. We define, for $D > 10$,
$\varphi = V_1 \psi \in H_{V_1}$,
\[ \begin{array}{lll}
     B^{\tmop{loc}_D}_{V_1} (\varphi) & \assign & \int_{B (0, D)} (1 -
     \tilde{\eta}) (| \nabla \varphi |^2 - (1 - | V_1 |^2) | \varphi |^2 +
     2\mathfrak{R}\mathfrak{e}^2 (\overline{V_1} \varphi))\\
     & - & \int_{B (0, D)} \nabla \tilde{\eta} . (\mathfrak{R}\mathfrak{e}
     (\nabla V_1 \overline{V_1}) | \psi |^2 - 2\mathfrak{I}\mathfrak{m}
     (\nabla V_1 \overline{V_1}) \mathfrak{R}\mathfrak{e} (\psi)
     \mathfrak{I}\mathfrak{m} (\psi))\\
     & + & \int_{B (0, D)} \tilde{\eta} (| \nabla \psi |^2 | V_1 |^2 +
     2\mathfrak{R}\mathfrak{e}^2 (\psi) | V_1 |^4 + 4\mathfrak{I}\mathfrak{m}
     (\nabla V_1 \overline{V_1}) \mathfrak{I}\mathfrak{m} (\nabla \psi)
     \mathfrak{R}\mathfrak{e} (\psi)),
   \end{array} \]
where $\tilde{\eta}$ is a smooth radial cutoff function such that
$\tilde{\eta} (x) = 0$ on $B (0, 1)$, $\tilde{\eta} (x) = 1$ on $\mathbbm{R}^2
\backslash B (0, 2)$.

\begin{lemma}
  \label{CP2L422602}There exist $K, R, D_0 > 0$ with $D_0 > R$, such that, for
  $D > D_0$ and $\varphi = V_1 \psi \in C^{\infty}_c (\mathbbm{R}^2 \backslash
  \{ 0 \}, \mathbbm{C})$, if the following three orthogonality conditions
  \[ \int_{B (0, R)} \mathfrak{R}\mathfrak{e} (\partial_{x_1} V_1
     \bar{\varphi}) = \int_{B (0, R)} \mathfrak{R}\mathfrak{e} (\partial_{x_2}
     V_1 \bar{\varphi}) = \int_{B (0, R) \backslash B (0, R / 2)}
     \mathfrak{I}\mathfrak{m} (\psi) = 0 \]
  are satisfied, then
  \[ B^{\tmop{loc}_D}_{V_1} (\varphi) \geqslant K \left( \int_{B (0, 10)} |
     \nabla \varphi |^2 + | \varphi |^2 + \int_{B (0, D) \backslash B (0, 5)}
     | \nabla \psi |^2 | V_1 |^2 +\mathfrak{R}\mathfrak{e}^2 (\psi) | V_1 |^4
     + \frac{| \psi |^2}{r^2 \ln^2 (r)} \right) . \]
\end{lemma}

\begin{proof}
  We decompose $\psi$ in harmonics $j \in \mathbbm{N}, l \in \{ 1, 2 \}$, with
  the same decomposition as (2.5) of {\cite{DP}}. This decomposition is
  adapted to the quadratic form $B^{\tmop{loc}_D}_{V_1}$, see equation (2.4)
  of {\cite{DP}}, that also holds if the integral is only on $B (0, D)$.
  
  For $j = 0$, the proof is identical. For $j \geqslant 2$, $l \in \{ 1, 2 \}$
  from equation (2.38) of {\cite{DP}} (that holds on $B (0, D)$ as the
  inequality is pointwise), the proof holds if it does for $j = 1$, \ $l \in
  \{ 1, 2 \}$.
  
  We therefore focus on the case $j = l = 1$. We write $\psi = \psi_1 (r) \cos
  (\theta) + i \psi_2 (r) \sin (\theta)$, with $\psi_1, \psi_2 \in
  C^{\infty}_c (\mathbbm{R}^{+ \ast}, \mathbbm{R})$. The other possibility ($l
  = 2$) is $\psi = \psi_1 (r) i \cos (\theta) + \psi_2 (r) \sin (\theta)$,
  which is done similarly. We will show a more general result, that is, for
  any $\varphi = V_1 \psi \in C^{\infty}_c (\mathbbm{R}^2 \backslash \{ 0 \},
  \mathbbm{C})$ satisfying the orthogonality conditions,
  \begin{eqnarray*}
    &  & B^{\tmop{loc}_D}_{V_1} (V_1 \psi^{\neq 0})\\
    & \geqslant & K \left( \int_{B (0, 10)} | \nabla (V_1 \psi^{\neq 0}) |^2
    + | V_1 \psi^{\neq 0} |^2 + \int_{B (0, D) \backslash B (0, 5)} | \nabla
    \psi^{\neq 0} |^2 | V_1 |^2 +\mathfrak{R}\mathfrak{e}^2 (\psi^{\neq 0}) |
    V_1 |^4 + \frac{| \psi^{\neq 0} |^2}{r^2} \right) .
  \end{eqnarray*}
  With the previous remark, it is enough to conlcude the proof of this lemma.
  In the rest of the proof, to simplify the notation, we write $\psi$ instead
  of $\psi^{\neq 0}$, but it still has no $0$-harmonic.
  
  We remark that, for $D > R_0 > 2$,
  \begin{eqnarray}
    &  & \int_{B (0, D) \backslash B (0, R_0)} | \nabla \psi |^2 | V_1 |^2 +
    2\mathfrak{R}\mathfrak{e}^2 (\psi) | V_1 |^4 + 4\mathfrak{I}\mathfrak{m}
    (\nabla V_1 \overline{V_1}) .\mathfrak{I}\mathfrak{m} (\nabla \psi)
    \mathfrak{R}\mathfrak{e} (\psi) \nonumber\\
    & \geqslant & \int_{B (0, D) \backslash B (0, R_0)} | \nabla \psi |^2 |
    V_1 |^2 + 2\mathfrak{R}\mathfrak{e}^2 (\psi) | V_1 |^4 - \frac{K | V_1
    |^2}{R_0} | \mathfrak{I}\mathfrak{m} (\nabla \psi)
    \mathfrak{R}\mathfrak{e} (\psi) | \nonumber\\
    & \geqslant & \frac{1}{2} \int_{B (0, D) \backslash B (0, R_0)} | \nabla
    \psi |^2 | V_1 |^2 + 2\mathfrak{R}\mathfrak{e}^2 (\psi) | V_1 |^4 
    \label{CP2jspfdm}
  \end{eqnarray}
  if $R_0$ is large enough. We therefore take $R_0 > R$ large enough such that
  (\ref{CP2jspfdm}) holds. For $\frac{D}{2} > \lambda > R_0$, we define
  $\chi_{\lambda}$ a smooth cutoff function such that $\chi_{\lambda} (r) = 1$
  if $r \leqslant \lambda$, $\chi_{\lambda} = 0$ if $r \geqslant 2 \lambda$,
  and $| \chi'_{\lambda} | \leqslant \frac{K}{\lambda}$. In particular, since
  $R_0 > 2$, we have $\tmop{Supp} (\chi_{\lambda}') \subset \tmop{Supp}
  (\tilde{\eta})$ and $\tmop{Supp} (1 - \tilde{\eta}) \subset \tmop{Supp}
  (\chi_{\lambda})$. This implies that
  \begin{eqnarray*}
    &  & \int_{B (0, D)} (1 - \tilde{\eta}) (| \nabla \varphi |^2 - (1 - |
    V_1 |^2) | \varphi |^2 + 2\mathfrak{R}\mathfrak{e}^2 (\overline{V_1}
    \varphi))\\
    & = & \int_{B (0, D)} (1 - \tilde{\eta}) (| \nabla (\chi_{\lambda}
    \varphi) |^2 - (1 - | V_1 |^2) | \chi_{\lambda} \varphi |^2 +
    2\mathfrak{R}\mathfrak{e}^2 (\overline{V_1} \chi_{\lambda} \varphi))
  \end{eqnarray*}
  and
  \begin{eqnarray*}
    &  & \int_{B (0, D)} \nabla \tilde{\eta} . (\mathfrak{R}\mathfrak{e}
    (\nabla V_1 \overline{V_1}) | \psi |^2 - 2\mathfrak{I}\mathfrak{m} (\nabla
    V_1 \overline{V_1}) \mathfrak{R}\mathfrak{e} (\psi)
    \mathfrak{I}\mathfrak{m} (\psi))\\
    & = & \int_{B (0, D)} \nabla \tilde{\eta} . (\mathfrak{R}\mathfrak{e}
    (\nabla V_1 \overline{V_1}) | \chi_{\lambda} \psi |^2 -
    2\mathfrak{I}\mathfrak{m} (\nabla V_1 \overline{V_1})
    \mathfrak{R}\mathfrak{e} (\chi_{\lambda} \psi) \mathfrak{I}\mathfrak{m}
    (\chi_{\lambda} \psi)) .
  \end{eqnarray*}
  Now, we decompose
  \begin{eqnarray*}
    &  & \int_{B (0, D)} \tilde{\eta} (| \nabla \psi |^2 | V_1 |^2 +
    2\mathfrak{R}\mathfrak{e}^2 (\psi) | V_1 |^4 + 4\mathfrak{I}\mathfrak{m}
    (\nabla V_1 \overline{V_1}) \mathfrak{I}\mathfrak{m} (\nabla \psi)
    \mathfrak{R}\mathfrak{e} (\psi))\\
    & = & \int_{B (0, D)} (1 - \chi_{\lambda}^2) \tilde{\eta} (| \nabla \psi
    |^2 | V_1 |^2 + 2\mathfrak{R}\mathfrak{e}^2 (\psi) | V_1 |^4 +
    4\mathfrak{I}\mathfrak{m} (\nabla V_1 \overline{V_1})
    \mathfrak{I}\mathfrak{m} (\nabla \psi) \mathfrak{R}\mathfrak{e} (\psi))\\
    & + & \int_{B (0, D)} \chi_{\lambda}^2 \tilde{\eta} (| \nabla \psi |^2 |
    V_1 |^2 + 2\mathfrak{R}\mathfrak{e}^2 (\psi) | V_1 |^4 +
    4\mathfrak{I}\mathfrak{m} (\nabla V_1 \overline{V_1})
    \mathfrak{I}\mathfrak{m} (\nabla \psi) \mathfrak{R}\mathfrak{e} (\psi)),
  \end{eqnarray*}
  and by equation (\ref{CP2jspfdm}),
  \begin{eqnarray*}
    &  & \int_{B (0, D)} (1 - \chi_{\lambda}^2) \tilde{\eta} (| \nabla \psi
    |^2 | V_1 |^2 + 2\mathfrak{R}\mathfrak{e}^2 (\psi) | V_1 |^4 +
    4\mathfrak{I}\mathfrak{m} (\nabla V_1 \overline{V_1})
    \mathfrak{I}\mathfrak{m} (\nabla \psi) \mathfrak{R}\mathfrak{e} (\psi))\\
    & \geqslant & K \int_{B (0, D)} (1 - \chi_{\lambda}^2) | \nabla \psi |^2
    | V_1 |^2 + 2\mathfrak{R}\mathfrak{e}^2 (\psi) | V_1 |^4 .
  \end{eqnarray*}
  Furthermore,
  \begin{eqnarray*}
    &  & \int_{B (0, D)} \chi_{\lambda}^2 \tilde{\eta} (| \nabla \psi |^2 |
    V_1 |^2 + 2\mathfrak{R}\mathfrak{e}^2 (\psi) | V_1 |^4 +
    4\mathfrak{I}\mathfrak{m} (\nabla V_1 \overline{V_1})
    \mathfrak{I}\mathfrak{m} (\nabla \psi) \mathfrak{R}\mathfrak{e} (\psi))\\
    & = & \int_{B (0, D)} \tilde{\eta} (| \nabla (\chi_{\lambda} \psi) |^2 |
    V_1 |^2 + 2\mathfrak{R}\mathfrak{e}^2 (\chi_{\lambda} \psi) | V_1 |^4 +
    4\mathfrak{I}\mathfrak{m} (\nabla V_1 \overline{V_1})
    \mathfrak{I}\mathfrak{m} (\nabla (\chi_{\lambda} \psi))
    \mathfrak{R}\mathfrak{e} (\chi_{\lambda} \psi))\\
    & - & \int_{B (0, D)} \tilde{\eta} ((| \nabla (\chi_{\lambda} \psi) -
    \nabla \chi_{\lambda} \psi |^2 - | \nabla (\chi_{\lambda} \psi) |^2) | V_1
    |^2 - 4\mathfrak{I}\mathfrak{m} (\nabla V_1 \overline{V_1}) . \nabla
    \chi_{\lambda} \mathfrak{I}\mathfrak{m} (\psi) \mathfrak{R}\mathfrak{e}
    (\chi_{\lambda} \psi)),
  \end{eqnarray*}
  and thus
  \begin{eqnarray*}
    &  & B^{\tmop{loc}_D}_{V_1} (V_1 \psi)\\
    & \geqslant & B^{\tmop{loc}_D}_{V_1} (V_1 \chi_{\lambda} \psi) + K
    \int_{B (0, D)} (1 - \chi_{\lambda}^2) | \nabla \psi |^2 | V_1 |^2 +
    2\mathfrak{R}\mathfrak{e}^2 (\psi) | V_1 |^4\\
    & - & \int_{B (0, D)} \tilde{\eta} ((| \nabla (\chi_{\lambda} \psi) -
    \nabla \chi_{\lambda} \psi |^2 - | \nabla (\chi_{\lambda} \psi) |^2) | V_1
    |^2 - 4\mathfrak{I}\mathfrak{m} (\nabla V_1 \overline{V_1}) . \nabla
    \chi_{\lambda} \mathfrak{I}\mathfrak{m} (\psi) \mathfrak{R}\mathfrak{e}
    (\chi_{\lambda} \psi)) .
  \end{eqnarray*}
  Since $V_1 \chi_{\lambda} \psi \in C^{\infty}_c (B (0, D))$, we have
  $B^{\tmop{loc}_D}_{V_1} (V_1 \chi_{\lambda} \psi) = B_{V_1} (V_1
  \chi_{\lambda} \psi)$, and since $\chi_{\lambda} = 1$ in $B (0, R)$ and $V_1
  \psi$ satisfied the orthogonality conditions, so does $V_1 \chi_{\lambda}
  \psi$. By Proposition \ref{CP2P811}, we deduce that
  \begin{eqnarray*}
    &  & B^{\tmop{loc}_D}_{V_1} (V_1 \chi_{\lambda} \psi)\\
    & \geqslant & K \int_{B (0, 10)} | \nabla (V_1 \chi_{\lambda} \psi) |^2 +
    | V_1 \chi_{\lambda} \psi |^2\\
    & + & K \int_{B (0, D) \backslash B (0, 5)} | \nabla (\chi_{\lambda}
    \psi) |^2 | V_1 |^2 +\mathfrak{R}\mathfrak{e}^2 (\chi_{\lambda} \psi) |
    V_1 |^4 + \frac{| \chi_{\lambda} \psi |^2}{r^2 \ln^2 (r)} .
  \end{eqnarray*}
  Now, remarking that
  \[ | \nabla (\chi_{\lambda} \psi) |^2 | V_1 |^2 \geqslant K_1 | \nabla \psi
     |^2 \chi^2_{\lambda} | V_1 |^2 - K_2 | \nabla \chi_{\lambda} |^2 | \psi
     |^2 | V_1 |^2, \]
  and since $\chi_{\lambda} = 1$ in $B (0, 10)$, we deduce that
  \begin{eqnarray}
    &  & B^{\tmop{loc}_D}_{V_1} (V_1 \psi) \nonumber\\
    & \geqslant & K \left( \int_{B (0, 10)} | \nabla \varphi |^2 + | \varphi
    |^2 + \int_{B (0, D) \backslash B (0, 5)} | \nabla \psi |^2 | V_1 |^2
    +\mathfrak{R}\mathfrak{e}^2 (\psi) | V_1 |^4 \right) \nonumber\\
    & - & K \int_{B (0, D)} \tilde{\eta} \left( | (| \nabla (\chi_{\lambda}
    \psi) - \nabla \chi_{\lambda} \psi |^2 - | \nabla (\chi_{\lambda} \psi)
    |^2) | | V_1 |^2 + | \mathfrak{I}\mathfrak{m} (\nabla V_1 \overline{V_1})
    . \nabla \chi_{\lambda} \mathfrak{I}\mathfrak{m} (\psi)
    \mathfrak{R}\mathfrak{e} (\chi_{\lambda} \psi) | \right) \nonumber\\
    & - & K \int_{B (0, D) \backslash B (0, 5)} | \nabla \chi_{\lambda} |^2 |
    \psi |^2 | V_1 |^2 .  \label{quar1}
  \end{eqnarray}
  Since $\nabla \chi_{\lambda}$ is supported in $B (0, 2 \lambda) \backslash B
  (0, \lambda)$ with $| \nabla \chi_{\lambda} | \leqslant \frac{K}{\lambda}$,
  we have
  \begin{equation}
    \int_{B (0, D) \backslash B (0, 5)} | \nabla \chi_{\lambda} |^2 | \psi |^2
    | V_1 |^2 \leqslant K \int_{B (0, 2 \lambda) \backslash B (0, \lambda)}
    \frac{| \psi |^2}{(1 + r)^2},
  \end{equation}
  and by Cauchy-Schwarz, we have that
  \[ \int_{B (0, D)} \tilde{\eta} | \mathfrak{I}\mathfrak{m} (\nabla V_1
     \overline{V_1}) . \nabla \chi_{\lambda} \mathfrak{I}\mathfrak{m} (\psi)
     \mathfrak{R}\mathfrak{e} (\chi_{\lambda} \psi) | \leqslant K
     \sqrt{\int_{B (0, 2 \lambda) \backslash B (0, \lambda)} \frac{| \psi
     |^2}{(1 + r)^2} \int_{B (0, D) \backslash B (0, 5)}
     \mathfrak{R}\mathfrak{e}^2 (\psi)} \]
  and
  \begin{eqnarray}
    &  & \int_{B (0, D)} \tilde{\eta} (| (| \nabla (\chi_{\lambda} \psi) -
    \nabla \chi_{\lambda} \psi |^2 - | \nabla (\chi_{\lambda} \psi) |^2) | |
    V_1 |^2) \nonumber\\
    & \leqslant & K \left( \sqrt{\int_{B (0, 2 \lambda) \backslash B (0,
    \lambda)} \frac{| \psi |^2}{(1 + r)^2} \int_{B (0, D) \backslash B (0, 5)}
    | \nabla \psi |^2 | V_1 |^2} + \int_{B (0, 2 \lambda) \backslash B (0,
    \lambda)} \frac{| \psi |^2}{(1 + r)^2} \right) . 
  \end{eqnarray}
  Since $\psi$ has no $0$ harmonics, we have that
  \[ \int_{B (0, D) \backslash B (0, 5)} \frac{| \psi |^2}{(1 + r)^2}
     \leqslant K \int_{B (0, D) \backslash B (0, 5)} | \nabla \psi |^2 | V_1
     |^2 . \]
  We infer that there exists $D_0 > R_0$ a large constant such that, for $D >
  D_0$, for all $\varphi = V_1 \psi \in C^{\infty}_c (\mathbbm{R}^2 \backslash
  \{ 0 \}, \mathbbm{C})$, there exists $\lambda \in \left[ R_0, \frac{D_0}{2}
  \right]$ such that
  \begin{equation}
    \int_{B (0, 2 \lambda) \backslash B (0, \lambda)} \frac{| \psi |^2}{(1 +
    r)^2} \leqslant \varepsilon \int_{B (0, D) \backslash B (0, 5)} | \nabla
    \psi |^2 | V_1 |^2 \label{quar2}
  \end{equation}
  for some small fixed constant $\varepsilon > 0$. Indeed, if this does not
  hold, then $\int_{B (0, D) \backslash B (0, 5)} | \nabla \psi |^2 | V_1 |^2
  \neq 0$ and
  \begin{eqnarray*}
    &  & \int_{B (0, D) \backslash B (0, 5)} \frac{| \psi |^2}{(1 + r)^2}\\
    & \geqslant & \int_{R_0}^{D_0} \frac{| \psi |^2}{(1 + r)^2} r d r\\
    & \geqslant & \sum_{n = 0}^{\left\lfloor \log_2 \left( \frac{D_0}{2 R_0}
    \right) \right\rfloor - 2} \int_{2^n R_0}^{2^{n + 1} R_0} \frac{| \psi
    |^2}{(1 + r)^2} r d r\\
    & \geqslant & \sum_{n = 0}^{\left\lfloor \log_2 \left( \frac{D_0}{2 R_0}
    \right) \right\rfloor - 2} \varepsilon \int_{B (0, D) \backslash B (0, 5)}
    | \nabla \psi |^2 | V_1 |^2\\
    & \geqslant & \varepsilon \left( \left\lfloor \log_2 \left( \frac{D_0}{2
    R_0} \right) \right\rfloor - 1 \right) \int_{B (0, D) \backslash B (0, 5)}
    | \nabla \psi |^2 | V_1 |^2\\
    & > & \frac{1}{K} \int_{B (0, D) \backslash B (0, 5)} | \nabla \psi |^2 |
    V_1 |^2
  \end{eqnarray*}
  for $D_0$ large enough. Taking $\varepsilon > 0$ small enough, with equation
  (\ref{quar1}) to (\ref{quar2}), we conclude the proof of this lemma.
\end{proof}

A consequence of Lemma \ref{CP2L422602} is that, for a function $\varphi = V_1
\psi \in C^{\infty}_c (\mathbbm{R}^2 \backslash \{ 0 \}, \mathbbm{C})$
satisfying the three orthogonality conditions in Lemma \ref{CP2L422602} and $D
> D_0$, then
\begin{equation}
  B_{V_1}^{\tmop{loc}_D \nobracket \nobracket} (\varphi) \geqslant K (D) \|
  \varphi \|^2_{H^1 (B (0, D))} . \label{CP2eq420}
\end{equation}

\subsection{Coercivity for a travelling wave near its zeros}

\quad We recall from Lemma \ref{CP2L412602} that, for $\varphi \in
C^{\infty}_c (\mathbbm{R}^2 \backslash \{ \tilde{d}_c \vec{e}_1, - \tilde{d}_c
\vec{e}_1 \}, \mathbbm{C})$, we have
\begin{eqnarray*}
  \langle L_{Q_c} (\varphi), \varphi \rangle & = & \int_{\mathbbm{R}^2} (1 -
  \eta) (| \nabla \varphi |^2 -\mathfrak{R}\mathfrak{e} (i c \partial_{x_2}
  \varphi \bar{\varphi}) - (1 - | Q_c |^2) | \varphi |^2 +
  2\mathfrak{R}\mathfrak{e}^2 (\overline{Q_c} \varphi))\\
  & - & \int_{\mathbbm{R}^2} \nabla \eta . (\mathfrak{R}\mathfrak{e} (\nabla
  Q_c \overline{Q_c}) | \psi |^2 - 2\mathfrak{I}\mathfrak{m} (\nabla Q_c
  \overline{Q_c}) \mathfrak{R}\mathfrak{e} (\psi) \mathfrak{I}\mathfrak{m}
  (\psi))\\
  & + & \int_{\mathbbm{R}^2} c \partial_{x_2} \eta \mathfrak{R}\mathfrak{e}
  (\psi) \mathfrak{I}\mathfrak{m} (\psi) | Q_c |^2\\
  & + & \int_{\mathbbm{R}^2} \eta (| \nabla \psi |^2 | Q_c |^2 +
  2\mathfrak{R}\mathfrak{e}^2 (\psi) | Q_c |^4)\\
  & + & \int_{\mathbbm{R}^2} \eta (4\mathfrak{I}\mathfrak{m} (\nabla Q_c
  \overline{Q_c}) \mathfrak{I}\mathfrak{m} (\nabla \psi)
  \mathfrak{R}\mathfrak{e} (\psi) + 2 c | Q_c |^2 \mathfrak{I}\mathfrak{m}
  (\partial_{x_2} \psi) \mathfrak{R}\mathfrak{e} (\psi)) .
\end{eqnarray*}
For $D > D_0$ ($D_0 > 0$ being defined in Lemma \ref{CP2L422602}), we define,
with $\varphi = Q_c \psi$,
\begin{eqnarray*}
  B_{Q_c}^{\tmop{loc}_{\pm 1, D}} (\varphi) & \assign & \int_{B (\pm
  \tilde{d}_c \overrightarrow{e_1}, D)} (1 - \eta) (| \nabla \varphi |^2
  -\mathfrak{R}\mathfrak{e} (i c \partial_{x_2} \varphi \bar{\varphi}) - (1 -
  | Q_c |^2) | \varphi |^2 + 2\mathfrak{R}\mathfrak{e}^2 (\overline{Q_c}
  \varphi))\\
  & - & \int_{B (\pm \tilde{d}_c \overrightarrow{e_1}, D)} \nabla \eta .
  (\mathfrak{R}\mathfrak{e} (\nabla Q_c \overline{Q_c}) | \psi |^2 -
  2\mathfrak{I}\mathfrak{m} (\nabla Q_c \overline{Q_c})
  \mathfrak{R}\mathfrak{e} (\psi) \mathfrak{I}\mathfrak{m} (\psi))\\
  & + & \int_{B (\pm \tilde{d}_c \overrightarrow{e_1}, D)} c \partial_{x_2}
  \eta \mathfrak{R}\mathfrak{e} (\psi) \mathfrak{I}\mathfrak{m} (\psi) | Q_c
  |^2\\
  & + & \int_{B (\pm \tilde{d}_c \overrightarrow{e_1}, D)} \eta (| \nabla
  \psi |^2 | Q_c |^2 + 2\mathfrak{R}\mathfrak{e}^2 (\psi) | Q_c |^4)\\
  & + & \int_{B (\pm \tilde{d}_c \overrightarrow{e_1}, D)} \eta
  (4\mathfrak{I}\mathfrak{m} (\nabla Q_c \overline{Q_c})
  \mathfrak{I}\mathfrak{m} (\nabla \psi) \mathfrak{R}\mathfrak{e} (\psi) + 2 c
  | Q_c |^2 \mathfrak{I}\mathfrak{m} (\partial_{x_2} \psi)
  \mathfrak{R}\mathfrak{e} (\psi)) .
\end{eqnarray*}
We infer that this quantity is close enough to
$B^{\tmop{loc}_D}_{\tilde{V}_{\pm 1}} (\varphi)$ for the coercivity to hold,
with $\tilde{V}_{\pm 1}$ being centered at $\pm \tilde{d}_c
\overrightarrow{e_1}$, the zero of $Q_c$ in the right half plane.

\begin{lemma}
  \label{CP2L4326}There exist $R, D_0 > 0$ with $D_0 > R$, such that, for $D >
  D_0$, $0 < c < c_0 (D)$ and $\varphi = Q_c \psi \in C^{\infty}_c
  (\mathbbm{R}^2 \backslash \{ \tilde{d}_c \overrightarrow{e_1} \},
  \mathbbm{C})$, if the following three orthogonality conditions
  \[ \int_{B (\tilde{d}_c \overrightarrow{e_1}, R)} \mathfrak{R}\mathfrak{e}
     (\partial_{x_1} \tilde{V}_1 \bar{\varphi}) = \int_{B (\tilde{d}_c
     \overrightarrow{e_1}, R)} \mathfrak{R}\mathfrak{e} (\partial_{x_2}
     \tilde{V}_1 \bar{\varphi}) = \int_{B (\tilde{d}_c \overrightarrow{e_1},
     R) \backslash B (\tilde{d}_c \overrightarrow{e_1}, R / 2)}
     \mathfrak{I}\mathfrak{m} (\psi) = 0 \]
  are satisfied, then
  \[ B_{Q_c}^{\tmop{loc}_{1, D}} (\varphi) \geqslant K (D) \| \varphi \|_{H^1
     (B (\tilde{d}_c \overrightarrow{e_1}, D))}^2 . \]
\end{lemma}

\begin{proof}
  First, remark that we write $\varphi = Q_c \psi$ and not $\varphi =
  \tilde{V}_1 \psi$, as we did in the proof of Proposition \ref{CP2P811}.
  Hence, to apply Lemma \ref{CP2L422602}, the third orthogonality condition
  becomes
  \[ \int_{B (\tilde{d}_c \overrightarrow{e_1}, R) \backslash B (\tilde{d}_c
     \overrightarrow{e_1}, R / 2)} \mathfrak{I}\mathfrak{m} \left( \psi
     \frac{Q_c}{\tilde{V}_1} \right) = 0. \]
  With Lemma \ref{CP2closecall}, we check that
  \begin{eqnarray*}
    &  & \left| \int_{B (\tilde{d}_c \overrightarrow{e_1}, R) \backslash B
    (\tilde{d}_c \overrightarrow{e_1}, R / 2)} \mathfrak{I}\mathfrak{m} \left(
    \psi \frac{Q_c}{\tilde{V}_1} \right) \right|\\
    & \leqslant & \left| \int_{B (\tilde{d}_c \overrightarrow{e_1}, R)
    \backslash B (\tilde{d}_c \overrightarrow{e_1}, R / 2)}
    \mathfrak{I}\mathfrak{m} (\psi) \right| + o_{c \rightarrow 0} (1) \| \psi
    \|_{L^2 (B (\tilde{d}_c \overrightarrow{e_1}, R) \backslash B (\tilde{d}_c
    \overrightarrow{e_1}, R / 2))}\\
    & \leqslant & \left| \int_{B (\tilde{d}_c \overrightarrow{e_1}, R)
    \backslash B (\tilde{d}_c \overrightarrow{e_1}, R / 2)}
    \mathfrak{I}\mathfrak{m} (\psi) \right| + o^D_{c \rightarrow 0} (1) \|
    \varphi \|_{H^1 (B (\tilde{d}_c \overrightarrow{e_1}, D))},
  \end{eqnarray*}
  therefore, by standard coercivity argument, we can change this orthogonality
  condition, given that $c$ is small enough (depending on $D$). With equation
  (\ref{CP2eq420}), it is therefore enough to show that
  \[ | B_{Q_c}^{\tmop{loc}_D} (\varphi) - B^{\tmop{loc}_D}_{\tilde{V}_1}
     (\varphi) | \leqslant o_{c \rightarrow 0}^D (1) \| \varphi \|_{H^1 (B
     (\tilde{d}_c \overrightarrow{e_1}, D))}^2 \]
  to complete the proof of this lemma. Thus, for $\varphi = Q_c \psi \in
  C^{\infty}_c (\mathbbm{R}^2 \backslash \{ \tilde{d}_c \overrightarrow{e_1}
  \}, \mathbbm{C})$, writing $\varphi = V_1 \left( \frac{Q_c}{V_1} \psi
  \right)$ in $B^{\tmop{loc}_D}_{\tilde{V}_1} (\varphi)$, we have
  \begin{eqnarray*}
    &  & B_{Q_c}^{\tmop{loc}_{1, D}} (\varphi) -
    B^{\tmop{loc}_D}_{\tilde{V}_1} (\varphi)\\
    & = & \int_{B (\tilde{d}_c \overrightarrow{e_1}, D)}
    -\mathfrak{R}\mathfrak{e} (i c \partial_{x_2} \varphi \bar{\varphi}) + (|
    Q_c |^2 - | \tilde{V}_1 |^2) | \varphi |^2 + 2 \left(
    \mathfrak{R}\mathfrak{e}^2 (\overline{Q_c} \varphi)
    -\mathfrak{R}\mathfrak{e}^2 \left( \overline{\widetilde{V_1}} \varphi
    \right) \right)\\
    & - & \int_{B (\tilde{d}_c \overrightarrow{e_1}, D)} \nabla \eta .
    (\mathfrak{R}\mathfrak{e} (\nabla Q_c \overline{Q_c}) | \psi |^2 -
    2\mathfrak{I}\mathfrak{m} (\nabla Q_c \overline{Q_c})
    \mathfrak{R}\mathfrak{e} (\psi) \mathfrak{I}\mathfrak{m} (\psi))\\
    & + & \int_{B (\tilde{d}_c \overrightarrow{e_1}, D)} \nabla \eta . \left(
    \mathfrak{R}\mathfrak{e} (\nabla \tilde{V}_1 \overline{\tilde{V}_1})
    \left| \frac{Q_c}{\tilde{V}_1} \psi \right|^2 - 2\mathfrak{I}\mathfrak{m}
    (\nabla \tilde{V}_1 \overline{\tilde{V}_1}) \mathfrak{R}\mathfrak{e}
    \left( \frac{Q_c}{\tilde{V}_1} \psi \right) \mathfrak{I}\mathfrak{m}
    \left( \frac{Q_c}{\tilde{V}_1} \psi \right) \right)\\
    & + & \int_{B (\tilde{d}_c \overrightarrow{e_1}, D)} c \partial_{x_2}
    \eta \mathfrak{R}\mathfrak{e} (\psi) \mathfrak{I}\mathfrak{m} (\psi) | Q_c
    |^2\\
    & + & \int_{B (\tilde{d}_c \overrightarrow{e_1}, D)} \eta (| \nabla \psi
    |^2 | Q_c |^2 + 2\mathfrak{R}\mathfrak{e}^2 (\psi) | Q_c |^4)\\
    & - & \int_{B (\tilde{d}_c \overrightarrow{e_1}, D)} \eta \left( \left|
    \nabla \left( \frac{Q_c}{\tilde{V}_1} \psi \right) \right|^2 | Q_c |^2 +
    2\mathfrak{R}\mathfrak{e}^2 \left( \frac{Q_c}{\tilde{V}_1} \psi \right) |
    Q_c |^4 \right)\\
    & + & \int_{B (\tilde{d}_c \overrightarrow{e_1}, D)} \eta
    (4\mathfrak{I}\mathfrak{m} (\nabla Q_c \overline{Q_c})
    \mathfrak{I}\mathfrak{m} (\nabla \psi) \mathfrak{R}\mathfrak{e} (\psi) + 2
    c | Q_c |^2 \mathfrak{I}\mathfrak{m} (\partial_{x_2} \psi)
    \mathfrak{R}\mathfrak{e} (\psi))\\
    & - & \int_{B (\tilde{d}_c \overrightarrow{e_1}, D)} \eta \left(
    4\mathfrak{I}\mathfrak{m} (\nabla Q_c \overline{Q_c})
    \mathfrak{I}\mathfrak{m} \left( \nabla \left( \frac{Q_c}{\tilde{V}_1} \psi
    \right) \right) \mathfrak{R}\mathfrak{e} \left( \frac{Q_c}{\tilde{V}_1}
    \psi \right) \right) .
  \end{eqnarray*}
  With Theorem \ref{th1} (for $p = + \infty$) and Cauchy-Schwarz, we check
  easily that
  \begin{eqnarray*}
    &  & \int_{B (\tilde{d}_c \overrightarrow{e_1}, D)} |
    \mathfrak{R}\mathfrak{e} (i c \partial_{x_2} \varphi \bar{\varphi}) | + |
    | Q_c |^2 - | \tilde{V}_1 |^2 | | \varphi |^2 + 2 \left|
    \mathfrak{R}\mathfrak{e}^2 (\overline{Q_c} \varphi)
    -\mathfrak{R}\mathfrak{e}^2 \left( \overline{\widetilde{V_1}} \varphi
    \right) \right|\\
    & \leqslant & o_{c \rightarrow 0}^D (1) \| \varphi \|_{H^1 (B
    (\tilde{d}_c \overrightarrow{e_1}, D))}^2 .
  \end{eqnarray*}
  Since $\nabla \eta$ is supported in $B (\tilde{d}_c \overrightarrow{e_1}, 2)
  \backslash B (\tilde{d}_c \overrightarrow{e_1}, 1)$, still with Theorem
  \ref{th1} (for $p = + \infty$), we check that
  \begin{eqnarray*}
    &  & \int_{B (\tilde{d}_c \overrightarrow{e_1}, D)} \left| \nabla \eta
    .\mathfrak{R}\mathfrak{e} (\nabla Q_c \overline{Q_c}) | \psi |^2 - \nabla
    \eta \mathfrak{R}\mathfrak{e} (\nabla \tilde{V}_1 \overline{\tilde{V}_1})
    \left| \frac{Q_c}{\tilde{V}_1} \psi \right|^2 \right|\\
    & \leqslant & K \int_{B (\tilde{d}_c \overrightarrow{e_1}, D)} \left|
    \nabla \eta .\mathfrak{R}\mathfrak{e} (\nabla Q_c \overline{Q_c}) |
    \varphi |^2 - \nabla \eta \mathfrak{R}\mathfrak{e} (\nabla \tilde{V}_1
    \overline{\tilde{V}_1}) \left| \frac{Q_c}{\tilde{V}_1} \varphi \right|^2
    \right|\\
    & \leqslant & \left\| \nabla \eta .\mathfrak{R}\mathfrak{e} (\nabla Q_c
    \overline{Q_c}) - \nabla \eta \mathfrak{R}\mathfrak{e} (\nabla \tilde{V}_1
    \overline{\tilde{V}_1}) \left| \frac{Q_c}{\tilde{V}_1} \right|^2
    \right\|_{L^{\infty} ((\tilde{d}_c \overrightarrow{e_1}, D))} \| \varphi
    \|_{H^1 (B (\tilde{d}_c \overrightarrow{e_1}, D))}\\
    & \leqslant & o_{c \rightarrow 0}^D (1) \| \varphi \|_{H^1 (B
    (\tilde{d}_c \overrightarrow{e_1}, D))}^2 .
  \end{eqnarray*}
  We check similarly that the same estimate hold for all the remaining error
  terms, using the fact that $\eta$ is supported in $\mathbbm{R}^2 \backslash
  B (\tilde{d}_c \overrightarrow{e_1}, 1)$.
\end{proof}

Remark that, by density argument (see the proof of Lemma \ref{CP2Ndensity}),
Lemma \ref{CP2L4326} holds for any $\varphi \in H^1 (B (0, D))$. Now, we want
to remove the orthogonality condition on the phase. For that, we have to
change the coercivity norm

\begin{lemma}
  \label{CP2L442702}There exist $R, D_0 > 0$ with $D_0 > R$, such that, for $D
  > D_0$, $0 < c < c_0 (D)$ and $\varphi = Q_c \psi \in C^{\infty}_c
  (\mathbbm{R}^2 \backslash \{ \tilde{d}_c \overrightarrow{e_1} \},
  \mathbbm{C})$, if the following two orthogonality conditions
  \[ \int_{B (\tilde{d}_c \overrightarrow{e_1}, R)} \mathfrak{R}\mathfrak{e}
     (\partial_{x_1} \tilde{V}_1 \overline{\tilde{V}_1 \psi}) = \int_{B
     (\tilde{d}_c \overrightarrow{e_1}, R)} \mathfrak{R}\mathfrak{e}
     (\partial_{x_2} \tilde{V}_1 \overline{\tilde{V}_1 \psi}) = 0 \]
  are satisfied, then
  \[ B_{Q_c}^{\tmop{loc}_{1, D}} (\varphi) \geqslant K (D) \int_{B
     (\tilde{d}_c \overrightarrow{e_1}, D)} | \nabla \psi |^2 | Q_c |^4
     +\mathfrak{R}\mathfrak{e}^2 (\psi) | Q_c |^4 . \]
\end{lemma}

\begin{proof}
  Take a function $\varphi \in H^1 (B (0, D))$ that satisfies the
  orthogonality conditions
  \[ \int_{B (\tilde{d}_c \overrightarrow{e_1}, R)} \mathfrak{R}\mathfrak{e}
     (\partial_{x_1} \tilde{V}_1 \overline{\tilde{V}_1 \psi}) = \int_{B
     (\tilde{d}_c \overrightarrow{e_1}, R)} \mathfrak{R}\mathfrak{e}
     (\partial_{x_2} \tilde{V}_1 \overline{\tilde{V}_1 \psi}) = \int_{B
     (\tilde{d}_c \overrightarrow{e_1}, R) \backslash B (\tilde{d}_c
     \overrightarrow{e_1}, R / 2)} \mathfrak{I}\mathfrak{m} (\psi) = 0, \]
  and let us show that $B_{Q_c}^{\tmop{loc}_{1, D}} (\varphi) \geqslant K \|
  \varphi \|_{H^1 (B (\tilde{d}_c \overrightarrow{e_1}, D))}^2$. Take
  $\varepsilon_1, \varepsilon_2, \varepsilon_3 \in \mathbbm{R}$ and we define
  \[ \tilde{\varphi} = \varphi - \varepsilon_1 \partial_{x_1} Q_c -
     \varepsilon_2 \partial_{x_2} Q_c - \varepsilon_3 i Q_c . \]
  We have, for $\varphi = Q_c \psi$, by Theorem \ref{th1} (for $p = + \infty$)
  and Lemma \ref{CP2closecall},
  \begin{eqnarray*}
    &  & \left| \int_{B (\tilde{d}_c \overrightarrow{e_1}, R)}
    \mathfrak{R}\mathfrak{e} (\partial_{x_1} \tilde{V}_1 \overline{\tilde{V}_1
    \psi}) - \int_{B (\tilde{d}_c \overrightarrow{e_1}, R)}
    \mathfrak{R}\mathfrak{e} (\partial_{x_1} Q_c \overline{Q_c \psi})
    \right|\\
    & \leqslant & \left| \int_{B (\tilde{d}_c \overrightarrow{e_1}, R)}
    \mathfrak{R}\mathfrak{e} \left( \partial_{x_1} \tilde{V}_1
    \frac{\tilde{V}_1}{Q_c} \bar{\varphi} - \partial_{x_1} Q_c \bar{\varphi}
    \right) \right|\\
    & \leqslant & K \left\| \partial_{x_1} \tilde{V}_1
    \frac{\tilde{V}_1}{Q_c} - \partial_{x_1} Q_c \right\|_{L^{\infty} (B
    (\tilde{d}_c \overrightarrow{e_1}, R))} \| \varphi \|_{H^1 (B (\tilde{d}_c
    \overrightarrow{e_1}, D))}\\
    & \leqslant & o^D_{c \rightarrow 0} (1) \| \varphi \|_{H^1 (B
    (\tilde{d}_c \overrightarrow{e_1}, D))} .
  \end{eqnarray*}
  Similar estimates hold for $\int_{B (\tilde{d}_c \overrightarrow{e_1}, R)}
  \mathfrak{R}\mathfrak{e} (\partial_{x_2} \tilde{V}_1 \overline{\tilde{V}_1
  \psi})$. By standard arguments, we check that there exists $\varepsilon_1,
  \varepsilon_2, \varepsilon_3 \in \mathbbm{R}$ with $| \varepsilon_1 | + |
  \varepsilon_2 | + | \varepsilon_2 | \leqslant o_{c \rightarrow 0} (1) \|
  \varphi \|_{H^1 (B (\tilde{d}_c \overrightarrow{e_1}, D))}$ such that
  $\tilde{\varphi}$ satisfies the three orthogonality conditions of Lemma
  \ref{CP2L4326}. We deduce that, since (by Theorem \ref{th1} for $p = +
  \infty$)
  \[ \| \partial_{x_1} Q_c \|_{H^1 (B (\tilde{d}_c \overrightarrow{e_1}, D))}
     + \| \partial_{x_2} Q_c \|_{H^1 (B (\tilde{d}_c \overrightarrow{e_1},
     D))} + \| i Q_c \|_{H^1 (B (\tilde{d}_c \overrightarrow{e_1}, D))}
     \leqslant K (D), \]
  \begin{eqnarray*}
    B_{Q_c}^{\tmop{loc}_{1, D}} (\varphi) & \geqslant &
    B_{Q_c}^{\tmop{loc}_{1, D}} (\tilde{\varphi}) - o^D_{c \rightarrow 0} (1)
    \| \varphi \|_{H^1 (B (\tilde{d}_c \overrightarrow{e_1}, D))}^2\\
    & \geqslant & K (D) \| \tilde{\varphi} \|^2_{H^1 (B (\tilde{d}_c
    \overrightarrow{e_1}, D))} - o^D_{c \rightarrow 0} (1) \| \varphi \|_{H^1
    (B (\tilde{d}_c \overrightarrow{e_1}, D))}^2\\
    & \geqslant & K (D) \| \varphi \|^2_{H^1 (B (\tilde{d}_c
    \overrightarrow{e_1}, D))} - o^D_{c \rightarrow 0} (1) \| \varphi \|_{H^1
    (B (\tilde{d}_c \overrightarrow{e_1}, D))}^2\\
    & \geqslant & K (D) \| \varphi \|^2_{H^1 (B (\tilde{d}_c
    \overrightarrow{e_1}, D))},
  \end{eqnarray*}
  given that $c$ is small enough (depending on $D$). For $\varphi = Q_c \psi$,
  we infer that
  \[ \int_{B (\tilde{d}_c \overrightarrow{e_1}, D)} | \nabla \psi |^2 | Q_c
     |^4 +\mathfrak{R}\mathfrak{e}^2 (\psi) | Q_c |^4 \leqslant K (D) \|
     \varphi \|^2_{H^1 (B (\tilde{d}_c \overrightarrow{e_1}, D))} . \]
  Indeed, we have
  \[ \int_{B (\tilde{d}_c \overrightarrow{e_1}, D)} \mathfrak{R}\mathfrak{e}^2
     (\psi) | Q_c |^4 \leqslant K \int_{B (\tilde{d}_c \overrightarrow{e_1},
     D)} \mathfrak{R}\mathfrak{e}^2 (\varphi) \leqslant K \| \varphi \|^2_{H^1
     (B (\tilde{d}_c \overrightarrow{e_1}, D))}, \]
  and
  \begin{eqnarray*}
    \int_{B (\tilde{d}_c \overrightarrow{e_1}, D)} | \nabla \psi |^2 | Q_c |^4
    & = & \int_{B (\tilde{d}_c \overrightarrow{e_1}, D)} | \nabla \varphi -
    \nabla Q_c \psi |^2 | Q_c |^2\\
    & \leqslant & K \left( \int_{B (\tilde{d}_c \overrightarrow{e_1}, D)} |
    \nabla \varphi |^2 + \int_{B (\tilde{d}_c \overrightarrow{e_1}, D)} |
    \nabla Q_c \psi |^2 | Q_c |^2 \right)\\
    & \leqslant & K \left( \int_{B (\tilde{d}_c \overrightarrow{e_1}, D)} |
    \nabla \varphi |^2 + \int_{B (\tilde{d}_c \overrightarrow{e_1}, D)} |
    \varphi |^2 \right) .
  \end{eqnarray*}
  We deduce that, under the three orthogonality conditions, for $\varphi = Q_c
  \psi$,
  \[ \int_{B (\tilde{d}_c \overrightarrow{e_1}, R)} \mathfrak{R}\mathfrak{e}
     (\partial_{x_1} \tilde{V}_1 \overline{\tilde{V}_1 \psi}) = \int_{B
     (\tilde{d}_c \overrightarrow{e_1}, R)} \mathfrak{R}\mathfrak{e}
     (\partial_{x_2} \tilde{V}_1 \overline{\tilde{V}_1 \psi}) = \int_{B
     (\tilde{d}_c \overrightarrow{e_1}, R) \backslash B (\tilde{d}_c
     \overrightarrow{e_1}, R / 2)} \mathfrak{I}\mathfrak{m} (\psi) = 0, \]
  then
  \[ B_{Q_c}^{\tmop{loc}_{1, D}} (\varphi) \geqslant K (D) \int_{B
     (\tilde{d}_c \overrightarrow{e_1}, D)} | \nabla \psi |^2 | Q_c |^4
     +\mathfrak{R}\mathfrak{e}^2 (\psi) | Q_c |^4 . \]

  Now, let us show that for any $\lambda \in \mathbbm{R}$, $\varphi \in H^1 (B
  (\tilde{d}_c \overrightarrow{e_1}, D))$,
  \[ B_{Q_c}^{\tmop{loc}_{1, D}} (\varphi - i \lambda Q_c) =
     B_{Q_c}^{\tmop{loc}_{1, D}} (\varphi) . \]
  For $\varphi \in C^{\infty}_c (\mathbbm{R}^2, \mathbbm{C})$, we have
  $L_{Q_c} (\varphi - i \lambda Q_c) = L_{Q_c} (\varphi) \in C^{\infty}_c
  (\mathbbm{R}^2, \mathbbm{C})$, thus $\langle L_{Q_c} (\varphi - i \lambda
  Q_c), \varphi - i \lambda Q_c \rangle$ is well defined, and
  \[ \langle L_{Q_c} (\varphi - i \lambda Q_c), \varphi - i \lambda Q_c
     \rangle = \langle L_{Q_c} (\varphi), \varphi - i \lambda Q_c \rangle =
     \langle \varphi, L_{Q_c} (\varphi - i \lambda Q_c) \rangle = \langle
     L_{Q_c} (\varphi), \varphi \rangle . \]
  With computations similar to the one of the proof of Lemma \ref{CP2L412602}
  and by density, using \ $\nabla (\psi - i \lambda) = \nabla \psi$ and
  $\mathfrak{R}\mathfrak{e} (\psi - i \lambda) =\mathfrak{R}\mathfrak{e}
  (\psi)$, we deduce that $B_{Q_c}^{\tmop{loc}_{1, D}} (\varphi - i \lambda
  Q_c) = B_{Q_c}^{\tmop{loc}_{1, D}} (\varphi)$.
  
  \
  
  Now, for $\lambda \in \mathbbm{R}$, $\tilde{\varphi} = \varphi - i \lambda
  Q_c$, $\tilde{\psi} = \psi - i \lambda$, $\tilde{\varphi} = Q_c
  \tilde{\psi}$, we have $B_{Q_c}^{\tmop{loc}_{1, D}} (\varphi) =
  B_{Q_c}^{\tmop{loc}_{1, D}} (\tilde{\varphi})$,
  \[ \int_{B (\tilde{d}_c \overrightarrow{e_1}, D)} | \nabla \psi |^2 | Q_c
     |^4 +\mathfrak{R}\mathfrak{e}^2 (\psi) | Q_c |^4 = \int_{B (\tilde{d}_c
     \overrightarrow{e_1}, D)} | \nabla \tilde{\psi} |^2 | Q_c |^4
     +\mathfrak{R}\mathfrak{e}^2 (\tilde{\psi}) | Q_c |^4 \]
  and
  \[ \int_{B (\tilde{d}_c \overrightarrow{e_1}, R)} \mathfrak{R}\mathfrak{e}
     (\nabla \tilde{V}_1 \overline{\tilde{V}_1 \psi}) = \int_{B (\tilde{d}_c
     \overrightarrow{e_1}, R)} \mathfrak{R}\mathfrak{e} (\nabla \tilde{V}_1
     \overline{\tilde{V}_1 \tilde{\psi}}) . \]
  For this last equality, it comes from the fact that $\int_{B (\tilde{d}_c
  \overrightarrow{e_1}, R)} \mathfrak{R}\mathfrak{e} (i \nabla \tilde{V}_1
  \overline{\tilde{V}_1}) = 0$, since $\mathfrak{R}\mathfrak{e} (i \nabla
  \tilde{V}_1 \overline{\tilde{V}_1})$ has no zero harmonic (see Lemma
  \ref{lemme3new}). We also check that
  \[ \int_{B (\tilde{d}_c \overrightarrow{e_1}, R) \backslash B (\tilde{d}_c
     \overrightarrow{e_1}, R / 2)} \mathfrak{I}\mathfrak{m} (\psi) = \int_{B
     (\tilde{d}_c \overrightarrow{e_1}, R) \backslash B (\tilde{d}_c
     \overrightarrow{e_1}, R / 2)} \mathfrak{I}\mathfrak{m} (\tilde{\psi}) + K
     \lambda \]
  for a universal constant $K > 0$. Therefore, choosing $\lambda \in
  \mathbbm{R}$ such that $\int_{B (\tilde{d}_c \overrightarrow{e_1}, R)
  \backslash B (\tilde{d}_c \overrightarrow{e_1}, R / 2)}
  \mathfrak{I}\mathfrak{m} (\tilde{\psi}) = 0$, we have, for a function
  $\varphi = Q_c \psi$ that satisfies
  \[ \int_{B (\tilde{d}_c \overrightarrow{e_1}, R)} \mathfrak{R}\mathfrak{e}
     (\partial_{x_1} \tilde{V}_1 \overline{\tilde{V}_1 \psi}) = \int_{B
     (\tilde{d}_c \overrightarrow{e_1}, R)} \mathfrak{R}\mathfrak{e}
     (\partial_{x_2} \tilde{V}_1 \overline{\tilde{V}_1 \psi}) = 0, \]
  that
  \begin{eqnarray*}
    B_{Q_c}^{\tmop{loc}_{1, D}} (\varphi) & = & B_{Q_c}^{\tmop{loc}_{1, D}}
    (\tilde{\varphi})\\
    & \geqslant & \int_{B (\tilde{d}_c \overrightarrow{e_1}, D)} | \nabla
    \tilde{\psi} |^2 | Q_c |^4 +\mathfrak{R}\mathfrak{e}^2 (\tilde{\psi}) |
    Q_c |^4\\
    & = & \int_{B (\tilde{d}_c \overrightarrow{e_1}, D)} | \nabla \psi |^2 |
    Q_c |^4 +\mathfrak{R}\mathfrak{e}^2 (\psi) | Q_c |^4 .
  \end{eqnarray*}
  This concludes the proof of this lemma.
\end{proof}

\subsection{Proof of Proposition \ref{CP205218}\label{CP2concl}}

\begin{proof}[of Proposition \ref{CP205218}]
  From Lemma \ref{CP2L412602}, we have, for $\varphi = Q_c \psi \in
  C^{\infty}_c (\mathbbm{R}^2 \backslash \{ \tilde{d}_c \overrightarrow{e_1},
  - \tilde{d}_c \overrightarrow{e_1} \}, \mathbbm{C})$ that
  \begin{eqnarray*}
    B_{Q_c} (\varphi) & = & \int_{\mathbbm{R}^2} (1 - \eta) (| \nabla \varphi
    |^2 -\mathfrak{R}\mathfrak{e} (i c \partial_{x_2} \varphi \bar{\varphi}) -
    (1 - | Q_c |^2) | \varphi |^2 + 2\mathfrak{R}\mathfrak{e}^2
    (\overline{Q_c} \varphi))\\
    & - & \int_{\mathbbm{R}^2} \nabla \eta . (\mathfrak{R}\mathfrak{e}
    (\nabla Q_c \overline{Q_c}) | \psi |^2 - 2\mathfrak{I}\mathfrak{m} (\nabla
    Q_c \overline{Q_c}) \mathfrak{R}\mathfrak{e} (\psi)
    \mathfrak{I}\mathfrak{m} (\psi))\\
    & + & \int_{\mathbbm{R}^2} c \partial_{x_2} \eta \mathfrak{R}\mathfrak{e}
    (\psi) \mathfrak{I}\mathfrak{m} (\psi) | Q_c |^2\\
    & + & \int_{\mathbbm{R}^2} \eta (| \nabla \psi |^2 | Q_c |^2 +
    2\mathfrak{R}\mathfrak{e}^2 (\psi) | Q_c |^4)\\
    & + & \int_{\mathbbm{R}^2} \eta (4\mathfrak{I}\mathfrak{m} (\nabla Q_c
    \overline{Q_c}) \mathfrak{I}\mathfrak{m} (\nabla \psi)
    \mathfrak{R}\mathfrak{e} (\psi) + 2 c | Q_c |^2 \mathfrak{I}\mathfrak{m}
    (\partial_{x_2} \psi) \mathfrak{R}\mathfrak{e} (\psi)) .
  \end{eqnarray*}
  We decompose the integral in three domains, $B (\pm \tilde{d}_c
  \overrightarrow{e_1}, D)$ (which yield $B_{Q_c}^{\tmop{loc}_{\pm 1, D}}
  (\varphi)$) and $\mathbbm{R}^2 \backslash (B (\tilde{d}_c
  \overrightarrow{e_1}, D) \cup B (- \tilde{d}_c \overrightarrow{e_1}, D))$
  for some $D > D_0 > 0$, where $D_0$ is defined in Lemma \ref{CP2L4326}.
  
  Then, with the four orthogonality conditions and Lemma \ref{CP2L4326}, we
  check that
  \[ B_{Q_c}^{\tmop{loc}_{1, D}} (\varphi) \geqslant K (D) \int_{B
     (\tilde{d}_c \overrightarrow{e_1}, D)} | \nabla \psi |^2 | Q_c |^4
     +\mathfrak{R}\mathfrak{e}^2 (\psi) | Q_c |^4, \]
  and, by symmetry of the problem around $B (\pm \tilde{d}_c
  \overrightarrow{e_1}, D)$, since $Q_c = - V_{- 1} (. + \tilde{d}_c
  \vec{e}_1) + o_{c \rightarrow 0} (1)$ in $L^{\infty} (B (- \tilde{d}_c
  \overrightarrow{e_1}, D))$, and checking that multiplying the vortex by $-
  1$ does not change the result, that
  \[ B_{Q_c}^{\tmop{loc}_{- 1, D}} (\varphi) \geqslant K (D) \int_{B (-
     \tilde{d}_c \overrightarrow{e_1}, D)} | \nabla \psi |^2 | Q_c |^4
     +\mathfrak{R}\mathfrak{e}^2 (\psi) | Q_c |^4 . \]
  Furthermore, there exist $K_1, K_2 > 0$, universal constants, such that,
  outside of $B (\widetilde{d_c} \overrightarrow{e_1}, 1) \cup B (-
  \widetilde{d_c} \overrightarrow{e_1}, 1)$ for $c$ small enough, we have
  \[ K_1 \geqslant | Q_c |^2 \geqslant K_2 \]
  by (\ref{CP2Qcpaszero}). We also have
  \[ | \mathfrak{I}\mathfrak{m} (\nabla Q_c \overline{Q_c}) | \leqslant K
     \left( \frac{1}{(1 + \tilde{r}_1)} + \frac{1}{(1 + \tilde{r}_{- 1})}
     \right) \]
  by (\ref{CP2221}). With these estimates and by Cauchy-Schwarz, for $D >
  D_0$,
  \begin{eqnarray*}
    &  & \int_{\mathbbm{R}^2 \backslash (B (\tilde{d}_c \overrightarrow{e_1},
    D) \cup B (- \tilde{d}_c \overrightarrow{e_1}, D))} 2 c | Q_c |^2
    \mathfrak{I}\mathfrak{m} (\partial_{x_2} \psi) \mathfrak{R}\mathfrak{e}
    (\psi)\\
    & \geqslant & - K c \int_{\mathbbm{R}^2 \backslash (B (\tilde{d}_c
    \overrightarrow{e_1}, D) \cup B (- \tilde{d}_c \overrightarrow{e_1}, D))}
    | \nabla \psi |^2 | Q_c |^4 +\mathfrak{R}\mathfrak{e}^2 (\psi) | Q_c |^4,
  \end{eqnarray*}
  and
  \begin{eqnarray*}
    &  & \int_{\mathbbm{R}^2 \backslash (B (\tilde{d}_c \overrightarrow{e_1},
    D) \cup B (- \tilde{d}_c \overrightarrow{e_1}, D))}
    4\mathfrak{I}\mathfrak{m} (\nabla Q_c \overline{Q_c})
    .\mathfrak{I}\mathfrak{m} (\nabla \psi) \mathfrak{R}\mathfrak{e} (\psi)\\
    & \geqslant & \frac{- K}{(1 + D)} \int_{\mathbbm{R}^2 \backslash (B
    (\tilde{d}_c \overrightarrow{e_1}, D) \cup B (- \tilde{d}_c
    \overrightarrow{e_1}, D))} | \nabla \psi |^2 | Q_c |^4
    +\mathfrak{R}\mathfrak{e}^2 (\psi) | Q_c |^4 .
  \end{eqnarray*}
  Therefore, taking $D > D_0$ large enough (independently of $c$ or $c_0$, $D
  \geqslant 10 K + 1$) and $c$ small enough ($c \leqslant \frac{10}{K}$), we
  have
  \begin{eqnarray*}
    &  & \int_{\mathbbm{R}^2 \backslash (B (\tilde{d}_c \overrightarrow{e_1},
    D) \cup B (- \tilde{d}_c \overrightarrow{e_1}, D))} | \nabla \psi |^2 |
    Q_c |^2 + 2\mathfrak{R}\mathfrak{e}^2 (\psi) | Q_c |^4\\
    & + & \int_{\mathbbm{R}^2 \backslash (B (\tilde{d}_c
    \overrightarrow{e_1}, D) \cup B (- \tilde{d}_c \overrightarrow{e_1}, D))}
    4\mathfrak{I}\mathfrak{m} (\nabla Q_c \overline{Q_c})
    .\mathfrak{I}\mathfrak{m} (\nabla \psi) \mathfrak{R}\mathfrak{e} (\psi) +
    2 c | Q_c |^2 \mathfrak{I}\mathfrak{m} (\partial_{x_2} \psi)
    \mathfrak{R}\mathfrak{e} (\psi)\\
    & \geqslant & K \int_{\mathbbm{R}^2 \backslash (B (\tilde{d}_c
    \overrightarrow{e_1}, D) \cup B (- \tilde{d}_c \overrightarrow{e_1}, D))}
    | \nabla \psi |^2 | Q_c |^4 +\mathfrak{R}\mathfrak{e}^2 (\psi) | Q_c |^4 .
  \end{eqnarray*}
  We deduce that, for $\varphi = Q_c \psi \in C^{\infty}_c (\mathbbm{R}^2
  \backslash \{ \widetilde{d_c} \overrightarrow{e_1}, - \widetilde{d_c}
  \overrightarrow{e_1} \}, \mathbbm{C})$,
  \[ B_{Q_c} (\varphi) \geqslant K \| \varphi \|^2_{\mathcal{C}} \]
  if
  \[ \int_{B (\tilde{d}_c \overrightarrow{e_1}, R)} \mathfrak{R}\mathfrak{e}
     \left( \partial_{x_1} \tilde{V}_1 \overline{_{} \widetilde{V_1} \psi}
     \right) = \int_{B (\tilde{d}_c \overrightarrow{e_1}, R)}
     \mathfrak{R}\mathfrak{e} \left( \partial_{x_2} \widetilde{V_1}
     \overline{\widetilde{V_1} \psi} \right) = 0, \]
  \[ \int_{B (- \tilde{d}_c \overrightarrow{e_1}, R)} \mathfrak{R}\mathfrak{e}
     (\partial_{x_1} \tilde{V}_{- 1} \overline{\tilde{V}_{- 1} \psi}) =
     \int_{B (- \tilde{d}_c \overrightarrow{e_1}, R)} \mathfrak{R}\mathfrak{e}
     (\partial_{x_2} \tilde{V}_{- 1} \overline{\tilde{V}_{- 1} \psi}) = 0. \]
  We argue by density to show this result in $H_{Q_c}$. From Lemma
  \ref{CP2farnormexist}, we know that $\| . \|_{\mathcal{C}}$ is continuous
  with respect to $\| . \|_{H_{Q_c}}$. Furthermore, we recall from Lemma
  \ref{CP2L2111706}, that
  \[ \int_{B (\tilde{d}_c \overrightarrow{e_1}, R)} \left|
     \mathfrak{R}\mathfrak{e} \left( \partial_{x_1} \widetilde{V_1}
     \overline{\widetilde{V_1} \psi} \right) \right| \leqslant K (c) \|
     \varphi \|_{H_{Q_c}}, \]
  and similar estimates hold for
  \[ \int_{B (\tilde{d}_c \overrightarrow{e_1}, R)} \mathfrak{R}\mathfrak{e}
     \left( \partial_{x_2} \widetilde{V_1} \overline{\widetilde{V_1} \psi}
     \right), \int_{B (- \widetilde{d_c} \overrightarrow{e_1}, R)}
     \mathfrak{R}\mathfrak{e} \left( \partial_{x_1} \widetilde{V_{}}_{- 1}
     \overline{\widetilde{V_{}}_{- 1} \psi} \right) \]
  and
  \begin{equation}
    \int_{B (- \tilde{d}_c \overrightarrow{e_1}, R)} \mathfrak{R}\mathfrak{e}
    (\partial_{x_2} \tilde{V}_{- 1} \overline{\tilde{V}_{- 1} \psi})
    \label{CP22262505} .
  \end{equation}
  In particular, we check that these quantities are continuous for the norm
  $\| . \|_{H_{Q_c}}$, and that we can pass to the limit by density in these
  quantities by Lemma \ref{CP2Ndensity}.
  
  We are left with the passage to the limit for the quadratic form. For
  $\varphi \in H_{Q_c}$, we recall from (\ref{CP2truebqc}) that
  \[ \begin{array}{lll}
       B_{Q_c} (\varphi) & = & \int_{\mathbbm{R}^2} | \nabla \varphi |^2 - (1
       - | Q_c |^2) | \varphi |^2 + 2\mathfrak{R}\mathfrak{e}^2
       (\overline{Q_c} \varphi)\\
       & + & c \int_{\mathbbm{R}^2} (1 - \eta) \mathfrak{R}\mathfrak{e} (i
       \partial_{x_2} \varphi \bar{\varphi}) + c \int_{\mathbbm{R}^2
       \nosymbol} \eta \mathfrak{R}\mathfrak{e} (i \partial_{x_2} Q_c
       \overline{Q_c}) | \psi |^2\\
       & - & 2 c \int_{\mathbbm{R}^2} \eta \mathfrak{R}\mathfrak{e} \psi
       \mathfrak{I}\mathfrak{m} \partial_{x_2} \psi | Q_c |^2 - c
       \int_{\mathbbm{R}^2} \partial_{x_2} \eta \mathfrak{R}\mathfrak{e} \psi
       \mathfrak{I}\mathfrak{m} \psi | Q_c |^2\\
       & - & c \int_{\mathbbm{R}^2} \eta \mathfrak{R}\mathfrak{e} \psi
       \mathfrak{I}\mathfrak{m} \psi \partial_{x_2} (| Q_c |^2) .
     \end{array} \]
  Following the proof of Lemma \ref{CP2finitebilinear}, we check easily that,
  for $\varphi_1 = Q_c \psi_1, \varphi_2 = Q_c \psi_2 \in H_{Q_c}$, we have
  \begin{eqnarray*}
    &  & \int_{\mathbbm{R}^2} | \nabla \varphi_1 \overline{\nabla \varphi_2}
    | + | (1 - | Q_c |^2) \varphi_1 \overline{\varphi_2} | + |
    \mathfrak{R}\mathfrak{e} (\overline{Q_c} \varphi_1)
    \mathfrak{R}\mathfrak{e} (\overline{Q_c} \varphi_2) |\\
    & + & \int_{\mathbbm{R}^2} (1 - \eta) | \mathfrak{R}\mathfrak{e} (i
    \partial_{x_2} \varphi_1 \overline{\varphi_2}) | + \int_{\mathbbm{R}^2
    \nosymbol} \eta | \mathfrak{R}\mathfrak{e} (i \partial_{x_2} Q_c
    \overline{Q_c}) | | \psi_1 \psi_2 |\\
    & + & \int_{\mathbbm{R}^2} \eta | \mathfrak{R}\mathfrak{e} \psi_1
    \mathfrak{I}\mathfrak{m} \partial_{x_2} \psi_2 | | Q_c |^2 +
    \int_{\mathbbm{R}^2} | \partial_{x_2} \eta \mathfrak{R}\mathfrak{e} \psi_1
    \mathfrak{I}\mathfrak{m} \psi_2 | | Q_c |^2\\
    & + & \int_{\mathbbm{R}^2} \eta | \mathfrak{R}\mathfrak{e} \psi_1
    \mathfrak{I}\mathfrak{m} \psi_2 \partial_{x_2} (| Q_c |^2) |\\
    & \leqslant & K (c) \| \varphi_1 \|_{H_{Q_c}} \| \varphi_2 \|_{H_{Q_c}},
  \end{eqnarray*}
  and thus we can pass at the limit in $B_{Q_c}$ by Lemma \ref{CP2Ndensity}.
  This concludes the proof of Proposition \ref{CP205218}.
\end{proof}

\section{Proof of Theorem \ref{CP2th2} and its corollaries}\label{CP2s32}

\subsection{Link between the sets of orthogononality conditions}

The first goal of this subsection is to show that the four particular
directions ($\partial_{x_1} Q_c, \partial_{x_2} Q_c, c^2 \partial_c Q_c, c
\partial_{c^{\bot}} Q_c$) are almost orthogonal between them near the zeros of
$Q_c$, and that they can replace the four orthogonality conditions of
Proposition \ref{CP205218}. This is computed in the following lemma.

\begin{lemma}
  \label{CP2133L35}For $R > 0$ given by Proposition \ref{CP205218}, there
  exist $K_1, K_2 > 0$, two constants independent of $c$, such that, for $Q_c$
  defined in Theorem \ref{th1},
  \[ K_1 \leqslant \int_{B (\pm \tilde{d}_c \overrightarrow{e_1}, R)} |
     \partial_{x_1} Q_c |^2 + \int_{B (\pm \tilde{d}_c \overrightarrow{e_1},
     R)} | \partial_{x_2} Q_c |^2 + \int_{B (\pm \tilde{d}_c
     \overrightarrow{e_1}, R)} | c^2 \partial_c Q_c |^2 + \int_{B (\pm
     \widetilde{d_c} \overrightarrow{e_1}, R)} | c \partial_{c^{\bot}} Q_c |^2
     \leqslant K_2 . \]
  Furthermore, for $A, B \in \{ \partial_{x_1} Q_c, \partial_{x_2} Q_c, c^2
  \partial_c Q_c, c \partial_{c^{\bot}} Q_c \}$, $A \neq B$, we have that, for
  $1 > \beta_0 > 0$ a small constant,
  \[ \int_{B (\tilde{d}_c \overrightarrow{e_1}, R) \cup B (- \tilde{d}_c
     \overrightarrow{e_1}, R)} \mathfrak{R}\mathfrak{e} (A \bar{B}) = o_{c
     \rightarrow 0} (c^{\beta_0}) . \]
\end{lemma}

\begin{proof}
  From Lemma \ref{CP283L33}, we have, in $B (\pm \widetilde{d_c}
  \overrightarrow{e_1}, R)$, that (for $0 < \sigma = 1 - \beta_0 < 1$)
  \[ Q_c (x) = V_1 (x - d_c \overrightarrow{e_1}) V_{- 1} (x + d_c
     \overrightarrow{e_1}) + o_{c \rightarrow 0} (c^{\beta_0}) \]
  and
  \[ \nabla Q_c (x) = \nabla (V_1 (x - d_c \overrightarrow{e_1}) V_{- 1} (x +
     d_c \overrightarrow{e_1})) + o_{c \rightarrow 0} (c^{\beta_0}) . \]
  In this proof a $o_{c \rightarrow 0} (c^{\beta_0})$ may depend on $R$, but
  we consider $R$ as a universal constant. From Lemmas \ref{lemme3new} and
  \ref{CP2zerosofQc} and equation (\ref{CP2218}), we show that, by the mean
  value theorem, in $B (\pm \widetilde{d_c} \overrightarrow{e_1}, R)$,
  \begin{equation}
    Q_c = V_1 V_{- 1} + o_{c \rightarrow 0} (c^{\beta_0}) = V_{\pm 1} + o_{c
    \rightarrow 0} (c^{\beta_0}) = \tilde{V}_{\pm 1} + o_{c \rightarrow 0}
    (c^{\beta_0}) \label{CP2073221}
  \end{equation}
  and, similarly,
  \begin{equation}
    \nabla Q_c = \nabla \tilde{V}_{\pm 1} + o_{c \rightarrow 0} (c^{\beta_0})
    . \label{CP2073222}
  \end{equation}
  Thus, in $B (\pm \widetilde{d_c} \overrightarrow{e_1}, R)$, we have
  \begin{equation}
    \partial_{x_1} Q_c = \partial_{x_1} \tilde{V}_{\pm 1} + o_{c \rightarrow
    0} (c^{\beta_0}) \label{CP273221}
  \end{equation}
  and
  \begin{equation}
    \partial_{x_2} Q_c = \partial_{x_2} \tilde{V}_{\pm 1} + o_{c \rightarrow
    0} (c^{\beta_0}) . \label{CP273222}
  \end{equation}
  Furthermore, by Lemma \ref{CP2dcQcsigma}, we have in particular that in $B
  (\pm \widetilde{d_c} \overrightarrow{e_1}, R)$,
  \[ c^2 \partial_c Q_c = (1 + o_{c \rightarrow 0} (c^{\beta_0})) \partial_d
     (V_1 (x - d \overrightarrow{e_1}) V_{- 1} (x + d
     \overrightarrow{e_1}))_{| d = d_c \nobracket} + o_{c \rightarrow 0}
     (c^{\beta_0}) . \]
  Thus, in $B (\pm \widetilde{d_c} \overrightarrow{e_1}, R)$, with Lemmas
  \ref{lemme3new} and \ref{CP2zerosofQc}, we estimate
  \begin{equation}
    c^2 \partial_c Q_c = \mp \partial_{x_1} \tilde{V}_{\pm 1} + o_{c
    \rightarrow 0} (c^{\beta_0}) . \label{CP273223}
  \end{equation}
  Finally, from Lemma \ref{CP2a}, we have
  \[ c \partial_{c^{\bot}} Q_c = - c x^{\bot} . \nabla Q_c \]
  with $x^{\bot} = (- x_2, x_1)$. In $B (\pm \widetilde{d_c}
  \overrightarrow{e_1}, R)$, we have, since $c \widetilde{d_c} = 1 + o_{c
  \rightarrow 0} (c^{\beta_0})$ and Lemma \ref{CP2zerosofQc},
  \[ c x^{\bot} = \mp \overrightarrow{e_2} + o_{c \rightarrow 0} (c^{\beta_0})
     . \]
  Therefore, in $B (\pm \widetilde{d_c} \overrightarrow{e_1}, R)$, we have
  \begin{equation}
    c \partial_{c^{\bot}} Q_c = \pm \partial_{x_2} \tilde{V}_{\pm 1} + o_{c
    \rightarrow 0} (c^{\beta_0}) . \label{CP273224}
  \end{equation}
  Now, from Lemma \ref{lemme3new}, we have
  \begin{equation}
    K_1 \leqslant \int_{B (\pm \tilde{d}_c \overrightarrow{e_1}, R)} |
    \partial_{x_1} \tilde{V}_{\pm 1} |^2 + \int_{B (\pm \tilde{d}_c
    \overrightarrow{e_1}, R)} | \partial_{x_2} \tilde{V}_{\pm 1} |^2 \leqslant
    K_2 \label{CP273225}
  \end{equation}
  for universal constant $K_1, K_2 > 0$ (depending only on $R$). By a change
  of variable, we have, writing $\tilde{V}_{\pm 1} = \rho (\tilde{r}_{\pm 1})
  e^{i \tilde{\theta}_{\pm 1}}$ (with the notations of Lemma \ref{lemme3new}),
  \begin{equation}
    \partial_{x_1} \tilde{V}_{\pm 1} = \left( \cos (\tilde{\theta}_{\pm 1})
    \frac{\rho' (\tilde{r}_{\pm 1})}{\rho (\tilde{r}_{\pm 1})} - \frac{\pm
    i}{\tilde{r}_{\pm 1}} \sin (\tilde{\theta}_{\pm 1}) \right) \tilde{V}_{\pm
    1} \label{CP2073228}
  \end{equation}
  and
  \begin{equation}
    \partial_{x_2} \tilde{V}_{\pm 1} = \left( \sin (\tilde{\theta}_{\pm 1})
    \frac{\rho' (\tilde{r}_{\pm 1})}{\rho (\tilde{r}_{\pm 1})} + \frac{\pm
    i}{\tilde{r}_{\pm 1}} \cos (\tilde{\theta}_{\pm 1}) \right) \tilde{V}_{\pm
    1} .
  \end{equation}
  Since
  \[ \mathfrak{R}\mathfrak{e} \left( \partial_{x_1} \tilde{V}_{\pm 1}
     \overline{\partial_{x_2} \tilde{V}_{\pm 1}} \right) = 2 \cos
     (\tilde{\theta}_{\pm 1}) \sin (\tilde{\theta}_{\pm 1}) \frac{\rho'
     (\tilde{r}_{\pm 1})}{\tilde{r}_{\pm 1} \rho (\tilde{r}_{\pm 1})} |
     \tilde{V}_{\pm 1} |^2, \]
  by integration in polar coordinates, we have
  \begin{equation}
    \int_{B (\pm \tilde{d}_c \overrightarrow{e_1}, R)}
    \mathfrak{R}\mathfrak{e} \left( \partial_{x_1} \tilde{V}_{\pm 1}
    \overline{\partial_{x_2} \tilde{V}_{\pm 1}} \right) = 0. \label{CP273226}
  \end{equation}
  Combining (\ref{CP273221}) to (\ref{CP273224}) with (\ref{CP273225}) and
  (\ref{CP273226}), we can do every estimate stated in the lemma.
\end{proof}

With (\ref{CP273221}) to (\ref{CP273224}), we check that these four directions
are close to the ones in the orthogonality conditions of Proposition
\ref{CP205218}. This will appear in the proof of Lemma \ref{CP2133L39}. Now,
we give a way to develop the quadratic form for some particular functions.

\

\begin{lemma}
  \label{CP2icicmieux}For $\varphi \in C^{\infty}_c (\mathbbm{R}^2 \backslash
  \{ \widetilde{d_c} \overrightarrow{e_1}, - \widetilde{d_c}
  \overrightarrow{e_1} \}, \mathbbm{C})$ and $A \in \tmop{Span} \{
  \partial_{x_1} Q_c, \partial_{x_2} Q_c, \partial_c Q_c, \partial_{c^{\bot}}
  Q_c \}$, we have
  \[ \langle L_{Q_c} (\varphi + A), \varphi + A \rangle = \langle L_{Q_c}
     (\varphi), \varphi \rangle + \langle 2 L_{Q_c} (A), \varphi \rangle +
     \langle L_{Q_c} (A), A \rangle . \]
  Furthermore, $\langle L_{Q_c} (\varphi + A), \varphi + A \rangle = B_{Q_c}
  (\varphi + A)$ and $\langle L_{Q_c} (A), A \rangle = B_{Q_c} (A)$.
\end{lemma}

\begin{proof}
  Since $\varphi \in C^{\infty}_c (\mathbbm{R}^2 \backslash \{ \widetilde{d_c}
  \overrightarrow{e_1}, - \widetilde{d_c} \overrightarrow{e_1} \},
  \mathbbm{C})$, it is enough to check that $\mathfrak{R}\mathfrak{e} (L_{Q_c}
  (A) \bar{A}) \in L^1 (\mathbbm{R}^2, \mathbbm{R})$ for $A \in \tmop{Span} \{
  \partial_{x_1} Q_c, \partial_{x_2} Q_c, \partial_c Q_c, \partial_{c^{\bot}}
  Q_c \}$ to show that
  \[ \langle L_{Q_c} (\varphi + A), \varphi + A \rangle = \langle L_{Q_c}
     (\varphi), \varphi \rangle + \langle 2 L_{Q_c} (A), \varphi \rangle +
     \langle L_{Q_c} (A), A \rangle . \]
  From Lemma \ref{CP20703L222}, we have, for $A = \mu_1 \partial_{x_1} Q_c +
  \mu_2 \partial_{x_2} Q_c + \mu_3 \partial_c Q_c + \mu_4 \partial_{c^{\bot}}
  Q_c$, that
  \[ L_{Q_c} (A) = \mu_3 i \partial_{x_2} Q_c - \mu_4 i \partial_{x_1} Q_c .
  \]
  Now, with (\ref{CP2N216}) (that holds also for $A$ by linearity) and
  (\ref{CP2220}), (\ref{CP2221}), we check easily that
  $\mathfrak{R}\mathfrak{e} (L_{Q_c} (A) \bar{A}) \in L^1 (\mathbbm{R}^2,
  \mathbbm{R})$.
  
  Now, from subsection \ref{CP22241103}, to show that for $\Phi = Q_c \Psi \in
  H_{Q_c} \cap C^2 (\mathbbm{R}^2, \mathbbm{C})$, we have \ $\langle L_{Q_c}
  (\Phi), \Phi \rangle = B_{Q_c} (\Phi)$, it is enough to show that
  $\int_{\mathbbm{R}^2} \partial_{x_2} (\eta \mathfrak{R}\mathfrak{e} \Psi
  \mathfrak{I}\mathfrak{m} \Psi | Q_c |^2)$ is well defined and is $0$. For
  $\Phi = A$ or $\Phi = \varphi + A$, this is a consequence of
  (\ref{CP2N216}), Lemma \ref{CP2N216} and $\varphi \in C^{\infty}_c
  (\mathbbm{R}^2 \backslash \{ \widetilde{d_c} \overrightarrow{e_1}, -
  \widetilde{d_c} \overrightarrow{e_1} \}, \mathbbm{C})$.
\end{proof}

\subsection{Some useful elliptic estimates}

We want to improve slightly the coercivity norm near the zeros of $Q_c$. This
is done in the following lemma. The improvement is in the exponent of the
weight in front of $f^2$.

\begin{lemma}
  \label{CP20503}There exists a universal constant $K > 0$ such that, for any
  $D > 2$, for $V_1$ centered at $0$ and any function $f \in C^{\infty}_c
  (\mathbbm{R}^2 \backslash \{ \tilde{d}_c \vec{e}_1, - \tilde{d}_c \vec{e}_1
  \}, \mathbbm{R})$, we have
  \[ \int_{B (\pm \tilde{d}_c \vec{e}_1, D)} f^2 | V_1 |^3 d x \leqslant K
     \int_{B (\pm \tilde{d}_c \vec{e}_1, D)} | \nabla f |^2 | V_1 |^4 + f^2 |
     V_1 |^4 d x. \]
  In particular, this implies that, for $\psi \in C^{\infty}_c (\mathbbm{R}^2
  \backslash \{ 0 \}, \mathbbm{C})$,
  \[ \int_{B (0, D)} \mathfrak{R}\mathfrak{e}^2 (\psi) | V_1 |^3 d x \leqslant
     K \int_{B (0, D)} | \nabla \psi |^2 | V_1 |^4 +\mathfrak{R}\mathfrak{e}^2
     (\psi) | V_1 |^4 d x. \]
\end{lemma}

This lemma, with Lemmas \ref{CP2closecall} and \ref{CP2Ndensity}, implies
that, for $\varphi = Q_c \psi \in H_{Q_c}$,
\begin{equation}
  \int_{\mathbbm{R}^2} \mathfrak{R}\mathfrak{e}^2 (\psi) | Q_c |^3 \leqslant K
  \| \varphi \|_{\mathcal{C}}^2 . \label{CP22142405}
\end{equation}
\begin{proof}
  Since $| V_1 | \geqslant K > 0$ outside of $B (0, 1)$, we take $\chi$ a
  radial smooth non negative cutoff with value $0$ in $B (0, 1)$ and value $1$
  outside $B (0, 3 / 2)$. We have
  \[ \int_{B (0, D)} \chi f^2 | V_1 |^3 d x \leqslant K \int_{B (0, D)} \chi
     f^2 | V_1 |^4 d x \leqslant K \int_{B (0, D)} f^2 | V_1 |^4 d x. \]
  In $B (0, 2)$, from Lemma \ref{lemme3new}, there exists $K_1, K_2 > 0$ such
  that $K_1 \geqslant \frac{| V_1 |}{r} \geqslant K_2$, and thus
  \[ \int_{B (0, D)} (1 - \chi) f^2 | V_1 |^3 d x \leqslant K \left( \int_0^{2
     \pi} \int_0^2 (1 - \chi (r)) f^2 (x) r^4 d r \right) d \theta . \]
  For $g \in C^{\infty}_c (\mathbbm{R} \backslash \{ 0 \}, \mathbbm{R})$, we
  have
  \begin{eqnarray*}
    \int_0^2 (1 - \chi (r)) g^2 (r) r^4 d r & = & \frac{- 1}{5} \int_0^2
    \partial_r ((1 - \chi) g^2) r^5 d r\\
    & = & \frac{- 2}{5} \int_0^2 (1 - \chi (r)) \partial_r g (r) g (r) r^5 d
    r + \frac{1}{4} \int_0^2 \chi' (r) g^2 (r) r^5 d r,
  \end{eqnarray*}
  and since $\chi' (r) \neq 0$ only for $r \in [1, 2]$, we have
  \[ \int_0^2 | \chi' (r) | g^2 (r) r^5 d r \leqslant K \int_0^2 g^2 (r) r^4 d
     r, \]
  and, by Cauchy-Schwarz,
  \[ \int_0^2 (1 - \chi (r)) | \partial_r g (r) g (r) | r^5 d r \leqslant
     \sqrt{\int_0^2 (\partial_r g)^2 r^5 d r \int_0^2 g^2 (r) r^5 d r} . \]
  We deduce that
  \[ \int_0^2 (1 - \chi (r)) g^2 (r) r^4 d r \leqslant K \left( \int_0^2
     (\partial_r g)^2 r^5 d r + \int_0^2 g^2 (r) r^5 d r \right), \]
  and taking, for any $\theta \in [0, 2 \pi]$, $g (r) = f (r \cos (\theta), r
  \sin (\theta))$, and since $r \leqslant K | V_1 |$ in $B (0, 2)$ (by Lemma
  \ref{lemme3new}), by integration with respect to $\theta$, we conclude that
  \[ \int_{B (0, D)} (1 - \chi) f^2 | V_1 |^3 d x \leqslant K \int_{B (0, D)}
     | \nabla f |^2 | V_1 |^4 + f^2 | V_1 |^4 d x, \]
  which ends the proof of this lemma.
\end{proof}

We estimate here some quantities with the coercivity norm. These computations
will be useful later on.

\begin{lemma}
  \label{CP2L390911}There exists $K > 0$, a universal constant independent of
  $c$, such that, if $c$ is small enough, for $Q_c$ defined in Theorem
  \ref{th1}, for $\varphi = Q_c \psi \in C^{\infty}_c (\mathbbm{R}^2
  \backslash \{ \widetilde{d_c} \overrightarrow{e_1}, - \widetilde{d_c}
  \overrightarrow{e_1} \}, \mathbbm{C})$, we have

  \[ \left| \int_{\mathbbm{R}^2} \mathfrak{R}\mathfrak{e} (\psi)
     \mathfrak{I}\mathfrak{m} (\nabla Q_c \overline{Q_c}) \right| \leqslant K
     \ln \left( \frac{1}{c} \right) \| \varphi \|_{\mathcal{C}} \]
  and
  \[ \left| \int_{\mathbbm{R}^2} \mathfrak{I}\mathfrak{m} (\psi)
     \mathfrak{R}\mathfrak{e} (\nabla Q_c \overline{Q_c}) \right| \leqslant K
     \| \varphi \|_{\mathcal{C}} . \]
\end{lemma}

\begin{proof}
  By Cauchy-Schwarz, Lemmas \ref{CP293L37} (with a slight modification near
  the zeros of $Q_c$ that does not change the result) and \ref{CP20503},
  \begin{eqnarray*}
    \left| \int_{\mathbbm{R}^2} \mathfrak{R}\mathfrak{e} (\psi)
    \mathfrak{I}\mathfrak{m} (\nabla Q_c \overline{Q_c}) \right| & \leqslant &
    \sqrt{\int_{\mathbbm{R}^2} \mathfrak{R}\mathfrak{e}^2 (\psi) | Q_c |^3
    \int_{\mathbbm{R}^2} \frac{| \mathfrak{I}\mathfrak{m} (\nabla Q_c
    \overline{Q_c}) |^2}{| Q_c |^3}}\\
    & \leqslant & K \ln \left( \frac{1}{c} \right) \sqrt{\int_{\mathbbm{R}^2}
    \mathfrak{R}\mathfrak{e}^2 (\psi) | Q_c |^3}\\
    & \leqslant & K \ln \left( \frac{1}{c} \right) \| \varphi
    \|_{\mathcal{C}} .
  \end{eqnarray*}
  We now focus on the second estimate. We take $\chi$ a smooth function with
  value $1$ outside of $\{ \tilde{r} \geqslant 2 \}$ and $0$ inside $\{
  \tilde{r} \leqslant 1 \}$, and that is radial around $\pm \tilde{d}_c
  \overrightarrow{e_1}$ in $B (\pm \tilde{d}_c \overrightarrow{e_1}, 2)$. We
  remark that
  \[ \mathfrak{R}\mathfrak{e} (\nabla Q_c \overline{Q_c}) = \frac{1}{2} \nabla
     (| Q_c |^2) = \frac{1}{2} \nabla (\chi (| Q_c |^2 - 1) + (1 - \chi) | Q_c
     |^2) + \frac{1}{2} \nabla \chi, \]
  thus, by integration by parts, we have
  \begin{eqnarray*}
    \int_{\mathbbm{R}^2} \mathfrak{I}\mathfrak{m} (\psi)
    \mathfrak{R}\mathfrak{e} (\nabla Q_c \overline{Q_c}) & = & \frac{1}{2}
    \int_{\mathbbm{R}^2} \mathfrak{I}\mathfrak{m} (\psi) \nabla (\chi (| Q_c
    |^2 - 1) + (1 - \chi) | Q_c |^2) + \frac{1}{2} \int_{\mathbbm{R}^2} \nabla
    \chi \mathfrak{I}\mathfrak{m} (\psi)\\
    & = & \frac{- 1}{2} \int_{\mathbbm{R}^2} \mathfrak{I}\mathfrak{m} (\nabla
    \psi) \chi (| Q_c |^2 - 1) - \frac{1}{2} \int_{\mathbbm{R}^2}
    \mathfrak{I}\mathfrak{m} (\nabla \psi) (1 - \chi) | Q_c |^2\\
    & + & \frac{1}{2} \int_{\mathbbm{R}^2} \nabla \chi
    \mathfrak{I}\mathfrak{m} (\psi) .
  \end{eqnarray*}
  and, since $\chi$ is radial around $\pm \tilde{d}_c \overrightarrow{e_1}$ in
  $B (\pm \tilde{d}_c \overrightarrow{e_1}, 2)$,
  \[ \int_{\mathbbm{R}^2} \mathfrak{I}\mathfrak{m} (\psi) \nabla \chi =
     \int_{B (\tilde{d}_c \overrightarrow{e_1}, 2) \cup B (- \tilde{d}_c
     \overrightarrow{e_1}, 2)} \mathfrak{I}\mathfrak{m} (\psi^{\neq 0}) \nabla
     \chi . \]
  Since $\nabla \chi$ is supported in $(B (\tilde{d}_c \overrightarrow{e_1},
  2) \cup B (- \tilde{d}_c \overrightarrow{e_1}, 2)) \backslash (B
  (\tilde{d}_c \overrightarrow{e_1}, 1) \cup B (- \tilde{d}_c
  \overrightarrow{e_1}, 1))$, by equations (\ref{CP2Qcpaszero}),
  (\ref{CP2jesaispasfairedesmaths}) and Cauchy-Schwarz,
  \[ \left| \int_{B (\tilde{d}_c \overrightarrow{e_1}, 2) \cup B (-
     \tilde{d}_c \overrightarrow{e_1}, 2)} \mathfrak{I}\mathfrak{m}
     (\psi^{\neq 0}) \nabla \chi \right| \leqslant K
     \sqrt{\int_{\mathbbm{R}^2} | \nabla \psi |^2 | Q_c |^4} . \]
  Now, by Cauchy-Schwarz, we check that
  \[ \left| \int_{\mathbbm{R}^2} \mathfrak{I}\mathfrak{m} (\nabla \psi) (1 -
     \chi) | Q_c |^2 \right| \leqslant K \sqrt{\int_{\mathbbm{R}^2} | \nabla
     \psi |^2 | Q_c |^4 \int_{\mathbbm{R}^2} (1 - \chi)^2} \leqslant K
     \sqrt{\int_{\mathbbm{R}^2} | \nabla \psi |^2 | Q_c |^4} . \]
  Furthermore, we check that ($\chi$ being supported in $\{ \tilde{r}
  \geqslant 1 \}$)
  \begin{eqnarray*}
    \left| \int_{\mathbbm{R}^2} \mathfrak{I}\mathfrak{m} (\nabla \psi)
    \nobracket \chi (| Q_c |^2 - 1 \nobracket) \right| & \leqslant &
    \sqrt{\int_{\mathbbm{R}^2} | \nabla \psi |^2 \chi \int_{\mathbbm{R}^2} (|
    Q_c |^2 - 1)^2}\\
    & \leqslant & K \sqrt{\int_{\mathbbm{R}^2} | \nabla \psi |^2 | Q_c |^4} .
  \end{eqnarray*}
  Indeed, we have, from equation (\ref{CP2217}) (for $\sigma = 1 / 2$), that
  \[ | | Q_c |^2 - 1 | \leqslant \frac{K}{(1 + \tilde{r})^{3 / 2}}, \]
  which is enough to show that
  \[ \int_{\mathbbm{R}^2} (| Q_c |^2 - 1)^2 \leqslant K. \]
  Combining these estimates, we conclude the proof of
  \[ \left| \int_{\mathbbm{R}^2} \mathfrak{I}\mathfrak{m} (\psi)
     \mathfrak{R}\mathfrak{e} (\nabla Q_c \overline{Q_c}) \right| \leqslant K
     \sqrt{\int_{\mathbbm{R}^2} | \nabla \psi |^2 | Q_c |^4} \leqslant K \|
     \varphi \|_{\mathcal{C}} . \]
\end{proof}

\subsection{Coercivity result under four othogonality conditions}

The next result is the first part of Theorem \ref{CP2th2}, the second part
(for the coercivity under three orthogonalities) is done in Lemma
\ref{CP2L73224} below. We recall that, in $B (\pm \tilde{d}_c
\overrightarrow{e_1}, R)$, we have $\psi^{\neq 0} (x) = \psi (x) - \psi^{0,
\pm 1} (\tilde{r}_{\pm 1})$ with $\psi^{0, \pm 1} (\tilde{r}_{\pm 1})$ the
$0$-harmonic centered around $\pm \tilde{d}_c \overrightarrow{e_1}$ of $\psi$.

\begin{lemma}
  \label{CP2133L39}There exist $R, K, c_0 > 0$ such that, for $0 < c \leqslant
  c_0$ and $\varphi = Q_c \psi \in H_{Q_c}$, $Q_c$ defined in Theorem
  \ref{th1}, if
  \[ \mathfrak{R}\mathfrak{e} \int_{B (\tilde{d}_c \overrightarrow{e_1}, R)
     \cup B (- \tilde{d}_c \overrightarrow{e_1}, R)} \partial_{x_1} Q_c
     \overline{Q_c \psi^{\neq 0}} =\mathfrak{R}\mathfrak{e} \int_{B
     (\tilde{d}_c \overrightarrow{e_1}, R) \cup B (- \tilde{d}_c
     \overrightarrow{e_1}, R)} \partial_{x_2} Q_c \overline{Q_c \psi^{\neq 0}}
     = 0, \]
  \[ \mathfrak{R}\mathfrak{e} \int_{B (\tilde{d}_c \overrightarrow{e_1}, R)
     \cup B (- \tilde{d}_c \overrightarrow{e_1}, R)} \partial_c Q_c
     \overline{Q_c \psi^{\neq 0}} =\mathfrak{R}\mathfrak{e} \int_{B
     (\tilde{d}_c \overrightarrow{e_1}, R) \cup B (- \tilde{d}_c
     \overrightarrow{e_1}, R)} \partial_{c^{\bot}} Q_c \overline{Q_c
     \psi^{\neq 0}} = 0, \]
  then
  \[ B_{Q_c} (\varphi) \geqslant K \| \varphi \|^2_{\mathcal{C}} . \]
\end{lemma}

\begin{proof}
  For $\varphi = Q_c \psi \in C^{\infty}_c (\mathbbm{R}^2 \backslash \{
  \widetilde{d_c} \overrightarrow{e_1}, - \widetilde{d_c} \overrightarrow{e_1}
  \}, \mathbbm{C})$, we take $\varepsilon_1, \varepsilon_2, \varepsilon_3,
  \varepsilon_4$ four real parameters and we define
  \[ \psi^{\ast} \assign \psi + \varepsilon_1 \frac{\partial_{x_1} Q_c}{Q_c} +
     \varepsilon_2 \frac{c^2 \partial_c Q_c}{Q_c} + \varepsilon_3
     \frac{\partial_{x_2} Q_c}{Q_c} + \varepsilon_4 \frac{c
     \partial_{c^{\bot}} Q_c}{Q_c} . \]
  Since, by Lemma \ref{CP20703L222}, $\partial_{x_1} Q_c, \partial_{x_2} Q_c,
  \partial_c Q_c, \partial_{c^{\bot}} Q_c \in H_{Q_c}$, we deduce that $Q_c
  \psi^{\ast} \in H_{Q_c}$. Furthermore, we have
  \begin{eqnarray*}
    \int_{B (\tilde{d}_c \overrightarrow{e_1}, R)} \mathfrak{R}\mathfrak{e}
    \left( \partial_{x_1} \widetilde{V_1} \overline{\widetilde{V_1}
    \psi^{\ast}} \right) & = & \int_{B (\tilde{d}_c \overrightarrow{e_1}, R)}
    \mathfrak{R}\mathfrak{e} \left( \partial_{x_1} \widetilde{V_1}
    \overline{\widetilde{V_1} \psi} \right)\\
    & + & \varepsilon_1 \int_{B (\tilde{d}_c \overrightarrow{e_1}, R)}
    \mathfrak{R}\mathfrak{e} \left( \partial_{x_1} \widetilde{V_1}
    \overline{\partial_{x_1} Q_c \frac{\widetilde{V_1}}{Q_c}} \right)\\
    & + & \varepsilon_2 \int_{B (\tilde{d}_c \overrightarrow{e_1}, R)}
    \mathfrak{R}\mathfrak{e} \left( \partial_{x_1} \widetilde{V_1} c^2
    \overline{\partial_c Q_c \frac{\widetilde{V_1}}{Q_c}} \right)\\
    & + & \varepsilon_3 \int_{B (\tilde{d}_c \overrightarrow{e_1}, R)}
    \mathfrak{R}\mathfrak{e} \left( \partial_{x_1} \widetilde{V_1}
    \overline{\partial_{x_2} Q_c \frac{\widetilde{V_1}}{Q_c}} \right)\\
    & + & \varepsilon_4 \int_{B (\tilde{d}_c \overrightarrow{e_1}, R)}
    \mathfrak{R}\mathfrak{e} \left( \partial_{x_1} \widetilde{V_1} c
    \overline{\partial_{c^{\bot}} Q_c \frac{\widetilde{V_1}}{Q_c}} \right) .
  \end{eqnarray*}
  From (\ref{CP2073228}), we compute
  \[ \partial_{x_1} \widetilde{V_1} \overline{\widetilde{V_1}} = \left( \cos
     (\tilde{\theta}_1) \frac{\rho' (\tilde{r}_1)}{\rho (\tilde{r}_1)} -
     \frac{i}{\tilde{r}_1} \sin (\tilde{\theta}_1) \right) | \widetilde{V_1}
     |^2, \]
  and in particular, it has no $0$-harmonic (since $| \widetilde{V_1} |^2$ is
  radial). Therefore,
  \[ \int_{B (\tilde{d}_c \overrightarrow{e_1}, R)} \mathfrak{R}\mathfrak{e}
     \left( \partial_{x_1} \widetilde{V_1} \overline{\widetilde{V_1} \psi}
     \right) = \int_{B (\tilde{d}_c \overrightarrow{e_1}, R)}
     \mathfrak{R}\mathfrak{e} \left( \partial_{x_1} \widetilde{V_1}
     \overline{\widetilde{V_1} \psi^{\neq 0}} \right) = \]
  \[ \int_{B (\tilde{d}_c \overrightarrow{e_1}, R)} \mathfrak{R}\mathfrak{e}
     (\partial_{x_1} Q_c \overline{Q_c \psi^{\neq 0}}) + \int_{B (\tilde{d}_c
     \overrightarrow{e_1}, R)} \mathfrak{R}\mathfrak{e} \left( \left(
     \partial_{x_1} \widetilde{V_1} \overline{\widetilde{V_1}} -
     \partial_{x_1} Q_c \overline{Q_c} \right) \psi^{\neq 0} \right) . \]
  By Cauchy-Schwarz and equation (\ref{CP2jesaistoujourspasfairedesmaths}),
  \begin{equation}
    \int_{B (\tilde{d}_c \overrightarrow{e_1}, R) \cup B (- \tilde{d}_c
    \overrightarrow{e_1}, R)} | Q_c \psi^{\neq 0} |^2 \leqslant K \int_{B
    (\tilde{d}_c \overrightarrow{e_1}, R) \cup B (- \tilde{d}_c
    \overrightarrow{e_1}, R)} | Q_c |^4 | \nabla \psi |^2 \leqslant K \|
    \varphi \|^2_{\mathcal{C}} . \label{CP23120403}
  \end{equation}
  Here, $K$ depends on $R$, but we consider $R$ as a universal constant. We
  remark, by equations (\ref{CP273221}), (\ref{CP273223}) and
  (\ref{CP23120403}) that
  \begin{eqnarray*}
    &  & \frac{1}{2} \mathfrak{R}\mathfrak{e} \int_{B (\tilde{d}_c
    \overrightarrow{e_1}, R) \cup B (- \tilde{d}_c \overrightarrow{e_1}, R)}
    (\partial_{x_1} Q_c - c^2 \partial_c Q_c) \overline{Q_c \psi^{\neq 0}}\\
    & = & \int_{B (\tilde{d}_c \overrightarrow{e_1}, R)}
    \mathfrak{R}\mathfrak{e} (\partial_{x_1} Q_c \overline{Q_c \psi^{\neq 0}})
    + o_{c \rightarrow 0} (c^{\beta_0}) K \| \varphi \|^2_{\mathcal{C}},
  \end{eqnarray*}
  where $\beta_0 > 0$ is a small constant. We supposed that
  \[ \mathfrak{R}\mathfrak{e} \int_{B (\tilde{d}_c \overrightarrow{e_1}, R)
     \cup B (- \tilde{d}_c \overrightarrow{e_1}, R)} \partial_{x_1} Q_c
     \overline{Q_c \psi^{\neq 0}} =\mathfrak{R}\mathfrak{e} \int_{B
     (\tilde{d}_c \overrightarrow{e_1}, R) \cup B (- \tilde{d}_c
     \overrightarrow{e_1}, R)} \partial_c Q_c \overline{Q_c \psi^{\neq 0}} =
     0, \]
  therefore
  \[ \int_{B (\tilde{d}_c \overrightarrow{e_1}, R)} \mathfrak{R}\mathfrak{e}
     (\partial_{x_1} Q_c \overline{Q_c \psi^{\neq 0}}) = o_{c \rightarrow 0}
     (c^{\beta_0}) K \| \varphi \|^2_{\mathcal{C}} . \]
  Furthermore, by equations (\ref{CP2218}),
  (\ref{CP2jesaistoujourspasfairedesmaths}), (\ref{CP273221}), Lemma
  \ref{CP2closecall} and Cauchy-Schwarz,
  \begin{eqnarray*}
    \left| \int_{B (\tilde{d}_c \overrightarrow{e_1}, R)}
    \mathfrak{R}\mathfrak{e} \left( \left( \partial_{x_1} \widetilde{V_1}
    \overline{\widetilde{V_1}} - \partial_{x_1} Q_c \overline{Q_c} \right)
    \psi^{\neq 0} \right) \right| & \leqslant & o_{c \rightarrow 0}
    (c^{\beta_0}) \sqrt{\int_{B (\tilde{d}_c \overrightarrow{e_1}, R)} |
    \psi^{\neq 0} |^2 | Q_c |^2}\\
    & \leqslant & o_{c \rightarrow 0} (c^{\beta_0}) K \| \varphi
    \|_{\mathcal{C}}
  \end{eqnarray*}
  Now, from Lemma \ref{CP2closecall} and equation (\ref{CP273221}), we
  estimate
  \[ \int_{B (\tilde{d}_c \overrightarrow{e_1}, R)} \mathfrak{R}\mathfrak{e}
     \left( \partial_{x_1} \widetilde{V_1} \overline{\partial_{x_1} Q_c
     \frac{\widetilde{V_1}}{Q_c}} \right) = \int_{B (\tilde{d}_c
     \overrightarrow{e_1}, R)} | \partial_{x_1} \widetilde{V_1} |^2 + o_{c
     \rightarrow 0} (1) . \]
  With (\ref{CP273222}), we check
  \[ \int_{B (\tilde{d}_c \overrightarrow{e_1}, R)} \mathfrak{R}\mathfrak{e}
     \left( \partial_{x_1} \widetilde{V_1} \overline{\partial_{x_2} Q_c
     \frac{\widetilde{V_1}}{Q_c}} \right) = o_{c \rightarrow 0} (1) . \]
  Similarly, by (\ref{CP273223}) and Lemma \ref{CP2closecall}, we have
  \[ \int_{B (\tilde{d}_c \overrightarrow{e_1}, R)} \mathfrak{R}\mathfrak{e}
     \left( \partial_{x_1} \widetilde{V_1} \overline{c^2 \partial_c Q_c
     \frac{\widetilde{V_1}}{Q_c}} \right) = - \int_{B (\tilde{d}_c
     \overrightarrow{e_1}, R)} | \partial_{x_1} \widetilde{V_1} |^2 + o_{c
     \rightarrow 0} (1) \]
  and by (\ref{CP273224}), we have
  \[ \int_{B (\tilde{d}_c \overrightarrow{e_1}, R)} \mathfrak{R}\mathfrak{e}
     \left( \partial_{x_1} \widetilde{V_1} \overline{c \partial_{c^{\bot}} Q_c
     \frac{\widetilde{V_1}}{Q_c}} \right) = o_{c \rightarrow 0} (1) . \]
  Thus, with (\ref{CP273225}) we deduce that, writing
  \[ K (R) = \int_{B (0, R)} | \partial_{x_1} V_1 (x) |^2 d x, \]
  since
  \[ K (R) = \int_{B (\tilde{d}_c \overrightarrow{e_1}, R)} | \partial_{x_1}
     \widetilde{V_1} |^2 = \int_{B (- \tilde{d}_c \overrightarrow{e_1}, R)} |
     \partial_{x_1} \tilde{V}_{- 1} |^2 = \int_{B (\tilde{d}_c
     \overrightarrow{e_1}, R)} | \partial_{x_2} \widetilde{V_1} |^2 = \int_{B
     (- \tilde{d}_c \overrightarrow{e_1}, R)} | \partial_{x_2} \tilde{V}_{- 1}
     |^2, \]
  we have
  \begin{eqnarray*}
    &  & \int_{B (\tilde{d}_c \overrightarrow{e_1}, R)}
    \mathfrak{R}\mathfrak{e} \left( \partial_{x_1} \widetilde{V_1}
    \overline{\widetilde{V_1} \psi^{\ast}} \right)\\
    & = & (\varepsilon_1 - \varepsilon_2) K (R) + o^{}_{c \rightarrow 0} (1)
    (\varepsilon_1 + \varepsilon_2 + \varepsilon_3 + \varepsilon_4) + o_{c
    \rightarrow 0} (c^{\beta_0}) K \| \varphi \|_{\mathcal{C}}
  \end{eqnarray*}
  Similarly we can do the same computation for every orthogonalities, and we
  have the system
  \begin{eqnarray*}
    \left(\begin{array}{c}
      \int_{B (\tilde{d}_c \overrightarrow{e_1}, R)} \mathfrak{R}\mathfrak{e}
      \left( \partial_{x_1} \widetilde{V_1} \overline{\widetilde{V_1}
      \psi^{\ast}} \right)\\
      \int_{B (- \tilde{d}_c \overrightarrow{e_1}, R)}
      \mathfrak{R}\mathfrak{e} (\partial_{x_1} \tilde{V}_{- 1}
      \overline{\tilde{V}_{- 1} \psi^{\ast}})\\
      \int_{B (\tilde{d}_c \overrightarrow{e_1}, R)} \mathfrak{R}\mathfrak{e}
      \left( \partial_{x_2} \widetilde{V_1} \overline{\widetilde{V_1}
      \psi^{\ast}} \right)\\
      \int_{B (- \tilde{d}_c \overrightarrow{e_1}, R)}
      \mathfrak{R}\mathfrak{e} (\partial_{x_2} \tilde{V}_{- 1}
      \overline{\tilde{V}_{- 1} \psi^{\ast}})
    \end{array}\right) & = & \left( K (R) \left(\begin{array}{cccc}
      1 & - 1 & 0 & 0\\
      1 & 1 & 0 & 0\\
      0 & 0 & 1 & - 1\\
      0 & 0 & 1 & 1
    \end{array}\right) + o_{c \rightarrow 0} (1) \right)
    \left(\begin{array}{c}
      \varepsilon_1\\
      \varepsilon_2\\
      \varepsilon_3\\
      \varepsilon_4
    \end{array}\right)\\
    & + & o_{c \rightarrow 0} (c^{\beta_0}) K \| \varphi \|_{\mathcal{C}} .
  \end{eqnarray*}
  Therefore, since the matrix is invertible and $K (R) > 0$, for $c$ small
  enough, we can find $\varepsilon_1, \varepsilon_2, \varepsilon_3,
  \varepsilon_4 \in \mathbbm{R}$ such that
  \begin{equation}
    | \varepsilon_1 | + | \varepsilon_2 | + | \varepsilon_3 | + |
    \varepsilon_4 | \leqslant o_{c \rightarrow 0} (c^{\beta_0}) K \| \varphi
    \|_{\mathcal{C}} \label{CP273231}
  \end{equation}
  and
  \[ \int_{B (\tilde{d}_c \overrightarrow{e_1}, R)} \mathfrak{R}\mathfrak{e}
     \left( \partial_{x_1} \widetilde{V_1} \overline{\widetilde{V_1}
     \psi^{\ast}} \right) = \int_{B (\tilde{d}_c \overrightarrow{e_1}, R)}
     \mathfrak{R}\mathfrak{e} \left( \partial_{x_2} \widetilde{V_1}
     \overline{\widetilde{V_1} \psi^{\ast}} \right) = 0, \]
  \[ \int_{B (- \tilde{d}_c \overrightarrow{e_1}, R)} \mathfrak{R}\mathfrak{e}
     (\partial_{x_1} \tilde{V}_{- 1} \overline{\tilde{V}_{- 1} \psi^{\ast}}) =
     \int_{B (- \tilde{d}_c \overrightarrow{e_1}, R)} \mathfrak{R}\mathfrak{e}
     (\partial_{x_2} \tilde{V}_{- 1} \overline{\tilde{V}_{- 1} \psi^{\ast}}) =
     0. \]
  Therefore, by Proposition \ref{CP205218}, since $Q_c \psi^{\ast} \in
  H_{Q_c}$, we have
  \[ B_{Q_c} (Q_c \psi^{\ast}) \geqslant K \| Q_c \psi^{\ast}
     \|^2_{\mathcal{C}} . \]
  From Lemma \ref{CP283L223}, we have,
  \[ \| \partial_{x_1} Q_c \|_{\mathcal{C}} + \| \partial_{x_2} Q_c
     \|_{\mathcal{C}} + \| c^2 \partial_c Q_c \|_{\mathcal{C}} + c^{\beta_0 /
     2} \| c \partial_{c^{\bot}} Q_c \|_{\mathcal{C}} \leqslant K (\beta_0) \]
  hence, since $Q_c (\psi^{\ast} - \psi) = \varepsilon_1 \partial_{x_1} Q_c +
  \varepsilon_2 c^2 \partial_c Q_c + \varepsilon_3 \partial_{x_2} Q_c +
  \varepsilon_4 c \partial_{c^{\bot}} Q_c$,
  \begin{eqnarray*}
    &  & \| Q_c \psi \|^2_{\mathcal{C}}\\
    & \leqslant & \| Q_c \psi^{\ast} \|^2_{\mathcal{C}} + \| Q_c (\psi -
    \psi^{\ast}) \|^2_{\mathcal{C}}\\
    & \leqslant & \| Q_c \psi^{\ast} \|^2_{\mathcal{C}} + K (\beta_0)  (|
    \varepsilon_1 | + | \varepsilon_2 | + | \varepsilon_3 | + c^{- \beta_0 /
    2} | \varepsilon_4 |)^2,
  \end{eqnarray*}
  therefore, for $c$ small enough, by (\ref{CP273231}), we have
  \[ \| Q_c \psi^{\ast} \|^2_{\mathcal{C}} \geqslant K \| Q_c \psi
     \|^2_{\mathcal{C}} \]
  and
  \[ B_{Q_c} (Q_c \psi^{\ast}) \geqslant K \| Q_c \psi \|^2_{\mathcal{C}} \]
  Finally, we compute, since $Q_c (\psi - \psi^{\ast}) = \varepsilon_1
  \partial_{x_1} Q_c + \varepsilon_2 c^2 \partial_c Q_c + \varepsilon_3
  \partial_{x_2} Q_c + \varepsilon_4 c \partial_{c^{\bot}} Q_c$, by Lemma
  \ref{CP2icicmieux}, that
  \[ B_{Q_c} (\varphi) = B_{Q_c} (Q_c \psi^{\ast}) + B_{Q_c} (Q_c (\psi -
     \psi^{\ast})) + 2 \langle Q_c \psi^{\ast}, L_{Q_c} (Q_c (\psi -
     \psi^{\ast})) \rangle . \]
  Furthermore, we compute, still by Lemma \ref{CP2icicmieux},
  \[ \langle Q_c \psi^{\ast}, L_{Q_c} (Q_c (\psi - \psi^{\ast})) \rangle = -
     B_{Q_c} (Q_c (\psi - \psi^{\ast})) + \langle Q_c \psi, L_{Q_c} (Q_c (\psi
     - \psi^{\ast})) \rangle, \]
  therefore
  \begin{eqnarray*}
    B_{Q_c} (\varphi) & = & B_{Q_c} (Q_c \psi^{\ast}) - B_{Q_c} (Q_c (\psi -
    \psi^{\ast})) + 2 \langle Q_c \psi, L_{Q_c} (Q_c (\psi - \psi^{\ast}))
    \rangle\\
    & \geqslant & K \| Q_c \psi \|^2_{\mathcal{C}} - B_{Q_c} (Q_c (\psi -
    \psi^{\ast})) + 2 \langle Q_c \psi, L_{Q_c} (Q_c (\psi - \psi^{\ast}))
    \rangle .
  \end{eqnarray*}
  We have
  \[ Q_c (\psi - \psi^{\ast}) = - (\varepsilon_1 \partial_{x_1} Q_c +
     \varepsilon_2 c^2 \partial_c Q_c + \varepsilon_3 \partial_{x_2} Q_c +
     \varepsilon_4 c \partial_{c^{\bot}} Q_c), \]
  and from Lemma \ref{CP20703L222}, we have
  \[ L_{Q_c} (Q_c (\psi - \psi^{\ast})) = - c^2 \varepsilon_2 i \partial_{x_2}
     Q_c + c^2 \varepsilon_4 i \partial_{x_1} Q_c . \]
  We compute
  \begin{eqnarray*}
    &  & B_{Q_c} (Q_c (\psi - \psi^{\ast}))\\
    & = & \langle - (\varepsilon_1 \partial_{x_1} Q_c + \varepsilon_2 c^2
    \partial_c Q_c + \varepsilon_3 \partial_{x_2} Q_c + \varepsilon_4 c
    \partial_{c^{\bot}} Q_c), - c^2 \varepsilon_2 i \partial_{x_2} Q_c + c^2
    \varepsilon_4 i \partial_{x_1} Q_c \rangle,
  \end{eqnarray*}
  and with (\ref{CP2sym}), we check that
  \[ B_{Q_c} (Q_c (\psi - \psi^{\ast})) = \varepsilon_2^2 c^4 \langle L_{Q_c}
     (\partial_c Q_c), \partial_c Q_c \rangle - \varepsilon_4^2 c^2 \langle
     L_{Q_c} (\partial_{c^{\bot}} Q_c), \partial_{c^{\bot}} Q_c \rangle . \]
  With Lemma \ref{CP2nend} and equation (\ref{CP273231}), we estimate
  \[ | B_{Q_c} (Q_c (\psi - \psi^{\ast})) | \leqslant K c^2 (\varepsilon_2^2 +
     \varepsilon_4^2) \leqslant o_{c \rightarrow 0} (1) \| Q_c \psi
     \|^2_{\mathcal{C}} . \]
  Finally, we have
  \[ \langle Q_c \psi, L_{Q_c} (Q_c (\psi - \psi^{\ast})) \rangle = \langle
     Q_c \psi, - c^2 \varepsilon_2 i \partial_{x_2} Q_c + c^2 \varepsilon_4 i
     \partial_{x_1} Q_c \rangle . \]
  We compute
  \[ c^2 \langle Q_c \psi, i \nabla Q_c \rangle = c^2 \int_{\mathbbm{R}^2}
     \mathfrak{I}\mathfrak{m} (\psi) \mathfrak{R}\mathfrak{e} (\nabla Q_c
     \overline{Q_c}) - c^2 \int_{\mathbbm{R}^2} \mathfrak{R}\mathfrak{e}
     (\psi) \mathfrak{I}\mathfrak{m} (\nabla Q_c \overline{Q_c}), \]
  and to finish the proof, we use
  \begin{equation}
    | c \langle Q_c \psi, i \nabla Q_c \rangle | \leqslant K c \ln \left(
    \frac{1}{c} \right) \| Q_c \psi \|_{\mathcal{C}} \label{CP2133312}
  \end{equation}
  for a constant $K > 0$ independent of $c$ by Lemma \ref{CP2L390911}, which
  is enough to show that
  \begin{eqnarray*}
    &  & | \langle Q_c \psi, L_{Q_c} (Q_c (\psi - \psi^{\ast})) \rangle |\\
    & \leqslant & o_{c \rightarrow 0} (1) (| \varepsilon_2 | + |
    \varepsilon_4 |) \| Q_c \psi \|_{\mathcal{C}}\\
    & \leqslant & o_{c \rightarrow 0} (1) \| Q_c \psi \|^2_{\mathcal{C}},
  \end{eqnarray*}
  since $c \ln \left( \frac{1}{c} \right) = o_{c \rightarrow 0} (1)$. We have
  shown that, for $\varphi \in C^{\infty}_c (\mathbbm{R}^2 \backslash \{
  \widetilde{d_c} \overrightarrow{e_1}, - \widetilde{d_c} \overrightarrow{e_1}
  \}, \mathbbm{C})$
  \begin{eqnarray*}
    B_{Q_c} (\varphi) & \geqslant & K \| Q_c \psi \|^2_{\mathcal{C}} - B_{Q_c}
    (Q_c (\psi - \psi^{\ast})) + 2 \langle Q_c \psi, L_{Q_c} (Q_c (\psi -
    \psi^{\ast})) \rangle\\
    & \geqslant & (K - o_{c \rightarrow 0} (1)) \| Q_c \psi
    \|^2_{\mathcal{C}}\\
    & \geqslant & \frac{K}{2} \| Q_c \psi \|^2_{\mathcal{C}}
  \end{eqnarray*}
  for $c$ small enough. Now, by Lemma \ref{CP2Ndensity}, we conclude by
  density as in the proof of Proposition \ref{CP205218}.
\end{proof}

\subsection{Coercivity under three orthogonality conditions}\label{CP2s34}

\begin{lemma}
  \label{CP2L73224}There exists $R, K > 0$ such that, for $0 < \beta <
  \beta_0$, $\beta_0$ a small constant, there exists $c_0 (\beta), K (\beta) >
  0$ with, for $0 < c < c_0 (\beta)$, $Q_c$ defined in Theorem \ref{th1},
  $\varphi = Q_c \psi \in H_{Q_c}$, if
  \[ \mathfrak{R}\mathfrak{e} \int_{B (\tilde{d}_c \overrightarrow{e_1}, R)
     \cup B (- \tilde{d}_c \overrightarrow{e_1}, R)} \partial_{x_1} Q_c
     \overline{Q_c \psi^{\neq 0}} =\mathfrak{R}\mathfrak{e} \int_{B
     (\tilde{d}_c \overrightarrow{e_1}, R) \cup B (- \tilde{d}_c
     \overrightarrow{e_1}, R)} \partial_{x_2} Q_c \overline{Q_c \psi^{\neq 0}}
     = 0, \]
  \[ \mathfrak{R}\mathfrak{e} \int_{B (\tilde{d}_c \overrightarrow{e_1}, R)
     \cup B (- \tilde{d}_c \overrightarrow{e_1}, R)} \partial_c Q_c
     \overline{Q_c \psi^{\neq 0}} = 0, \]
  then
  \[ B_{Q_c} (\varphi) \geqslant K (\beta) c^{2 + \beta} \| \varphi
     \|^2_{\mathcal{C}} . \]
\end{lemma}

\begin{proof}
  As for the proof of Lemma \ref{CP2133L39}, we show the result for $\varphi =
  Q_c \psi \in C^{\infty}_c (\mathbbm{R}^2 \backslash \{ \widetilde{d_c}
  \overrightarrow{e_1}, - \widetilde{d_c} \overrightarrow{e_1} \},
  \mathbbm{C})$, and we conclude by density for $\varphi \in H_{Q_c}$.
  
  For $\varphi = Q_c \psi \in C^{\infty}_c (\mathbbm{R}^2 \backslash \{
  \widetilde{d_c} \overrightarrow{e_1}, - \widetilde{d_c} \overrightarrow{e_1}
  \}, \mathbbm{C})$, we take $\varepsilon_1, \varepsilon_2, \varepsilon_3,
  \varepsilon_4$ four real parameters and we define
  \[ \psi^{\ast} \assign \psi + \varepsilon_1 \frac{\partial_{x_1} Q_c}{Q_c} +
     \varepsilon_2 \frac{c^2 \partial_c Q_c}{Q_c} + \varepsilon_3
     \frac{\partial_{x_2} Q_c}{Q_c} + \varepsilon_4 \frac{c
     \partial_{c^{\bot}} Q_c}{Q_c} . \]
  With the same computation as in the proof of Lemma \ref{CP2133L39}, we check
  that $Q_c \psi^{\ast} \in H_{Q_c}$, and using similarly the estimates of
  Lemma \ref{CP2133L35}, we can take $\varepsilon_1, \varepsilon_2,
  \varepsilon_3, \varepsilon_4 \in \mathbbm{R}$ such that
  \[ | \varepsilon_1 | + | \varepsilon_2 | + | \varepsilon_3 | = o_{c
     \rightarrow 0} (c^{\beta_0}) \| \varphi \|_{\mathcal{C}}, \]
  $| \varepsilon_4 | \leqslant K \| \varphi \|_{\mathcal{C}}$ and such that
  $\psi^{\ast}$ satisfies the four orthogonality conditions of Lemma
  \ref{CP2133L39}. Therefore,
  \begin{equation}
    B_{Q_c} (Q_c \psi^{\ast}) \geqslant K \| Q_c \psi^{\ast}
    \|_{\mathcal{C}}^2 . \label{CP2133320}
  \end{equation}
  We write
  \[ T = \varepsilon_1 \partial_{x_1} Q_c + \varepsilon_2 c^2 \partial_c Q_c +
     \varepsilon_3 \partial_{x_2} Q_c, \]
  and we develop, by Lemma \ref{CP2icicmieux},
  \begin{eqnarray*}
    &  & B_{Q_c} (Q_c \psi)\\
    & = & B_{Q_c} (Q_c \psi^{\ast}) + c^2 \varepsilon_4^2 B_{Q_c}
    (\partial_{c^{\bot}} Q_c) + B_{Q_c} (T)\\
    & - & 2 \langle Q_c \psi^{\ast}, c \varepsilon_4 L_{Q_c}
    (\partial_{c^{\bot}} Q_c) \rangle - 2 \langle Q_c \psi^{\ast}, L_{Q_c} (T)
    \rangle + 2 c \varepsilon_4 \langle L_{Q_c} (\partial_{c^{\bot}} Q_c), T
    \rangle .
  \end{eqnarray*}
  Using Lemmas \ref{CP20703L222} and \ref{CP2nend}, we compute
  \begin{eqnarray}
    | B_{Q_c} (T) | & = & | \langle L_{Q_c} (T), T \rangle | = | \langle
    L_{Q_c} (\varepsilon_2 c^2 \partial_c Q_c), \varepsilon_2 c^2 \partial_c
    Q_c \rangle | \nonumber\\
    & = & \varepsilon^2_2 c^4 | \langle L_{Q_c} (\partial_c Q_c), \partial_c
    Q_c \rangle | \nonumber\\
    & \leqslant & K \varepsilon_2^2 c^2 = o_{c \rightarrow 0} (c^{2 + 2
    \beta_0}) \| \varphi \|_{\mathcal{C}}^2 
  \end{eqnarray}
  Now, we compute, by Lemma \ref{CP20703L222}, that
  \[ \langle Q_c \psi^{\ast}, c \varepsilon_4 L_{Q_c} (\partial_{c^{\bot}}
     Q_c) \rangle = \varepsilon_4 c^2 \langle Q_c \psi^{\ast}, i
     \partial_{x_1} Q_c \rangle . \]
  From Lemma \ref{CP2L390911}, we have
  \[ | c \langle Q_c \psi^{\ast}, i \partial_{x_1} Q_c \rangle | \leqslant
     o_{c \rightarrow 0} (c^{1 - \beta_0 / 2}) \| \varphi^{\ast}
     \|_{\mathcal{C}}, \]
  therefore
  \begin{equation}
    | \langle Q_c \psi^{\ast}, c \varepsilon_4 L_{Q_c} (\partial_{c^{\bot}}
    Q_c) \rangle | \leqslant o_{c \rightarrow 0} (c^{1 + \beta_0 / 2}) \|
    \varphi^{\ast} \|_{\mathcal{C}} \| \varphi \|_{\mathcal{C}} .
  \end{equation}
  Similarly, we compute
  \[ \langle Q_c \psi^{\ast}, L_{Q_c} (T) \rangle = \langle Q_c \psi^{\ast},
     \varepsilon_2 c^2 L_{Q_c} (\partial_c Q_c) \rangle = \varepsilon_2 c^2
     \langle Q_c \psi^{\ast}, i \partial_{x_2} Q_c \rangle . \]
  Still from Lemma \ref{CP2L390911}, we have
  \[ | c \langle Q_c \psi^{\ast}, i \partial_{x_2} Q_c \rangle | \leqslant K c
     \ln \left( \frac{1}{c} \right) \| \varphi^{\ast} \|_{\mathcal{C}}, \]
  therefore
  \begin{equation}
    | \langle Q_c \psi^{\ast}, L_{Q_c} (T) \rangle | \leqslant K |
    \varepsilon_2 | c^2 \ln \left( \frac{1}{c} \right) \| \varphi^{\ast}
    \|_{\mathcal{C}} \leqslant o_{c \rightarrow 0} (c^{1 + \beta_0}) \|
    \varphi^{\ast} \|_{\mathcal{C}} \| \varphi \|_{\mathcal{C}} .
  \end{equation}
  Finally, we compute similarly that
  \[ c | \varepsilon_4 \langle L_{Q_c} (\partial_{c^{\bot}} Q_c), T \rangle |
     = c | \varepsilon_4 \langle i c \partial_{x_1} Q_c, T \rangle | = c^2 |
     \varepsilon_4 \langle i \partial_{x_1} Q_c, \varepsilon_2 c^2 \partial_c
     Q_c + \varepsilon_3 \partial_{x_2} Q_c \rangle | . \]
  Using Lemma \ref{CP2L390911} for $\varphi = c^2 \partial_c Q_c$ (with Lemma
  \ref{CP2Ndensity}), we infer
  \[ | \langle i \partial_{x_1} Q_c, c^2 \partial_c Q_c \rangle | \leqslant K
     \| c^2 \partial_c Q_c \|_{\mathcal{C}}, \]
  and $\| c^2 \partial_c Q_c \|_{\mathcal{C}} \leqslant K$ by Lemma
  \ref{CP283L223}. Furthermore, since $Q_c (- x_1, x_2) = Q_c (x_1, x_2)$, we
  have
  \[ \langle i \partial_{x_1} Q_c, \partial_{x_2} Q_c \rangle = 0. \]
  We conclude that
  \begin{equation}
    | c \varepsilon_4 \langle L_{Q_c} (\partial_{c^{\bot}} Q_c), T \rangle |
    \leqslant K c^2  | \varepsilon_4 | (| \varepsilon_2 | + | \varepsilon_3 |)
    = o_{c \rightarrow 0} (c^{2 + \beta_0 / 2}) \| \varphi \|_{\mathcal{C}}^2
    . \label{CP2133324}
  \end{equation}
  Now, combining (\ref{CP2133320}) to (\ref{CP2133324}), and with $B_{Q_c}
  (\partial_{c^{\bot}} Q_c) = 2 \pi + o_{c \rightarrow 0} (1)$ from Lemma
  \ref{CP2nend}, we have
  \[ B_{Q_c} (\varphi) \geqslant K \| \varphi^{\ast} \|^2_{\mathcal{C}} + K
     \varepsilon_4^2 c^2 - o_{c \rightarrow 0} (c^{2 + \beta_0 / 2}) \|
     \varphi \|^2_{\mathcal{C}} - o_{c \rightarrow 0} (c^{1 + \beta_0 / 2}) \|
     \varphi^{\ast} \|_{\mathcal{C}} \| \varphi \|_{\mathcal{C}} . \]
  Similarly as in the proof of Lemma \ref{CP2133L39}, we have from Lemma
  \ref{CP283L223} that, for any $\beta_0 / 2 > \beta > 0$,
  \[ \| \varphi \|_{\mathcal{C}}^2 \leqslant K \| \varphi^{\ast}
     \|^2_{\mathcal{C}} + K (\beta) \varepsilon^2_4 c^{- \beta}, \]
  hence
  \[ \varepsilon_4^2 c^2 \geqslant K (\beta) c^{2 + \beta} (\| \varphi
     \|_{\mathcal{C}}^2 - \| \varphi^{\ast} \|^2_{\mathcal{C}}), \]
  therefore
  \begin{eqnarray*}
    B_{Q_c} (\varphi) & \geqslant & K_1 (\beta) (\| \varphi^{\ast}
    \|^2_{\mathcal{C}} + c^{2 + \beta} \| \varphi \|_{\mathcal{C}}^2) - K_2
    (\beta) c^{2 + \beta} \| \varphi^{\ast} \|^2_{\mathcal{C}} - o_{c
    \rightarrow 0} (c^{2 + \beta_0 / 2}) \| \varphi \|_{\mathcal{C}}^2\\
    & - & o_{c \rightarrow 0} (c^{1 + \beta_0}) \| \varphi^{\ast}
    \|_{\mathcal{C}} \| \varphi \|_{\mathcal{C}}\\
    & \geqslant & K (\beta) c^{2 + \beta} \| \varphi \|_{\mathcal{C}}^2
  \end{eqnarray*}
  for $c$ small enough (depending on $\beta$).
\end{proof}

Lemmas \ref{CP2zerosofQc}, \ref{CP2133L39} and \ref{CP2L73224} together end
the proof of Theorem \ref{CP2th2}. Remark that in both Lemmas \ref{CP2133L39}
and \ref{CP2L73224}, we could replace the orthogonality condition
$\mathfrak{R}\mathfrak{e} \int_{B (\tilde{d}_c \overrightarrow{e_1}, R) \cup B
(- \tilde{d}_c \overrightarrow{e_1}, R)} \partial_c Q_c \overline{Q_c
\psi^{\neq 0}} = 0$ by
\begin{equation}
  \mathfrak{R}\mathfrak{e} \int_{B (\tilde{d}_c \overrightarrow{e_1}, R) \cup
  B (- \tilde{d}_c \overrightarrow{e_1}, R)} \partial_d (V_1 (x - d \vec{e}_1)
  V_{- 1} (x + d \vec{e}_1))_{| d = d_c \nobracket} \overline{Q_c \psi^{\neq
  0}} (x) d x = 0, \label{CP2corto}
\end{equation}
since, by Theorem \ref{th1} (for $p = + \infty$),
\[ \| c^2 \partial_c Q_c - \partial_d (V_1 (x - d \vec{e}_1) V_{- 1} (x + d
   \vec{e}_1))_{| d = d_c \nobracket} \|_{C^1 (B (\tilde{d}_c
   \overrightarrow{e_1}, R) \cup B (- \tilde{d}_c \overrightarrow{e_1}, R))} =
   o_{c \rightarrow 0} (1), \]
and thus this replacement creates an error term that can be estimate as the
other ones in the proof of Lemma \ref{CP2133L39}.

\subsection{Proof of the corollaries of Theorem
\ref{CP2th2}}\label{CP2proofcor}

\subsubsection{Proof of Corollary \ref{CP2Cor41}}

\begin{proof}
  We start with the proof that $(i)$ implies $(i i)$. We start by showing
  that, for $\varphi_0 \in C^{\infty}_c (\mathbbm{R}^2, \mathbbm{C})$,
  \[ B_{Q_c} (\varphi + \varphi_0) = B_{Q_c} (\varphi_0) . \]
  We take $\varphi_0 = Q_c \psi_0 \in C^{\infty}_c (\mathbbm{R}^2,
  \mathbbm{C})$ and, by integration by parts, from ($i$), we check that
  \[ \langle L_{Q_c} (\varphi_0), \varphi \rangle = 0. \]
  Furthermore, we check (for $\varphi \in C^{\infty}_c (\mathbbm{R}^2
  \backslash \{ \widetilde{d_c} \overrightarrow{e_1}, - \widetilde{d_c}
  \overrightarrow{e_1} \}, \mathbbm{C})$ and then by density for $\varphi \in
  H_{Q_c}$) that for $\varphi_0 \in C^{\infty}_c (\mathbbm{R}^2,
  \mathbbm{C})$,
  \[ B_{Q_c} (\varphi + \varphi_0) = B_{Q_c} (\varphi) + B_{Q_c} (\varphi_0)
     + 2 \langle \varphi, L_{Q_c} (\varphi_0) \rangle, \]
  hence
  \begin{equation}
    B_{Q_c} (\varphi + \varphi_0) = B_{Q_c} (\varphi) + B_{Q_c} (\varphi_0) .
    \label{CP2tototo}
  \end{equation}
  Similarly as in the proof of Proposition \ref{CP205218}, we argue by density
  that this result holds for $\varphi_0 \in H_{Q_c}$. Now, taking $\varphi_0 =
  - \varphi$, we infer from (\ref{CP2tototo}) that $B_{Q_c} (\varphi) = 0$,
  thus, for $\varphi \in H_{Q_c}$,
  \begin{equation}
    B_{Q_c} (\varphi + \varphi_0) = B_{Q_c} (\varphi_0) . \label{CP2llabel5}
  \end{equation}
  Now, similarly as the proof of Lemma \ref{CP2133L39}, we decompose $\varphi
  = Q_c \psi \in H_{Q_c}$ in
  \[ \varphi = \varphi^{\ast} + \varepsilon_1 \partial_{x_1} Q_c +
     \varepsilon_2 \partial_{x_2} Q_c + \varepsilon_3 c^2 \partial_c Q_c \]
  with
  \[ | \varepsilon_1 | + | \varepsilon_2 | + | \varepsilon_3 | \leqslant K \|
     \varphi \|_{\mathcal{C}}, \]
  such that $\varphi^{\ast}$ verifies the three orthogonality conditions of
  Lemma \ref{CP2L73224}. We write
  \[ A = \varepsilon_1 \partial_{x_1} Q_c + \varepsilon_2 \partial_{x_2} Q_c +
     \varepsilon_3 c^2 \partial_c Q_c \in H_{Q_c} \]
  by Lemma \ref{CP20703L222}, and using (\ref{CP2llabel5}), we have
  \[ B_{Q_c} (\varphi^{\ast}) = B_{Q_c} (\varphi - A) = B_{Q_c} (A) . \]
  From Lemma \ref{CP2L73224}, we have $B_{Q_c} (\varphi^{\ast}) \geqslant K
  c^{2 + \beta_0 / 2} \| \varphi^{\ast} \|_{\mathcal{C}}^2$. Furthermore, from
  Lemmas \ref{CP20703L222} and \ref{CP2nend},
  \[ B_{Q_c} (A) = \varepsilon_3^2 c^2 B_{Q_c} (\partial_c Q_c) = (- 2 \pi +
     o_{c \rightarrow 0} (1)) \varepsilon_3^2 \leqslant 0. \]
  We deduce that $\varepsilon_3 = 0$ and $\| \varphi^{\ast} \|_{\mathcal{C}} =
  0$, hence $\varphi^{\ast} = i \mu Q_c$ for some $\mu \in \mathbbm{R}$. Since
  $\varphi^{\ast} = \varphi - R \in H_{Q_c}$, we deduce that $\mu = 0$ (or
  else $\| \varphi^{\ast} \|^2_{H_{Q_c}} \geqslant \int_{\mathbbm{R}^2}
  \frac{| \varphi^{\ast} |^2}{(1 + \tilde{r})^2} = + \infty$). Therefore,
  \[ \varphi = \varepsilon_1 \partial_{x_1} Q_c + \varepsilon_2 \partial_{x_2}
     Q_c \in \tmop{Span}_{\mathbbm{R}} (\partial_{x_1} Q_c, \partial_{x_2}
     Q_c) . \]

  Finally, the fact that $(i i)$ implies $(i)$ is a consequence of Lemma
  \ref{CP20703L222}. This concludes the proof of this lemma.
\end{proof}

\subsubsection{Spectral stability}

We have $H^1 (\mathbbm{R}^2) \subset H_{Q_c}$, therefore $B_{Q_c} (\varphi)$
is well defined for $\varphi \in H^1 (\mathbbm{R}^2)$. Furthermore, the fact
that $i \partial_{x_2} Q_c \in L^2 (\mathbbm{R}^2)$ is a consequence of
Theorem \ref{CP2Qcbehav}, and in particular this justifies that $\langle
\varphi, i \partial_{x_2} Q_c \rangle$ is well defined for $\varphi \in H^1
(\mathbbm{R}^2)$. For $\varphi \in H^1 (\mathbbm{R}^2)$, there are no issue in
the definition of the quadratic form, as shown in the following lemma.

\begin{lemma}
  \label{CP2L3150403}There exists $c_0 > 0$ such that, for $0 < c < c_0$,
  $Q_c$ defined in Theorem \ref{th1}, if $\varphi \in H^1 (\mathbbm{R}^2)$,
  then
  \[ B_{Q_c} (\varphi) = \int_{\mathbbm{R}^2} | \nabla \varphi |^2
     -\mathfrak{R}\mathfrak{e} (i c \partial_{x_2} \varphi \bar{\varphi}) - (1
     - | Q_c |^2) | \varphi |^2 + 2\mathfrak{R}\mathfrak{e}^2 (\overline{Q_c}
     \varphi) . \]
\end{lemma}

\begin{proof}
  We recall that $H^1 (\mathbbm{R}^2) \subset H_{Q_c}$ and, for $\varphi = Q_c
  \psi \in H^1 (\mathbbm{R}^2)$,
  \begin{eqnarray*}
    B_{Q_c} (\varphi) & = & \int_{\mathbbm{R}^2} | \nabla \varphi |^2 - (1 - |
    Q_c |^2) | \varphi |^2 + 2\mathfrak{R}\mathfrak{e}^2 (\overline{Q_c}
    \varphi)\\
    & - & c \int_{\mathbbm{R}^2} (1 - \eta) \mathfrak{R}\mathfrak{e} (i
    \partial_{x_2} \varphi \bar{\varphi}) - c \int_{\mathbbm{R}^2 \nosymbol}
    \eta \mathfrak{R}\mathfrak{e}i \partial_{x_2} Q_c \overline{Q_c} | \psi
    |^2\\
    & + & 2 c \int_{\mathbbm{R}^2} \eta \mathfrak{R}\mathfrak{e} \psi
    \mathfrak{I}\mathfrak{m} \partial_{x_2} \psi | Q_c |^2 + c
    \int_{\mathbbm{R}^2} \partial_{x_2} \eta \mathfrak{R}\mathfrak{e} \psi
    \mathfrak{I}\mathfrak{m} \psi | Q_c |^2\\
    & + & c \int_{\mathbbm{R}^2} \eta \mathfrak{R}\mathfrak{e} \psi
    \mathfrak{I}\mathfrak{m} \psi \partial_{x_2} (| Q_c |^2) .
  \end{eqnarray*}
  Since $\varphi \in H^1 (\mathbbm{R}^2)$, the integral $\int_{\mathbbm{R}^2}
  \mathfrak{R}\mathfrak{e} (i c \partial_{x_2} \varphi \bar{\varphi})$ is well
  defined as the scalar product of two $L^2 (\mathbbm{R}^2)$ functions. Now,
  still because $\varphi = Q_c \psi \in H^1 (\mathbbm{R}^2)$, we can integrate
  by parts, and we check that
  \begin{eqnarray*}
    \int_{\mathbbm{R}^2} \eta \mathfrak{R}\mathfrak{e} \psi
    \mathfrak{I}\mathfrak{m} \partial_{x_2} \psi | Q_c |^2 & = & -
    \int_{\mathbbm{R}^2} \eta \mathfrak{R}\mathfrak{e} \partial_{x_2} \psi
    \mathfrak{I}\mathfrak{m} \psi | Q_c |^2\\
    & - & \int_{\mathbbm{R}^2} \partial_{x_2} \eta \mathfrak{R}\mathfrak{e}
    \psi \mathfrak{I}\mathfrak{m} \psi | Q_c |^2 - \int_{\mathbbm{R}^2} \eta
    \mathfrak{R}\mathfrak{e} \psi \mathfrak{I}\mathfrak{m} \psi \partial_{x_2}
    (| Q_c |^2) .
  \end{eqnarray*}
  We conclude by expanding
  \begin{eqnarray*}
    \int_{\mathbbm{R}^2} \eta \mathfrak{R}\mathfrak{e} (i \partial_{x_2}
    \varphi \bar{\varphi}) & = & \int_{\mathbbm{R}^2} \eta
    \mathfrak{R}\mathfrak{e} (i \partial_{x_2} Q_c \overline{Q_c}) | \psi |^2
    + \int_{\mathbbm{R}^2} \eta \mathfrak{R}\mathfrak{e} (i \partial_{x_2}
    \psi \bar{\psi}) | Q_c |^2\\
    & = & \int_{\mathbbm{R}^2} \eta \mathfrak{R}\mathfrak{e} (i
    \partial_{x_2} Q_c \overline{Q_c}) | \psi |^2 + \int_{\mathbbm{R}^2} \eta
    \mathfrak{R}\mathfrak{e} (\partial_{x_2} \psi) \mathfrak{I}\mathfrak{m}
    \psi | Q_c |^2\\
    & + & \int_{\mathbbm{R}^2} \eta \mathfrak{R}\mathfrak{e} (\psi)
    \mathfrak{I}\mathfrak{m} \partial_{x_2} \psi | Q_c |^2 .
  \end{eqnarray*}
\end{proof}

The rest of this subsection is devoted to the proofs of Corollary
\ref{CP2cor177}, Proposition \ref{CP2p188} and Corollary \ref{CP2th199}.

\begin{proof}[of Corollary \ref{CP2cor177}]
  For $\varphi \in H^1 (\mathbbm{R}^2)$ such that $\langle \varphi, i
  \partial_{x_2} Q_c \rangle = 0$, we decompose it in
  \[ \varphi = \varphi^{\ast} + \varepsilon_1 \partial_{x_1} Q_c +
     \varepsilon_2 \partial_{x_2} Q_c + c^2 \varepsilon_3 \partial_c Q_c . \]
  Similarly as in the proof of Lemma \ref{CP2133L39}, we can find
  $\varepsilon_1, \varepsilon_2, \varepsilon_3 \in \mathbbm{R}$ such that
  $\varphi^{\ast}$ satisfies the three orthogonality conditions of Lemma
  \ref{CP2L73224}, and thus (since $\varphi \in H^1 (\mathbbm{R}^2) \subset
  H_{Q_c}$, for $\beta = \beta_0 / 2$)
  \[ B_{Q_c} (\varphi^{\ast}) \geqslant K c^{2 + \beta_0 / 2} \|
     \varphi^{\ast} \|_{\mathcal{C}}^2 . \]
  Now, we compute, by Lemma \ref{CP2icicmieux} and with a density argument,
  that
  \[ B_{Q_c} (\varphi) = B_{Q_c} (\varphi^{\ast}) + 2 \langle \varphi^{\ast},
     L_{Q_c} (\varepsilon_1 \partial_{x_1} Q_c + \varepsilon_2 \partial_{x_2}
     Q_c + c^2 \varepsilon_3 \partial_c Q_c) \rangle + \varepsilon_3^2 c^4
     B_{Q_c} (\partial_c Q_c) . \]
  We have from Lemma \ref{CP20703L222} that $L_{Q_c} (\varepsilon_1
  \partial_{x_1} Q_c + \varepsilon_2 \partial_{x_2} Q_c + c^2 \varepsilon_3
  \partial_c Q_c) = c^2 \varepsilon_3 i \partial_{x_2} Q_c$, therefore
  \[ B_{Q_c} (\varphi) \geqslant K c^{2 + \beta_0 / 2} \| \varphi^{\ast}
     \|^2_{\mathcal{C}} + 2 c^2 \varepsilon_3 \langle \varphi^{\ast}, i
     \partial_{x_2} Q_c \rangle + \varepsilon_3^2 c^4 B_{Q_c} (\partial_c Q_c)
     . \]
  Since $\langle \varphi, i \partial_{x_2} Q_c \rangle = 0$ and $\varphi =
  \varphi^{\ast} + \varepsilon_1 \partial_{x_1} Q_c + \varepsilon_2
  \partial_{x_2} Q_c + c^2 \varepsilon_3 \partial_c Q_c$, we have
  \[ \langle \varphi^{\ast}, i \partial_{x_2} Q_c \rangle = - \langle
     \varepsilon_1 \partial_{x_1} Q_c + \varepsilon_2 \partial_{x_2} Q_c + c^2
     \varepsilon_3 \partial_c Q_c, i \partial_{x_2} Q_c \rangle . \]
  Since $\partial_{x_1} Q_c$ is odd in $x_1$ and $i \partial_{x_2} Q_c$ is
  even in $x_1$, we have $\langle \varepsilon_1 \partial_{x_1} Q_c, i
  \partial_{x_2} Q_c \rangle = 0$. Furthermore,
  \[ \langle \varepsilon_2 \partial_{x_2} Q_c, i \partial_{x_2} Q_c \rangle =
     \varepsilon_2 \int_{\mathbbm{R}^2} \mathfrak{R}\mathfrak{e} (i |
     \partial_{x_2} Q_c |^2) = 0, \]
  and, from Lemma \ref{CP2nend}, we have
  \[ B_{Q_c} (\partial_c Q_c) = \langle \partial_c Q_c, i \partial_{x_2} Q_c
     \rangle = \frac{- 2 \pi + o_{c \rightarrow 0} (1)}{c^2}, \]
  thus
  \[ \langle \varphi^{\ast}, L_{Q_c} (\varepsilon_1 \partial_{x_1} Q_c +
     \varepsilon_2 \partial_{x_2} Q_c + c^2 \varepsilon_3 \partial_c Q_c)
     \rangle = (2 \pi + o_{c \rightarrow 0} (1)) \varepsilon_3 B_{Q_c}
     (\partial_c Q_c), \]
  and
  \[ B_{Q_c} (\varphi) \geqslant K c^{2 + \beta_0 / 2} \| \varphi^{\ast}
     \|^2_{\mathcal{C}} - \varepsilon_3^2 c^4 B_{Q_c} (\partial_c Q_c)
     \geqslant K c^{2 + \beta_0 / 2} \| \varphi^{\ast} \|^2_{\mathcal{C}} + 2
     \pi \varepsilon_3^2 c^2 (1 + o_{c \rightarrow 0} (1)) \geqslant 0 \]
  for $c$ small enough. This also shows that if $\varphi \in H^1
  (\mathbbm{R}^2)$, $B_{Q_c} (\varphi) = 0$ and $\langle \varphi, i
  \partial_{x_2} Q_c \rangle = 0$, then $\varphi \in \tmop{Span}_{\mathbbm{R}}
  \{ \partial_{x_1} Q_c, \partial_{x_2} Q_c \}$.
\end{proof}

We can now finish the proof of Proposition \ref{CP2p188}.

\begin{proof}[of Proposition \ref{CP2p188}]
  First, we have from Theorem \ref{CP2Qcbehav} that $i \partial_{x_2} Q_c \in
  L^2 (\mathbbm{R}^2)$. Now, with Corollary \ref{CP2cor177}, it is easy to
  check that $n^- (L_{Q_c}) \leqslant 1$. Indeed, if it is false, we can find
  $u, v \in H^1 (\mathbbm{R}^2)$ such that for all $\lambda, \mu \in
  \mathbbm{R}$ with $(\lambda, \mu) \neq (0, 0)$, $\lambda u + \mu v \neq 0$
  and $B_{Q_c} (\lambda u + \mu v) < 0$. Then, we can take $(\lambda, \mu)
  \neq (0, 0)$ such that
  \[ \langle \lambda u + \mu v, i \partial_{x_2} Q_c \rangle = 0, \]
  which implies $B_{Q_c} (\lambda u + \mu v) \geqslant 0$ and therefore a
  contradiction.
  
  Let us show that $L_{Q_c}$ has at least one negative eigenvalue (with
  eigenvector in $H^1 (\mathbbm{R}^2)$), which implies that $n^- (L_{Q_c}) =
  1$ and that it is the only negative eigenvalue. We consider
  \[ \alpha_c \assign \inf_{\varphi \in H^1 (\mathbbm{R}^2), \| \varphi
     \|_{L^2 (\mathbbm{R}^2)} = 1} B_{Q_c} (\varphi) . \]
  We recall, from Lemma \ref{CP2L3150403}, that (since $\varphi \in H^1
  (\mathbbm{R}^2)$)
  \[ B_{Q_c} (\varphi) = \int_{\mathbbm{R}^2} | \nabla \varphi |^2
     -\mathfrak{R}\mathfrak{e} (i c \partial_{x_2} \varphi \bar{\varphi}) - (1
     - | Q_c |^2) | \varphi |^2 + 2\mathfrak{R}\mathfrak{e}^2 (\overline{Q_c}
     \varphi), \]
  and if $\varphi \in H^1 (\mathbbm{R}^2)$ with $\| \varphi \|_{L^2
  (\mathbbm{R}^2)} = 1$, we have, by Cauchy-Schwarz,
  \[ B_{Q_c} (\varphi) \geqslant \int_{\mathbbm{R}^2} | \nabla \varphi |^2 - K
     c \| \partial_{x_2} \varphi \|_{L^2 (\mathbbm{R}^2)} - K \geqslant - K
     (c) . \]
  In particular, this implies that $\alpha_c \neq - \infty$.
  
  Now, assume that there exists no $\varphi \in C^{\infty}_c (\mathbbm{R}^2,
  \mathbbm{C})$ such that $B_{Q_c} (\varphi) < 0$. Then, for any $\varphi \in
  C^{\infty}_c (\mathbbm{R}^2, \mathbbm{C})$, we have $B_{Q_c} (\varphi)
  \geqslant 0$. Following the density argument at the end of the proof of
  Proposition \ref{CP205218}, we have $B_{Q_c} (\varphi) \geqslant 0$ for all
  $\varphi \in H_{Q_c}$, and in particular $B_{Q_c} (\partial_c Q_c) \geqslant
  0$ (we recall that $\partial_c Q_c \in H_{Q_c}$ but is not a priori in $H^1
  (\mathbbm{R}^2)$), which is in contradiction with Lemma \ref{CP2nend}.
  Therefore, there exists $\varphi \in C^{\infty}_c (\mathbbm{R}^2,
  \mathbbm{C}) \subset H^1 (\mathbbm{R}^2)$ such that $B_{Q_c} (\varphi) < 0$,
  and in particular $B_{Q_c} \left( \frac{\varphi}{\| \varphi \|_{L^2
  (\mathbbm{R}^2)}} \right) < 0$ and $\left\| \frac{\varphi}{\| \varphi
  \|_{L^2 (\mathbbm{R}^2)}} \right\|_{L^2 (\mathbbm{R}^2)} = 1$, hence
  $\alpha_c < 0$.
  
  Remark that we did not show that $\partial_c Q_c \in L^2 (\mathbbm{R}^2)$,
  and we believe this to be false. This estimation on $\alpha_c$ is the only
  time we need to work specifically with $Q_c$ from Theorem \ref{th1}. From
  now on, we can suppose that $Q_c$ is a travelling wave with finite energy
  such that $\alpha_c < 0$.
  
  To show that there exists at least one negative eigenvalue, it is enough to
  show that $\alpha_c$ is achieved for a function $\varphi \in H^1
  (\mathbbm{R}^2)$. Let us take a minimizing sequence $\varphi_n \in H^1
  (\mathbbm{R}^2)$ such that $\| \varphi_n \|_{L^2 (\mathbbm{R}^2)} = 1$ and
  $B_{Q_c} (\varphi_n) \rightarrow \alpha_c$. We have
  \[ \int_{\mathbbm{R}^2} | \nabla \varphi_n |^2 = B_{Q_c} (\varphi_n) +
     \int_{\mathbbm{R}^2} \mathfrak{R}\mathfrak{e} (i c \partial_{x_2}
     \varphi_n \overline{\varphi_n}) + (1 - | Q_c |^2) | \varphi_n |^2 -
     2\mathfrak{R}\mathfrak{e}^2 (\overline{Q_c} \varphi_n), \]
  therefore, by Cauchy-Schwarz,
  \[ \int_{\mathbbm{R}^2} | \nabla \varphi_n |^2 \leqslant | \alpha_c | + K c
     \| \nabla \varphi_n \|_{L^2 (\mathbbm{R}^2)} + K. \]
  We deduce that, for $c$ small enough,
  \[ \| \nabla \varphi_n \|^2_{L^2 (\mathbbm{R}^2)} - K c \| \nabla \varphi_n
     \|_{L^2 (\mathbbm{R}^2)} \leqslant K (c), \]
  hence $\| \nabla \varphi_n \|^2_{L^2 (\mathbbm{R}^2)}$ is bounded uniformly
  in $n$ given that $c < c_0$ for some constant $c_0$ small enough. We deduce
  that $\varphi_n$ is bounded in $H^1 (\mathbbm{R}^2)$, therefore, up to a
  subsequence, $\varphi_n \rightarrow \varphi$ weakly in $H^1
  (\mathbbm{R}^2)$.
  
  Now, we remark that for any $\varphi \in H^1 (\mathbbm{R}^2)$, by
  integration by parts (see Lemma \ref{CP2L3150403}),
  \begin{eqnarray*}
    \int_{\mathbbm{R}^2} -\mathfrak{R}\mathfrak{e} (i c \partial_{x_2} \varphi
    \bar{\varphi}) & = & - c \int_{\mathbbm{R}^2} \mathfrak{R}\mathfrak{e}
    (\partial_{x_2} \varphi) \mathfrak{I}\mathfrak{m} (\varphi) + c
    \int_{\mathbbm{R}^2} \mathfrak{R}\mathfrak{e} (\varphi)
    \mathfrak{I}\mathfrak{m} (\partial_{x_2} \varphi)\\
    & = & 2 c \int_{\mathbbm{R}^2} \mathfrak{R}\mathfrak{e} (\varphi)
    \mathfrak{I}\mathfrak{m} (\partial_{x_2} \varphi) .
  \end{eqnarray*}
  For $R > 0$, since $\varphi_n \rightarrow \varphi$ weakly in $H^1
  (\mathbbm{R}^2)$, this implies that $\varphi_n \rightarrow \varphi$ strongly
  in $L^2 (B (0, R))$ by Rellich-Kondrakov theorem. In particular, we have
  \[ \int_{B (0, R)} \mathfrak{R}\mathfrak{e} (\varphi_n)
     \mathfrak{I}\mathfrak{m} (\partial_{x_2} \varphi_n) \rightarrow \int_{B
     (0, R)} \mathfrak{R}\mathfrak{e} (\varphi) \mathfrak{I}\mathfrak{m}
     (\partial_{x_2} \varphi) . \]
  since $\varphi_n \rightarrow \varphi$ strongly in $L^2 (B (0, R))$ and
  $\partial_{x_2} \varphi_n \rightarrow \partial_{x_2} \varphi$ weakly in $L^2
  (B (0, R))$. We deduce that, up to a subsequence,
  \begin{eqnarray*}
    &  & \int_{B (0, R)} | \nabla \varphi |^2 + 2 c\mathfrak{R}\mathfrak{e}
    (\varphi) \mathfrak{I}\mathfrak{m} (\partial_{x_2} \varphi) - (1 - | Q_c
    |^2) | \varphi |^2 + 2\mathfrak{R}\mathfrak{e}^2 (\overline{Q_c}
    \varphi)\\
    & \leqslant & \liminf_{n \rightarrow \infty} \int_{B (0, R)} | \nabla
    \varphi_n |^2 + 2 c\mathfrak{R}\mathfrak{e} (\varphi_n)
    \mathfrak{I}\mathfrak{m} (\partial_{x_2} \varphi_n) - (1 - | Q_c |^2) |
    \varphi_n |^2 + 2\mathfrak{R}\mathfrak{e}^2 (\overline{Q_c} \varphi_n) +
    o^R_{n \rightarrow \infty} (1) .
  \end{eqnarray*}
  Furthermore, we have, by weak convergence
  \[ \| \varphi \|_{H^1 (\mathbbm{R}^2)} \leqslant \liminf_{n \rightarrow
     \infty} \| \varphi_n \|_{H^1 (\mathbbm{R}^2)} \leqslant K (c) \]
  therefore, we estimate
  \begin{eqnarray*}
    &  & \int_{\mathbbm{R}^2 \backslash B (0, R)} | \nabla \varphi |^2 + 2
    c\mathfrak{R}\mathfrak{e} (\varphi) \mathfrak{I}\mathfrak{m}
    (\partial_{x_2} \varphi) - (1 - | Q_c |^2) | \varphi |^2 +
    2\mathfrak{R}\mathfrak{e}^2 (\overline{Q_c} \varphi)\\
    & \leqslant & K \| \varphi \|^2_{H^1 (\mathbbm{R}^2 \backslash B (0, R))}
    = o_{R \rightarrow \infty} (1) .
  \end{eqnarray*}
  We deduce that
  \begin{eqnarray*}
    B_{Q_c} (\varphi) & \leqslant & \liminf_{n \rightarrow \infty} \int_{B (0,
    R)} | \nabla \varphi_n |^2 + 2 c\mathfrak{R}\mathfrak{e} (\varphi_n)
    \mathfrak{I}\mathfrak{m} (\partial_{x_2} \varphi_n) - (1 - | Q_c |^2) |
    \varphi_n |^2 + 2\mathfrak{R}\mathfrak{e}^2 (\overline{Q_c} \varphi_n)\\
    & + & o^R_{n \rightarrow \infty} (1) + o_{R \rightarrow \infty} (1) .
  \end{eqnarray*}
  Now, we have
  \begin{eqnarray*}
    &  & \liminf_{n \rightarrow \infty} \int_{B (0, R)} | \nabla \varphi_n
    |^2 + 2 c\mathfrak{R}\mathfrak{e} (\varphi_n) \mathfrak{I}\mathfrak{m}
    (\partial_{x_2} \varphi_n) - (1 - | Q_c |^2) | \varphi_n |^2 +
    2\mathfrak{R}\mathfrak{e}^2 (\overline{Q_c} \varphi_n)\\
    & = & \liminf_{n \rightarrow \infty} B_{Q_c} (\varphi_n),\\
    & - & \liminf_{n \rightarrow \infty} \int_{\mathbbm{R}^2 \backslash B (0,
    R)} | \nabla \varphi_n |^2 + 2 c\mathfrak{R}\mathfrak{e} (\varphi_n)
    \mathfrak{I}\mathfrak{m} (\partial_{x_2} \varphi_n) - (1 - | Q_c |^2) |
    \varphi_n |^2 + 2\mathfrak{R}\mathfrak{e}^2 (\overline{Q_c} \varphi_n)
  \end{eqnarray*}
  and $B_{Q_c} (\varphi_n) \rightarrow \alpha_c$, therefore
  \begin{eqnarray*}
    B_{Q_c} (\varphi) & \leqslant & \alpha_c + o^R_{n \rightarrow \infty} (1)
    + o_{R \rightarrow \infty} (1)\\
    & - & \liminf_{n \rightarrow \infty} \int_{\mathbbm{R}^2 \backslash B (0,
    R)} | \nabla \varphi_n |^2 + 2 c\mathfrak{R}\mathfrak{e} (\varphi_n)
    \mathfrak{I}\mathfrak{m} (\partial_{x_2} \varphi_n) - (1 - | Q_c |^2) |
    \varphi_n |^2 + 2\mathfrak{R}\mathfrak{e}^2 (\overline{Q_c} \varphi_n) .
  \end{eqnarray*}
  From Theorem \ref{CP2Qcbehav}, we have $(1 - | Q_c |^2) (x) \rightarrow 0$
  when $| x | \rightarrow \infty$, therefore, since $\| \varphi_n \|_{L^2
  (\mathbbm{R}^2)} = 1$, we have by dominated convergence that
  \[ \int_{\mathbbm{R}^2 \backslash B (0, R)} (1 - | Q_c |^2) | \varphi_n |^2
     \leqslant \sqrt{\int_{\mathbbm{R}^2 \backslash B (0, R)} (1 - | Q_c
     |^2)^2 \int_{\mathbbm{R}^2} | \varphi_n |^2} \leqslant o_{R \rightarrow
     \infty} (1) . \]
  Furthermore, we check easily that (since $(A + B)^2 \geqslant \frac{1}{2}
  A^2 - B^2$)
  \[ \int_{\mathbbm{R}^2 \backslash B (0, R)} \mathfrak{R}\mathfrak{e}^2
     (\overline{Q_c} \varphi_n) \geqslant \frac{1}{2} \int_{\mathbbm{R}^2
     \backslash B (0, R)} \mathfrak{R}\mathfrak{e}^2 (Q_c)
     \mathfrak{R}\mathfrak{e}^2 (\varphi_n) - \int_{\mathbbm{R}^2 \backslash B
     (0, R)} \mathfrak{I}\mathfrak{m}^2 (Q_c) \mathfrak{I}\mathfrak{m}^2
     (\varphi_n), \]
  and from Theorem \ref{CP2Qcbehav}, $\mathfrak{I}\mathfrak{m} (Q_c) (x)
  \rightarrow 0$ and $\mathfrak{R}\mathfrak{e} (Q_c) (x) \rightarrow 1$ when
  $| x | \rightarrow \infty$, therefore, since $\| \varphi_n \|_{L^2
  (\mathbbm{R}^2)} = 1$, by dominated convergence,
  \[ \int_{\mathbbm{R}^2 \backslash B (0, R)} 2\mathfrak{R}\mathfrak{e}^2
     (\overline{Q_c} \varphi_n) \geqslant \int_{\mathbbm{R}^2 \backslash B (0,
     R)} \mathfrak{R}\mathfrak{e}^2 (\varphi_n) - o_{R \rightarrow \infty} (1)
     . \]
  We deduce that, since $c < \sqrt{2}$,
  \begin{eqnarray*}
    B_{Q_c} (\varphi) & \leqslant & \alpha_c + o^R_{n \rightarrow \infty} (1)
    + o_{R \rightarrow \infty} (1)\\
    & - & \liminf_{n \rightarrow \infty} \left( \int_{\mathbbm{R}^2
    \backslash B (0, R)} | \nabla \varphi_n |^2 + 2 c\mathfrak{R}\mathfrak{e}
    (\varphi_n) \mathfrak{I}\mathfrak{m} (\partial_{x_2} \varphi_n)
    +\mathfrak{R}\mathfrak{e}^2 (\varphi_n) \right)\\
    & \leqslant & \alpha_c + o^R_{n \rightarrow \infty} (1) + o_{R
    \rightarrow \infty} (1)\\
    & - & \liminf_{n \rightarrow \infty} \left( \int_{\mathbbm{R}^2
    \backslash B (0, R)} (| \nabla \varphi_n | + c\mathfrak{R}\mathfrak{e}
    (\varphi_n))^2 + (2 - c^2) \mathfrak{R}\mathfrak{e}^2 (\varphi_n)
    \right)\\
    & \leqslant & \alpha_c + o^R_{n \rightarrow \infty} (1) + o_{R
    \rightarrow \infty} (1) .
  \end{eqnarray*}
  Thus, by letting $n \rightarrow \infty$ and then $R \rightarrow \infty$,
  \[ B_{Q_c} (\varphi) \leqslant \alpha_c . \]
  In particular, this implies that $\| \varphi \|_{L^2 (\mathbbm{R}^2)} \neq
  0$, or else $B_{Q_c} (\varphi) = 0 \leqslant \alpha_c$ and we know that
  $\alpha_c < 0$. Furthermore, by weak convergence, we have $\| \varphi
  \|_{L^2 (\mathbbm{R}^2)} \leqslant 1$, and if it is not $1$, then, since
  $\alpha_c < 0$,
  \[ B_{Q_c} \left( \frac{\varphi}{\| \varphi \|_{L^2 (\mathbbm{R}^2)}}
     \right) \leqslant \frac{\alpha_c}{\| \varphi \|^2_{L^2 (\mathbbm{R}^2)}}
     < \alpha_c \]
  which is in contradiction with the definition of $\alpha_c$. Therefore $\|
  \varphi \|_{L^2 (\mathbbm{R}^2)} = 1$ and $B_{Q_c} (\varphi) = \alpha_c$.
  This concludes the proof of Proposition \ref{CP2p188}.
\end{proof}

\begin{proof}[of Corollary \ref{CP2th199}]
  The hypothesis to have the spectral stability from Theorem 11.8 of
  {\cite{LZ}} are:
  
  - The curve of travelling waves is $C^1$ from $] 0, c_0 [$ to $C^1
  (\mathbbm{R}^2, \mathbbm{C})$ with respect to the speed. This is a
  consequence of Theorem \ref{th1}. This is enough to legitimate the
  computations done in the proof of Theorem 11.8 of {\cite{LZ}}.
  
  - $\mathfrak{R}\mathfrak{e} (Q_c) - 1 \in H^1 (\mathbbm{R}^2)$, $\nabla Q_c
  \in L^2 (\mathbbm{R}^2)$, $| Q_c | \rightarrow 1$ at infinity and $\| Q_c
  \|_{C^1 (\mathbbm{R}^2)} \leqslant K$. These are consequences of Theorem 7
  of {\cite{G1}}.
  
  - $n^- (L_{Q_c}) \leqslant 1$. This is a consequence of Proposition
  \ref{CP2p188}.
  
  - $\partial_c P_2 (Q_c) < 0$. This is a consequence of Proposition
  \ref{CP2prop5}.
\end{proof}

\section{Coercivity results with an orthogonality on the phase}\label{CP2ss41}

This section is devoted to the proofs of Proposition \ref{CP2prop16},
\ref{CP2prop17} and Theorem \ref{CP2p41}.

\subsection{Properties of the space $H_{Q_c}^{\exp}$}\label{CP2HQc911}

In this subsection, we look at the space $H^{\exp}_{Q_c}$. We recall the norm
\[ \| \varphi \|_{H^{\exp}_{Q_c}}^2 = \| \varphi \|_{H^1 (\{ \check{r}
   \leqslant 10 \})}^2 + \int_{\{ \tilde{r} \geqslant 5 \}} | \nabla \psi |^2
   +\mathfrak{R}\mathfrak{e}^2 (\psi) + \frac{| \psi |^2}{\tilde{r}^2 \ln
   (\tilde{r})^2} . \]
The quadratic form we look at is
\begin{eqnarray*}
  B^{\exp}_{Q_c} (\varphi) & = & \int_{\mathbbm{R}^2} \eta (| \nabla \varphi
  |^2 -\mathfrak{R}\mathfrak{e} (i c \partial_{x_2} \varphi \bar{\varphi}) -
  (1 - | Q_c |^2) | \varphi |^2 + 2\mathfrak{R}\mathfrak{e}^2 (\overline{Q_c}
  \varphi))\\
  & - & \int_{\mathbbm{R}^2} \nabla \eta . (\mathfrak{R}\mathfrak{e} (\nabla
  Q_c \overline{Q_c}) | \psi |^2 - 2\mathfrak{I}\mathfrak{m} (\nabla Q_c
  \overline{Q_c}) \mathfrak{R}\mathfrak{e} (\psi) \mathfrak{I}\mathfrak{m}
  (\psi))\\
  & + & \int_{\mathbbm{R}^2} c \partial_{x_2} \eta | Q_c |^2
  \mathfrak{R}\mathfrak{e} (\psi) \mathfrak{I}\mathfrak{m} (\psi)\\
  & + & \int_{\mathbbm{R}^2} (1 - \eta) (| \nabla \psi |^2 | Q_c |^2 +
  2\mathfrak{R}\mathfrak{e}^2 (\psi) | Q_c |^4)\\
  & + & \int_{\mathbbm{R}^2} (1 - \eta) (4\mathfrak{I}\mathfrak{m} (\nabla
  Q_c \overline{Q_c}) \mathfrak{I}\mathfrak{m} (\nabla \psi)
  \mathfrak{R}\mathfrak{e} (\psi) + 2 c | Q_c |^2 \mathfrak{I}\mathfrak{m}
  (\partial_{x_2} \psi) \mathfrak{R}\mathfrak{e} (\psi))
\end{eqnarray*}
We will show in Lemma \ref{CP2L3133105} that $B^{\exp}_{Q_c} (\varphi)$ is
well defined for $\varphi \in H^{\exp}_{Q_c}$. The main difference between
$B_{Q_c}$ and $B^{\exp}_{Q_c}$ is the space on which they are defined. In
particular, we can check easily for instance that, for $\varphi \in
C^{\infty}_c (\mathbbm{R}^2)$ with support far from the zeros of $Q_c$, we
have $B^{\exp}_{Q_c} (\varphi) = B_{Q_c} (\varphi)$, as the terms with the
gradient of the cutoff are exactly the ones coming from the integrations by
parts. We start with a lemma about the space $H^{\exp}_{Q_c}$.

\begin{lemma}
  \label{CP2L3133105}The following properties of $H^{\exp}_{Q_c}$ hold:
  \[ H_{Q_c} \subset H^{\exp}_{Q_c}, \]
  \[ i Q_c \in H^{\exp}_{Q_c} . \]
  Furthermore, there exists $K (c) > 0$ such that, for $\varphi \in
  H^{\exp}_{Q_c}$,
  \begin{equation}
    \| \varphi \|_{\mathcal{C}} \leqslant K \| \varphi \|_{H^{\exp}_{Q_c}},
  \end{equation}
  \begin{equation}
    \| \varphi \|_{H^{\exp}_{Q_c}} \leqslant K (c) \| \varphi \|_{H_{Q_c}} .
    \label{CP2410606}
  \end{equation}
  and the integrands of $B^{\exp}_{Q_c} (\varphi)$, defined in
  (\ref{CP2Btilda}), are in $L^1 (\mathbbm{R}^2)$ for $\varphi \in
  H^{\exp}_{Q_c}$, and $B_{Q_c}^{\exp}$ does not depend on the choice of
  $\eta$. Finally, if $\varphi \in H_{Q_c} \subset H_{Q_c}^{\exp}$,
  \[ B_{Q_c} (\varphi) = B_{Q_c}^{\exp} (\varphi) . \]
\end{lemma}

\begin{proof}
  First, let us show (\ref{CP2410606}). We have
  \[ \| \varphi \|_{H^1 (\{ \tilde{r} \leqslant 10 \})} \leqslant K \| \varphi
     \|_{H_{Q_c}}, \]
  and, by equation (\ref{CP2Qcpaszero}) and Lemma \ref{CP2sensmanqu}, we check
  that
  \[ \int_{\{ \tilde{r} \geqslant 5 \}} \mathfrak{R}\mathfrak{e}^2 (\psi)
     \leqslant K \| \varphi \|_{H_{Q_c}}^2, \]
  and also that
  \[ \int_{\{ \tilde{r} \geqslant 5 \}} \frac{| \psi |^2}{\tilde{r}^2 \ln
     (\tilde{r})^2} \leqslant K \int_{\{ \tilde{r} \geqslant 5 \}} \frac{|
     \varphi |^2}{(1 + \tilde{r})^2} \leqslant K (c) \| \varphi \|_{H_{Q_c}}^2
     . \]
  Furthermore, we compute, by equations (\ref{CP2Qcpaszero}),
  (\ref{CP22102est}) and Proposition \ref{CP2Qcbehav},
  \[ \int_{\{ \tilde{r} \geqslant 5 \}} | \nabla \psi |^2 \leqslant K \int_{\{
     \tilde{r} \geqslant 5 \}} | \nabla \psi |^2 | Q_c |^4 \leqslant K \left(
     \int_{\{ \tilde{r} \geqslant 5 \}} | \nabla \varphi |^2 + \int_{\{
     \tilde{r} \geqslant 5 \}} | \nabla Q_c |^2 | \varphi |^2 \right)
     \leqslant K (c) \| \varphi \|_{H_{Q_c}}^2 . \]
  We deduce that (\ref{CP2410606}) holds, and therefore $H_{Q_c} \subset
  H^{\exp}_{Q_c}$. Now, we check that
  \begin{equation}
    \| i Q_c \|^2_{H^{\exp}_{Q_c}} \leqslant \| i Q_c \|^2_{H^1 (\{ \tilde{r}
    \leqslant 10 \})} + K \int_{\{ \tilde{r} \geqslant 5 \}} \frac{| i
    |^2}{\tilde{r}^2 \ln (\tilde{r})^2} + \int_{\{ \tilde{r} \geqslant 5 \}} |
    \nabla i |^2 < + \infty . \label{CP2IQC}
  \end{equation}
  With regards to the definition of $\| . \|_{\mathcal{C}}$, we check easily
  that
  \[ \| \varphi \|_{\mathcal{C}} \leqslant \| \varphi \|_{H^{\exp}_{Q_c}} . \]
  Finally, we recall the definition of $B^{\exp}_{Q_c} (\varphi) $from
  equation (\ref{CP2Btilda}),
  \begin{eqnarray*}
    B^{\exp}_{Q_c} (\varphi) & = & \int_{\mathbbm{R}^2} (1 - \eta) (| \nabla
    \varphi |^2 -\mathfrak{R}\mathfrak{e} (i c \partial_{x_2} \varphi
    \bar{\varphi}) - (1 - | Q_c |^2) | \varphi |^2 +
    2\mathfrak{R}\mathfrak{e}^2 (\overline{Q_c} \varphi))\\
    & - & \int_{\mathbbm{R}^2} \nabla \eta . (\mathfrak{R}\mathfrak{e}
    (\nabla Q_c \overline{Q_c}) | \psi |^2 - 2\mathfrak{I}\mathfrak{m} (\nabla
    Q_c \overline{Q_c}) \mathfrak{R}\mathfrak{e} (\psi)
    \mathfrak{I}\mathfrak{m} (\psi))\\
    & + & \int_{\mathbbm{R}^2} c \partial_{x_2} \eta | Q_c |^2
    \mathfrak{R}\mathfrak{e} (\psi) \mathfrak{I}\mathfrak{m} (\psi)\\
    & + & \int_{\mathbbm{R}^2} \eta (| \nabla \psi |^2 | Q_c |^2 +
    2\mathfrak{R}\mathfrak{e}^2 (\psi) | Q_c |^4)\\
    & + & \int_{\mathbbm{R}^2} \eta (4\mathfrak{I}\mathfrak{m} (\nabla Q_c
    \overline{Q_c}) \mathfrak{I}\mathfrak{m} (\nabla \psi)
    \mathfrak{R}\mathfrak{e} (\psi) + 2 c | Q_c |^2 \mathfrak{I}\mathfrak{m}
    (\partial_{x_2} \psi) \mathfrak{R}\mathfrak{e} (\psi)) .
  \end{eqnarray*}
  For $\lambda > 0$, we have $\| \varphi \|_{H^1 (B (0, \lambda))} \leqslant K
  (c, \lambda) \| \varphi \|_{H^{\exp}_{Q_c}}$, therefore (since $1 - \eta$ is
  compactly supported) we only have to check that the integrands in the last
  two lines are in $L^1 (\mathbbm{R}^2)$, and this is a consequence of
  Cauchy-Schwarz, since
  \[ \int_{\mathbbm{R}^2} \eta \left( | \nabla \psi |^2 | Q_c |^2 +
     2\mathfrak{R}\mathfrak{e}^2 (\psi) | Q_c |^4 + 4 \left|
     \mathfrak{I}\mathfrak{m} (\nabla Q_c \overline{Q_c})
     \mathfrak{I}\mathfrak{m} (\nabla \psi) \mathfrak{R}\mathfrak{e} (\psi)
     \right| + 2 c | Q_c |^2 | \mathfrak{I}\mathfrak{m} (\partial_{x_2} \psi)
     \mathfrak{R}\mathfrak{e} (\psi) | \right) \]
  \[ \leqslant K \int_{\mathbbm{R}^2} \eta (| \nabla \psi |^2
     +\mathfrak{R}\mathfrak{e}^2 (\psi)) \leqslant K \| \varphi
     \|^2_{H^{\exp}_{Q_c}} . \]
  Furthermore, for two cutoffs $\eta, \eta'$ such that they are both $0$ near
  the zeros of $Q_c$ and $1$ at infinity, we have
  \begin{eqnarray*}
    &  & B^{\exp}_{Q_c, \eta} (\varphi) - B^{\exp}_{Q_c, \eta'} (\varphi)\\
    & = & \int_{\mathbbm{R}^2} (\eta' - \eta) (| \nabla \varphi |^2
    -\mathfrak{R}\mathfrak{e} (i c \partial_{x_2} \varphi \bar{\varphi}) - (1
    - | Q_c |^2) | \varphi |^2 + 2\mathfrak{R}\mathfrak{e}^2 (\overline{Q_c}
    \varphi))\\
    & + & \int_{\mathbbm{R}^2} \nabla (\eta - \eta') .
    (\mathfrak{R}\mathfrak{e} (\nabla Q_c \overline{Q_c}) | \psi |^2 -
    2\mathfrak{I}\mathfrak{m} (\nabla Q_c \overline{Q_c})
    \mathfrak{R}\mathfrak{e} (\psi) \mathfrak{I}\mathfrak{m} (\psi)) - c
    \partial_{x_2} (\eta - \eta') | Q_c |^2 \mathfrak{R}\mathfrak{e} (\psi)
    \mathfrak{I}\mathfrak{m} (\psi)\\
    & + & \int_{\mathbbm{R}^2} (\eta' - \eta) (| \nabla \psi |^2 | Q_c |^2 +
    2\mathfrak{R}\mathfrak{e}^2 (\psi) | Q_c |^4)\\
    & + & \int_{\mathbbm{R}^2} (\eta' - \eta) (4\mathfrak{I}\mathfrak{m}
    (\nabla Q_c \overline{Q_c}) \mathfrak{I}\mathfrak{m} (\nabla \psi)
    \mathfrak{R}\mathfrak{e} (\psi) + 2 c | Q_c |^2 \mathfrak{I}\mathfrak{m}
    (\partial_{x_2} \psi) \mathfrak{R}\mathfrak{e} (\psi))
  \end{eqnarray*}
  and, developping $\varphi = Q_c \psi$ (see the proof of Lemma
  \ref{CP2L412602}) and by integration by parts using that $\eta - \eta' \neq
  0$ only in a compact domain far from the zeros of $Q_c$, we check that it is
  $0$.
  
  Finally, for $\varphi \in H_{Q_c}$, $B_{Q_c} (\varphi)$ and $B^{\exp}_{Q_c}
  (\varphi)$ are both well defined. We recall
  \[ \begin{array}{lll}
       B_{Q_c} (\varphi) & = & \int_{\mathbbm{R}^2} | \nabla \varphi |^2 - (1
       - | Q_c |^2) | \varphi |^2 + 2\mathfrak{R}\mathfrak{e}^2
       (\overline{Q_c} \varphi)\\
       & - & c \int_{\mathbbm{R}^2} (1 - \eta) \mathfrak{R}\mathfrak{e} (i
       \partial_{x_2} \varphi \bar{\varphi}) - c \int_{\mathbbm{R}^2
       \nosymbol} \eta \mathfrak{R}\mathfrak{e}i \partial_{x_2} Q_c
       \overline{Q_c} | \psi |^2\\
       & + & 2 c \int_{\mathbbm{R}^2} \eta \mathfrak{R}\mathfrak{e} \psi
       \mathfrak{I}\mathfrak{m} \partial_{x_2} \psi | Q_c |^2 + c
       \int_{\mathbbm{R}^2} \partial_{x_2} \eta \mathfrak{R}\mathfrak{e} \psi
       \mathfrak{I}\mathfrak{m} \psi | Q_c |^2\\
       & + & c \int_{\mathbbm{R}^2} \eta \mathfrak{R}\mathfrak{e} \psi
       \mathfrak{I}\mathfrak{m} \psi \partial_{x_2} (| Q_c |^2) .
     \end{array} \]
  With the same computation as in the proof of Lemma \ref{CP2L412602}, we
  check that for $\varphi \in C^{\infty}_c (\mathbbm{R}^2 \backslash \{
  \tilde{d}_c \vec{e}_1, - \tilde{d}_c \vec{e}_1 \}, \mathbbm{C})$, we have
  \[ B_{Q_c} (\varphi) = B^{\exp}_{Q_c} (\varphi) . \]
  With the same arguments as in the density proof at the end of the proof of
  Proposition \ref{CP205218}, we check that this equality holds for $\varphi
  \in H_{Q_c}$.
\end{proof}

Now, we state some lemmas that where shown previously in $H_{Q_c}$, that we
have to extend to $H_{Q_c}^{\exp}$ to replace some arguments that were used in
the proof of Propositions \ref{CP205218} for the proofs of Propositions
\ref{CP2prop16}, \ref{CP2prop17} and Theorem \ref{CP2p41}. We start with the
density argument.

\begin{lemma}
  \label{CP2NNN}$C^{\infty}_c (\mathbbm{R}^2 \backslash \{ \tilde{d}_c
  \vec{e}_1, - \tilde{d}_c \vec{e}_1 \}, \mathbbm{C})$ is dense in
  $H_{Q_c}^{\exp}$ for $\| . \|_{H_{Q_c}^{\exp}}$.
\end{lemma}

\begin{proof}
  The proof is identical to the one of Lemma \ref{CP2Ndensity}, as we check
  easily that, for $\lambda > \frac{10}{c}$ large enough,
  \[ \| \varphi \|_{H^1 (\{ \check{r} \leqslant 10 \})}^2 + \int_{\{ \tilde{r}
     \geqslant 5 \} \cap B (0, \lambda)} | \nabla \psi |^2
     +\mathfrak{R}\mathfrak{e}^2 (\psi) + \frac{| \psi |^2}{\tilde{r}^2 \ln
     (\tilde{r})^2} \leqslant K_1 (\lambda, c) \| \varphi \|^2_{H^1 (B (0,
     \lambda))} \]
  and
  \[ \| \varphi \|_{H^1 (\{ \check{r} \leqslant 10 \})}^2 + \int_{\{ \tilde{r}
     \geqslant 5 \} \cap B (0, \lambda)} | \nabla \psi |^2
     +\mathfrak{R}\mathfrak{e}^2 (\psi) + \frac{| \psi |^2}{\tilde{r}^2 \ln
     (\tilde{r})^2} \geqslant K_2 (\lambda, c) \| \varphi \|^2_{H^1 (B (0,
     \lambda))} \]
\end{proof}

We also want to decompose the quadratic form, but with a fifth possible
direction: $i Q_c$.

\begin{lemma}
  \label{CP2destrucs}For $\varphi \in C^{\infty}_c (\mathbbm{R}^2 \backslash
  \{ \widetilde{d_c} \overrightarrow{e_1}, - \widetilde{d_c}
  \overrightarrow{e_1} \}, \mathbbm{C})$ and $A \in \tmop{Span} \{
  \partial_{x_1} Q_c, \partial_{x_2} Q_c, \partial_c Q_c, \partial_{c^{\bot}}
  Q_c, i Q_c \}$, we have
  \[ \langle L_{Q_c} (\varphi + A), \varphi + A \rangle = \langle L_{Q_c}
     (\varphi), \varphi \rangle + \langle 2 L_{Q_c} (A), \varphi \rangle +
     \langle L_{Q_c} (A), A \rangle . \]
  Furthermore, $\langle L_{Q_c} (\varphi + A), \varphi + A \rangle =
  B^{\exp}_{Q_c} (\varphi + A)$, $L_{Q_c} (i Q_c) = 0$ and
  \[ \| \partial_{x_1} Q_c \|_{H_{Q_c}^{\exp}} + \| \partial_{x_2} Q_c
     \|_{H_{Q_c}^{\exp}} + \| c^2 \partial_c Q_c \|_{H_{Q_c}^{\exp}} +
     c^{\beta_0 / 2} \| c \partial_{c^{\bot}} Q_c \|_{H_{Q_c}^{\exp}} + \| i
     Q_c \|_{H_{Q_c}^{\exp}} \leqslant K (\beta_0) . \]
\end{lemma}

\begin{proof}
  As for the proof of Lemma \ref{CP2icicmieux}, we only have to show that
  $\mathfrak{R}\mathfrak{e} (L_{Q_c} (A) \bar{A}) \in L^1 (\mathbbm{R}^2)$ to
  show the first equality.
  
  By simple computation (or by invariance of the phase), we check that
  $L_{Q_c} (i Q_c) = 0$. Writing $A = T + \varepsilon i Q_c$ for $\varepsilon
  \in \mathbbm{R}, T \in \tmop{Span} \{ \partial_{x_1} Q_c, \partial_{x_2}
  Q_c, \partial_c Q_c, \partial_{c^{\bot}} Q_c \}$, we compute from Lemma
  \ref{CP20703L222} that
  \[ L_{Q_c} (A) = L_{Q_c} (T) \in \tmop{Span}_{\mathbbm{R}} (i \partial_{x_1}
     Q_c, i \partial_{x_2} Q_c), \]
  thus
  \[ \mathfrak{R}\mathfrak{e} (L_{Q_c} (A) \bar{A}) =\mathfrak{R}\mathfrak{e}
     (L_{Q_c} (T) \overline{T + \varepsilon i Q_c}) =\mathfrak{R}\mathfrak{e}
     (L_{Q_c} (T) \bar{T}) + \varepsilon \mathfrak{R}\mathfrak{e} (L_{Q_c} (T)
     \overline{i Q_c}) . \]
  From the proof of Lemma \ref{CP2icicmieux}, we have
  $\mathfrak{R}\mathfrak{e} (L_{Q_c} (T) \bar{T}) \in L^1 (\mathbbm{R}^2)$,
  and since $L_{Q_c} (T) \in \tmop{Span}_{\mathbbm{R}} (i \partial_{x_1} Q_c,
  i \partial_{x_2} Q_c)$, with Theorem \ref{CP2Qcbehav}, we have
  \[ | \mathfrak{R}\mathfrak{e} (L_{Q_c} (T) \overline{i Q_c}) | \leqslant
     \frac{K (c)}{(1 + r)^3} \in L^1 (\mathbbm{R}^2) . \]
  Let us check that, for $\varphi \in H_{Q_c}^{\exp}$, $B_{Q_c}^{\exp}
  (\varphi + \varepsilon i Q_c) = B_{Q_c}^{\exp} (\varphi)$ for $\varepsilon
  \in \mathbbm{R}$.
  
  We check, from (\ref{CP2Btilda}), that, for $\varphi \in C^{\infty}_c
  (\mathbbm{R}^2 \backslash \{ \widetilde{d_c} \overrightarrow{e_1}, -
  \widetilde{d_c} \overrightarrow{e_1} \}, \mathbbm{C})$, this equality holds
  by integration by parts and because $\mathfrak{R}\mathfrak{e} (\psi + i)
  =\mathfrak{R}\mathfrak{e} (\psi)$, $\mathfrak{I}\mathfrak{m} (\nabla (\psi +
  i)) =\mathfrak{I}\mathfrak{m} (\nabla \psi)$. We then argue by density, as
  in the proof of Proposition \ref{CP205218}.
  
  We deduce, from Lemmas \ref{CP20703L222} and \ref{CP2icicmieux}, that for
  $\varphi \in C^{\infty}_c (\mathbbm{R}^2 \backslash \{ \widetilde{d_c}
  \overrightarrow{e_1}, - \widetilde{d_c} \overrightarrow{e_1} \},
  \mathbbm{C})$,
  \begin{eqnarray*}
    B^{\exp}_{Q_c} (\varphi + A) & = & B^{\exp}_{Q_c} (\varphi + T) = B_{Q_c}
    (\varphi + T)\\
    & = & \langle L_{Q_c} (\varphi + T), \varphi + T \rangle = \langle
    L_{Q_c} (\varphi + A), \varphi + T \rangle\\
    & = & \langle L_{Q_c} (\varphi + A), \varphi + A \rangle - \langle
    L_{Q_c} (\varphi + A), \varepsilon i Q_c \rangle,
  \end{eqnarray*}
  and we check, with Lemma \ref{CP20703L222}, that for some $v \in
  \mathbbm{R}^2$ depending on $A$,
  \begin{eqnarray*}
    \langle L_{Q_c} (\varphi + A), \varepsilon i Q_c \rangle & = & \langle
    L_{Q_c} (\varphi), \varepsilon i Q_c \rangle + \langle L_{Q_c} (P),
    \varepsilon i Q_c \rangle\\
    & = & \varepsilon \langle \varphi, L_{Q_c} (i Q_c) \rangle + \varepsilon
    v. \int_{\mathbbm{R}^2} \mathfrak{R}\mathfrak{e} (\nabla Q_c
    \overline{Q_c})\\
    & = & 0.
  \end{eqnarray*}
  From Lemma \ref{CP283L223}, we have,
  \[ \| \partial_{x_1} Q_c \|_{\mathcal{C}} + \| \partial_{x_2} Q_c
     \|_{\mathcal{C}} + \| c^2 \partial_c Q_c \|_{\mathcal{C}} + c^{\beta_0 /
     2} \| c \partial_{c^{\bot}} Q_c \|_{\mathcal{C}} \leqslant K (\beta_0),
  \]
  and with Lemmas \ref{lemme3new}, \ref{CP2dcQcsigma} and equations
  (\ref{CP2220}), (\ref{CP2221}), (\ref{CP2222}), we check with the definition
  of $\| . \|_{H_{Q_c}^{\exp}}$ and $\| . \|_{\mathcal{C}}$ that, for $A \in
  \{ \partial_{x_1} Q_c, \partial_{x_2} Q_c, c^2 \partial_c Q_c, c^{1 +
  \beta_0 / 2} \partial_{c^{\bot}} Q_c \}$,
  \[ \| A \|^2_{H_{Q_c}^{\exp}} \leqslant K \| A \|^2_{H^1 (\{ \check{r}
     \leqslant 10 \})} + \| A \|^2_{\mathcal{C}} \leqslant K (\beta_0) . \]
  Finally, we check that
  \[ \| i Q_c \|_{H_{Q_c}^{\exp}}^2 = \| i Q_c \|_{H^1 (\{ \check{r} \leqslant
     10 \})}^2 + \int_{\{ \tilde{r} \geqslant 5 \}} | \nabla i |^2
     +\mathfrak{R}\mathfrak{e}^2 (i) + \frac{| i |^2}{\tilde{r}^2 \ln
     (\tilde{r})^2} \leqslant K. \]
\end{proof}

We can now end the proof of Proposition \ref{CP2prop16}.

\begin{proof}[of Proposition \ref{CP2prop16}]
  From Theorem \ref{CP2th2}, for $\varphi \in C^{\infty}_c (\mathbbm{R}^2
  \backslash \{ \widetilde{d_c} \overrightarrow{e_1}, - \widetilde{d_c}
  \overrightarrow{e_1} \}, \mathbbm{C})$, under the four orthogonality
  conditions of Proposition \ref{CP2prop16}, we have, by lemma
  \ref{CP2L3133105},
  \[ B_{Q_c}^{\exp} (\varphi) = B_{Q_c} (\varphi) = \langle L_{Q_c} (\varphi),
     \varphi \rangle \geqslant K \| \varphi \|_{\mathcal{C}}^2 . \]
  We then conclude by density, as in the proof of Proposition \ref{CP205218},
  using Lemma \ref{CP2NNN}. The proof for the density in $B_{Q_c}^{\exp}$ is
  similar to the one for $B_{Q_c}$ in the proof of Proposition \ref{CP205218}.
  The coercivity under three orthogonality conditions can be shown similarly.
  
  \
  
  Then, for the computation of the kernel, the proof is identical to the one
  of Corollary \ref{CP2Cor41}. With Lemma \ref{CP2L3133105}, we check easily
  that we can do the same computation simply by replacing $B_{Q_c} (\varphi)$
  by $B_{Q_c}^{\exp} (\varphi)$. The only difference is at the end, when we
  have $\| \varphi^{\ast} \|_{\mathcal{C}} = 0$, it implies that
  $\varphi^{\ast} = \lambda i Q_c$ for some $\lambda \in \mathbbm{R}$, and we
  can not conclude that $\lambda = 0$, since we only have $\varphi^{\ast} \in
  H_{Q_c}^{\exp}$ instead of $\varphi^{\ast} \in H_{Q_c}$. This implies that
  \[ \varphi \in \tmop{Span}_{\mathbbm{R}} (\partial_{x_1} Q_c, \partial_{x_2}
     Q_c, i Q_c) . \]
  Using Lemma \ref{CP20703L222} and \ref{CP2destrucs}, we check easily the
  implication from $(i i)$ to $(i)$.
\end{proof}

\subsection{Change of the coercivity norm with an orthogonality on the
phase}\label{CP2ss62}

We now focus on the proofs of Proposition \ref{CP2prop17} and Theorem
\ref{CP2p41}. In these results, we add an orthogonality condition on the
phase. We start with a lemma giving the coercivity result but with the
original orthogonality conditions on the vortices, adding the one on the
phase.

\begin{lemma}
  \label{CP2L661510}For $\varphi = Q_c \psi \in H_{Q_c}^{\exp}$, if the
  following four orthogonality conditions are satisfied:
  \[ \int_{B (\tilde{d}_c \overrightarrow{e_1}, R)} \mathfrak{R}\mathfrak{e}
     \left( \partial_{x_1} \tilde{V}_1 \overline{_{} \widetilde{V_1} \psi}
     \right) = \int_{B (\tilde{d}_c \overrightarrow{e_1}, R)}
     \mathfrak{R}\mathfrak{e} \left( \partial_{x_2} \widetilde{V_1}
     \overline{\widetilde{V_1} \psi} \right) = 0, \]
  \[ \int_{B (- \tilde{d}_c \overrightarrow{e_1}, R)} \mathfrak{R}\mathfrak{e}
     (\partial_{x_1} \tilde{V}_{- 1} \overline{\tilde{V}_{- 1} \psi}) =
     \int_{B (- \tilde{d}_c \overrightarrow{e_1}, R)} \mathfrak{R}\mathfrak{e}
     (\partial_{x_2} \tilde{V}_{- 1} \overline{\tilde{V}_{- 1} \psi}) = 0, \]
  then, if $\mathfrak{R}\mathfrak{e} \int_{B (0, R)} i \psi = 0,$ we have
  (with $K (c) \leqslant 1$)
  \[ B_{Q_c}^{\exp} (\varphi) \geqslant K (c) \| \varphi \|_{H_{Q_c}^{\exp}}^2
     + K \| \varphi \|_{\mathcal{C}}^2, \]
  or if $\forall x \in \mathbbm{R}, \varphi (x_1, x_2) = \varphi (- x_1, x_2)$
  and $\mathfrak{R}\mathfrak{e} \int_{B (\tilde{d}_c \overrightarrow{e_1}, R)
  \cup B (- \tilde{d}_c \overrightarrow{e_1}, R)} i Q_c \bar{\varphi} = 0$,
  then
  \[ B_{Q_c}^{\exp} (\varphi) \geqslant K \| \varphi \|_{H_{Q_c}^{\exp}}^2 .
  \]
\end{lemma}

\begin{proof}
  Let us show these results for $\varphi = Q_c \psi \in C^{\infty}_c
  (\mathbbm{R}^2 \backslash \{ \tilde{d}_c \vec{e}_1, - \tilde{d}_c \vec{e}_1
  \}, \mathbbm{C})$. We then conclude by density. We start with the
  nonsymmetric case.
  
  By Lemma \ref{CP2L442702}, for $\varphi = Q_c \psi \in C^{\infty}_c
  (\mathbbm{R}^2 \backslash \{ \tilde{d}_c \vec{e}_1, - \tilde{d}_c \vec{e}_1
  \}, \mathbbm{C})$ such that
  \[ \int_{B (\tilde{d}_c \overrightarrow{e_1}, R)} \mathfrak{R}\mathfrak{e}
     \left( \partial_{x_1} \tilde{V}_1 \overline{_{} \widetilde{V_1} \psi}
     \right) = \int_{B (\tilde{d}_c \overrightarrow{e_1}, R)}
     \mathfrak{R}\mathfrak{e} \left( \partial_{x_2} \widetilde{V_1}
     \overline{\widetilde{V_1} \psi} \right) = 0, \]
  we have
  \[ B_{Q_c}^{\tmop{loc}_{1, D}} (\varphi) \geqslant K (D) \int_{B
     (\tilde{d}_c \overrightarrow{e_1}, D)} | \nabla \psi |^2 | Q_c |^4
     +\mathfrak{R}\mathfrak{e}^2 (\psi) | Q_c |^4 . \]
  By Lemma \ref{CP2L4326}, we infer, by a standard proof by contradiction
  (with the first two orthogonality conditions),
  \[ B_{Q_c}^{\tmop{loc}_{1, D}} (\varphi) \geqslant K_1 (D) \| \varphi
     \|_{H^1 (B (\tilde{d}_c \overrightarrow{e_1}, D))}^2 - K_2 (D) \left(
     \int_{B (\tilde{d}_c \overrightarrow{e_1}, R) \backslash B (\tilde{d}_c
     \overrightarrow{e_1}, R / 2)} \mathfrak{I}\mathfrak{m} (\psi) \right)^2 .
  \]
  We deduce, with Lemma \ref{CP2L4326}, that for any small $\varepsilon > 0$
  \begin{eqnarray*}
    B_{Q_c}^{\tmop{loc}_{1, D}} (\varphi) & \geqslant & K (D) (1 -
    \varepsilon) \int_{B (\tilde{d}_c \overrightarrow{e_1}, D)} | \nabla \psi
    |^2 | Q_c |^4 +\mathfrak{R}\mathfrak{e}^2 (\psi) | Q_c |^4\\
    & + & K_1 (D) \varepsilon \| \varphi \|^2_{H^1 (B (\tilde{d}_c
    \overrightarrow{e_1}, D))} - K_2 (D) \varepsilon \left( \int_{B
    (\tilde{d}_c \overrightarrow{e_1}, R) \backslash B (\tilde{d}_c
    \overrightarrow{e_1}, R / 2)} \mathfrak{I}\mathfrak{m} (\psi) \right)^2 .
  \end{eqnarray*}
  By Poincar{\'e} in{\'e}quality, if $\mathfrak{R}\mathfrak{e} \int_{B (0, R)}
  i \psi = 0$, then
  \begin{eqnarray*}
    \int_{B (\tilde{d}_c \overrightarrow{e_1}, R) \backslash B (\tilde{d}_c
    \overrightarrow{e_1}, R / 2)} \mathfrak{I}\mathfrak{m} (\psi) & \leqslant
    & K (c) \sqrt{\int_{\mathbbm{R}^2 \backslash (B (\tilde{d}_c
    \overrightarrow{e_1}, R / 2) \cup B (- \tilde{d}_c \overrightarrow{e_1}, R
    / 2))} | \nabla \psi |^2}\\
    & \leqslant & K (c) \sqrt{\int_{\mathbbm{R}^2} | \nabla \psi |^2 | Q_c
    |^4} .
  \end{eqnarray*}
  Therefore, for any small $\mu > 0$, taking $\varepsilon > 0$ small enough
  (depending on $c, D$ and $\mu$),
  \begin{eqnarray*}
    B_{Q_c}^{\tmop{loc}_{1, D}} (\varphi) & \geqslant & K (D) \int_{B
    (\tilde{d}_c \overrightarrow{e_1}, D)} | \nabla \psi |^2 | Q_c |^4
    +\mathfrak{R}\mathfrak{e}^2 (\psi) | Q_c |^4\\
    & + & K_1 (D, c, \mu) \| \varphi \|^2_{H^1 (B (\tilde{d}_c
    \overrightarrow{e_1}, D))} - \mu \int_{\mathbbm{R}^2} | \nabla \psi |^2 |
    Q_c |^4 .
  \end{eqnarray*}
  With similar arguments, we have a similar result for $B_{Q_c}^{\tmop{loc}_{-
  1, D}} (\varphi)$. Now, as in the proof of Proposition \ref{CP205218}, we
  have, taking $\mu > 0$ small enough and $D > 0$ large enough,
  \begin{eqnarray*}
    B_{Q_c} (\varphi) & \geqslant & B_{Q_c}^{\tmop{loc}_{1, D}} (\varphi) +
    B_{Q_c}^{\tmop{loc}_{- 1, D}} (\varphi)\\
    & + & K \left( \int_{\mathbbm{R}^2 \backslash (B (\tilde{d}_c
    \overrightarrow{e_1}, D) \cup B (- \tilde{d}_c \overrightarrow{e_1}, D))}
    | \nabla \psi |^2 | Q_c |^4 +\mathfrak{R}\mathfrak{e}^2 (\psi) | Q_c |^4
    \right)\\
    & \geqslant & K \int_{\mathbbm{R}^2} | \nabla \psi |^2 | Q_c |^4
    +\mathfrak{R}\mathfrak{e}^2 (\psi) | Q_c |^4 + K_1 (c, \mu) \| \varphi
    \|^2_{H^1 (B (\tilde{d}_c \overrightarrow{e_1}, 10))}\\
    & - & \mu \int_{\mathbbm{R}^2} | \nabla \psi |^2 | Q_c |^4\\
    & \geqslant & K \| \varphi \|_{\mathcal{C}}^2 + K (c) \| \varphi
    \|^2_{H^1 (B (\tilde{d}_c \overrightarrow{e_1}, 10))} .
  \end{eqnarray*}
  Then, by the same Hardy type inequality as in the proof of Proposition
  \ref{CP205218}, we show that
  \[ \int_{\mathbbm{R}^2} \frac{| \varphi |^2}{(1 + \tilde{r})^2 \ln^2 (2 +
     \tilde{r})} \leqslant K \left( \| \varphi \|^2_{H^1 (B (\tilde{d}_c
     \overrightarrow{e_1}, 10))} + \int_{\mathbbm{R}^2} | \nabla \psi |^2 |
     Q_c |^4 \right), \]
  therefore
  \[ B_{Q_c} (\varphi) \geqslant K \| \varphi \|_{\mathcal{C}}^2 + K (c) \|
     \varphi \|_{H_{Q_c}^{\exp}}^2 . \]

  In the symmetric case, the proof is identical, exept that, by symmetry,
  \[ \mathfrak{R}\mathfrak{e} \int_{B (\tilde{d}_c \overrightarrow{e_1}, R)} i
     Q_c \bar{\varphi} = 0, \]
  and we check by Poincar{\'e} inequality that for a function $\varphi$
  satisfying this orthogonality condition, $\varphi = Q_c \psi$,
  \[ \left| \int_{B (\tilde{d}_c \overrightarrow{e_1}, R) \backslash B
     (\tilde{d}_c \overrightarrow{e_1}, R / 2)} \mathfrak{I}\mathfrak{m}
     (\psi) \right| \leqslant K \| \varphi \|_{H^1 (B (\tilde{d}_c
     \overrightarrow{e_1}, R))}, \]
  for a universal constant $K > 0$. By a similar computation as previously, we
  conclude the proof of this lemma.
\end{proof}

We now have all the elements necessary to conclude the proof of Proposition
\ref{CP2prop17}.

\begin{proof}[of Proposition \ref{CP2prop17}]
  This proof follows the proof of Lemma \ref{CP2133L39}. For $\varphi \in
  C^{\infty}_c (\mathbbm{R}^2 \backslash \{ \tilde{d}_c \vec{e}_1, -
  \tilde{d}_c \vec{e}_1 \}, \mathbbm{C})$ and five real-valued parameters
  $\varepsilon_1, \varepsilon_2, \varepsilon_3, \varepsilon_4, \varepsilon_5$
  we define $\varphi^{\ast} = Q_c \psi^{\ast}$ by
  \[ \psi^{\ast} = \psi + \varepsilon_1 \frac{\partial_{x_1} Q_c}{Q_c} +
     \varepsilon_2 \frac{c^2 \partial_c Q_c}{Q_c} + \varepsilon_3
     \frac{\partial_{x_2} Q_c}{Q_c} + \varepsilon_4 \frac{c
     \partial_{c^{\bot}} Q_c}{Q_c} + \varepsilon_5 i. \]
  From Lemma \ref{CP2destrucs}, we check that $\varphi^{\ast} \in
  H_{Q_c}^{\exp}$. Now, similarly as the proof of Lemma \ref{CP2133L39}, we
  check that
  \[ \begin{array}{lll}
       \int_{B (\tilde{d}_c \overrightarrow{e_1}, R)} \mathfrak{R}\mathfrak{e}
       \left( \partial_{x_1} \widetilde{V_1} \overline{\widetilde{V_1}
       \psi^{\ast}} \right) & = & \int_{B (\tilde{d}_c \overrightarrow{e_1},
       R)} \mathfrak{R}\mathfrak{e} \left( \partial_{x_1} \widetilde{V_1}
       \overline{\widetilde{V_1} \psi} \right)\\
       & + & \varepsilon_1 \int_{B (\tilde{d}_c \overrightarrow{e_1}, R)}
       \mathfrak{R}\mathfrak{e} \left( \partial_{x_1} \widetilde{V_1}
       \overline{\partial_{x_1} Q_c \frac{\widetilde{V_1}}{Q_c}} \right)\\
       & + & \varepsilon_2 \int_{B (\tilde{d}_c \overrightarrow{e_1}, R)}
       \mathfrak{R}\mathfrak{e} \left( \partial_{x_1} \widetilde{V_1} c^2
       \overline{\partial_c Q_c \frac{\widetilde{V_1}}{Q_c}} \right)\\
       & + & \varepsilon_3 \int_{B (\tilde{d}_c \overrightarrow{e_1}, R)}
       \mathfrak{R}\mathfrak{e} \left( \partial_{x_1} \widetilde{V_1}
       \overline{\partial_{x_2} Q_c \frac{\widetilde{V_1}}{Q_c}} \right)\\
       & + & \varepsilon_4 \int_{B (\tilde{d}_c \overrightarrow{e_1}, R)}
       \mathfrak{R}\mathfrak{e} \left( \partial_{x_1} \widetilde{V_1} c
       \overline{\partial_{c^{\bot}} Q_c \frac{\widetilde{V_1}}{Q_c}}
       \right)\\
       & + & \varepsilon_5 \int_{B (\tilde{d}_c \overrightarrow{e_1}, R)}
       \mathfrak{R}\mathfrak{e} \left( \partial_{x_1} \widetilde{V_1}
       \overline{i \widetilde{V_1}} \right) .
     \end{array} \]
  Furthermore, with Lemma \ref{lemme3new}, we check that
  \[ \int_{B (\tilde{d}_c \overrightarrow{e_1}, R)} \mathfrak{R}\mathfrak{e}
     \left( \partial_{x_1} V_1 \overline{i \widetilde{V_1}} \right) = 0, \]
  and the other terms are estimated as in the proof of Lemma \ref{CP2133L39}.
  Similarly,
  \[ \int_{B (\tilde{d}_c \overrightarrow{e_1}, R)} \mathfrak{R}\mathfrak{e}
     \left( \partial_{x_2} V_1 \overline{i \widetilde{V_1}} \right) = \int_{B
     (- \tilde{d}_c \overrightarrow{e_1}, R)} \mathfrak{R}\mathfrak{e} \left(
     \partial_{x_1} V_{- 1} \overline{i \widetilde{V_{- 1}}} \right) = \int_{B
     (- \tilde{d}_c \overrightarrow{e_1}, R)} \mathfrak{R}\mathfrak{e} \left(
     \partial_{x_2} V_{- 1} \overline{i \widetilde{V_{- 1}}} \right) = 0. \]
  We also check that, from (\ref{CP2220}), (\ref{CP2221}), Lemmas
  \ref{CP2dcQcsigma} and \ref{CP2a} that
  \begin{eqnarray*}
    &  & \left| \int_{B (0, R)} \mathfrak{R}\mathfrak{e} \left( i
    \frac{\partial_{x_1} Q_c}{Q_c} \right) \right| + \left| \int_{B (0, R)}
    \mathfrak{R}\mathfrak{e} \left( i \frac{\partial_{x_2} Q_c}{Q_c} \right)
    \right|\\
    & + & \left| \int_{B (0, R)} \mathfrak{R}\mathfrak{e} \left( i c^2
    \frac{\partial_c Q_c}{Q_c} \right) \right| + \left| \int_{B (0, R)}
    \mathfrak{R}\mathfrak{e} \left( c i \frac{\partial_{c^{\bot}} Q_c}{Q_c}
    \right) \right|\\
    & = & o_{c \rightarrow 0} (1),
  \end{eqnarray*}
  and
  \[ \int_{B (0, R)} \mathfrak{R}\mathfrak{e} (i \times i) = - \pi R^2 < 0. \]
  We deduce, as in the proof of Lemma \ref{CP2133L39}, that
  \begin{eqnarray*}
    &  & \left(\begin{array}{c}
      \int_{B (\tilde{d}_c \overrightarrow{e_1}, R)} \mathfrak{R}\mathfrak{e}
      \left( \partial_{x_1} \widetilde{V_1} \overline{\widetilde{V_1}
      \psi^{\ast}} \right)\\
      \int_{B (- \tilde{d}_c \overrightarrow{e_1}, R)}
      \mathfrak{R}\mathfrak{e} (\partial_{x_1} \tilde{V}_{- 1}
      \overline{\tilde{V}_{- 1} \psi^{\ast}})\\
      \int_{B (\tilde{d}_c \overrightarrow{e_1}, R)} \mathfrak{R}\mathfrak{e}
      \left( \partial_{x_2} \widetilde{V_1} \overline{\widetilde{V_1}
      \psi^{\ast}} \right)\\
      \int_{B (- \tilde{d}_c \overrightarrow{e_1}, R)}
      \mathfrak{R}\mathfrak{e} (\partial_{x_2} \tilde{V}_{- 1}
      \overline{\tilde{V}_{- 1} \psi^{\ast}})\\
      \mathfrak{R}\mathfrak{e} \int_{B (0, R)} i \psi = 0
    \end{array}\right)\\
    & = & \left( \left(\begin{array}{ccccc}
      K (R) & - K (R) & 0 & 0 & 0\\
      K (R) & K (R) & 0 & 0 & 0\\
      0 & 0 & K (R) & - K (R) & 0\\
      0 & 0 & K (R) & K (R) & 0\\
      0 & 0 & 0 & 0 & - \pi R^2
    \end{array}\right) + o_{c \rightarrow 0} (1) \right)
    \left(\begin{array}{c}
      \varepsilon_1\\
      \varepsilon_2\\
      \varepsilon_3\\
      \varepsilon_4\\
      \varepsilon_5
    \end{array}\right)\\
    & + & o_{c \rightarrow 0} (c^{\beta_0}) K \| \varphi \|_{\mathcal{C}} .
  \end{eqnarray*}
  Therefore, we can find $\varepsilon_1, \varepsilon_2, \varepsilon_3,
  \varepsilon_4, \varepsilon_5 \in \mathbbm{R}$ such that
  \[ | \varepsilon_1 | + | \varepsilon_2 | + | \varepsilon_3 | + |
     \varepsilon_4 | + | \varepsilon_5 | \leqslant o_{c \rightarrow 0}
     (c^{\beta_0}) \| \varphi \|_{\mathcal{C}} \]
  and $\varphi^{\ast}$ satisfies the five orthogonality conditions of Lemma
  \ref{CP2L661510}. Therefore,
  \[ B_{Q_c}^{\exp} (\varphi^{\ast}) \geqslant K (c) \| \varphi^{\ast}
     \|_{H_{Q_c}^{\exp}}^2 + K \| \varphi^{\ast} \|_{\mathcal{C}}^2 . \]
  We continue as in the proof of Lemma \ref{CP2133L39}, and with the same
  arguments, we have
  \[ B_{Q_c}^{\exp} (\varphi) \geqslant K (c) \| \varphi^{\ast}
     \|_{H_{Q_c}^{\exp}}^2 + K \| \varphi \|_{\mathcal{C}}^2 . \]
  Now, by Lemma \ref{CP2destrucs}, we have
  \begin{eqnarray*}
    \| \varphi^{\ast} \|_{H_{Q_c}^{\exp}} & \geqslant & \| \varphi
    \|_{H_{Q_c}^{\exp}} - \| \varepsilon_1 \partial_{x_1} Q_c + \varepsilon_2
    c^2 \partial_c Q_c + \varepsilon_3 \partial_{x_2} Q_c + \varepsilon_4 c
    \partial_{c^{\bot}} Q_c + \varepsilon_5 i \|_{H_{Q_c}^{\exp}}\\
    & \geqslant & \| \varphi \|_{H_{Q_c}^{\exp}} - o_{c \rightarrow 0}
    (c^{\beta_0 / 2}) \| \varphi \|_{\mathcal{C}},
  \end{eqnarray*}
  thus, since we can take $K (c) \leqslant 1$, we have
  \[ B_{Q_c}^{\exp} (\varphi) \geqslant K (c) \| \varphi
     \|_{H_{Q_c}^{\exp}}^2 . \]
  We conclude by density as in the proof of Proposition \ref{CP205218}, thanks
  to Lemma \ref{CP2NNN}. We are left with the proof of $B_{Q_c}^{\exp}
  (\varphi) \leqslant K \| \varphi \|^2_{H^{\exp}_{Q_c}}$. With regards to
  (\ref{CP2Btilda}), the local terms can be estimated by $K \| \varphi \|_{H^1
  (\{ \tilde{r} \leqslant 10 \})}^2 \leqslant K \| \varphi
  \|^2_{H^{\exp}_{Q_c}}$ and the terms at infinity, by Cauchy Schwarz, can be
  estimated by $K \int_{\{ \tilde{r} \geqslant 5 \}} | \nabla \psi |^2
  +\mathfrak{R}\mathfrak{e}^2 (\psi) + \frac{| \psi |^2}{\tilde{r}^2 \ln^2
  (\tilde{r})} \leqslant K \| \varphi \|^2_{H^{\exp}_{Q_c}}$.
\end{proof}

As for the remark above equation (\ref{CP2corto}), we can replace the
orthogonality condition $\mathfrak{R}\mathfrak{e} \int_{B (\tilde{d}_c
\overrightarrow{e_1}, R) \cup B (- \tilde{d}_c \overrightarrow{e_1}, R)}
\partial_c Q_c \overline{Q_c \psi^{\neq 0}} = 0$ by
\begin{equation}
  \mathfrak{R}\mathfrak{e} \int_{B (\tilde{d}_c \overrightarrow{e_1}, R) \cup
  B (- \tilde{d}_c \overrightarrow{e_1}, R)} \partial_d (V_1 (x - d \vec{e}_1)
  V_{- 1} (x + d \vec{e}_1))_{| d = d_c \nobracket} \overline{Q_c \psi^{\neq
  0}} (x) d x = 0 \label{CP2corto2}
\end{equation}
in Propositions \ref{CP2prop16} and \ref{CP2prop17}.

\begin{proof}[of Theorem \ref{CP2p41}]
  This proof follows closely the proof of Proposition \ref{CP2prop17},
  
  First, From Lemma \ref{CP2sym} and the definition of $\partial_{c^{\bot}}
  Q_c$ in Lemma \ref{CP2a}, we check that $\partial_{x_1} Q_c$ and
  $\partial_{c^{\bot}} Q_c$ are odd in $x_1$, and for $\varphi = Q_c \psi \in
  C^{\infty}_c (\mathbbm{R}^2 \backslash \{ \tilde{d}_c \vec{e}_1, -
  \tilde{d}_c \vec{e}_1 \}, \mathbbm{C})$ with $\forall (x_1, x_2) \in
  \mathbbm{R}^2, \varphi (x_1, x_2) = \varphi (- x_1, x_2)$, we check that in
  $B (\tilde{d}_c \overrightarrow{e_1}, R) \cup B (- \tilde{d}_c
  \overrightarrow{e_1}, R)$, $Q_c \psi^{\neq 0}$ is even in $x_1$. Therefore,
  these two orthogonality conditions are freely given.
  
  We decompose as previously for, $\varepsilon_1, \varepsilon_2,
  \varepsilon_3$ three real-valued parameters,
  \[ \varphi = \varphi^{\ast} + \varepsilon_1 i Q_c + \varepsilon_2
     \partial_{x_2} Q_c + \varepsilon_3 c^2 \partial_c Q_c . \]
  We suppose that
  \[ \mathfrak{R}\mathfrak{e} \int_{B (\tilde{d}_c \overrightarrow{e_1}, R)
     \cup B (- \tilde{d}_c \overrightarrow{e_1}, R)} \partial_c Q_c
     \bar{\varphi} =\mathfrak{R}\mathfrak{e} \int_{B (\tilde{d}_c
     \overrightarrow{e_1}, R) \cup B (- \tilde{d}_c \overrightarrow{e_1}, R)}
     \partial_{x_2} Q_c \bar{\varphi} = 0, \]
  \[ \mathfrak{R}\mathfrak{e} \int_{B (\tilde{d}_c \overrightarrow{e_1}, R)
     \cup B (- \tilde{d}_c \overrightarrow{e_1}, R)} i Q_c \bar{\varphi} = 0,
  \]
  and we show, as in the proof of Lemma \ref{CP2133L39}, that we can find
  $\varepsilon_1, \varepsilon_2, \varepsilon_3 \in \mathbbm{R}$ such that
  \[ | \varepsilon_1 | + | \varepsilon_2 | + | \varepsilon_3 | \leqslant o_{c
     \rightarrow 0} (c^{\beta_0}) \| \varphi \|_{H_{Q_c}^{\exp}}, \]
  and $\varphi^{\ast}$ satisfies the five orthogononality conditions of Lemma
  \ref{CP2L661510} (we recall that two of them are given by symmetry). Here,
  since we did not remove the $0$-harmonics, the error is only controlled by
  $\| \varphi \|_{H_{Q_c}^{\exp}}$ instead of $\| \varphi \|_{\mathcal{C}}$.
  For instance, we have
  \[ \int_{B (\tilde{d}_c \overrightarrow{e_1}, R)} \left|
     \mathfrak{R}\mathfrak{e} \left( \left( \partial_{x_2} \widetilde{V_1}
     \overline{\widetilde{V_1}} - \partial_{x_2} Q_c \overline{Q_c} \right)
     \psi \right) \right| \leqslant o_{c \rightarrow 0} (1) \| Q_c \psi
     \|_{L^2 (B (\tilde{d}_c \overrightarrow{e_1}, R))} = o_{c \rightarrow 0}
     (1) \| \varphi \|_{H_{Q_c}^{\exp}} . \]
  Now, from Lemma \ref{CP2L661510}, since $\varphi^{\ast} \in H_{Q_c}^{\exp}$,
  we have
  \[ B^{\exp}_{Q_c} (\varphi^{\ast}) \geqslant K \| \varphi^{\ast}
     \|_{H_{Q_c}^{\exp}}^2 . \]
  We continue, as in the proof of Lemma \ref{CP2133L39}, with $| \varepsilon_1
  | + | \varepsilon_2 | + | \varepsilon_3 | = o_{c \rightarrow 0} (1) \|
  \varphi \|_{H_{Q_c}^{\exp}}$ and Lemma \ref{CP2destrucs}. We show that
  \[ B^{\exp}_{Q_c} (\varphi) \geqslant K \| \varphi \|_{H_{Q_c}^{\exp}}^2 .
  \]
  We conclude the proof of Theorem \ref{CP2p41} by density.
\end{proof}

\section{Local uniqueness result}\label{CP2ss42}

This section is devoted to the proof of Theorem \ref{CP2th16}. This proof will
follow classical schemes for local uniqueness using the coercivity. Here, we
will use Propositions \ref{CP2prop16} and \ref{CP2prop17}, with the remark
(\ref{CP2corto2}).

\subsection{Construction of a perturbation}

For a given $\vec{c}' \in \mathbbm{R}^2$, $0 < | \vec{c}' | \leqslant c_0$
($c_0$ defined in Theorem \ref{th1}), $X \in \mathbbm{R}^2$ and $\gamma \in
\mathbbm{R}$, we define, thanks to (\ref{CP2rotQc}), the travelling wave
\begin{equation}
  Q \assign Q_{\vec{c}'} (. - X) e^{i \gamma} . \label{CP2280542}
\end{equation}
We define a smooth cutoff function $\eta$, with value $0$ in $B (\pm
\tilde{d}_c \overrightarrow{e_1}, R + 1)$ ($R > 10$ is defined in Theorem
\ref{CP2th2}), and $1$ outside of $B (\tilde{d}_c \overrightarrow{e_1}, R + 2)
\cup B (- \tilde{d}_c \overrightarrow{e_1}, R + 2)$. The first step is to
define a function $\psi$ such that
\begin{equation}
  (1 - \eta) Q \psi + \eta Q (e^{\psi} - 1) = Z - Q, \label{CP2422805}
\end{equation}
with $Q \psi$ satisfying the orthogonality conditions of Propositions
\ref{CP2prop16} and \ref{CP2prop17}. We start by showing that there exists a
function $\psi$ solution of (\ref{CP2422805}). We denote $\delta^{| . |} (c
\overrightarrow{e_2}, \vec{c}') \assign | c \overrightarrow{e_2} . \nobracket
\frac{\vec{c}'}{| \vec{c}' |} - \nobracket \vec{c}' |$ and $\delta^{\bot} (c
\overrightarrow{e_2}, \vec{c}') \assign | c \overrightarrow{e_2} . \nobracket
\frac{\vec{c}^{\prime \bot}}{| \vec{c}' |} - \nobracket \vec{c}' |$. At fixed
$c$, these two quantities characterize $\vec{c}'$. We will use them as
variables instead of $\vec{c}'$, this decomposition being well adapted to the
problem.

Since both $Z$ and $| Q |$ go to $1$ at infinity, we have that such a function
$\psi$ is bounded at infinity.

The perturbation here is chosen additively close to the zeros of the
travelling wave, and multiplicatively at infinity. This seems to be a fit form
for the perturbation, and we have already used it in the construction of
$Q_c$.

\begin{lemma}
  \label{CP2L422905}There exits $c_0 > 0$ such that, for $0 < c < c_0$ and any
  $\Lambda > \frac{10}{c}$, with $Z$ a function satisfying the hypothesis of
  Theorem \ref{CP2th16} and $Q$ defined by (\ref{CP2280542}) with $\frac{c}{2}
  \leqslant | \vec{c}' | \leqslant 2 c$, there exist $K, K (\Lambda) > 0$ such
  that
  \[ \| Z - Q \|_{C^1 (B (0, \Lambda))} \leqslant K (\Lambda) \| Z - Q_c
     \|_{H^{\exp}_{Q_c}} + K \left( | X | + \frac{\delta^{| . |} (c
     \overrightarrow{e_2}, \vec{c}')}{c^2} + \frac{\delta^{\bot} (\nobracket c
     \overrightarrow{e_2}, \vec{c}') \nobracket}{c} + | \gamma | \right) . \]
\end{lemma}

We will mainly use this result for $\Lambda = \lambda + 1$, $\lambda > 0$
defined in Theorem \ref{CP2th16}.

\begin{proof}
  We recall that such a function $Z$ is in $C^{\infty} (\mathbbm{R}^2,
  \mathbbm{C})$ by elliptic regularity.
  
  We start with the estimation of $w \assign Q_c - Z$ in $B (0, \Lambda)$.
  Since both $Z$ and $Q_c$ solve $(\tmop{TW}_c)$, we have
  \[ - \Delta w = (1 - | Q_c |^2) Q_c - (1 - | Z |^2) Z + i c \partial_{x_2}
     w. \]
  From Theorem 8.8 of {\cite{MR1814364}}, we have that for $x \in
  \mathbbm{R}^2$, $\Omega \assign B (0, \Lambda)$, $2 \Omega = B (0, 2
  \Lambda)$,
  \[ \| w \|_{W^{2, 2} (\Omega)} \leqslant K (\Lambda) (\| w \|_{H^1 (\Omega)}
     + \| i c \partial_{x_2} w + (1 - | Q_c |^2) Q_c - (1 - | Z |^2) Z \|_{L^2
     (2 \Omega)}) . \]
  We compute that
  \[ (1 - | Q_c |^2) Q_c - (1 - | Z |^2) Z = (Q_c - Z) (1 - | Q_c |^2) + Z (|
     Q_c | - | Z |) (| Q_c | + | Z |) . \]
  From {\cite{MR2001707}}, we have that any travelling wave of finite energy
  is bounded in $L^{\infty} (\mathbbm{R}^2)$ by a universal constant, i.e.
  \begin{equation}
    | Q_c | + | Z | \leqslant K, \label{CP2591806}
  \end{equation}
  therefore
  \[ | 1 - | Q_c |^2 | + | Z |  (| Q_c | + | Z |) \leqslant K \]
  for a universal constant $K$. Thus,
  \[ \| (1 - | Q_c |^2) Q_c - (1 - | Z |^2) Z \|_{L^2 (2 \Omega)} \leqslant K
     \| w \|_{L^2 (2 \Omega)}, \]
  and we deduce, from Lemma \ref{CP2sensmanqu}, that
  \[ \| w \|_{W^{2, 2} (\Omega)} \leqslant K (\Lambda) (\| w \|_{H^1 (2
     \Omega)} + \| i c \partial_{x_2} w \|_{L^2 (2 \Omega)} + \| w \|_{L^2 (2
     \Omega)}) \leqslant K (\Lambda) \| w \|_{H^{\exp}_{Q_c}} . \]
  By standard elliptic arguments, we have that for every $k \geqslant 2$,
  \[ \| w \|_{W^{k, 2} (\Omega)} \leqslant K (\Lambda, k) \| w
     \|_{H^{\exp}_{Q_c}} . \]
  By Sobolev embeddings, we estimate
  \begin{equation}
    \| w \|_{C^1 (\Omega)} \leqslant K (\Lambda) \| w \|_{W^{4, 2} (\Omega)}
    \leqslant K (\Lambda) \| w \|_{H^{\exp}_{Q_c}} . \label{CP25918062}
  \end{equation}
  From (\ref{CP25918062}), we have
  \[ \| Z - Q \|_{L^{\infty} (\Omega)} \leqslant \| Q - Q_c \|_{L^{\infty}
     (\Omega)} + \| w \|_{L^{\infty} (\Omega)} \leqslant \| Q - Q_c
     \|_{L^{\infty} (\mathbbm{R}^2)} + K (\Lambda) \| w \|_{H^{\exp}_{Q_c}} .
  \]
  We estimate
  \begin{eqnarray*}
    \| Q - Q_c \|_{L^{\infty} (\mathbbm{R}^2)} & = & \| Q_{\vec{c}'} (. - X)
    e^{i \gamma} - Q_c \|_{L^{\infty} (\mathbbm{R}^2)}\\
    & \leqslant & \| Q_{\vec{c}'} (. - X) e^{i \gamma} - Q_{\vec{c}'} (. - X)
    \|_{L^{\infty} (\mathbbm{R}^2)} + \| Q_{\vec{c}'} (. - X) - Q_{\vec{c}'}
    \|_{L^{\infty} (\mathbbm{R}^2)}\\
    & + & \| Q_{\vec{c}'} - Q_{| \vec{c}' | \vec{e}_2} \|_{L^{\infty}
    (\mathbbm{R}^2)} + \| Q_{| \vec{c}' | \vec{e}_2} - Q_c \|_{L^{\infty}
    (\mathbbm{R}^2)} .
  \end{eqnarray*}
  We check, with Theorem \ref{th1} and Lemma \ref{CP2a} that $\| \nabla Q
  \|_{L^{\infty} (\mathbbm{R}^2)} + c^2 \| \partial_c Q \|_{L^{\infty}
  (\mathbbm{R}^2)} + c \| \partial_{c^{\bot}} Q \|_{L^{\infty}
  (\mathbbm{R}^2)} + \| i Q \|_{L^{\infty} (\mathbbm{R}^2)} \leqslant K$, and
  that it also holds for any travelling wave of the form $Q_{\vec{\varsigma}}
  (. - Y) e^{i \beta}$ if $2 c \geqslant | \vec{\varsigma} | \geqslant c / 2,
  Y \in \mathbbm{R}^2$ and $\beta \in \mathbbm{R}$.
  
  We check that $\| Q_{\vec{c}'} (. - X) e^{i \gamma} - Q_{\vec{c}'} (. - X)
  \|_{L^{\infty} (\mathbbm{R}^2)} \leqslant | e^{i \gamma} - 1 | \|
  Q_{\vec{c}'} (. - X) \|_{L^{\infty} (\mathbbm{R}^2)} \leqslant K | \gamma
  |,$ and we estimate (by the mean value theorem)
  \[ \| Q_{\vec{c}'} (. - X) - Q_{\vec{c}'} \|_{L^{\infty} (\mathbbm{R}^2)}
     \leqslant K | X | \| \nabla Q_{\vec{c}'} \|_{L^{\infty} (\mathbbm{R}^2)}
     \leqslant K | X | . \]
  Similarly, we have
  \[ \| Q_{\vec{c}'} - Q_{| \vec{c}' | \vec{e}_2} \|_{L^{\infty}
     (\mathbbm{R}^2)} \leqslant K \frac{\delta^{\bot} (\nobracket c
     \overrightarrow{e_2}, \vec{c}') + \delta^{| . |} (c \overrightarrow{e_2},
     \vec{c}') \nobracket}{c} \]
  and $\| Q_{| \vec{c}' | \vec{e}_2} - Q_c \|_{L^{\infty} (\mathbbm{R}^2)}
  \leqslant K \frac{\delta^{| . |} (c \overrightarrow{e_2}, \vec{c}')}{c^2}$.
  We deduce that (since $c \leqslant 1$)
  \begin{equation}
    \| Q - Q_c \|_{L^{\infty} (\mathbbm{R}^2)} \leqslant K \left( | X | +
    \frac{\delta^{| . |} (c \overrightarrow{e_2}, \vec{c}')}{c^2} +
    \frac{\delta^{\bot} (\nobracket c \overrightarrow{e_2}, \vec{c}')
    \nobracket}{c} + | \gamma | \right), \label{CP2471510}
  \end{equation}
  and thus
  \begin{equation}
    \| Z - Q \|_{L^{\infty} (B (0, \Lambda))} \leqslant K (\Lambda) \| Z - Q_c
    \|_{H^{\exp}_{Q_c}} + K \left( | X | + \frac{\delta^{| . |} (c
    \overrightarrow{e_2}, \vec{c}')}{c^2} + \frac{\delta^{\bot} (\nobracket c
    \overrightarrow{e_2}, \vec{c}') \nobracket}{c} + | \gamma | \right) .
    \label{CP24102905}
  \end{equation}
  Finally, from Lemmas \ref{lemme3new}, \ref{CP283L33} and \ref{CP2dcQcsigma},
  $\partial_{c^{\bot}} Q_c = - x^{\bot} . \nabla Q_c$ and equation
  (\ref{CP2222}), we have
  \[ \| \nabla \partial_{x_2} Q \|_{L^{\infty} (\mathbbm{R}^2)} + c^2 \|
     \nabla \partial_c Q \|_{L^{\infty} (\mathbbm{R}^2)} + c \| \nabla
     \partial_{c^{\bot}} Q \|_{L^{\infty} (\mathbbm{R}^2)} + \| i \nabla Q_c
     \|_{L^{\infty} (\mathbbm{R}^2)} \leqslant K. \]
  We deduce that
  \[ \| \nabla (Q - Q_c) \|_{L^{\infty} (\mathbbm{R}^2)} \leqslant K \left( |
     X | + \frac{\delta^{| . |} (c \overrightarrow{e_2}, \vec{c}')}{c^2} +
     \frac{\delta^{\bot} (\nobracket c \overrightarrow{e_2}, \vec{c}')
     \nobracket}{c} + | \gamma | \right), \]
  and, by (\ref{CP25918062}),
  \[ \| \nabla (Z - Q) \|_{L^{\infty} (B (0, \Lambda))} \leqslant K (\Lambda)
     \| Z - Q_c \|_{H^{\exp}_{Q_c}} + K \left( | X | + \frac{\delta^{| . |} (c
     \overrightarrow{e_2}, \vec{c}')}{c^2} + \frac{\delta^{\bot} (\nobracket c
     \overrightarrow{e_2}, \vec{c}') \nobracket}{c} + | \gamma | \right) . \]
  
\end{proof}

\begin{lemma}
  \label{CP2L432905}There exists $\varepsilon_0 (c) > 0$ small such that, for
  $Z$ a function satisfying the hyptothesis of Theorem \ref{CP2th16} with
  \[ | X | + \frac{\delta^{| . |} (c \overrightarrow{e_2}, \vec{c}')}{c^2} +
     \frac{\delta^{\bot} (c \overrightarrow{e_2}, \vec{c}')}{c} + | \gamma |
     \leqslant \varepsilon_0 (c), \]
  there exists a function $Q \psi \in C^1 (\mathbbm{R}^2, \mathbbm{C})$ such
  that (\ref{CP2422805}) holds. Furthermore, for any $\Lambda > \frac{10}{c}$,
  there exists $K, K (\Lambda) > 0$ such that
  \[ \| Q \psi \|_{C^1 (B (0, \Lambda))} \leqslant K (\Lambda) \| Z - Q_c
     \|_{H^{\exp}_{Q_c}} + K \left( | X | + \frac{\delta^{| . |} (c
     \overrightarrow{e_2}, \vec{c}')}{c^2} + \frac{\delta^{\bot} (c
     \overrightarrow{e_2}, \vec{c}')}{c} + | \gamma | \right) . \]
\end{lemma}

\begin{proof}
  First, taking $\varepsilon_0 (c)$ small enough (depending on $c$), we check
  that $\frac{c}{2} \leqslant | \vec{c}' | \leqslant 2 c$.
  
  We recall equation (\ref{CP2422805}):
  \[ (1 - \eta) Q \psi + \eta Q (e^{\psi} - 1) = Z - Q. \]
  We write it in the form
  \[ \psi + \eta (e^{\psi} - 1 - \psi) = \frac{Z - Q}{Q}, \]
  and in $\{ \eta = 0 \}$, we therefore define
  \begin{equation}
    \psi = \frac{Z - Q}{Q} \label{CP24112905} .
  \end{equation}
  Now, we define the set $\Omega \assign B (0, \lambda + 1) \backslash (B (d_c
  \vec{e}_1, R - 1) \cup B (- d_c \vec{e}_1, R - 1))$. In this set, we have
  that
  \[ \left\| \frac{Z - Q}{Q} \right\|_{C^1 (\Omega)} \leqslant K \varepsilon_0
     (c) + K (\lambda) \| Z - Q_c \|_{H^{\exp}_{Q_c}} \]
  by Lemma \ref{CP2L422905} and (\ref{CP2Qcpaszero}). Therefore, since
  $e^{\psi} - 1 - \psi$ is at least quadratic in $\psi \in C^1 (\Omega,
  \mathbbm{C})$, by a fixed point argument (on $H (\psi) \assign \frac{Z -
  Q}{Q} - \eta (e^{\psi} - 1 - \psi)$, which is a contraction on $\| \psi
  \|_{L^{\infty} (\{ \eta \neq 0 \})} < \mu$ for $\mu > 0$ small enough), we
  deduce that on $\Omega$, given that $\varepsilon_0$ and $\| Z - Q_c
  \|_{H^{\exp}_{Q_c}}$ are small enough (depending on $\lambda$ for $\| Z -
  Q_c \|_{H^{\exp}_{Q_c}}$), there exists a unique function $\psi \in C^1
  (\Omega, \mathbbm{C})$ such that $\psi + \eta (e^{\psi} - 1 - \psi) =
  \frac{Z - Q}{Q}$ in $\Omega$. By unicity, since we have a solution of the
  same problem on $\{ \eta = 0 \}$ which intersect $\Omega$, we can construct
  $Q \psi \in C^1 (B (0, \lambda + 1), \mathbbm{C})$ such that $\eta Q \psi +
  (1 - \eta) Q (e^{\psi} - 1) = Z - Q$ in $B (0, \lambda + 1)$.
  
  Furthermore, we use here the hypothesis that, outside of $B (0, \lambda)$, \
  $| Z - Q_c | \leqslant \mu_0$. We deduce that (taking $\mu_0 < \frac{1}{4}$)
  there exists $\delta > 0$ such that $| Z | > \delta$ outside of $B (0,
  \lambda)$. In particular, since $\lambda$ can be taken large, we have that
  outside of $B (0, \lambda)$, $\eta = 1$. The equation on $\psi$ then becomes
  \[ e^{\psi} = \frac{Z}{Q}, \]
  and by equation (\ref{CP2Qcpaszero}) and $| Z | > \delta$, we deduce that
  there exists a unique solution to this problem in $C^1 (\mathbbm{R}^2
  \backslash B (0, \lambda), \mathbbm{C})$ that is equal on $B (0, \lambda +
  1) \backslash B (0, \lambda)$ to the previously constructed function $\psi$.
  
  Therefore, there exists $Q \psi \in C^1 (\mathbbm{R}^2, \mathbbm{C})$ such
  that $(1 - \eta) Q \psi + \eta Q (e^{\psi} - 1) = Z - Q$ in $\mathbbm{R}^2$.
  Furthermore, we check that (by the fixed point argument), since $\{ \eta
  \neq 1 \} \subset B (0, \lambda)$,
  \begin{eqnarray*}
    \| \psi \|_{C^1 (\{ \eta \neq 1 \})} & \leqslant & K_{} \left\| \frac{Z -
    Q}{Q} \right\|_{C^1 (\{ \eta \neq 1 \})}\\
    & \leqslant & K (\lambda) \| Z - Q_c \|_{H^{\exp}_{Q_c}} + K \left( | X |
    + \frac{\delta^{| . |} (c \overrightarrow{e_2}, \vec{c}')}{c^2} +
    \frac{\delta^{\bot} (c \overrightarrow{e_2}, \vec{c}')}{c} + | \gamma |
    \right) .
  \end{eqnarray*}
  From equation (\ref{CP2Qcpaszero}) and Lemma \ref{CP2L422905}, we have
  \begin{eqnarray*}
    \| Q \psi \|_{C^1 (B (0, \Lambda))} & \leqslant & \| Z - Q \|_{C^1 (B (0,
    \Lambda))} + K \| \psi \|_{C^1 (\{ \eta \neq 1 \})} + K (\Lambda) \| Z -
    Q_c \|_{H^{\exp}_{Q_c}}\\
    & \leqslant & K (\Lambda) \| Z - Q_c \|_{H^{\exp}_{Q_c}} + K \left( | X |
    + \frac{\delta^{| . |} (c \overrightarrow{e_2}, \vec{c}')}{c^2} +
    \frac{\delta^{\bot} (c \overrightarrow{e_2}, \vec{c}')}{c} + | \gamma |
    \right) .
  \end{eqnarray*}
  This concludes the proof of the lemma.
\end{proof}

\begin{lemma}
  \label{CP2L531806}The functions Q and $\psi$, defined respectively in
  (\ref{CP2280542}) and Lemma \ref{CP2L432905}, satisfy
  \[ \varphi \assign Q \psi \in H^{\exp}_Q . \]
  Furthermore, $\varphi \in C^2 (\mathbbm{R}^2, \mathbbm{C})$ and there exists
  $K \left( \lambda, c, \| Z - Q_c \|_{H^{\exp}_{Q_c}}, \varepsilon_0, Z
  \right) > 0$ such that, in $\{ \eta = 1 \}$ (i.e. far from the vortices),
  \[ | \nabla \psi | + | \mathfrak{R}\mathfrak{e} (\psi) | + | \Delta \psi |
     \leqslant \frac{K \left( \lambda, c, \| Z - Q_c \|_{H^{\exp}_{Q_c}},
     \varepsilon_0, Z \right)}{(1 + r)^2}, \]
  \[ | \nabla \mathfrak{R}\mathfrak{e} (\psi) | \leqslant \frac{K \left(
     \lambda, c, \| Z - Q_c \|_{H^{\exp}_{Q_c}}, \varepsilon_0, Z \right)}{(1
     + r)^3} \]
  and
  \[ | \mathfrak{I}\mathfrak{m} (\psi + i \gamma) | \leqslant \frac{K \left(
     \lambda, c, \| Z - Q_c \|_{H^{\exp}_{Q_c}}, \varepsilon_0, Z \right)}{(1
     + r)} . \]
\end{lemma}

\begin{proof}
  From Lemma \ref{CP2L432905}, for any $\Lambda > \frac{10}{c}$,
  \begin{equation}
    \| Q \psi \|_{C^1 (B (0, \Lambda))} \leqslant K (\Lambda) \| Z - Q_c
    \|_{H^{\exp}_{Q_c}} + K \left( | X | + \frac{\delta^{| . |} (c
    \overrightarrow{e_2}, \vec{c}')}{c^2} + \frac{\delta^{\bot} (c
    \overrightarrow{e_2}, \vec{c}')}{c} + | \gamma | \right),
    \label{CP25121806}
  \end{equation}
  therefore, we only have to check the integrability at infinity of $Q \psi$
  to show that $\varphi = Q \psi \in H^{\exp}_Q .$ In $\{ \eta = 1 \}$, we
  have
  \[ e^{\psi} = \frac{Z}{Q} . \]
  We have shown in the proof of Lemma \ref{CP2L432905} that $K > \left|
  \frac{Z}{Q} \right| > \delta / 2$ outside of $B (0, \lambda)$ for some
  $\delta > 0$, and together with (\ref{CP25121806}), we check that
  \begin{equation}
    \| \psi \|_{C^0 (\{ \eta = 1 \})} \leqslant K \left( \lambda, \| Z - Q_c
    \|_{H^{\exp}_{Q_c}}, \varepsilon_0 \right) . \label{CP2psiinfini}
  \end{equation}
  This implies that
  \[ \int_{\{ \eta = 1 \}} \frac{| Q \psi |^2}{\tilde{r}^2 \nobracket \ln
     (\tilde{r} \nobracket)^2} < + \infty . \]
  Similarly, we check that, in $\{ \eta = 1 \}$, since $e^{\psi} =
  \frac{Z}{Q}$,
  \[ \nabla \psi = \frac{e^{- \psi}}{Q} \nabla (Z - Q) - \frac{\nabla Q}{Q} (1
     - e^{- \psi}), \]
  therefore
  \begin{equation}
    | \nabla \psi | \leqslant K \left( \lambda, \| Z - Q_c
    \|_{H^{\exp}_{Q_c}}, \varepsilon_0 \right) (| \nabla (Z - Q) | + | \nabla
    Q |) \label{CP2491110} .
  \end{equation}
  From Theorem \ref{CP2Qcbehav}, we have
  \[ | \nabla Z | + | \nabla Q | \leqslant \frac{K (c, Z)}{(1 + r)^2}, \]
  therefore,
  \[ \int_{\{ \eta = 1 \}} | \nabla Q |^2 | \psi |^2 < + \infty \]
  and
  \begin{eqnarray*}
    \int_{\{ \eta = 1 \}} | \nabla (Z - Q) |^2 & \leqslant & \int_{\{ \eta = 1
    \}} \frac{K (c, Z)}{(1 + r)^4} < + \infty .
  \end{eqnarray*}
  We deduce that $\int_{\{ \eta = 1 \}} | \nabla \psi |^2 < + \infty$, and,
  furthermore, equation (\ref{CP2491110}) shows that
  \[ | \nabla \psi | \leqslant \frac{K \left( \lambda, c, \| Z - Q_c
     \|_{H^{\exp}_{Q_c}}, \varepsilon_0, Z \right)}{(1 + r)^2} \]
  in $\{ \eta = 1 \}$.
  
  Now, still in $\{ \eta = 1 \}$, we have
  \[ Q e^{\psi} = Z, \]
  we deduce that $Q e^{- i \gamma} (e^{\psi + i \gamma} - 1) = Z - Q e^{- i
  \gamma}$. Now, we recall that $\| \psi \|_{C^0 (\{ \eta = 1 \})} \leqslant K
  \left( \lambda, \| Z - Q_c \|_{H^{\exp}_{Q_c}}, \varepsilon_0 \right)$, thus
  $| \mathfrak{R}\mathfrak{e} (e^{\psi + i \gamma} - 1 - (\psi + i \gamma)) |
  \leqslant K \left( \lambda, \| Z - Q_c \|_{H^{\exp}_{Q_c}}, \varepsilon_0
  \right) | \mathfrak{R}\mathfrak{e} (e^{\psi + i \gamma} - 1) |$. We deduce
  from this, with (\ref{CP25121806}) that, in $\{ \eta = 1 \}$, with
  $\frac{1}{4} \| \psi + i \gamma \|_{L^{\infty} (\mathbbm{R}^2)} \leqslant |
  \mathfrak{R}\mathfrak{e} (e^{\psi + i \gamma} - 1) | \leqslant K \| \psi + i
  \gamma \|_{L^{\infty} (\mathbbm{R}^2)}$,
  \begin{eqnarray*}
    | \mathfrak{R}\mathfrak{e} (\psi) | & = & | \mathfrak{R}\mathfrak{e} (\psi
    + i \gamma) |\\
    & \leqslant & | \mathfrak{R}\mathfrak{e} (e^{\psi + i \gamma} - 1) | + |
    \mathfrak{R}\mathfrak{e} (e^{\psi + i \gamma} - 1 - (\psi + i \gamma)) |\\
    & \leqslant & K \left( \lambda, \| Z - Q_c \|_{H^{\exp}_{Q_c}},
    \varepsilon_0 \right) | \mathfrak{R}\mathfrak{e} (e^{\psi + i \gamma} - 1)
    |\\
    & \leqslant & K \left( \lambda, \| Z - Q_c \|_{H^{\exp}_{Q_c}},
    \varepsilon_0 \right) \left| \mathfrak{R}\mathfrak{e} \left( \frac{(Z - Q
    e^{- i \gamma} ) \overline{Q e^{i \gamma}}}{| Q |^2} \right) \right|\\
    & \leqslant & K \left( \lambda, \| Z - Q_c \|_{H^{\exp}_{Q_c}},
    \varepsilon_0 \right) (| \mathfrak{R}\mathfrak{e} (Z - Q e^{- i \gamma} )
    | + | \mathfrak{I}\mathfrak{m} (Z - Q e^{- i \gamma})
    \mathfrak{I}\mathfrak{m} (Q e^{i \gamma} - 1) |) .
  \end{eqnarray*}
  From Theorem \ref{CP2Qcbehav},
  \[ | \mathfrak{R}\mathfrak{e} (Z - Q e^{- i \gamma} ) | \leqslant |
     \mathfrak{R}\mathfrak{e} (\nobracket Z - 1) \nobracket | + |
     \mathfrak{R}\mathfrak{e} (\nobracket 1 - Q e^{- i \gamma} ) \nobracket |
     \leqslant \frac{K (c, Z)}{(1 + r)^2} \]
  and
  \[ | \mathfrak{I}\mathfrak{m} (Z - Q e^{- i \gamma})
     \mathfrak{I}\mathfrak{m} (Q e^{i \gamma} - 1) | \leqslant \frac{K (c,
     Z)}{(1 + r)^2} . \]
  We conclude that, in $\{ \eta = 1 \}$, we have $| \mathfrak{R}\mathfrak{e}
  (\psi) | \leqslant \frac{K \left( \lambda, c, \| Z - Q_c
  \|_{H^{\exp}_{Q_c}}, \varepsilon_0, Z \right)}{(1 + r)^2}$ hence
  \[ \int_{\{ \eta = 1 \}} \mathfrak{R}\mathfrak{e}^2 (\psi) < + \infty . \]
  This conclude the proof of $\varphi = Q \psi \in H^{\exp}_{Q_c}$. We are
  left with the proof of the following estimates, $| \Delta \psi | \leqslant
  \frac{K \left( \lambda, c, \| Z - Q_c \|_{H^{\exp}_{Q_c}}, \varepsilon_0, Z
  \right)}{(1 + r)^2}$, $| \mathfrak{I}\mathfrak{m} (\psi + i \gamma) |
  \leqslant \frac{K \left( \lambda, c, \| Z - Q_c \|_{H^{\exp}_{Q_c}},
  \varepsilon_0, Z \right)}{(1 + r)}$ and $| \mathfrak{R}\mathfrak{e} (\nabla
  \psi) | \leqslant \frac{K \left( \lambda, c, \| Z - Q_c \|_{H^{\exp}_{Q_c}},
  \varepsilon_0, Z \right)}{(1 + r)^3}$ in $\{ \eta = 1 \}$.
  
  We recall that, in $\{ \eta = 1 \}$, $\nabla \psi = \frac{e^{- \psi}}{Q}
  \nabla (Z - Q) - \frac{\nabla Q}{Q} (1 - e^{- \psi})$, from which we
  compute, by differentiating a second time,
  \begin{eqnarray*}
    \Delta \psi & = & - \frac{\nabla \psi . \nabla (Z - Q)}{Q} e^{- \psi} -
    \frac{\nabla Q}{Q} e^{- \psi} . \nabla (Z - Q) + \frac{e^{- \psi}}{Q}
    \Delta (Z - Q)\\
    & - & \frac{\Delta Q}{Q} (1 - e^{- \psi}) + \frac{\nabla Q. \nabla
    Q}{Q^2} (1 - e^{- \psi}) - \frac{\nabla Q}{Q} . \nabla \psi e^{- \psi} .
  \end{eqnarray*}
  Using Theorem \ref{CP2Qcbehav}, $\Delta Q = - i \vec{c}' . \nabla Q - (1 - |
  Q |^2) Q$, $Z = - i c \partial_{x_2} Z - (1 - | Z |^2) Z$ and previous
  estimates on $\psi$, we check that, in $\{ \eta = 1 \}$,
  \[ | \Delta \psi | \leqslant \frac{K \left( \lambda, c, \| Z - Q_c
     \|_{H^{\exp}_{Q_c}}, \varepsilon_0, Z \right)}{(1 + r)^2} . \]
  We have $Q e^{- i \gamma} (e^{\psi + i \gamma} - 1) = Z - Q e^{- i \gamma}$
  in $\{ \eta = 1 \}$, therefore
  \[ e^{\psi + i \gamma} - 1 = \frac{Z}{Q e^{- i \gamma}} - 1 \]
  We check, since $\| \psi \|_{C^0 (\{ \eta = 1 \})} \leqslant K \left(
  \lambda, \| Z - Q_c \|_{H^{\exp}_{Q_c}}, \varepsilon_0 \right)$, that we
  have by Theorem \ref{CP2Qcbehav}
  \begin{eqnarray*}
    | \mathfrak{I}\mathfrak{m} (\psi + i \gamma) | & \leqslant & K \left(
    \lambda, \| Z - Q_c \|_{H^{\exp}_{Q_c}}, \varepsilon_0 \right) |
    \mathfrak{I}\mathfrak{m} (e^{\psi + i \gamma} - 1) |\\
    & \leqslant & K \left( \lambda, \| Z - Q_c \|_{H^{\exp}_{Q_c}},
    \varepsilon_0 \right) \left| \frac{Z}{\tmop{Qe}^{- i \gamma}} - 1
    \right|\\
    & \leqslant & \frac{K \left( \lambda, c, \| Z - Q_c \|_{H^{\exp}_{Q_c}},
    \varepsilon_0, Z \right)}{(1 + r)} .
  \end{eqnarray*}
  Finally, since $\nabla \psi = \frac{e^{- \psi}}{Q} \nabla (Z - Q) -
  \frac{\nabla Q}{Q} (1 - e^{- \psi}) = \frac{\nabla Z}{Q} e^{- \psi} -
  \frac{\nabla Q}{Q}$, we check with Theorem \ref{CP2Qcbehav} that, in $\{
  \eta = 1 \}$,
  \begin{eqnarray*}
    | \nabla \mathfrak{R}\mathfrak{e} (\psi) | & \leqslant & \left|
    \mathfrak{R}\mathfrak{e} \left( \frac{\nabla Z}{Q} e^{- \psi} \right)
    \right| + \left| \mathfrak{R}\mathfrak{e} \left( \frac{\nabla Q}{Q}
    \right) \right|\\
    & \leqslant & \left| \mathfrak{R}\mathfrak{e} \left( \nabla Z \bar{Z}
    \frac{e^{- \psi}}{Q \bar{Z}} \right) \right| + \frac{|
    \mathfrak{R}\mathfrak{e} (\nabla Q \bar{Q}) |}{| Q |^2}\\
    & \leqslant & \left| \mathfrak{I}\mathfrak{m} (\nabla Z \bar{Z})
    \mathfrak{I}\mathfrak{m} \left( \frac{e^{- \psi}}{Q \bar{Z}} \right)
    \right| + | \mathfrak{R}\mathfrak{e} (\nabla Z \bar{Z}) | \left|
    \mathfrak{R}\mathfrak{e} \left( \frac{e^{- \psi}}{Q \bar{Z}} \right)
    \right| + \frac{| \nobracket \nabla (| Q |^2 \nobracket) |}{2 | Q |^2}\\
    & \leqslant & \frac{K \left( \lambda, c, \| Z - Q_c \|_{H^{\exp}_{Q_c}},
    \varepsilon_0, Z \right)}{(1 + r)^2} \left| \mathfrak{I}\mathfrak{m}
    \left( \frac{e^{- \psi}}{Q \bar{Z}} \right) \right| + \frac{K \left(
    \lambda, c, \| Z - Q_c \|_{H^{\exp}_{Q_c}}, \varepsilon_0, Z \right)}{(1 +
    r)^3} .
  \end{eqnarray*}
  We compute in $\{ \eta = 1 \}$, still using Theorem \ref{CP2Qcbehav},
  \begin{eqnarray*}
    \left| \mathfrak{I}\mathfrak{m} \left( \frac{e^{- \psi}}{Q \bar{Z}}
    \right) \right| & = & \frac{1}{| Q Z |^2} | \mathfrak{I}\mathfrak{m} (e^{-
    \psi - i \gamma} \bar{Q} Z e^{i \gamma}) |\\
    & \leqslant & K (| \mathfrak{I}\mathfrak{m} (e^{- \psi - i \gamma} - 1)
    \mathfrak{R}\mathfrak{e} (\bar{Q} Z e^{i \gamma}) | + |
    \mathfrak{R}\mathfrak{e} (e^{- \psi - i \gamma}) \mathfrak{I}\mathfrak{m}
    (\bar{Q} Z e^{i \gamma}) |)\\
    & \leqslant & \frac{K \left( \lambda, c, \| Z - Q_c \|_{H^{\exp}_{Q_c}},
    \varepsilon_0, Z \right)}{1 + r} + K \left( \lambda, c, \| Z - Q_c
    \|_{H^{\exp}_{Q_c}}, \varepsilon_0, Z \right) | \mathfrak{I}\mathfrak{m}
    (\bar{Q} Z e^{i \gamma}) |\\
    & \leqslant & \frac{K \left( \lambda, c, \| Z - Q_c \|_{H^{\exp}_{Q_c}},
    \varepsilon_0, Z \right)}{(1 + r)} + K \left( \lambda, c, \| Z - Q_c
    \|_{H^{\exp}_{Q_c}}, \varepsilon_0, Z \right) (| Q e^{- i \gamma} - 1 | +
    | Z - 1 |)\\
    & \leqslant & \frac{K \left( \lambda, c, \| Z - Q_c \|_{H^{\exp}_{Q_c}},
    \varepsilon_0, Z \right)}{(1 + r)} .
  \end{eqnarray*}
  This concludes the proof of this lemma.
\end{proof}

We remark that here, since $\psi \nrightarrow 0$ at infinity (if $\gamma \neq
0$), we do not have $Q \psi \in H_Q$. This is one of the main reasons to
introduce the space $H^{\exp}_Q$.

\

\begin{lemma}
  \label{CP2L3162011}The functions Q and $\psi$, defined respectively in
  (\ref{CP2280542}) and Lemma \ref{CP2L432905}, satisfy, with $\varphi = Q
  \psi$,
  \[ \langle L^{\exp}_Q (\varphi), (\varphi + i \gamma Q) \rangle = B^{\exp}_Q
     (\varphi), \]
  where $L^{\exp}_Q (\varphi) = (1 - \eta) L_Q (\varphi) + \eta Q L'_Q
  (\psi)$, with
  \[ L'_Q (\psi) = - \Delta \psi - 2 \frac{\nabla Q}{Q} . \nabla \psi + i
     \vec{c} . \nabla \psi + 2\mathfrak{R}\mathfrak{e} (\psi) | Q |^2 . \]
  Furthermore,
  \[ L_Q (\varphi) = Q L'_Q (\psi) . \]
\end{lemma}

The equality $\langle L^{\exp}_Q (\varphi), (\varphi + i \gamma Q) \rangle =
B^{\exp}_Q (\varphi)$ is not obvious for functions $\varphi \in C^2
(\mathbbm{R}^2, \mathbbm{C}) \cap H_{Q_c}^{\exp}$ (because of some integration
by parts to justify) and we need to check that, for the particular function
$\varphi$ we have contructed, this result holds. We will use mainly the decay
estimates of Lemma \ref{CP2L531806}.

Morally, we are showing that, since $L_Q (i \gamma Q) = 0$, that we can do the
following computation: $\langle L_Q (\varphi), \varphi + i \gamma Q \rangle =
\langle \varphi, L_Q (\varphi + i \gamma Q) \rangle = \langle \varphi, L_Q
(\varphi) \rangle = B_Q (\varphi)$. The goal of this lemma is simply to check
that, with the estimates of Lemma \ref{CP2L531806}, the integrands are
integrable and the integration by parts can be done to have $\langle
L^{\exp}_Q (\varphi), (\varphi + i \gamma Q) \rangle = B^{\exp}_Q (\varphi)$.

\begin{proof}
  First, let us show that $L_{Q_c} (\Phi) = Q_c L'_{Q_c} (\Psi)$ if $\Phi =
  Q_c \Psi \in C^2 (\mathbbm{R}^2, \mathbbm{C})$. With equation
  (\ref{CP2280542}), it implies that $L_Q (\varphi) = Q L'_Q (\psi) .$ We
  recall that
  \[ L_{Q_c} (\Phi) = - \Delta \Phi - i c \partial_{x_2} \Phi - (1 - | Q_c
     |^2) \Phi + 2\mathfrak{R}\mathfrak{e} (\overline{Q_c} \Phi) Q_c, \]
  and we develop with $\Phi = Q_c \Psi$,
  \[ L_{Q_c} (\Phi) = \tmop{TW}_c (Q_c) \Psi - Q_c \Delta \Psi - 2 \nabla Q_c
     . \nabla \Psi - i c Q_c \partial_{x_2} \Psi + 2\mathfrak{R}\mathfrak{e}
     (\Psi) | Q_c |^2 Q_c, \]
  thus, since $(\tmop{TW}_c) (Q_c) = 0$, we have $L_{Q_c} (\Phi) = Q_c
  L'_{Q_c} (\Psi)$.
  
  \
  
  Now, for $\varphi = Q \psi$, we have
  \begin{eqnarray*}
    &  & \langle (1 - \eta) L_Q (\varphi) + \eta Q L'_Q (\psi), (\varphi + i
    \gamma Q) \rangle\\
    & = & \int_{\mathbbm{R}^2} \mathfrak{R}\mathfrak{e} ((1 - \eta) L_Q
    (\varphi) \overline{(\varphi + i \gamma Q)})\\
    & + & \int_{\mathbbm{R}^2} \eta | Q |^2 \mathfrak{R}\mathfrak{e} \left(
    \left( - \Delta \psi - 2 \frac{\nabla Q}{Q} . \nabla \psi + i
    \overrightarrow{c} . \nabla \psi \right) (\overline{\psi + i \gamma})
    \right) + \eta | Q |^4 \mathfrak{R}\mathfrak{e}^2 (\psi) .
  \end{eqnarray*}
  With Lemma \ref{CP2L531806}, we check that all the terms are integrable
  independently (in particular since $\varphi + i \gamma Q = Q (\psi + i
  \gamma)$ and $\| (\psi + i \gamma) (1 + r) \|_{L^{\infty} (\{ \eta = 1 \})}
  < + \infty$ by Lemma \ref{CP2L531806}). We recall that $L_Q (\varphi) = -
  \Delta \varphi + i \vec{c} . \nabla \varphi - (1 - | Q |^2) \varphi +
  2\mathfrak{R}\mathfrak{e} (\bar{Q} \varphi) Q$, and thus
  \begin{eqnarray*}
    \int_{\mathbbm{R}^2} \mathfrak{R}\mathfrak{e} ((1 - \eta) L_Q (\varphi)
    \overline{(\varphi + i \gamma Q)}) & = & \int_{\mathbbm{R}^2} (1 - \eta)
    (\mathfrak{R}\mathfrak{e} (i \vec{c} . \nabla \varphi \bar{\varphi}) - (1
    - | Q |^2) | \varphi |^2 + 2\mathfrak{R}\mathfrak{e}^2 (\bar{Q}
    \varphi))\\
    & + & \int_{\mathbbm{R}^2} (1 - \eta) \mathfrak{R}\mathfrak{e} (- \Delta
    \varphi \bar{\varphi}) + \gamma \int_{\mathbbm{R}^2} (1 - \eta)
    \mathfrak{R}\mathfrak{e} (L_Q (\varphi) \overline{i Q}) .
  \end{eqnarray*}
  We recall that $1 - \eta$ is compactly supported and that $\varphi \in C^2
  (\mathbbm{R}^2, \mathbbm{C})$. By integration by parts,
  \[ \int_{\mathbbm{R}^2} (1 - \eta) \mathfrak{R}\mathfrak{e} (- \Delta
     \varphi \bar{\varphi}) = \int_{\mathbbm{R}^2} (1 - \eta) | \nabla \varphi
     |^2 - \int_{\mathbbm{R}^2} \nabla \eta .\mathfrak{R}\mathfrak{e} (\nabla
     \varphi \bar{\varphi}), \]
  and we decompose
  \begin{eqnarray*}
    \int_{\mathbbm{R}^2} (1 - \eta) \mathfrak{R}\mathfrak{e} (\eta L_Q
    (\varphi) \overline{i Q}) & = & \int_{\mathbbm{R}^2} (1 - \eta)
    \mathfrak{R}\mathfrak{e} (- \Delta \varphi \overline{i Q} + \vec{c} .
    \nabla \varphi \overline{Q})\\
    & - & \int_{\mathbbm{R}^2} (1 - \eta) \mathfrak{R}\mathfrak{e} ((1 - | Q
    |^2) \varphi \overline{i Q}) .
  \end{eqnarray*}
  By integration by parts, that we have
  \[ \int_{\mathbbm{R}^2} (1 - \eta) \mathfrak{R}\mathfrak{e} (\vec{c} .
     \nabla \varphi \overline{Q}) = - \vec{c} . \int_{\mathbbm{R}^2} - \nabla
     \eta \mathfrak{R}\mathfrak{e} (\varphi \bar{Q}) + (1 - \eta)
     \mathfrak{R}\mathfrak{e} (\varphi \nabla \bar{Q}) \]
  and
  \begin{eqnarray*}
    \int_{\mathbbm{R}^2} (1 - \eta) \mathfrak{R}\mathfrak{e} (- \Delta \varphi
    \overline{i Q}) & = & \int_{\mathbbm{R}^2} - \nabla \eta .
    (\mathfrak{R}\mathfrak{e} (i \varphi \nabla \bar{Q})
    -\mathfrak{R}\mathfrak{e} (i \nabla \varphi \bar{Q})) +
    \int_{\mathbbm{R}^2} (1 - \eta) \mathfrak{R}\mathfrak{e} (i \varphi
    \bar{Q}) .
  \end{eqnarray*}
  Combining this computations, we infer
  \[ \begin{array}{lll}
       &  & \int_{\mathbbm{R}^2} \mathfrak{R}\mathfrak{e} ((1 - \eta) L_Q
       (\varphi) \overline{(\varphi + i \gamma Q)})\\
       & = & \int_{\mathbbm{R}^2} (1 - \eta) (| \nabla \varphi |^2
       +\mathfrak{R}\mathfrak{e} (i \vec{c} . \nabla \varphi \bar{\varphi}) -
       (1 - | Q |^2) | \varphi |^2 + 2\mathfrak{R}\mathfrak{e}^2 (\bar{Q}
       \varphi))\\
       & - & \int_{\mathbbm{R}^2} \nabla \eta .\mathfrak{R}\mathfrak{e}
       (\nabla \varphi \bar{\varphi}) \gamma \vec{c} . \int_{\mathbbm{R}^2}
       \nabla \eta \mathfrak{R}\mathfrak{e} (\varphi \bar{Q})\\
       & - & \gamma \left( \int_{\mathbbm{R}^2} \nabla \eta .
       (\mathfrak{R}\mathfrak{e} (i \varphi \nabla \bar{Q})
       -\mathfrak{R}\mathfrak{e} (i \nabla \varphi \bar{Q})) \right)\\
       & + & \gamma \int_{\mathbbm{R}^2} (1 - \eta) \mathfrak{R}\mathfrak{e}
       (\varphi (- \vec{c} . \nabla \bar{Q} + i (1 - | Q |^2) \bar{Q} + i
       \Delta \bar{Q})) .
     \end{array} \]
  Since $- \Delta Q + i \vec{c} . \nabla Q - (1 - | Q |^2) Q = 0$, we have $-
  \vec{c} . \nabla \bar{Q} + i (1 - | Q |^2) \bar{Q} + i \Delta \bar{Q} = 0$,
  therefore
  \[ \begin{array}{lll}
       &  & \int_{\mathbbm{R}^2} \mathfrak{R}\mathfrak{e} ((1 - \eta) L_Q
       (\varphi) \overline{(\varphi + i \gamma Q)})\\
       & = & \int_{\mathbbm{R}^2} (1 - \eta) (| \nabla \varphi |^2
       +\mathfrak{R}\mathfrak{e} (i \vec{c} . \nabla \varphi \bar{\varphi}) -
       (1 - | Q |^2) | \varphi |^2 + 2\mathfrak{R}\mathfrak{e}^2 (\bar{Q}
       \varphi))\\
       & - & \int_{\mathbbm{R}^2} \nabla \eta .\mathfrak{R}\mathfrak{e}
       (\nabla \varphi \bar{\varphi})\\
       & - & \gamma \left( - \vec{c} . \int_{\mathbbm{R}^2} \nabla \eta
       \mathfrak{R}\mathfrak{e} (\varphi \bar{Q}) + \int_{\mathbbm{R}^2}
       \nabla \eta . (\mathfrak{R}\mathfrak{e} (i \varphi \nabla \bar{Q})
       -\mathfrak{R}\mathfrak{e} (i \nabla \varphi \bar{Q})) \right) .
     \end{array} \]
  Until now, all the integrals were on bounded domain (since $1 - \eta$ is
  compactly supported).
  
  Now, by integration by parts, (that can be done thanks to Lemma
  \ref{CP2L531806} and Theorem \ref{CP2Qcbehav})
  \begin{eqnarray*}
    \int_{\mathbbm{R}^2} \eta | Q |^2 \mathfrak{R}\mathfrak{e} (- \Delta \psi
    (\overline{\psi + i \gamma})) & = & \int_{\mathbbm{R}^2} \nabla \eta . | Q
    |^2 \mathfrak{R}\mathfrak{e} (\nabla \psi (\overline{\psi + i \gamma}))\\
    & + & \int_{\mathbbm{R}^2} \eta \nabla (| Q |^2)
    .\mathfrak{R}\mathfrak{e} (\nabla \psi (\overline{\psi + i \gamma}))\\
    & + & \int_{\mathbbm{R}^2} \eta | Q |^2 | \nabla \psi |^2 .
  \end{eqnarray*}
  Now, we decompose (and we check that each term is well defined at each step
  with Lemma \ref{CP2L531806} and Theorem \ref{CP2Qcbehav})
  \begin{eqnarray*}
    &  & \int_{\mathbbm{R}^2} \eta | Q |^2 \mathfrak{R}\mathfrak{e} \left(
    \left. - 2 \frac{\nabla Q}{Q} . \nabla \psi \right) (\overline{\psi + i
    \gamma}) \right)\\
    & = & - 2 \int_{\mathbbm{R}^2} \eta \mathfrak{R}\mathfrak{e} (\nabla Q
    \bar{Q} . \nabla \psi \bar{\psi}) - 2 \int_{\mathbbm{R}^2} \eta
    \mathfrak{R}\mathfrak{e} (\nabla Q \bar{Q} . \nabla \psi \overline{(i
    \gamma)}),
  \end{eqnarray*}
  with
  \begin{eqnarray*}
    - 2 \int_{\mathbbm{R}^2} \eta \mathfrak{R}\mathfrak{e} (\nabla Q \bar{Q} .
    \nabla \psi \bar{\psi}) & = & - 2 \int_{\mathbbm{R}^2} \eta
    \mathfrak{R}\mathfrak{e} (\nabla Q \bar{Q}) .\mathfrak{R}\mathfrak{e}
    (\nabla \psi \bar{\psi})\\
    & + & 2 \int_{\mathbbm{R}^2} \eta \mathfrak{I}\mathfrak{m} (\nabla Q
    \bar{Q}) .\mathfrak{I}\mathfrak{m} (\nabla \psi \bar{\psi}),
  \end{eqnarray*}
  and since $\nabla (| Q |^2) = 2\mathfrak{R}\mathfrak{e} (\nabla Q \bar{Q})$,
  we have
  \begin{eqnarray*}
    &  & \int_{\mathbbm{R}^2} \eta | Q |^2 \mathfrak{R}\mathfrak{e} \left(
    \left( - \Delta \psi - 2 \frac{\nabla Q}{Q} . \nabla \psi \right)
    (\overline{\psi + i \gamma}) \right)\\
    & = & \int_{\mathbbm{R}^2} \eta | Q |^2 | \nabla \psi |^2 + 2
    \int_{\mathbbm{R}^2} (1 - \eta) \mathfrak{I}\mathfrak{m} (\nabla Q
    \bar{Q}) .\mathfrak{I}\mathfrak{m} (\nabla \psi \bar{\psi})\\
    & + & \int_{\mathbbm{R}^2} \nabla \eta . | Q |^2 \mathfrak{R}\mathfrak{e}
    (\nabla \psi (\overline{\psi + i \gamma})) + 2 \int_{\mathbbm{R}^2} \eta
    \mathfrak{I}\mathfrak{m} (\nabla Q \bar{Q}) .\mathfrak{I}\mathfrak{m}
    (\nabla \psi \nobracket \overline{(i \gamma)})) .
  \end{eqnarray*}
  We continue. We have
  \begin{eqnarray*}
    2 \int_{\mathbbm{R}^2} \eta \mathfrak{I}\mathfrak{m} (\nabla Q \bar{Q})
    .\mathfrak{I}\mathfrak{m} (\nabla \psi \bar{\psi}) & = & 2
    \int_{\mathbbm{R}^2} \eta \mathfrak{I}\mathfrak{m} (\nabla Q \bar{Q})
    .\mathfrak{R}\mathfrak{e} (\psi) \mathfrak{I}\mathfrak{m} (\nabla \psi)\\
    & - & 2 \int_{\mathbbm{R}^2} \eta \mathfrak{I}\mathfrak{m} (\nabla Q
    \bar{Q}) .\mathfrak{R}\mathfrak{e} (\nabla \psi) \mathfrak{I}\mathfrak{m}
    (\psi),
  \end{eqnarray*}
  and by integration by parts (still using Lemma \ref{CP2L531806} and Theorem
  \ref{CP2Qcbehav}),
  \begin{eqnarray*}
    &  & - 2 \int_{\mathbbm{R}^2} \eta \mathfrak{I}\mathfrak{m} (\nabla Q
    \bar{Q}) .\mathfrak{R}\mathfrak{e} (\nabla \psi) \mathfrak{I}\mathfrak{m}
    (\psi)\\
    & = & 2 \int_{\mathbbm{R}^2} \eta \mathfrak{I}\mathfrak{m} (\nabla Q
    \bar{Q}) .\mathfrak{R}\mathfrak{e} (\psi) \mathfrak{I}\mathfrak{m} (\nabla
    \psi)\\
    & + & 2 \int_{\mathbbm{R}^2} \eta \mathfrak{I}\mathfrak{m} (\Delta Q
    \bar{Q}) \mathfrak{R}\mathfrak{e} (\psi) \mathfrak{I}\mathfrak{m} (\psi)\\
    & + & 2 \int_{\mathbbm{R}^2} \nabla \eta .\mathfrak{I}\mathfrak{m}
    (\nabla Q \bar{Q}) \mathfrak{R}\mathfrak{e} (\psi)
    \mathfrak{I}\mathfrak{m} (\psi) .
  \end{eqnarray*}
  We have $\mathfrak{I}\mathfrak{m} (\Delta Q \bar{Q})
  =\mathfrak{I}\mathfrak{m} (i \vec{c} . \nabla Q - (1 - | Q |^2 Q) \bar{Q})
  =\mathfrak{R}\mathfrak{e} (\vec{c} . \nabla Q \bar{Q})$, therefore
  \begin{eqnarray*}
    &  & \int_{\mathbbm{R}^2} \eta | Q |^2 \mathfrak{R}\mathfrak{e} \left(
    \left( - \Delta \psi - 2 \frac{\nabla Q}{Q} . \nabla \psi \right)
    (\overline{\psi + i \gamma}) \right)\\
    & = & \int_{\mathbbm{R}^2} \eta | Q |^2 | \nabla \psi |^2 + 4
    \int_{\mathbbm{R}^2} \eta \mathfrak{I}\mathfrak{m} (\nabla Q \bar{Q})
    .\mathfrak{R}\mathfrak{e} (\psi) \mathfrak{I}\mathfrak{m} (\nabla \psi)\\
    & + & 2 \int_{\mathbbm{R}^2} \eta \mathfrak{I}\mathfrak{m} (\nabla Q
    \bar{Q}) .\mathfrak{I}\mathfrak{m} (\nabla \psi \nobracket \overline{(i
    \gamma)}))\\
    & + & 2 \int_{\mathbbm{R}^2} \eta \mathfrak{R}\mathfrak{e} (\vec{c} .
    \nabla Q \bar{Q}) \mathfrak{R}\mathfrak{e} (\psi) \mathfrak{I}\mathfrak{m}
    (\psi)\\
    & + & \int_{\mathbbm{R}^2} \nabla \eta (| Q |^2 \mathfrak{R}\mathfrak{e}
    (\nabla \psi (\overline{\psi + i \gamma})) + 2\mathfrak{I}\mathfrak{m}
    (\nabla Q \bar{Q}) \mathfrak{R}\mathfrak{e} (\psi)
    \mathfrak{I}\mathfrak{m} (\psi)) .
  \end{eqnarray*}
  Now, we compute
  \begin{eqnarray*}
    \vec{c} . \int_{\mathbbm{R}^2} \eta | Q |^2 \mathfrak{R}\mathfrak{e} (i
    \nabla \psi (\overline{\psi + i \gamma})) & = & \vec{c} .
    \int_{\mathbbm{R}^2} \eta | Q |^2 \mathfrak{R}\mathfrak{e} (\nabla \psi)
    \mathfrak{I}\mathfrak{m} (\psi + i \gamma)\\
    & - & \vec{c} . \int_{\mathbbm{R}^2} \eta | Q |^2
    \mathfrak{I}\mathfrak{m} (\nabla \psi) \mathfrak{R}\mathfrak{e} (\psi),
  \end{eqnarray*}
  and by integration by parts (still using Lemma \ref{CP2L531806} and Theorem
  \ref{CP2Qcbehav}),
  \begin{eqnarray*}
    \vec{c} . \int_{\mathbbm{R}^2} \eta | Q |^2 \mathfrak{R}\mathfrak{e}
    (\nabla \psi) \mathfrak{I}\mathfrak{m} (\psi + i \gamma) & = & - \vec{c} .
    \int_{\mathbbm{R}^2} \nabla \eta | Q |^2 \mathfrak{R}\mathfrak{e} (\psi)
    \mathfrak{I}\mathfrak{m} (\psi + i \gamma)\\
    & - & \vec{c} . \int_{\mathbbm{R}^2} \eta \nabla (| Q |^2)
    \mathfrak{R}\mathfrak{e} (\psi) \mathfrak{I}\mathfrak{m} (\psi + i
    \gamma)\\
    & - & \vec{c} . \int_{\mathbbm{R}^2} \eta | Q |^2
    \mathfrak{R}\mathfrak{e} (\psi) \mathfrak{I}\mathfrak{m} (\nabla \psi) .
  \end{eqnarray*}
  Since $\nabla (| Q |^2) = 2\mathfrak{R}\mathfrak{e} (\nabla Q \bar{Q})$, we
  infer
  \begin{eqnarray*}
    &  & \int_{\mathbbm{R}^2} \eta | Q |^2 \mathfrak{R}\mathfrak{e} \left(
    \left( - \Delta \psi - 2 \frac{\nabla Q}{Q} . \nabla \psi - i \vec{c} .
    \nabla \psi \right) (\overline{\psi + i \gamma}) \right)\\
    & = & \int_{\mathbbm{R}^2} \eta (| Q |^2 | \nabla \psi |^2 +
    4\mathfrak{I}\mathfrak{m} (\nabla Q \bar{Q}) .\mathfrak{R}\mathfrak{e}
    (\psi) \mathfrak{I}\mathfrak{m} (\nabla \psi) - 2 \vec{c}
    .\mathfrak{I}\mathfrak{m} (\nabla \psi) \mathfrak{R}\mathfrak{e} (\psi))\\
    & + & 2 \int_{\mathbbm{R}^2} \eta \mathfrak{I}\mathfrak{m} (\nabla Q
    \bar{Q}) .\mathfrak{I}\mathfrak{m} (\nabla \psi \nobracket \overline{(i
    \gamma)}))\\
    & - & 2 \gamma \int_{\mathbbm{R}^2} \eta \mathfrak{R}\mathfrak{e}
    (\vec{c} . \nabla Q \bar{Q}) \mathfrak{R}\mathfrak{e} (\psi)\\
    & + & \int_{\mathbbm{R}^2} \nabla \eta . (| Q |^2
    \mathfrak{R}\mathfrak{e} (\nabla \psi (\overline{\psi + i \gamma})) +
    2\mathfrak{I}\mathfrak{m} (\nabla Q \bar{Q}) \mathfrak{R}\mathfrak{e}
    (\psi) \mathfrak{I}\mathfrak{m} (\psi))\\
    & + & \vec{c} . \int_{\mathbbm{R}^2} \nabla \eta | Q |^2
    \mathfrak{R}\mathfrak{e} (\psi) \mathfrak{I}\mathfrak{m} (\psi + i \gamma)
    .
  \end{eqnarray*}
  Combining these computation yields
  \begin{eqnarray*}
    \int_{\mathbbm{R}^2} \mathfrak{R}\mathfrak{e} (L^{\exp}_Q (\varphi)
    \overline{(\varphi + i \gamma Q)}) & = & B^{\exp}_Q (\varphi)\\
    & - & \gamma \left( - \vec{c} . \int_{\mathbbm{R}^2} \nabla \eta
    \mathfrak{R}\mathfrak{e} (\varphi \bar{Q}) + \int_{\mathbbm{R}^2} \nabla
    \eta . (\mathfrak{R}\mathfrak{e} (i \varphi \nabla \bar{Q})
    -\mathfrak{R}\mathfrak{e} (i \nabla \varphi \bar{Q})) \right)\\
    & + & 2 \int_{\mathbbm{R}^2} \eta \mathfrak{I}\mathfrak{m} (\nabla Q
    \bar{Q}) .\mathfrak{I}\mathfrak{m} (\nabla \psi \nobracket \overline{(i
    \gamma)}))\\
    & - & 2 \gamma \int_{\mathbbm{R}^2} \eta \mathfrak{R}\mathfrak{e}
    (\vec{c} . \nabla Q \bar{Q}) \mathfrak{R}\mathfrak{e} (\psi)\\
    & + & \int_{\mathbbm{R}^2} \nabla \eta . | Q |^2 \mathfrak{R}\mathfrak{e}
    (\nabla \psi \overline{(i \gamma)})\\
    & - & \vec{c} . \gamma \int_{\mathbbm{R}^2} \nabla \eta | Q |^2
    \mathfrak{R}\mathfrak{e} (\psi) .
  \end{eqnarray*}
  We compute, by integration by parts (still using Lemma \ref{CP2L531806} and
  Theorem \ref{CP2Qcbehav}), that
  \begin{eqnarray*}
    2 \int_{\mathbbm{R}^2} \eta \mathfrak{I}\mathfrak{m} (\nabla Q \bar{Q})
    .\mathfrak{I}\mathfrak{m} (\nabla \psi \nobracket \overline{(i \gamma)}))
    & = & - 2 \gamma \int_{\mathbbm{R}^2} \eta \mathfrak{I}\mathfrak{m}
    (\nabla Q \bar{Q}) .\mathfrak{R}\mathfrak{e} (\nabla \psi)\\
    & = & 2 \gamma \int_{\mathbbm{R}^2} \nabla \eta .\mathfrak{I}\mathfrak{m}
    (\nabla Q \bar{Q}) \mathfrak{R}\mathfrak{e} (\psi)\\
    & + & 2 \gamma \int_{\mathbbm{R}^2} \eta \mathfrak{I}\mathfrak{m} (\Delta
    Q \bar{Q}) \mathfrak{R}\mathfrak{e} (\psi),
  \end{eqnarray*}
  and since $\mathfrak{I}\mathfrak{m} (\Delta Q \bar{Q})
  =\mathfrak{R}\mathfrak{e} (\vec{c} . \nabla Q \bar{Q})$ and
  $\mathfrak{R}\mathfrak{e} (\nabla \psi \overline{(i \gamma)}) = \gamma
  \mathfrak{I}\mathfrak{m} (\nabla \psi)$, we have
  \begin{eqnarray*}
    \int_{\mathbbm{R}^2} \mathfrak{R}\mathfrak{e} (L^{\exp}_Q (\varphi)
    \overline{(\varphi + i \gamma Q)}) & = & B^{\exp}_Q (\varphi)\\
    & - & \gamma \left( - \vec{c} . \int_{\mathbbm{R}^2} \nabla \eta
    \mathfrak{R}\mathfrak{e} (\varphi \bar{Q}) + \int_{\mathbbm{R}^2} \nabla
    \eta . (\mathfrak{R}\mathfrak{e} (i \varphi \nabla \bar{Q})
    -\mathfrak{R}\mathfrak{e} (i \nabla \varphi \bar{Q})) \right)\\
    & + & 2 \gamma \int_{\mathbbm{R}^2} \nabla \eta .\mathfrak{I}\mathfrak{m}
    (\nabla Q \bar{Q}) \mathfrak{R}\mathfrak{e} (\psi)\\
    & + & \gamma \int_{\mathbbm{R}^2} \nabla \eta . | Q |^2
    \mathfrak{I}\mathfrak{m} (\nabla \psi)\\
    & - & \vec{c} . \gamma \int_{\mathbbm{R}^2} \nabla \eta | Q |^2
    \mathfrak{R}\mathfrak{e} (\psi) .
  \end{eqnarray*}
  we check that $\mathfrak{R}\mathfrak{e} (\varphi \bar{Q}) = | Q |^2
  \mathfrak{R}\mathfrak{e} (\psi)$, $\mathfrak{R}\mathfrak{e} (i \varphi
  \nabla \bar{Q}) = -\mathfrak{R}\mathfrak{e} (\nabla Q \bar{Q})
  \mathfrak{I}\mathfrak{m} (\psi) +\mathfrak{I}\mathfrak{m} (\nabla Q \bar{Q})
  \mathfrak{R}\mathfrak{e} (\psi)$ and that
  \begin{eqnarray*}
    -\mathfrak{R}\mathfrak{e} (i \nabla \varphi \bar{Q}) & = &
    -\mathfrak{R}\mathfrak{e} (i \nabla Q_c \bar{Q} \psi)
    -\mathfrak{R}\mathfrak{e} (i \nabla \psi) | Q |^2\\
    & = & \mathfrak{I}\mathfrak{m} (\nabla Q \bar{Q})
    \mathfrak{R}\mathfrak{e} (\psi) +\mathfrak{R}\mathfrak{e} (\nabla Q
    \bar{Q}) \mathfrak{I}\mathfrak{m} (\psi) +\mathfrak{I}\mathfrak{m} (\nabla
    \psi) | Q |^2,
  \end{eqnarray*}
  thus concluding the proof of
  \[ \int_{\mathbbm{R}^2} \mathfrak{R}\mathfrak{e} (L^{\exp}_Q (\varphi)
     \overline{(\varphi + i \gamma Q)}) = B^{\exp}_Q (\varphi) . \]
\end{proof}

\subsection{Properties of the perturbation}

We look for the equation satisfied by $\varphi = Q \psi$ in the next lemma.

\begin{lemma}
  \label{CP2L442905}The functions Q and $\psi$, defined respectively in
  (\ref{CP2280542}) and Lemma \ref{CP2L432905}, with $\varphi = Q \psi$,
  satisfy the equation
  \[ L_Q (Q \psi) - i (c \overrightarrow{e_2} - \vec{c}') .H (\psi) +
     \tmop{NL}_{\tmop{loc}} (\psi) + F (\psi) = 0, \]
  with $L_Q$ the linearized operator around $Q$: $L_Q (\varphi) \assign -
  \Delta \varphi - i \vec{c} . \nabla \varphi - (1 - | Q |^2) \varphi +
  2\mathfrak{R}\mathfrak{e} (\bar{Q} \varphi) Q$,
  \[ S (\psi) \assign e^{2\mathfrak{R}\mathfrak{e} (\psi)} - 1 -
     2\mathfrak{R}\mathfrak{e} (\psi), \]
  \[ F (\psi) \assign Q \eta (- \nabla \psi . \nabla \psi + | Q |^2 S (\psi)),
  \]
  \[ H (\psi) \assign \nabla Q + \frac{\nabla (Q \psi) (1 - \eta) + Q \nabla
     \psi \eta e^{\psi}}{(1 - \eta) + \eta e^{\psi}} \]
  and $\tmop{NL}_{\tmop{loc}} (\psi)$ is a sum of terms at least quadratic in
  $\psi$, localized in the area where $\eta \neq 1$. Furthermore,
  \[ | \langle \tmop{NL}_{\tmop{loc}} (\psi), Q (\psi + i \gamma) \rangle |
     \leqslant K (\| Q \psi \|_{C^1 (\{ \eta \neq 1 \})} + | \gamma |) \| Q
     \psi \|^2_{H^1 (\{ \eta \neq 1 \})} . \]
\end{lemma}

Remark that here, the equation satisfied by $\varphi$ has a ``source'' term,
$i (c \overrightarrow{e_2} - \vec{c}') .H (\psi)$, coming from the fact that
$Z$ and $Q_c$ might not have the same speed at this point. We will estimate it
later on.

\begin{proof}
  The function $Z$ solves $(\tmop{TW}_c)$, hence,
  \[ i (c \overrightarrow{e_2} - \vec{c}') . \nabla Z = - i \vec{c}' . \nabla
     Z - \Delta Z - (1 - | Z |^2) Z. \]
  From (\ref{CP2422805}), we have
  \[ Z = Q + (1 - \eta) Q \psi + \eta Q (e^{\psi} - 1) . \]
  We define
  \[ \zeta \assign 1 + \psi - e^{\psi} . \]
  We replace $Z = Q + (1 - \eta) Q \psi + \eta Q (e^{\psi} - 1)$ in $- i
  \vec{c}' . \nabla Z - \Delta Z - (1 - | Z |^2) Z$ exactly as in the proof of
  Lemma 2.7 of {\cite{CP1}}, by simply changing $V, \Psi, c \vec{e}_2, \eta$
  respectively to $Q, \psi, \vec{c}', 1 - \eta$. In particular, $E - i c
  \partial_{x_2} V$ becomes $0$ (since $\tmop{TW}_{\vec{c}'} (Q) = 0$). This
  computation yields
  \[ i (c \overrightarrow{e_2} - \vec{c}') . \nabla Z = ((1 - \eta) + \eta
     e^{\psi}) (L_Q (Q \psi) + \widetilde{\tmop{NL}}_{\tmop{loc}} (\psi) + F
     (\psi)) . \]
  Furthermore, we have that $((1 - \eta) + \eta e^{\psi}) \neq 0$ by Lemma
  \ref{CP2L432905} and equation (\ref{CP2psiinfini}) (for the same reason as
  in the proof of Lemma 2.7 of {\cite{CP1}}), and we compute (as in Lemma 2.7
  of {\cite{CP1}}) that
  \begin{equation}
    \frac{\eta e^{\psi}}{(1 - \eta) + \eta e^{\psi}} = \eta + \eta (1 - \eta)
    \left( \frac{e^{\psi} - 1}{(1 - \eta) + \eta e^{\psi}} \right) .
    \label{CP24133005}
  \end{equation}
  Furthermore, we have
  \begin{eqnarray*}
    \nabla Z & = & \nabla Q - Q \nabla \eta \zeta + \nabla Q ((1 - \eta) \psi
    + \eta (e^{\psi} - 1)) + Q \nabla \psi ((1 - \eta) + \eta e^{\psi})\\
    & = & \nabla Q (1 - \eta + \eta e^{\psi}) - Q \nabla \eta \zeta + \nabla
    (Q \psi) (1 - \eta) + Q \nabla \psi \eta e^{\psi},
  \end{eqnarray*}
  hence
  \[ \frac{\nabla Z}{(1 - \eta) + \eta e^{\psi}} = \nabla Q - \frac{Q \nabla
     \eta \zeta}{(1 - \eta) + \eta e^{\psi}} + \frac{\nabla (Q \psi) (1 -
     \eta) + Q \nabla \psi \eta e^{\psi}}{(1 - \eta) + \eta e^{\psi}}, \]
  therefore, with \ $\tmop{NL}_{\tmop{loc}} (\psi) =
  \widetilde{\tmop{NL}}_{\tmop{loc}} (\psi) + i (c \overrightarrow{e_2} -
  \vec{c}') . \frac{- Q \nabla \eta \zeta}{(1 - \eta) + \eta e^{\psi}}$, we
  have
  \[ L_Q (Q \psi) - i (c \overrightarrow{e_2} - \vec{c}') .H (\psi) +
     \tmop{NL}_{\tmop{loc}} (\psi) + F (\psi) = 0. \]
  Finally, we check, similarly as in the proof of Lemma 2.7 of {\cite{CP1}},
  that
  \[ | \langle \tmop{NL}_{\tmop{loc}} (\psi), Q (\psi + i \gamma) \rangle |
     \leqslant K (\| Q \psi \|_{C^1 (\{ \eta \neq 1 \})} + | \gamma |)
     \int_{\mathbbm{R}^2} | \tmop{NL}_{\tmop{loc}} (\psi) |, \]
  hence
  \[ | \langle \tmop{NL}_{\tmop{loc}} (\psi), Q (\psi + i \gamma) \rangle |
     \leqslant K (\| Q \psi \|_{C^1 (\{ \eta \neq 1 \})} + | \gamma |) \| Q
     \psi \|^2_{H^1 (\{ \eta \neq 1 \})} . \]
\end{proof}

Now, we want to choose the right parameters $\gamma, \vec{c}', X$ so that
$\varphi$ satisfies the orthogonality conditions of Proposition
\ref{CP2prop16} and \ref{CP2prop17} (with remark (\ref{CP2corto2}).

\begin{lemma}
  \label{CP2L452905}For the functions Q and $\psi$, defined respectively in
  (\ref{CP2280542}) and Lemma \ref{CP2L432905}, there exist $X, \vec{c}' \in
  \mathbbm{R}^2, \gamma \in \mathbbm{R}$ such that
  \[ | X | + \frac{\delta^{| . |} (c \overrightarrow{e_2}, \vec{c}')}{c^2} +
     \frac{\delta^{\bot} (\nobracket c \overrightarrow{e_2}, \vec{c}')
     \nobracket}{c} + | \gamma | \leqslant o^{\lambda, c}_{\| Z - Q_c
     \|_{H^{\exp}_{Q_c}} \rightarrow 0} (1), \]
  and
  \[ \mathfrak{R}\mathfrak{e} \int_{B (\tmmathbf{d}_{\vec{c}', 1}, R) \cup B
     (\tmmathbf{d}_{\vec{c}', 2}, R)} \partial_{x_1} Q \overline{Q \psi^{\neq
     0}} =\mathfrak{R}\mathfrak{e} \int_{B (\tmmathbf{d}_{\vec{c}', 1}, R)
     \cup B (\tmmathbf{d}_{\vec{c}', 2}, R)} \partial_{x_2} Q \overline{Q
     \psi^{\neq 0}} = 0, \]
  \[ \mathfrak{R}\mathfrak{e} \int_{B (\tmmathbf{d}_{\vec{c}', 1}, R) \cup B
     (\tmmathbf{d}_{\vec{c}', 2}, R)} \partial_{c^{\bot}} Q \overline{Q
     \psi^{\neq 0}} = 0, \]
  \[ \mathfrak{R}\mathfrak{e} \int_{B (\tmmathbf{d}_{\vec{c}', 1}, R) \cup B
     (\tmmathbf{d}_{\vec{c}', 2}, R)} \partial_{\tmmathbf{d}} \tmmathbf{V}
     \overline{Q \psi^{\neq 0}} = 0 \]
  \[ \mathfrak{R}\mathfrak{e} \int_{B ((\tmmathbf{d}_{\vec{c}', 1}
     +\tmmathbf{d}_{\vec{c}', 2}) / 2, R)} i \psi = 0, \]
  where $\tmmathbf{d}_{\vec{c}', 1}$ and $\tmmathbf{d}_{\vec{c}', 2}$ are the
  zeros of $Q$, $\tmmathbf{d} _{\vec{c}', 1}$ being the closest one of
  $\tilde{d}_c \overrightarrow{e_1}$, and $\partial_{\tmmathbf{d}}
  \tmmathbf{V}$ is the first order of $Q$ by Theorem \ref{th1} and
  (\ref{CP2rotQc}).
\end{lemma}

Here, the notations for the harmonics are done for $Q$, and are therefore
centered around $\tmmathbf{d}_{\vec{c}', 1}$ or $\tmmathbf{d}_{\vec{c}', 2}$.
This means that $\psi^{\neq 0} (x) = \psi (x) - \psi^{\tmmathbf{0}, 1}
(\tmmathbf{r}_1)$ with $\tmmathbf{r}_1 \assign | x -\tmmathbf{d}_{\vec{c}', 1}
|$, $x -\tmmathbf{d}_{\vec{c}', 1} =\tmmathbf{r}_1 e^{i\tmmathbf{\theta}_1}
\in \mathbbm{R}^2$ and $\psi^{\tmmathbf{0}, 1}$ being the $0$-harmonic of
$\psi$ around $\tmmathbf{d}_{\vec{c}', 1}$ in $B (\tmmathbf{d}_{\vec{c}', 1},
R)$, and $\psi^{\neq 0} (x) = \psi (x) - \psi^{\tmmathbf{0}, 2}
(\tmmathbf{r}_2)$ with $\tmmathbf{r}_2 \assign | x -\tmmathbf{d}_{\vec{c}', 2}
|$ in $B (\tmmathbf{d}_{\vec{c}', 2}, R)$ and $\psi^{\tmmathbf{0}, 1}$ being
the $0$-harmonic of $\psi$ around $\tmmathbf{d}_{\vec{c}', 2}$. We will denote
$\psi^{\tmmathbf{0}} (x)$ the quantity equal to $\psi^{\tmmathbf{0}, 1}
(\tmmathbf{r}_1)$ in the right half-plane and to $\psi^{\tmmathbf{0}, 2}
(\tmmathbf{r}_2)$ in the left half-plane. Remark that $\tmmathbf{d}_{\vec{c}',
1} \in \mathbbm{R}^2$, whereas $\tilde{d}_c \in \mathbbm{R}$. We recall that,
taking $\| Z - Q_c \|_{H^{\exp}_{Q_c}}$ small enough, we have $\frac{\delta^{|
. |} (c \overrightarrow{e_2}, \vec{c}')}{c^2} \leqslant 1$, and in particular,
for $c$ small enough, it implies that $\frac{c}{2} \leqslant | \vec{c}' |
\leqslant 2 c$. We recall that $o^{\lambda, c}_{\| Z - Q_c \|_{H^{\exp}_{Q_c}}
\rightarrow 0} (1)$ is a quantity going to $0$ when $\| Z - Q_c
\|_{H^{\exp}_{Q_c}} \rightarrow 0$ at fixed $\lambda$ and $c$.

\begin{proof}
  For $X = (X_1, X_2), \vec{c}' \in \mathbbm{R}^2$, we define, as previously,
  the function
  \[ Q = Q_{\vec{c}'} (. - X) e^{i \gamma} . \]
  We define, to simplify the notations,
  \[ \Omega \assign B (\tmmathbf{d}_{\vec{c}', 1}, R) \cup B
     (\tmmathbf{d}_{\vec{c}', 2}, R) \]
  and
  \[ \Omega' \assign B \left( \frac{(\tmmathbf{d}_{\vec{c}', 1}
     +\tmmathbf{d}_{\vec{c}', 2})}{2}, R \right), \]
  which is between the two vortices. We define
  \[ G \left(\begin{array}{c}
       X_1\\
       X_2\\
       \delta_1\\
       \delta_2\\
       \gamma
     \end{array}\right) \assign \left(\begin{array}{c}
       \mathfrak{R}\mathfrak{e} \int_{\Omega} \partial_{x_1} Q \overline{Q
       \psi^{\neq 0}}\\
       \mathfrak{R}\mathfrak{e} \int_{\Omega} \partial_{x_2} Q \overline{Q
       \psi^{\neq 0}}\\
       c^2 \mathfrak{R}\mathfrak{e} \int_{\Omega} \partial_{\tmmathbf{d}}
       \tmmathbf{V} \overline{Q \psi^{\neq 0}}\\
       c\mathfrak{R}\mathfrak{e} \int_{\Omega} \partial_{c^{\bot}} Q
       \overline{Q \psi^{\neq 0}}\\
       \mathfrak{R}\mathfrak{e} \int_{\Omega'} i \psi
     \end{array}\right), \]
  where $\vec{c}'$ (used to defined $Q = Q_{\vec{c}'} (. - X) e^{i \gamma}$)
  is given by $\delta_1 = \delta^{| . |} (c \overrightarrow{e_2}, \vec{c}')$
  and $\delta_2 = \delta^{\bot} (\nobracket c \overrightarrow{e_2}, \vec{c}')
  \nobracket$.
  
  Here, we use the notation $\partial_c Q_{}$ for $\partial_c Q_{c | c = c'
  \nobracket}$. We remark from (\ref{CP24112905}) and the definition of
  $\eta$, that in $\Omega$, we have
  \[ Q \psi = Z - Q. \]
  First, we have
  \begin{equation}
    \| Q \psi \|_{C^1 (\Omega)} \leqslant o^{\lambda, c}_{\| Z - Q_c
    \|_{H^{\exp}_{Q_c}} \rightarrow 0} (1) + K \left( | X | + \frac{\delta^{|
    . |} (c \overrightarrow{e_2}, \vec{c}')}{c^2} + \frac{\delta^{\bot}
    (\nobracket c \overrightarrow{e_2}, \vec{c}') \nobracket}{c} + | \gamma |
    \right), \label{CP244124}
  \end{equation}
  which is a consequence of Lemma \ref{CP2L422905}. By Lemma \ref{CP2133L35},
  we compute that
  \[ \left| G \left(\begin{array}{c}
       0\\
       0\\
       0\\
       0\\
       0
     \end{array}\right) \right| \leqslant o^{\lambda, c}_{\| Z - Q_c
     \|_{H^{\exp}_{Q_c}} \rightarrow 0} (1) . \]
  Let us compute $\partial_{X_2} G$. We recall that $Q \psi \in C^1
  (\mathbbm{R}^2, \mathbbm{C})$. Since $\Omega$ depends on $X$, we have
  \begin{eqnarray*}
    \partial_{X_2} \mathfrak{R}\mathfrak{e} \int_{\Omega} \partial_{x_2} Q
    \overline{Q \psi^{\neq 0}} & = & \int_{\partial \Omega}
    \mathfrak{R}\mathfrak{e} (\partial_{x_2} Q \overline{Q \psi^{\neq 0}})\\
    & - & \int_{\Omega} \mathfrak{R}\mathfrak{e} (\partial^2_{x_2 x_2} Q
    \overline{Q \psi^{\neq 0}})\\
    & + & \int_{\Omega} \mathfrak{R}\mathfrak{e} \left( \partial_{x_2} Q
    \overline{\partial_{X_2} (Q \psi^{\neq 0})} \right) .
  \end{eqnarray*}
  By estimate (\ref{CP244124}), we have
  \[ \left| \int_{\partial \Omega} \mathfrak{R}\mathfrak{e} (\partial_{x_2} Q
     \overline{Q \psi^{\neq 0}}) \right| + \left| \int_{\Omega}
     \mathfrak{R}\mathfrak{e} (\partial^2_{x_2 x_2} Q \overline{Q \psi^{\neq
     0}}) \right| \leqslant \]
  \[ o^{\lambda, c}_{\| Z - Q_c \|_{H^{\exp}_{Q_c}} \rightarrow 0} (1) + K
     \left( | X | + \frac{\delta^{| . |} (c \overrightarrow{e_2},
     \vec{c}')}{c^2} + \frac{\delta^{\bot} (\nobracket c \overrightarrow{e_2},
     \vec{c}') \nobracket}{c} \right), \]
  and since $Q \psi = Z - Q$ and $\psi^{\neq 0} = \psi - \psi^{\tmmathbf{0}}$
  in $\Omega$, we check that,
  \[ \int_{\Omega} \mathfrak{R}\mathfrak{e} \left( \partial_{x_2} Q
     \overline{\partial_{X_2} (Q \psi^{\neq 0})} \right) = - \int_{\Omega} |
     \partial_{x_2} Q |^2 + \int_{\Omega} \mathfrak{R}\mathfrak{e}
     (\partial_{x_2} Q \overline{\partial_X (Q \psi^{\tmmathbf{0}})}) . \]
  Now, using $Q \psi = Z - Q$, we check that, in $B (\tmmathbf{d}_{\vec{c}',
  1}, R)$, where $x =\tmmathbf{r}_1 e^{i\tmmathbf{\theta}_1}$,
  \begin{eqnarray*}
    2 \pi \partial_{X_2} (Q \psi^0) & = & \partial_{X_2} \left( Q \int_0^{2
    \pi} \frac{Z - Q}{Q} d\tmmathbf{\theta}_1 \right)\\
    & = & \partial_{x_2} Q \int_0^{2 \pi} \frac{Z - Q}{Q}
    d\tmmathbf{\theta}_1\\
    & + & Q \int_0^{2 \pi} \frac{- \partial_{x_2} Q}{Q} d\tmmathbf{\theta}_1
    + Q \int_0^{2 \pi} \frac{- (Z - Q) \partial_{x_2} Q}{Q^2}
    d\tmmathbf{\theta}_1\\
    & + & Q \int_0^{2 \pi} \partial_{x_2} \left( \frac{Z - Q}{Q} \right)
    d\tmmathbf{\theta}_1
  \end{eqnarray*}
  Therefore, we estimate (since $R$ is a universal constant)
  \[ \left| \int_{B (\tmmathbf{d}_{\vec{c}', 1}, R)} \mathfrak{R}\mathfrak{e}
     (\partial_{x_2} Q \overline{\partial_X (Q \psi^0)}) \right| \leqslant \]
  \[ \left| \int_{B (\tmmathbf{d}_{\vec{c}', 1}, R)} \mathfrak{R}\mathfrak{e}
     \left( \overline{\partial_{x_2} Q} Q \int_0^{2 \pi} \frac{-
     \partial_{x_2} Q}{Q} d\tmmathbf{\theta}_1 \right) \right| + K \| Z - Q
     \|_{C^1 (\Omega)} . \]
  Let us show that, in $B (\tmmathbf{d}_{\vec{c}', 1}, R)$,
  \begin{equation}
    Q \int_0^{2 \pi} \frac{- \partial_{x_2} Q}{Q} d\tmmathbf{\theta}_1 = o_{c
    \rightarrow 0} (1) . \label{CP24121110}
  \end{equation}
  We have in this domain that $\frac{Q}{\tmmathbf{V}_1} = 1 + o_{c \rightarrow
  0} (1)$ and $| \nabla Q_c - \nabla \tilde{V}_1 | = o_{c \rightarrow 0} (1)$
  by Lemmas \ref{CP20524} and \ref{CP2closecall}, where $\mathbf{V}_1$ is the
  vortex centered at $\tmmathbf{d}_{\vec{c}', 1}$. We deduce that, in $B
  (\tmmathbf{d}_{\vec{c}', 1}, R)$,
  \[ Q \int_0^{2 \pi} \frac{- \partial_{x_2} Q}{Q} d\tmmathbf{\theta}_1
     =\mathbf{V}_1 \int_0^{2 \pi} \frac{- \partial_{x_2}
     \mathbf{V}_1}{\mathbf{V}_1} d\tmmathbf{\theta}_1 + o_{c \rightarrow 0}
     (1) . \]
  Finally, by Lemma \ref{lemme3new}, we check that $\frac{\partial_{x_2}
  \mathbf{V}_1}{\mathbf{V}_1}$ has no $0$-harmonic around
  $\tmmathbf{d}_{\vec{c}', 1}$, therefore
  \begin{equation}
    \mathbf{V}_1 \int_0^{2 \pi} \frac{- \partial_{x_2}
    \mathbf{V}_1}{\mathbf{V}_1} d\tmmathbf{\theta}_1 = 0. \label{CP2shogun}
  \end{equation}
  By symmetry, the same proof holds in $B (\tmmathbf{d}_{\vec{c}', 2}, R)$.
  
  Adding up these estimates, we get
  \[ \left| \partial_{X_2} \mathfrak{R}\mathfrak{e} \int_{\Omega}
     \partial_{x_2} Q \overline{Q \psi^{\neq 0}} + \int_{\Omega} |
     \partial_{x_2} Q |^2 \right| \leqslant \]
  \[ o^{\lambda, c}_{\| Z - Q_c \|_{H^{\exp}_{Q_c}} \rightarrow 0} (1) + o_{c
     \rightarrow 0} (1) + K \left( | X | + \frac{\delta^{| . |} (c
     \overrightarrow{e_2}, \vec{c}')}{c^2} + \frac{\delta^{\bot} (\nobracket c
     \overrightarrow{e_2}, \vec{c}') \nobracket}{c} + | \gamma | \right) . \]
  By a similar computation, we have
  \[ \left| \partial_{X_2} \mathfrak{R}\mathfrak{e} \int_{\Omega}
     \partial_{\tmmathbf{d}} \tmmathbf{V} \overline{Q \psi^{\neq 0}} \right. -
     \left. \int_{\Omega} \mathfrak{R}\mathfrak{e} \left(
     \partial_{\tmmathbf{d}} \tmmathbf{V} \overline{\partial_{x_2} Q} \right)
     \right| \leqslant \]
  \[ o^{\lambda, c}_{\| Z - Q_c \|_{H^{\exp}_{Q_c}} \rightarrow 0} (1) + o_{c
     \rightarrow 0} (1) + K \left( | X | + \frac{\delta^{| . |} (c
     \overrightarrow{e_2}, \vec{c}')}{c^2} + \frac{\delta^{\bot} (\nobracket c
     \overrightarrow{e_2}, \vec{c}') \nobracket}{c} + | \gamma | \right) . \]
  By Lemma \ref{CP2133L35} and Theorem \ref{th1} (for $p = + \infty$), we have
  \[ \left| \int_{\Omega} \mathfrak{R}\mathfrak{e} \left(
     \partial_{\tmmathbf{d}} \tmmathbf{V} \overline{\partial_{x_2} Q} \right)
     \right| \leqslant \left| \int_{\Omega} \mathfrak{R}\mathfrak{e} \left(
     c^2 \partial_c Q \overline{\partial_{x_2} Q} \right) \right| + \left|
     \int_{\Omega} \mathfrak{R}\mathfrak{e} \left( (\partial_{\tmmathbf{d}}
     \tmmathbf{V}- c^2 \partial_c Q) \overline{\partial_{x_2} Q} \right)
     \right| = o_{c \rightarrow 0} (1) . \]
  Similarly, we check
  \[ \left| \partial_{X_2} \int_{\Omega} \partial_{x_1} Q \overline{Q
     \psi^{\neq 0}} \right| - \left| \int_{\Omega} \mathfrak{R}\mathfrak{e}
     \left( \partial_{x_1} Q \overline{\partial_{x_2} Q} \right) \right|
     \leqslant \]
  \[ o^{\lambda, c}_{\| Z - Q_c \|_{H^{\exp}_{Q_c}} \rightarrow 0} (1) + o_{c
     \rightarrow 0} (1) + K \left( | X | + \frac{\delta^{| . |} (c
     \overrightarrow{e_2}, \vec{c}')}{c^2} + \frac{\delta^{\bot} (\nobracket c
     \overrightarrow{e_2}, \vec{c}') \nobracket}{c} + | \gamma | \right) . \]
  Still by Lemma \ref{CP2133L35}, we have
  \[ \left| \int_{\Omega} \mathfrak{R}\mathfrak{e} \left( \partial_{x_1} Q
     \overline{\partial_{x_2} Q} \right) \right| = o_{c \rightarrow 0} (1) .
  \]
  With the same arguments, we check that
  \[ \left| \partial_{X_2} \int_{\Omega} c \partial_{c^{\bot}} Q \overline{Q
     \psi^{\neq 0}} \right| \leqslant \]
  \[ o^{\lambda, c}_{\| Z - Q_c \|_{H^{\exp}_{Q_c}} \rightarrow 0} (1) + o_{c
     \rightarrow 0} (1) + K \left( | X | + \frac{\delta^{| . |} (c
     \overrightarrow{e_2}, \vec{c}')}{c^2} + \frac{\delta^{\bot} (\nobracket c
     \overrightarrow{e_2}, \vec{c}') \nobracket}{c} + | \gamma | \right) . \]
  Finally, with equations (\ref{CP2217}) to (\ref{CP2221}) and
  (\ref{CP244124}), we check easily that
  \[ \partial_{X_2} \left( \mathfrak{R}\mathfrak{e} \int_{\Omega'} i \psi
     \right) \leqslant o^{\lambda, c}_{\| Z - Q_c \|_{H^{\exp}_{Q_c}}
     \rightarrow 0} (1) + o_{c \rightarrow 0} (1) + K \left( | X | +
     \frac{\delta^{| . |} (c \overrightarrow{e_2}, \vec{c}')}{c^2} +
     \frac{\delta^{\bot} (\nobracket c \overrightarrow{e_2}, \vec{c}')
     \nobracket}{c} + | \gamma | \right) . \]
  We deduce that
  \[ \left| \partial_{X_2} G \left(\begin{array}{c}
       X_1\\
       X_2\\
       \delta_1\\
       \delta_2\\
       \gamma
     \end{array}\right) + \left(\begin{array}{c}
       0\\
       \int_{\Omega} | \partial_{x_2} Q |^2\\
       0\\
       0\\
       0
     \end{array}\right) \right| \leqslant \]
  \[ o^{\lambda, c}_{\| Z - Q_c \|_{H^{\exp}_{Q_c}} \rightarrow 0} (1) + o_{c
     \rightarrow 0} (1) + K \left( | X | + \frac{\delta^{| . |} (c
     \overrightarrow{e_2}, \vec{c}')}{c^2} + \frac{\delta^{\bot} (\nobracket c
     \overrightarrow{e_2}, \vec{c}') \nobracket}{c} + | \gamma | \right) . \]
  We can also check, with similar computations, that
  \[ \left| \partial_{X_1} G \left(\begin{array}{c}
       X_1\\
       X_2\\
       \delta_1\\
       \delta_2\\
       \gamma
     \end{array}\right) + \left(\begin{array}{c}
       \int_{\Omega} | \partial_{x_1} Q |^2\\
       0\\
       0\\
       0\\
       0
     \end{array}\right) \right| \leqslant \]
  \[ o^{\lambda, c}_{\| Z - Q_c \|_{H^{\exp}_{Q_c}} \rightarrow 0} (1) + o_{c
     \rightarrow 0} (1) + K \left( | X | + \frac{\delta^{| . |} (c
     \overrightarrow{e_2}, \vec{c}')}{c^2} + \frac{\delta^{\bot} (\nobracket c
     \overrightarrow{e_2}, \vec{c}') \nobracket}{c} + | \gamma | \right) . \]

  We infer that this also holds with a similar proof for the last two
  directions, namely

  \[ \left| c^2 \partial_{\delta_1} G \left(\begin{array}{c}
       X_1\\
       X_2\\
       \delta_1\\
       \delta_2\\
       \gamma
     \end{array}\right) + \left(\begin{array}{c}
       0\\
       0\\
       \int_{\Omega} | c^2 \partial_c Q |^2\\
       0\\
       0
     \end{array}\right) \right| \leqslant \]
  \[ o^{\lambda, c}_{\| Z - Q_c \|_{H^{\exp}_{Q_c}} \rightarrow 0} (1) + o_{c
     \rightarrow 0} (1) + K \left( | X | + \frac{\delta^{| . |} (c
     \overrightarrow{e_2}, \vec{c}')}{c^2} + \frac{\delta^{\bot} (\nobracket c
     \overrightarrow{e_2}, \vec{c}') \nobracket}{c} + | \gamma | \right) \]
  (using the fact that $\partial_{\tmmathbf{d}} \tmmathbf{V}$ is
  differentiable with respect to $\delta_1$, which is not obvious for $c^2
  \partial_c Q$ and is the reason we have to use this orthogonality) and
  \[ \left| c^{} \partial_{\delta_2} G \left(\begin{array}{c}
       X_1\\
       X_2\\
       \delta_1\\
       \delta_2\\
       \gamma
     \end{array}\right) + \left(\begin{array}{c}
       0\\
       0\\
       0\\
       \int_{\Omega} | c \partial_{c^{\bot}} Q |^2\\
       0
     \end{array}\right) \right| \leqslant \]
  \[ o^{\lambda, c}_{\| Z - Q_c \|_{H^{\exp}_{Q_c}} \rightarrow 0} (1) + o_{c
     \rightarrow 0} (1) + K \left( | X | + \frac{\delta^{| . |} (c
     \overrightarrow{e_2}, \vec{c}')}{c^2} + \frac{\delta^{\bot} (\nobracket c
     \overrightarrow{e_2}, \vec{c}') \nobracket}{c} + | \gamma | \right) . \]
  We will only show for these directions that, in $B (\tmmathbf{d}_{\vec{c}',
  1}, R)$,
  \[ \left| Q \int_0^{2 \pi} \frac{c^2 \partial_c Q}{Q} d\tmmathbf{\theta}_1
     \right| + \left| Q \int_0^{2 \pi} \frac{c \partial_{c^{\bot}} Q}{Q}
     d\tmmathbf{\theta}_1 \right| = o_{c \rightarrow 0} (1), \]
  the other computations are similar to the ones done for $\partial_{X_2} F$
  (using Lemma \ref{CP2133L35}).
  
  We recall from Lemma \ref{CP2dcQcsigma} that, in $B (\tmmathbf{d}_{\vec{c}',
  1}, R)$,
  \[ \| c^2 \partial_c Q - \partial_{\tmmathbf{d}} \tmmathbf{V} \|_{C^1 (B
     (\tmmathbf{d}_{\vec{c}', 1}, R))} = o_{c \rightarrow 0} (1), \]
  where $\| \partial_{\tmmathbf{d}} \tmmathbf{V}+ \partial_{x_1} V_1 \|_{C^1
  (B (\tmmathbf{d}_{\vec{c}', 1}, R))} = o_{c \rightarrow 0} (1)$, $V_1$ being
  centered around a point $d_{\vec{c}'} \in \mathbbm{R}^2$ such that
  \[ | d_{\vec{c}'} -\tmmathbf{d}_{\vec{c}', 1} | = o_{c \rightarrow 0} (1) .
  \]
  Therefore, we check that
  \begin{eqnarray*}
    \left| Q \int_0^{2 \pi} \frac{c^2 \partial_c Q}{Q} d\tmmathbf{\theta}_1
    \right| & \leqslant & \left| \mathbf{V}_1 \int_0^{2 \pi}
    \frac{\partial_{x_1} \mathbf{V}_1}{\mathbf{V}_1} d\tmmathbf{\theta}_1
    \right| + o_{c \rightarrow 0} (1)\\
    & = & o_{c \rightarrow 0} (1)
  \end{eqnarray*}
  from (\ref{CP2shogun}). Finally, we have, from Lemma \ref{CP2a} that
  $\partial_{c^{\bot}} Q = - x^{\bot, \delta^{\bot} (\nobracket c
  \overrightarrow{e_2}, \vec{c}') \nobracket} . \nabla Q$, where $x^{\bot,
  \delta^{\bot} (\nobracket c \overrightarrow{e_2}, \vec{c}') \nobracket}$ is
  $x^{\bot}$ rotated by an angle $\delta^{\bot} (\nobracket c
  \overrightarrow{e_2}, \vec{c}') \nobracket$. We remark that, in $B
  (\tmmathbf{d}_{\vec{c}', 1}, R)$,
  \[ \left| Q \int_0^{2 \pi} \frac{c\tmmathbf{d}_{\vec{c}', 1} . \nabla Q}{Q}
     d\tmmathbf{\theta}_1 \right| \leqslant \left| \mathbf{V}_1 \int_0^{2 \pi}
     \frac{c\tmmathbf{d}_{\vec{c}', 1} . \nabla \mathbf{V}_1}{\mathbf{V}_1}
     d\tmmathbf{\theta}_1 \right| + o_{c \rightarrow 0} (1) \]
  and
  \[ \left| \mathbf{V}_1 \int_0^{2 \pi} \frac{c\tmmathbf{d}_{\vec{c}', 1} .
     \nabla \mathbf{V}_1}{\mathbf{V}_1} d\tmmathbf{\theta}_1 \right| = 0 \]
  by (\ref{CP2shogun}) and the same result for $\partial_{x_1}$ instead of
  $\partial_{x_2}$. Therefore, since $\left| x^{\bot, \delta^{\bot}
  (\nobracket c \overrightarrow{e_2}, \vec{c}') \nobracket}
  -\tmmathbf{d}_{\vec{c}', 1} \right| \leqslant K$ in $B
  (\tmmathbf{d}_{\vec{c}', 1}, R)$,
  \begin{eqnarray*}
    \left| Q \int_0^{2 \pi} \frac{c \partial_{c^{\bot}} Q}{Q}
    d\tmmathbf{\theta}_1 \right| & \leqslant & \left| Q \int_0^{2 \pi} \frac{c
    \left( x^{\bot, \delta^{\bot} (\nobracket c \overrightarrow{e_2},
    \vec{c}') \nobracket} -\tmmathbf{d}_{\vec{c}', 1} \right) . \nabla Q}{Q}
    d\tmmathbf{\theta}_1 \right| + o_{c \rightarrow 0} (1)\\
    & \leqslant & K c + o_{c \rightarrow 0} (1)\\
    & = & o_{c \rightarrow 0} (1) .
  \end{eqnarray*}
  Finally, we infer that
  \[ \left| \partial_{\gamma} G \left(\begin{array}{c}
       X_1\\
       X_2\\
       \delta_1\\
       \delta_2\\
       \gamma
     \end{array}\right) + \left(\begin{array}{c}
       0\\
       0\\
       0\\
       0\\
       \mathfrak{R}\mathfrak{e} \int_{\Omega'} Q
     \end{array}\right) \right| \leqslant \]
  \[ o^{\lambda, c}_{\| Z - Q_c \|_{H^{\exp}_{Q_c}} \rightarrow 0} (1) + o_{c
     \rightarrow 0} (1) + K \left( | X | + \frac{\delta^{| . |} (c
     \overrightarrow{e_2}, \vec{c}')}{c^2} + \frac{\delta^{\bot} (\nobracket c
     \overrightarrow{e_2}, \vec{c}') \nobracket}{c} + | \gamma | \right) . \]
  The proof is similar of the previous computations, and we will only show
  that, in $\Omega$,
  \[ | \partial_{\gamma} (Q \psi^{\neq 0}) | \leqslant o^{\lambda, c}_{\| Z -
     Q_c \|_{H^{\exp}_{Q_c}} \rightarrow 0} (1) . \]
  We have
  \begin{eqnarray*}
    | \partial_{\gamma} (Q \psi^{\neq 0}) | & = & | \partial_{\gamma} (Q \psi)
    - \partial_{\gamma} (Q \psi^{\tmmathbf{0}}) |\\
    & \leqslant & \left| - i Q - \frac{Q}{2 \pi} \int_0^{2 \pi} \frac{- i
    Q}{Q} d \theta \right| + o^{\lambda, c}_{\| Z - Q_c \|_{H^{\exp}_{Q_c}}
    \rightarrow 0} (1)\\
    & \leqslant & o^{\lambda, c}_{\| Z - Q_c \|_{H^{\exp}_{Q_c}} \rightarrow
    0} (1) .
  \end{eqnarray*}
  From Theorem \ref{th1}, $\mathfrak{R}\mathfrak{e} \int_{\Omega'} Q
  =\mathfrak{R}\mathfrak{e} \int_{\Omega'} - 1 + o_{c \rightarrow 0} (1)
  \leqslant - K < 0$. We conclude, by Lemma \ref{CP2133L35}, that, for $c$ and
  $\| Z - Q_c \|_{H^{\exp}_{Q_c}}$ small enough, $d G$ is invertible in a
  vicinity of $(0, 0, 0, 0, 0)$ of size independent of $\| Z - Q_c
  \|_{H^{\exp}_{Q_c}}$. Therefore, by the implicit function theorem, taking
  $c$ small enough and $\varepsilon (c, \lambda)$ small enough, we can find
  $X, \vec{c}' \in \mathbbm{R}^2, \gamma \in \mathbbm{R}$ such that
  \[ | X | + \frac{\delta^{| . |} (c \overrightarrow{e_2}, \vec{c}')}{c^2} +
     \frac{\delta^{\bot} (\nobracket c \overrightarrow{e_2}, \vec{c}')
     \nobracket}{c} + | \gamma | \leqslant o^{\lambda, c}_{\| Z - Q_c
     \|_{H^{\exp}_{Q_c}} \rightarrow 0} (1), \]
  and satisfying
  \[ \mathfrak{R}\mathfrak{e} \int_{B (\tmmathbf{d}_{\vec{c}', 1}, R) \cup B
     (\tmmathbf{d}_{\vec{c}', 2}, R)} \partial_{x_1} Q \overline{Q \psi^{\neq
     0}} =\mathfrak{R}\mathfrak{e} \int_{B (\tmmathbf{d}_{\vec{c}', 1}, R)
     \cup B (\tmmathbf{d}_{\vec{c}', 2}, R)} \partial_{x_2} Q \overline{Q
     \psi^{\neq 0}} = 0, \]
  \[ \mathfrak{R}\mathfrak{e} \int_{B (\tmmathbf{d}_{\vec{c}', 1}, R) \cup B
     (\tmmathbf{d}_{\vec{c}', 2}, R)} \partial_{\tmmathbf{d}} \tmmathbf{V}
     \overline{Q \psi^{\neq 0}} =\mathfrak{R}\mathfrak{e} \int_{B
     (\tmmathbf{d}_{\vec{c}', 1}, R) \cup B (\tmmathbf{d}_{\vec{c}', 2}, R)}
     \partial_{c^{\bot}} Q \overline{Q \psi^{\neq 0}} = 0, \]
  \[ \mathfrak{R}\mathfrak{e} \int_{B ((\tmmathbf{d}_{\vec{c}', 1}
     +\tmmathbf{d}_{\vec{c}', 2}) / 2, R)} i \psi = 0. \]
\end{proof}

\subsection{End of the proof of Theorem \ref{CP2th16}}

From Lemmas \ref{CP2L531806} and \ref{CP2L452905}, we can find $\varphi = Q
\psi \in H^{\exp}_Q$ such that
\begin{equation}
  | X | + \frac{\delta^{| . |} (c \overrightarrow{e_2}, \vec{c}')}{c^2} +
  \frac{\delta^{\bot} (\nobracket c \overrightarrow{e_2}, \vec{c}')
  \nobracket}{c} + | \gamma | \leqslant o^{\lambda, c}_{\| Z - Q_c
  \|_{H^{\exp}_{Q_c}} \rightarrow 0} (1), \label{CP24143005}
\end{equation}
and
\[ \mathfrak{R}\mathfrak{e} \int_{B (\tmmathbf{d}_{\vec{c}', 1}, R) \cup B
   (\tmmathbf{d}_{\vec{c}', 2}, R)} \partial_{x_1} Q \overline{Q \psi^{\neq
   0}} =\mathfrak{R}\mathfrak{e} \int_{B (\tmmathbf{d}_{\vec{c}', 1}, R) \cup
   B (\tmmathbf{d}_{\vec{c}', 2}, R)} \partial_{x_2} Q \overline{Q \psi^{\neq
   0}} = 0, \]
\[ \mathfrak{R}\mathfrak{e} \int_{B (\tmmathbf{d}_{\vec{c}', 1}, R) \cup B
   (\tmmathbf{d}_{\vec{c}', 2}, R)} \partial_{\tmmathbf{d}} \tmmathbf{V}
   \overline{Q \psi^{\neq 0}} =\mathfrak{R}\mathfrak{e} \int_{B
   (\tmmathbf{d}_{\vec{c}', 1}, R) \cup B (\tmmathbf{d}_{\vec{c}', 2}, R)}
   \partial_{c^{\bot}} Q \overline{Q \psi^{\neq 0}} = 0, \]
\[ \mathfrak{R}\mathfrak{e} \int_{B ((\tmmathbf{d}_{\vec{c}', 1}
   +\tmmathbf{d}_{\vec{c}', 2}) / 2, R)} i \psi = 0. \]
Now, from Lemma \ref{CP2L442905}, $\psi$ satisfies the equation
\begin{equation}
  L_Q (Q \psi) - i (\vec{c}' - c \overrightarrow{e_2}) .H (\psi) +
  \tmop{NL}_{\tmop{loc}} (\psi) + F (\psi) = 0. \label{CP25152610}
\end{equation}
We remark that
\[ L_Q (Q \psi) = (1 - \eta) L_Q (Q \psi) + \eta Q L'_Q (\psi), \]
and by Lemmas \ref{CP2L531806} and \ref{CP2L3162011},
\[ \langle (1 - \eta) L_Q (Q \psi) + \eta Q L'_Q (\psi), Q (\psi + i \gamma)
   \rangle = B^{\exp}_Q (\varphi) . \]
We deduce that
\begin{equation}
  B^{\exp}_Q (\varphi) - \langle i (\vec{c}' - c \overrightarrow{e_2}) . H
  (\psi), Q (\psi + i \gamma) \rangle + \langle \tmop{NL}_{\tmop{loc}} (\psi),
  Q (\psi + i \gamma) \rangle + \langle F (\psi), Q (\psi + i \gamma) \rangle
  = 0. \label{CP22905}
\end{equation}
Since $Q \psi \in H^{\exp}_Q$ by Lemma \ref{CP2L531806}, with the
orthogonality conditions satisfied (see Lemma \ref{CP2L452905}), we can apply
Propositions \ref{CP2prop16} and \ref{CP2prop17} with remark
(\ref{CP2corto2}). We have
\begin{equation}
  B^{\exp}_Q (\varphi) \geqslant K \| \varphi \|_{\mathcal{C}}^2 + K (c) \|
  \varphi \|^2_{H^{\exp}_Q} . \label{CP229052}
\end{equation}

\subsubsection{Better estimates on $\vec{c}' - c \protect\overrightarrow{e_2}$}

The term $i (\vec{c}' - c \overrightarrow{e_2}) .H (\psi)$ contains a
``source'' term, because $Z$ and $Q$ do not satisfy the same equation (since
the travelling waves $Z$ and $Q$ may not have the same speed at this point).
We want to show the following estimates:
\begin{equation}
  \delta^{| . |} (c \overrightarrow{e_2}, \vec{c}') \leqslant \left( K c^2 \ln
  \left( \frac{1}{c} \right) + o^{\lambda, c}_{\| Z - Q_c \|_{H^{\exp}_{Q_c}}
  \rightarrow 0} (1) \right) \| \varphi \|_{\mathcal{C}} + o^{\lambda, c}_{\|
  Z - Q_c \|_{H^{\exp}_{Q_c}} \rightarrow 0} (1) \| \varphi
  \|_{H^{\exp}_{Q_c}} \label{CP24182910}
\end{equation}
and
\begin{equation}
  \delta^{\bot} (c \overrightarrow{e_2}, \vec{c}') \leqslant \left( K c^2 \ln
  \left( \frac{1}{c} \right) + o^{\lambda, c}_{\| Z - Q_c \|_{H^{\exp}_{Q_c}}
  \rightarrow 0} (1) \right) \| \varphi \|_{\mathcal{C}} + o^{\lambda, c}_{\|
  Z - Q_c \|_{H^{\exp}_{Q_c}} \rightarrow 0} (1) \| \varphi
  \|_{H^{\exp}_{Q_c}}, \label{CP24192910}
\end{equation}
where $\delta^{| . |} (c \overrightarrow{e_2}, \vec{c}') = | c
\overrightarrow{e_2} . \nobracket \frac{\vec{c}'}{| \vec{c}' |} - \nobracket
\vec{c}' |$ and $\delta^{\bot} (c \overrightarrow{e_2}, \vec{c}') = | c
\overrightarrow{e_2} . \nobracket \frac{\vec{c}^{\prime \bot}}{| \vec{c}' |} -
\nobracket \vec{c}' |$.

\

This subsection is devoted to the proof of (\ref{CP24182910}) and
(\ref{CP24192910}).

\begin{tmindent}
  Step 1.  We have the estimate (\ref{CP24182910}).
\end{tmindent}

We take the scalar product of (\ref{CP25152610}) with $c^2 \partial_c Q$,
which yields
\[ \langle i (\vec{c}' - c \overrightarrow{e_2}) . H (\psi), c^2 \partial_c Q
   \rangle = \langle Q \psi, c^2 L_Q (\partial_c Q) \rangle + \langle
   \tmop{NL}_{\tmop{loc}} (\psi) + F (\psi), c^2 \partial_c Q \rangle . \]
We check here, with the $L^{\infty}$ estimates on $\psi$ and its derivatives,
as well as on $\partial_c Q$ (see Lemma \ref{CP2dcQcsigma} and
\ref{CP2L531806}), that $\langle L_Q (Q \psi), c^2 \partial_c Q \rangle$ is
well defined and that all the integrations by parts can be done.

We recall that $H (\psi) = \nabla Q + \frac{\nabla (Q \psi) (1 - \eta) + Q
\nabla \psi \eta e^{\psi}}{(1 - \eta) + \eta e^{\psi}}$, and we check that,
since $1 - \eta$ is compactly supported (in a domain with size independent of
$c, \vec{c}'$), with equation (\ref{CP24133005})
\begin{eqnarray*}
  \left| \left\langle i (\vec{c}' - c \overrightarrow{e_2}) . \frac{\nabla (Q
  \psi) (1 - \eta) + Q \nabla \psi \eta e^{\psi}}{(1 - \eta) + \eta e^{\psi}},
  c^2 \partial_c Q \right\rangle \right| & \leqslant & K | (\vec{c}' - c
  \overrightarrow{e_2}) . \langle \eta i Q \nabla \psi, c^2 \partial_c Q
  \rangle |\\
  & + & K | \vec{c}' - c \overrightarrow{e_2} | \| \varphi
  \|_{H^{\exp}_{Q_c}} .
\end{eqnarray*}
We compute with Lemma \ref{CP2dcQcsigma} that
\begin{eqnarray*}
  | \langle \eta i Q \nabla \psi, c^2 \partial_c Q \rangle | & = & \left|
  \int_{\mathbbm{R}^2} \eta \mathfrak{R}\mathfrak{e} (\nabla \psi i Q c^2
  \overline{\partial_c Q}) \right|\\
  & \leqslant & \left| \int_{\mathbbm{R}^2} \eta \mathfrak{R}\mathfrak{e}
  (\nabla \psi) \mathfrak{I}\mathfrak{m} (Q c^2 \overline{\partial_c Q})
  \right| + \left| \int_{\mathbbm{R}^2} \eta \mathfrak{I}\mathfrak{m} (\nabla
  \psi) \mathfrak{R}\mathfrak{e} (Q c^2 \overline{\partial_c Q}) \right|\\
  & \leqslant & \left| \int_{\mathbbm{R}^2} \eta \mathfrak{R}\mathfrak{e}
  (\psi) \nabla (\mathfrak{I}\mathfrak{m} (Q c^2 \overline{\partial_c Q}))
  \right| + K \| \varphi \|_{H^{\exp}_{Q_c}}\\
  & + & \| \varphi \|_{\mathcal{C}} \sqrt{\int_{\mathbbm{R}^2} \eta
  \mathfrak{R}\mathfrak{e}^2 (Q c^2 \overline{\partial_c Q})} .
\end{eqnarray*}
From Lemmas \ref{CP283L33} and \ref{CP2dcQcsigma}, we check that
$\int_{\mathbbm{R}^2} \eta \mathfrak{R}\mathfrak{e}^2 (Q c^2
\overline{\partial_c Q}) \leqslant K$, and furthermore,
\[ | \nabla (\mathfrak{I}\mathfrak{m} (Q c^2 \overline{\partial_c Q})) |
   \leqslant c^2 | \partial_c Q | | \nabla Q | + K c^2 | \nabla \partial_c Q |
\]
and with Lemma \ref{CP2dcQcsigma} (with $\sigma = 1 / 2$), we check that
\[ | \nabla (\mathfrak{I}\mathfrak{m} (Q c^2 \overline{\partial_c Q})) |
   \leqslant \frac{K}{(1 + \tilde{r})^{3 / 2}}, \]
thus, by Cauchy-Schwarz,
\[ \left| \int_{\mathbbm{R}^2} \eta \mathfrak{R}\mathfrak{e} (\psi) \nabla
   (\mathfrak{I}\mathfrak{m} (Q c^2 \overline{\partial_c Q})) \right|
   \leqslant K \| \varphi \|_{\mathcal{C}} . \]
Using $| \vec{c}' - c \overrightarrow{e_2} | \leqslant K (c) (\delta^{| . |}
(c \overrightarrow{e_2}, \vec{c}') + \delta^{\bot} (c \overrightarrow{e_2},
\vec{c}')) \leqslant o^{\lambda, c}_{\| Z - Q_c \|_{H^{\exp}_{Q_c}}
\rightarrow 0} (1)$ and $\| \varphi \|_{\mathcal{C}} \leqslant K \| \varphi
\|_{H^{\exp}_{Q_c}}$, we deduce that
\[ \left| \left\langle i (\vec{c}' - c \overrightarrow{e_2}) . \frac{(1 -
   \eta) \nabla (Q \psi) + \eta e^{\psi} Q \nabla \psi}{(1 - \eta) + \eta
   e^{\psi}}, c^2 \partial_c Q \right\rangle \right| \leqslant o^{\lambda,
   c}_{\| Z - Q_c \|_{H^{\exp}_{Q_c}} \rightarrow 0} (1) . \]

Furthermore, we check that, by symmetry (see (\ref{CP2sym})),
\[ \langle i (\vec{c}' - c \overrightarrow{e_2}) . \nabla_x Q, c^2 \partial_c
   Q \rangle = \delta^{| . |} (c \overrightarrow{e_2}, \vec{c}') \left\langle
   i \frac{\vec{c}'}{| \vec{c}' |} . \nabla Q, c^2 \partial_c Q \right\rangle
   . \]
Furthermore, from Lemma \ref{CP20703L222}, we have $L_Q (\partial_c Q) = i
\nabla_{\vec{c}'} Q$, therefore, from Proposition \ref{CP2prop5},
\[ \left\langle i \frac{\vec{c}'}{| \vec{c}' |} . \nabla Q, c^2 \partial_c Q
   \right\rangle = c^2 B_Q (\partial_c Q) = - 2 \pi + o_{c \rightarrow 0} (1)
   . \]
We deduce that
\[ \delta^{| . |} (c \overrightarrow{e_2}, \vec{c}') \leqslant K | \langle Q
   \psi, c^2 L_Q (\partial_c Q) \rangle + \langle \tmop{NL}_{\tmop{loc}}
   (\psi) + F (\psi), c^2 \partial_c Q \rangle | + o^{\lambda, c}_{\| Z - Q_c
   \|_{H^{\exp}_{Q_c}} \rightarrow 0} (1) \| \varphi \|_{H^{\exp}_Q} . \]
Now, since $L_Q (\partial_c Q) = i \frac{\vec{c}'}{| \vec{c}' |} . \nabla Q$,
we check that
\[ \langle Q \psi, c^2 L_Q (\partial_c Q) \rangle = c^2 \left\langle Q \psi, i
   \frac{\vec{c}'}{| \vec{c}' |} . \nabla Q \right\rangle, \]
and
\[ \left| \left\langle Q \psi, i \frac{\vec{c}'}{| \vec{c}' |} . \nabla Q
   \right\rangle \right| \leqslant \left| \int_{\mathbbm{R}^2}
   \mathfrak{R}\mathfrak{e} (\psi) \mathfrak{I}\mathfrak{m} \left(
   \frac{\vec{c}'}{| \vec{c}' |} . \nabla Q \bar{Q} \right) \right| + \left|
   \int_{\mathbbm{R}^2} \mathfrak{I}\mathfrak{m} (\psi)
   \mathfrak{R}\mathfrak{e} \left( \frac{\vec{c}'}{| \vec{c}' |} . \nabla Q
   \bar{Q} \right) \right| . \]
From Lemma \ref{CP2L390911}, we deduce that
\[ | \langle Q \psi, c^2 L_Q (\partial_c Q) \rangle | \leqslant K c^2 \ln
   \left( \frac{1}{c} \right) \| \varphi \|_{\mathcal{C}} . \]
Now, we check easily that, with Lemmas \ref{CP2L422905} and \ref{CP2L442905},
\[ | \langle \tmop{NL}_{\tmop{loc}} (\psi), c^2 \partial_c Q \rangle |
   \leqslant K (c) \| \varphi \|_{H^{\exp}_Q} \| \varphi \|_{C^1 (B (0,
   \lambda))} \leqslant o^{\lambda, c}_{\| Z - Q_c \|_{H^{\exp}_{Q_c}}
   \rightarrow 0} (1) \| \varphi \|_{H^{\exp}_Q} . \]
To conclude the proof of estimate (\ref{CP24182910}), we shall estimate
\[ | \langle F (\psi), c^2 \partial_c Q \rangle | \leqslant o^{\lambda, c}_{\|
   Z - Q_c \|_{H^{\exp}_{Q_c}} \rightarrow 0} (1) \| \varphi \|_{H^{\exp}_Q} +
   \left( K \lambda_0 + o^{\lambda, c}_{\| Z - Q_c \|_{H^{\exp}_{Q_c}}
   \rightarrow 0} (1) \right) \| \varphi \|_C, \]
with $F (\psi) = Q \eta (- \nabla \psi . \nabla \psi + | Q |^2 S (\psi))$.
First, we estimate, for $\Lambda > \lambda > \frac{10}{c}$, with Lemma
\ref{CP2L432905},
\begin{eqnarray*}
  | \langle - Q \eta \nabla \psi . \nabla \psi, c^2 \partial_c Q \rangle | & =
  & \left| \int_{\mathbbm{R}^2} \eta \mathfrak{R}\mathfrak{e} (\nabla \psi .
  \nabla \psi c^2 \bar{Q} \partial_c Q) \right|\\
  & \leqslant & \int_{\mathbbm{R}^2} \eta | \nabla \psi |^2 | c^2 \bar{Q}
  \partial_c Q |\\
  & \leqslant & K \| \nabla \psi \|_{L^{\infty} (B (0, \lambda) \cap \{ \eta
  \neq 0 \})} \sqrt{\int_{B (0, \lambda)} \eta | \nabla \psi |^2}
  \sqrt{\int_{B (0, \lambda)} \eta | c^2 \bar{Q} \partial_c Q |^2}\\
  & + & \| c^2 \bar{Q} \partial_c Q \|_{L^{\infty} (\mathbbm{R}^2 \backslash
  B (0, \Lambda))} \int_{\mathbbm{R}^2 \backslash B (0, \lambda)} \eta |
  \nabla \psi |^2\\
  & \leqslant & o^{\Lambda, c}_{\| Z - Q_c \|_{H^{\exp}_{Q_c}} \rightarrow 0}
  (1) \| \varphi \|_{\mathcal{C}} + o_{\Lambda \rightarrow \infty} (1) \|
  \varphi \|_{\mathcal{C}},
\end{eqnarray*}
since, by Lemma \ref{CP2dcQcsigma}, $| c^2 \bar{Q} \partial_c Q | \leqslant
\frac{K}{(1 + \tilde{r})^{1 / 2}}$. We deduce that
\[ | \langle - Q \eta \nabla \psi . \nabla \psi, c^2 \partial_c Q \rangle |
   \leqslant o^{\lambda, c}_{\| Z - Q_c \|_{H^{\exp}_{Q_c}} \rightarrow 0} (1)
   \| \varphi \|_{\mathcal{C}} . \]
Now, in $\{ \eta = 1 \}$, since $e^{\psi} = \frac{Z}{Q}$ and $1 - K \lambda_0
\leqslant \frac{| Z |}{| Q |} \leqslant 1 + K \mu_0$ (by our assumptions on
$Z$), we have $| \mathfrak{R}\mathfrak{e} (\psi) | \leqslant K \mu_0$. We
deduce, with Lemma \ref{CP2L422905}, that in $\{ \eta \neq 0 \}$,
\[ | \mathfrak{R}\mathfrak{e} (\psi) | \leqslant K \mu_0 + o^{\lambda, c}_{\|
   Z - Q_c \|_{H^{\exp}_{Q_c}} \rightarrow 0} (1) . \]
With $S (\psi) = e^{2\mathfrak{R}\mathfrak{e} (\psi)} - 1 -
2\mathfrak{R}\mathfrak{e} (\psi)$, we check that, in $\eta \neq 0$, $| S
(\psi) | \leqslant K | \mathfrak{R}\mathfrak{e} (\psi) |^2$ (given that
$\mu_0$ and $\| Z - Q_c \|_{H^{\exp}_{Q_c}}$ are small enough), and with
similar computations as for $| \langle - Q \eta \nabla \psi . \nabla \psi, c^2
\partial_c Q \rangle |$, we conclude that
\[ | \langle F (\psi), c^2 \partial_c Q \rangle | \leqslant o^{\lambda, c}_{\|
   Z - Q_c \|_{H^{\exp}_{Q_c}} \rightarrow 0} (1) \| \varphi \|_{\mathcal{C}}
   . \]
This concludes the proof of
\[ \delta^{| . |} (c \overrightarrow{e_2}, \vec{c}') \leqslant o^{\lambda,
   c}_{\| Z - Q_c \|_{H^{\exp}_{Q_c}} \rightarrow 0} (1) \| \varphi
   \|_{H^{\exp}_{Q_{}}} + \left( K c^2 \ln \left( \frac{1}{c} \right) +
   o^{\lambda, c}_{\| Z - Q_c \|_{H^{\exp}_{Q_c}} \rightarrow 0} (1) \right)
   \| \varphi \|_{\mathcal{C}} . \]

\begin{tmindent}
  Step 2.  We have the estimate (\ref{CP24192910}).
\end{tmindent}

Now, we take the scalar product of (\ref{CP25152610}) with $c
\partial_{c^{\bot}} Q$:
\[ \langle i (\vec{c}' - c \overrightarrow{e_2}) . H (\psi), c
   \partial_{c^{\bot}} Q \rangle = \langle Q \psi, c L_Q (\partial_{c^{\bot}}
   Q) \rangle + \langle \tmop{NL}_{\tmop{loc}} (\psi) + F (\psi), c
   \partial_{c^{\bot}} Q \rangle . \]
We check that, since
\[ \begin{array}{lll}
     \left\langle i (\vec{c}' - c \overrightarrow{e_2}) . \frac{\nabla (Q
     \psi) (1 - \eta) + Q \nabla \psi \eta e^{\psi}}{(1 - \eta) + \eta
     e^{\psi}}, c \partial_{c^{\bot}} Q \right\rangle & \leqslant & K |
     (\vec{c}' - c \overrightarrow{e_2}) . \langle (1 - \eta) i Q \nabla \psi,
     c \partial_{c^{\bot}} Q \rangle |\\
     & + & K | \vec{c}' - c \overrightarrow{e_2} | \| \varphi
     \|_{H^{\exp}_{Q_c}},
   \end{array} \]
and
\[ \begin{array}{lll}
     | \langle \eta i Q \nabla \psi, c \partial_{c^{\bot}} Q \rangle | & = &
     \left| \int_{\mathbbm{R}^2} \eta \mathfrak{R}\mathfrak{e} \left( \nabla
     \psi i Q c \overline{\partial_{c^{\bot}} Q} \right) \right|\\
     & \leqslant & \left| \int_{\mathbbm{R}^2} \eta \mathfrak{R}\mathfrak{e}
     (\nabla \psi) \mathfrak{I}\mathfrak{m} \left( Q c
     \overline{\partial_{c^{\bot}} Q} \right) \right| + \left|
     \int_{\mathbbm{R}^2} \eta \mathfrak{I}\mathfrak{m} (\nabla \psi)
     \mathfrak{R}\mathfrak{e} \left( Q c \overline{\partial_{c^{\bot}} Q}
     \right) \right|\\
     & \leqslant & \left| \int_{\mathbbm{R}^2} \eta \mathfrak{R}\mathfrak{e}
     (\psi) \nabla \left( \mathfrak{I}\mathfrak{m} \left( Q c
     \overline{\partial_{c^{\bot}} Q} \right) \right) \right| + K \| \varphi
     \|_{H^{\exp}_{Q_c}}\\
     & + & \| \varphi \|_{\mathcal{C}} \int_{\mathbbm{R}^2} \eta
     \mathfrak{R}\mathfrak{e}^2 \left( Q c \overline{\partial_{c^{\bot}} Q}
     \right) .
   \end{array} \]
We check, with Lemmas \ref{CP283L33} and \ref{CP2dcQcsigma}, that
\[ \int_{\mathbbm{R}^2} \eta \mathfrak{R}\mathfrak{e}^2 \left( Q c
   \overline{\partial_{c^{\bot}} Q} \right) \leqslant K \]
and
\[ \left| \nabla \left( \mathfrak{I}\mathfrak{m} \left( Q
   \overline{\partial_{c^{\bot}} Q} \right) \right) \right| \leqslant | \nabla
   Q | | \partial_{c^{\bot}} Q | + | \nabla \partial_{c^{\bot}} Q | \leqslant
   \frac{K (c)}{(1 + r)^2}, \]
therefore, as for the previous estimation,
\[ \left| \left\langle i (\vec{c}' - c \overrightarrow{e_2}) . \frac{(1 -
   \eta) \nabla (Q \psi) + \eta e^{\psi} Q \nabla \psi}{(1 - \eta) + \eta
   e^{\psi}}, c \partial_{c^{\bot}} Q \right\rangle \right| \leqslant
   o^{\lambda, c}_{\| Z - Q_c \|_{H^{\exp}_{Q_c}} \rightarrow 0} (1) \|
   \varphi \|_{H^{\exp}_Q} . \]
We check that, by symmetry (see equation (\ref{CP2sym}))
\[ \langle i (\vec{c}' - c \overrightarrow{e_2}) . \nabla Q, c
   \partial_{c^{\bot}} Q \rangle = \delta^{\bot} (c \overrightarrow{e_2},
   \vec{c}') \left\langle i \frac{\vec{c}'}{| \vec{c}' |} . \nabla Q, c
   \partial_{c^{\bot}} Q \right\rangle \]
Furthermore, from Lemma \ref{CP20703L222}, we have $L_Q (\partial_{c^{\bot}}
Q) = - i c \frac{\vec{c}^{\prime \bot}}{| \vec{c}' |} . \nabla Q$, therefore,
from Proposition \ref{CP2prop5},
\[ c \left\langle i \frac{\vec{c}^{\prime \bot}}{| \vec{c}' |} . \nabla Q,
   \partial_{c^{\bot}} Q \right\rangle = - B_Q (\partial_{c^{\bot}} Q) = - 2
   \pi + o_{c \rightarrow 0} (1) . \]
We deduce that
\[ \delta^{\bot} (c \overrightarrow{e_2}, \vec{c}') \leqslant \]
\[ K | \langle Q \psi, c L_Q (\partial_{c^{\bot}} Q) \rangle + \langle
   \tmop{NL}_{\tmop{loc}} (\psi) + F (\psi), c \partial_{c^{\bot}} Q \rangle |
   + o^{\lambda, c}_{\| Z - Q_c \|_{H^{\exp}_{Q_c}} \rightarrow 0} (1) \|
   \varphi \|_{H^{\exp}_Q} . \]
As previously, we check that
\[ | \langle \tmop{NL}_{\tmop{loc}} (\psi) + F (\psi), c \partial_{c^{\bot}} Q
   \rangle | \leqslant o^{\lambda, c}_{\| Z - Q_c \|_{H^{\exp}_{Q_c}}
   \rightarrow 0} (1) \| \varphi \|_{H^{\exp}_{Q_{}}} + o^{\lambda, c}_{\| Z -
   Q_c \|_{H^{\exp}_{Q_c}} \rightarrow 0} (1) \| \varphi \|_{\mathcal{C}} \]
and from Lemma \ref{CP20703L222}, we have
\begin{eqnarray*}
  | \langle Q \psi, L_Q (\partial_{c^{\bot}} Q) \rangle | & = & \left|
  \left\langle Q \psi, i \frac{\vec{c}^{\prime \bot}}{| \vec{c}' |} . \nabla Q
  \right\rangle \right|\\
  & \leqslant & \left| \int_{\mathbbm{R}^2} \mathfrak{R}\mathfrak{e} (\psi)
  \mathfrak{I}\mathfrak{m} \left(  \frac{\vec{c}^{\prime \bot}}{| \vec{c}' |}
  . \nabla Q \bar{Q} \right) \right| + \left| \int_{\mathbbm{R}^2}
  \mathfrak{I}\mathfrak{m} (\psi) \mathfrak{R}\mathfrak{e} \left( 
  \frac{\vec{c}^{\prime \bot}}{| \vec{c}' |} . \nabla Q \bar{Q} \right)
  \right|,
\end{eqnarray*}
and with Lemma \ref{CP2L390911}, we deduce that
\[ c | \langle Q \psi, L_Q (\partial_{c^{\bot}} Q) \rangle | \leqslant K c \ln
   \left( \frac{1}{c} \right) \| \varphi \|_{\mathcal{C}} . \]
We conclude that
\[ \delta^{\bot} (c \overrightarrow{e_2}, \vec{c}') \leqslant \left( K c^2 \ln
   \left( \frac{1}{c} \right) + o^{\lambda, c}_{\| Z - Q_c \|_{H^{\exp}_{Q_c}}
   \rightarrow 0} (1) \right) \| \varphi \|_{\mathcal{C}} + o^{\lambda, c}_{\|
   Z - Q_c \|_{H^{\exp}_{Q_c}} \rightarrow 0} (1) \| \varphi
   \|_{H^{\exp}_{Q_c}} . \]

\subsubsection{Estimations on the remaining terms}

Let us show in this subsection that
\begin{eqnarray}
  &  & | \langle i (\vec{c}' - c \overrightarrow{e_2}) . H (\psi), Q (\psi +
  i \gamma) \rangle | + | \langle \tmop{NL}_{\tmop{loc}} (\psi), Q (\psi + i
  \gamma) \rangle | + | \langle F (\psi), Q (\psi + i \gamma) \rangle |
  \nonumber\\
  & \leqslant & \left( o_{c \rightarrow 0} (1) + o^{\lambda, c}_{\| Z - Q_c
  \|_{H^{\exp}_{Q_c}} \rightarrow 0} (1) + K \lambda_0 \right) \| \varphi
  \|_{\mathcal{C}}^2 + o^{\lambda, c}_{\| Z - Q_c \|_{H^{\exp}_{Q_c}}
  \rightarrow 0} (1) \| \varphi \|^2_{H^{\exp}_Q} .  \label{CP24191910}
\end{eqnarray}

\begin{tmindent}
  Step 1.  Proof of $| \langle \tmop{NL}_{\tmop{loc}} (\psi), Q (\psi + i
  \gamma) \rangle | \leqslant o^{\lambda, c}_{\| Z - Q_c \|_{H^{\exp}_{Q_c}}
  \rightarrow 0} (1) \| \varphi \|^2_{H^{\exp}_Q}$.
\end{tmindent}

From Lemma \ref{CP2L442905}, we have
\[ | \langle \tmop{NL}_{\tmop{loc}} (\psi), Q (\psi + i \gamma) \rangle |
   \leqslant K (\| Q \psi \|_{C^1 (\{ \eta \neq 1 \})} + | \gamma |) \|
   \varphi \|^2_{H^1 (\{ \eta \neq 1 \})}, \]
therefore, from Lemmas \ref{CP2L432905}, \ref{CP2L452905} and equation
(\ref{CP24143005}), we deduce
\[ | \langle \tmop{NL}_{\tmop{loc}} (\psi), Q \psi \rangle | \leqslant
   o^{\lambda, c}_{\| Z - Q_c \|_{H^{\exp}_{Q_c}} \rightarrow 0} (1) \|
   \varphi \|^2_{H^{\exp}_Q} . \]

\begin{tmindent}
  Step 2.  Proof of
  \[ | \langle i (\vec{c}' - c \overrightarrow{e_2}) . H (\psi), Q (\psi + i
     \gamma) \rangle | \]
  \[ \leqslant \left( o_{c \rightarrow 0} (1) + o^{\lambda, c}_{\| Z - Q_c
     \|_{H^{\exp}_{Q_c}} \rightarrow 0} (1) \right) \| \varphi
     \|_{\mathcal{C}}^2 + o^{\lambda, c}_{\| Z - Q_c \|_{H^{\exp}_{Q_c}}
     \rightarrow 0} (1) \| \varphi \|^2_{H^{\exp}_Q} . \]
\end{tmindent}

We separate the estimation in two parts. First, we look at $\langle i
(\vec{c}' - c \overrightarrow{e_2}) . H (\psi), Q \psi \rangle$. We recall
that $H (\psi) = \nabla Q + \frac{(1 - \eta) \nabla (Q \psi) + \eta e^{\psi} Q
\nabla \psi}{(1 - \eta) + \eta e^{\psi}}$, and, since $| \vec{c}' - c
\overrightarrow{e_2} | \leqslant o^{\lambda, c}_{\| Z - Q_c
\|_{H^{\exp}_{Q_c}} \rightarrow 0} (1)$ and $1 - \eta$ is compactly supported,
we check easily that
\[ \left| \left\langle i (\vec{c}' - c \overrightarrow{e_2}) . \frac{(1 -
   \eta) \nabla (Q \psi) + \eta e^{\psi} Q \nabla \psi}{(1 - \eta) + \eta
   e^{\psi}}, Q \psi \right\rangle \right| \leqslant \]
\[ o^{\lambda, c}_{\| Z - Q_c \|_{H^{\exp}_{Q_c}} \rightarrow 0} (1) (|
   \langle \eta i Q \nabla \psi, Q \psi \rangle | + K (c) \| \varphi
   \|^2_{H^{\exp}_Q}) . \]
Furthermore, we check that
\[ | \langle \eta i Q \nabla \psi, Q \psi \rangle | \leqslant \left|
   \int_{\mathbbm{R}^2} \mathfrak{R}\mathfrak{e} (\psi)
   \mathfrak{I}\mathfrak{m} (\nabla \psi) | Q |^2 \eta \right| + \left|
   \int_{\mathbbm{R}^2} \mathfrak{I}\mathfrak{m} (\psi)
   \mathfrak{R}\mathfrak{e} (\nabla \psi) | Q |^2 \eta \right|, \]
and by Cauchy-Scwharz, $\left| \int_{\mathbbm{R}^2} \mathfrak{R}\mathfrak{e}
(\psi) \mathfrak{I}\mathfrak{m} (\nabla \psi) | Q |^2 \eta \right| \leqslant K
\| \varphi \|_{\mathcal{C}}^2$. Now, by integration by parts (using Lemma
\ref{CP2L531806}), we have
\begin{eqnarray*}
  \left| \int_{\mathbbm{R}^2} \mathfrak{I}\mathfrak{m} (\psi)
  \mathfrak{R}\mathfrak{e} (\nabla \psi) | Q |^2 \eta \right| & \leqslant &
  \left| \int_{\mathbbm{R}^2} \mathfrak{R}\mathfrak{e} (\psi)
  \mathfrak{I}\mathfrak{m} (\nabla \psi) | Q |^2 \eta \right|\\
  & + & \left| \int_{\mathbbm{R}^2} \mathfrak{I}\mathfrak{m} (\psi)
  \mathfrak{R}\mathfrak{e} (\psi) \nabla (| Q |^2) \eta \right|\\
  & + & \left| \int_{\mathbbm{R}^2} \mathfrak{I}\mathfrak{m} (\psi)
  \mathfrak{R}\mathfrak{e} (\psi) | Q |^2 \nabla \eta \right|,
\end{eqnarray*}
and by Cauchy-Schwarz, we check that
\[ \left| \int_{\mathbbm{R}^2} \mathfrak{I}\mathfrak{m} (\psi)
   \mathfrak{R}\mathfrak{e} (\nabla \psi) | Q |^2 \eta \right| \leqslant K \|
   \varphi \|^2_{H^{\exp}_Q} . \]
We deduce that
\[ \left| \left\langle i (\vec{c}' - c \overrightarrow{e_2}) . \frac{(1 -
   \eta) \nabla (Q \psi) + \eta e^{\psi} Q \nabla \psi}{(1 - \eta) + \eta
   e^{\psi}}, Q \psi \right\rangle \right| \leqslant o^{\lambda, c}_{\| Z -
   Q_c \|_{H^{\exp}_{Q_c}} \rightarrow 0} (1) \| \varphi \|^2_{H^{\exp}_Q} .
\]
Finally, we write
\[ | \langle i (\vec{c}' - c \overrightarrow{e_2}) . \nabla Q, Q \psi \rangle
   | \leqslant \delta^{| . |} (c \overrightarrow{e_2}, \vec{c}') \left|
   \left\langle i \frac{\vec{c}'}{| \vec{c}' |} . \nabla Q, Q \psi
   \right\rangle \right| + \delta^{\bot} (c \overrightarrow{e_2}, \vec{c}')
   \left| \left\langle i \frac{\vec{c}^{\prime \bot}}{| \vec{c}' |} . \nabla
   Q, Q \psi \right\rangle \right| . \]
With Lemma \ref{CP2L390911}, we check that
\[ \left| \left\langle i \frac{\vec{c}'}{| \vec{c}' |} . \nabla Q, Q \psi
   \right\rangle \right| + \left| \left\langle i \frac{\vec{c}^{\prime
   \bot}}{| \vec{c}' |} . \nabla Q, Q \psi \right\rangle \right| \leqslant K
   \ln \left( \frac{1}{c} \right) \| \varphi \|_{\mathcal{C}} . \]
With (\ref{CP24182910}) and (\ref{CP24192910}), we deduce that
\begin{eqnarray*}
  | \langle i (\vec{c}' - c \overrightarrow{e_2}) . \nabla Q, Q \psi \rangle |
  & \leqslant & \left( K c \ln^2 \left( \frac{1}{c} \right) + o^{\lambda,
  c}_{\| Z - Q_c \|_{H^{\exp}_{Q_c}} \rightarrow 0} (1) \right) \| \varphi
  \|_{\mathcal{C}}^2 + o^{\lambda, c}_{\| Z - Q_c \|_{H^{\exp}_{Q_c}}
  \rightarrow 0} (1) \| \varphi \|^2_{H^{\exp}_Q}\\
  & \leqslant & \left( o_{c \rightarrow 0} (1) + o^{\lambda, c}_{\| Z - Q_c
  \|_{H^{\exp}_{Q_c}} \rightarrow 0} (1) \right) \| \varphi \|_{\mathcal{C}}^2
  + o^{\lambda, c}_{\| Z - Q_c \|_{H^{\exp}_{Q_c}} \rightarrow 0} (1) \|
  \varphi \|^2_{H^{\exp}_Q} .
\end{eqnarray*}
Now, we look at $\langle i (\vec{c}' - c \overrightarrow{e_2}) . H (\psi), Q i
\gamma \rangle$. We check that
\[ \langle i \nabla Q, Q i \gamma \rangle = \gamma \int_{\mathbbm{R}^2}
   \mathfrak{R}\mathfrak{e} (\nabla Q \bar{Q}) = \frac{\gamma}{2}
   \int_{\mathbbm{R}^2} \nabla (| Q |^2 - 1) = 0, \]
thus
\[ \langle i (\vec{c}' - c \overrightarrow{e_2}) . H (\psi), Q i \gamma
   \rangle = \left\langle i (\vec{c}' - c \overrightarrow{e_2}) . \frac{(1 -
   \eta) \nabla (Q \psi) + \eta e^{\psi} Q \nabla \psi}{(1 - \eta) + \eta
   e^{\psi}}, Q i \gamma \right\rangle . \]
In the area $\{ \eta \neq 0 \}$, since $| \gamma | = o^{\lambda, c}_{\| Z -
Q_c \|_{H^{\exp}_{Q_c}} \rightarrow 0} (1)$ by Lemma \ref{CP2L452905}, since
\[ | \vec{c}' - c \overrightarrow{e_2} | \leqslant K \left( c \ln \left(
   \frac{1}{c} \right) + o^{\lambda, c}_{\| Z - Q_c \|_{H^{\exp}_{Q_c}}
   \rightarrow 0} (1) \right) \| \varphi \|_{\mathcal{C}} + o^{\lambda, c}_{\|
   Z - Q_c \|_{H^{\exp}_{Q_c}} \rightarrow 0} (1) \| \varphi \|_{H^{\exp}_Q}
\]
by estimates (\ref{CP24182910}) and (\ref{CP24192910}), we check that
\[ \int_{\{ \eta \neq 0 \}} \mathfrak{R}\mathfrak{e} \left( i (\vec{c}' - c
   \overrightarrow{e_2}) . \frac{(1 - \eta) \nabla (Q \psi) + \eta e^{\psi} Q
   \nabla \psi}{(1 - \eta) + \eta e^{\psi}} \overline{Q i \gamma} \right)
   \leqslant o^{\lambda, c}_{\| Z - Q_c \|_{H^{\exp}_{Q_c}} \rightarrow 0} (1)
   \| \varphi \|^2_{H^{\exp}_Q}, \]
and therefore (with Lemma \ref{CP2L531806} that justifies the integrability)
\[ | \langle i (\vec{c}' - c \overrightarrow{e_2}) . H (\psi), Q i \gamma
   \rangle | \leqslant \left| \gamma (\vec{c}' - c \overrightarrow{e_2}) .
   \int_{\mathbbm{R}^2} \eta | Q |^2 \mathfrak{R}\mathfrak{e} (\nabla \psi)
   \right| + o^{\lambda, c}_{\| Z - Q_c \|_{H^{\exp}_{Q_c}} \rightarrow 0} (1)
   \| \varphi \|^2_{H^{\exp}_Q} . \]
By integration by parts (since $| \mathfrak{R}\mathfrak{e} (\psi) | \leqslant
\frac{K \left( \lambda, c, \| Z - Q_c \|_{H^{\exp}_{Q_c}}, \varepsilon_0, Z
\right)}{(1 + r)^2}$ and $| \mathfrak{R}\mathfrak{e} (\nabla \psi) | \leqslant
\frac{K \left( \lambda, c, \| Z - Q_c \|_{H^{\exp}_{Q_c}}, \varepsilon_0, Z
\right)}{(1 + r)^3}$ by Lemma \ref{CP2L531806}) and Cauchy-Schwarz,
\begin{eqnarray*}
  \left| \int_{\mathbbm{R}^2} \eta | Q |^2 \mathfrak{R}\mathfrak{e} (\nabla
  \psi) \right| & \leqslant & \left| \int_{\mathbbm{R}^2} \nabla \eta | Q |^2
  \mathfrak{R}\mathfrak{e} (\psi) \right| + \left| \int_{\mathbbm{R}^2} \eta
  \nabla (| Q |^2) \mathfrak{R}\mathfrak{e} (\psi) \right|\\
  & \leqslant & K (c) \| \varphi \|_{H^{\exp}_Q} .
\end{eqnarray*}
Since $| \gamma | = o^{\lambda, c}_{\| Z - Q_c \|_{H^{\exp}_{Q_c}} \rightarrow
0} (1)$ by Lemma \ref{CP2L452905} and $| \vec{c}' - c \overrightarrow{e_2} |
\leqslant \left( K (c) + o^{\lambda, c}_{\| Z - Q_c \|_{H^{\exp}_{Q_c}}
\rightarrow 0} (1) \right) \| \varphi \|_{H^{\exp}_Q}$ by (\ref{CP24182910}),
(\ref{CP24192910}) and Lemma \ref{CP2L3133105}, we conclude that
\[ | \langle i (\vec{c}' - c \overrightarrow{e_2}) . H (\psi), Q i \gamma
   \rangle | \leqslant o^{\lambda, c}_{\| Z - Q_c \|_{H^{\exp}_{Q_c}}
   \rightarrow 0} (1) \| \varphi \|^2_{H^{\exp}_Q} . \]

\begin{tmindent}
  Step 3.  Proof of $| \langle F (\psi), Q (\psi + i \gamma) \rangle |
  \leqslant \left( o^{\lambda, c}_{\| Z - Q_c \|_{H^{\exp}_{Q_c}} \rightarrow
  0} (1) + K \lambda_0 \right) \| \varphi \|^2_{\mathcal{C}}$.
\end{tmindent}

We recall
\[ F (\psi) = Q \eta (- \nabla \psi . \nabla \psi + | Q |^2 S (\psi)), \]
\[ S (\psi) = e^{2\mathfrak{R}\mathfrak{e} (\psi)} - 1 -
   2\mathfrak{R}\mathfrak{e} (\psi) . \]
First, we look at $\langle F (\psi), Q \psi \rangle$. We have
\[ | \langle F (\psi), Q \psi \rangle | \leqslant | \langle Q (1 - \eta)
   \nabla \psi . \nabla \psi, Q \psi \rangle | + | \langle Q (1 - \eta) | Q
   |^2 S (\psi), Q \psi \rangle | . \]
We check that $\| \varphi \|_{L^{\infty} (\mathbbm{R}^2)} \leqslant K \| \psi
\|_{L^{\infty} (\mathbbm{R}^2 \backslash B (0, \lambda))} + K \| \varphi
\|_{L^{\infty} (B (0, \lambda))} \leqslant K \lambda_0 + o^{\lambda, c}_{\| Z
- Q_c \|_{H^{\exp}_{Q_c}} \rightarrow 0} (1)$
\[ | \langle Q \eta \nabla \psi . \nabla \psi, Q \psi \rangle | \leqslant \|
   \varphi \|_{L^{\infty} (\mathbbm{R}^2)} \int_{\mathbbm{R}^2} \eta | \nabla
   \psi |^2 \leqslant \left( K \lambda_0 +_{} o^{\lambda, c}_{\| Z - Q_c
   \|_{H^{\exp}_{Q_c}} \rightarrow 0} (1) \right) \| \varphi
   \|^2_{\mathcal{C}} . \]
Finally, since $\| \varphi \|_{L^{\infty} (\mathbbm{R}^2)} \leqslant K$ a
uniform constant for $c$ and $\| Z - Q_c \|_{H^{\exp}_{Q_c}}$ small enough,
\[ | \langle Q \eta | Q |^2 S (\psi), Q \psi \rangle | \leqslant \| \varphi
   \|_{L^{\infty} (\mathbbm{R}^2)} \int_{\mathbbm{R}^2} \eta
   \mathfrak{R}\mathfrak{e}^2 (\psi) \leqslant \left( K \lambda_0 +_{}
   o^{\lambda, c}_{\| Z - Q_c \|_{H^{\exp}_{Q_c}} \rightarrow 0} (1) \right)
   \| \varphi \|^2_{\mathcal{C}} . \]
Now, we compute
\begin{eqnarray*}
  | \langle F (\psi), Q i \gamma \rangle | & \leqslant & | \gamma | \left|
  \int_{\mathbbm{R}^2} -\mathfrak{R}\mathfrak{e} (\eta i \nabla \psi . \nabla
  \psi) | Q |^2 + \eta | Q |^4 \mathfrak{R}\mathfrak{e} (S (\psi) i) \right|,
\end{eqnarray*}
and since $S (\psi)$ is real-valued, we check that, since $| \gamma | =
o^{\lambda, c}_{\| Z - Q_c \|_{H^{\exp}_{Q_c}} \rightarrow 0} (1)$ by Lemma
\ref{CP2L452905},
\[ | \langle F (\psi), Q i \gamma \rangle | \leqslant | \gamma |
   \int_{\mathbbm{R}^2} \eta | \nabla \psi |^2 | Q |^2 \leqslant o^{\lambda,
   c}_{\| Z - Q_c \|_{H^{\exp}_{Q_c}} \rightarrow 0} (1) \| \varphi
   \|^2_{\mathcal{C}} . \]

\subsubsection{Conclusion}

Combining the steps 1 to 3 and (\ref{CP229052}) in (\ref{CP22905}), we deduce
that, taking $c$ small enough, and then $\| Z - Q_c \|_{H^{\exp}_{Q_c}}$ small
enough (depending on $c$ and $\lambda$), we have
\begin{eqnarray*}
  0 & \geqslant & K \| \varphi \|_{\mathcal{C}}^2 + K (c) \| \varphi
  \|^2_{H^{\exp}_Q}\\
  & - & \left( o_{c \rightarrow 0} (1) + K \mu_0 +_{} o^{\lambda, c}_{\| Z -
  Q_c \|_{H^{\exp}_{Q_c}} \rightarrow 0} (1) \right) \| \varphi
  \|_{\mathcal{C}}^2 - o^{\lambda, c}_{\| Z - Q_c \|_{H^{\exp}_{Q_c}}
  \rightarrow 0} (1) \| \varphi \|^2_{H^{\exp}_Q},
\end{eqnarray*}
hence, if $\mu_0$ is taken small enough (independently of any other
parameters) then $c$ small enough and $\| Z - Q_c \|_{H_{Q_c}^{\exp}}$ small
enough (depending on $\lambda$ and $c$),
\[ K (c) \| \varphi \|_{H^{\exp}_{Q_c}}^2 + K \| \varphi \|^2_{\mathcal{C}}
   \leqslant 0. \]
We deduce that $\varphi = 0$, thus $Z = Q$. Furthermore, from
(\ref{CP24182910}) and (\ref{CP24192910}) we deduce that $\vec{c}' = c
\overrightarrow{e_2}$, and since $Z \rightarrow 1$ at infinity, we also have
$\gamma = 0$ (or else $\| Z - Q_c \|_{H^{\exp}_{Q_c}} = + \infty$). This
concludes the proof of Theorem \ref{CP2th16}.

\end{document}